\documentclass[10pt]{amsart}

\usepackage[utf8]{inputenc}
\usepackage{amssymb,amsmath}
\usepackage{amscd}

\usepackage{mathtools}

\usepackage[all]{xy}
\usepackage[onehalfspacing]{setspace}
\usepackage{graphicx}

\usepackage{url}
\usepackage[backref=page]{hyperref}

\usepackage{enumitem}

\usepackage[margin=2cm]{geometry}

\usepackage{tikz-cd}
\usepackage{array}
\usepackage{verbatim}
\usepackage{xspace}
\usepackage{xcolor} 
\usepackage{caption}

\usepackage{dsfont}

\usepackage{pdfpages}
\usepackage{lipsum,atbegshi,etoolbox}

\newcommand{\QQ}{\ensuremath{\mathbb{Q}}\xspace}

\newcommand{\ZZ}{\ensuremath{\mathbb{Z}}\xspace}
\newcommand{\Zp}{\ensuremath{\mathbb{Z}_{(p)}}\xspace}
\newcommand{\Z}[1]{\ensuremath{\mathbb{Z}_{(#1)}}\xspace}
\newcommand{\NN}{\ensuremath{\mathbb{N}}\xspace}
\newcommand{\LL}{\ensuremath{\mathbb{L}}\xspace}

\newcommand{\OO}{\ensuremath{\mathcal{O}}\xspace}

\newcommand{\rarr}{\rightarrow}

\newcommand{\xrarr}[1]{\xrightarrow{#1}}

\newcommand{\dasharr}{\dashrightarrow}

\newcommand{\id}{\ensuremath{\mbox{id}}}

\newcommand{\ot}{\otimes}

\newcommand{\Sm}{\mathcal{S}\mathrm m}
\newcommand{\SmProj}{\mathcal{S}\mathrm m\mathcal{P}\mathrm{roj}}

\newcommand{\Spec}{\mathrm{Spec\ }}

\newcommand{\End}{\ensuremath{\mathrm{End}}}

\newcommand{\F}[1]{\mathbb{F}_{#1}}

\newcommand{\un}{\ensuremath{\mathds{1}}}

\newcommand{\PM}{\mathrm{CM}}
\newcommand{\CM}{\mathrm{CM}}
\newcommand{\HH}{\mathrm{H}}
\newcommand{\Pic}{\mathrm{Pic}}

\newcommand{\PP}{\ensuremath{\mathbb{P}}}

\newcommand{\et}{\ensuremath{\mathrm{\acute{e}t}}}

\newcommand{\CH}{\ensuremath{\mathrm{CH}}}
\newcommand{\Ch}{\ensuremath{\mathrm{Ch}}}

\newcommand{\pra}[1]{\mathrm{pr}_{#1}}

\newcommand{\Br}{\ensuremath{\mathrm{Br}}}

\newcommand{\kmtech}{\mathrm{K}^{\mathrm{M}}}
\newcommand{\kma}[2]{\ensuremath{\kmtech_{#2}\!\left(#1\right)\!/2}} 
\newcommand{\kmk}[1]{\ensuremath{\kma{k}{#1}}} 
\newcommand{\km}{\ensuremath{\kmk{n+1}}}

\newcommand{\Hom}{\ensuremath{\mathrm{Hom}}}
\newcommand{\Ker}{\ensuremath{\mathrm{Ker}}}
\newcommand{\Gal}[1]{\ensuremath{{\mathrm{Gal}(#1)}}}
\newcommand{\Inv}[3]{\ensuremath{{\mathrm{Inv}^{#2}(#1,\,#3)}}}

\newcommand{\Hnpo}[1]{\ensuremath{\HH^{n+1}\left(#1,\,\mathbb Z/2\right)}} 
\newcommand{\Wnr}[1]{\ensuremath{W_{\mathrm{nr}}\left(#1/k\right)}} 

\newcommand{\ibar}{\ensuremath{\overline{i}}}   

\newcommand{\bundle}[1]{\ensuremath{{#1}}}   
\newcommand{\TB}[1]{\ensuremath{\bundle{T}_{#1}}}   
\newcommand{\LB}{\ensuremath{\bundle L}}   

\newcommand{\hyp}{\ensuremath{\mathbb{H}}}   
\newcommand{\sh}[1]{\ensuremath{(#1)}}   
\newcommand{\iw}{\ensuremath{i_\mathrm{W}}}   
\newcommand{\an}{\ensuremath{\mathrm{an}}}   

\newcommand{\A}{\ensuremath{A}}   
\newcommand{\B}{\ensuremath{B}}
\newcommand{\Kn}{\ensuremath{\mathrm{K}(n)}}
\newcommand{\KK}{\ensuremath{\mathrm{K}_0}}
\newcommand{\K}[1]{\ensuremath{\mathrm{K}(#1)}}
\newcommand{\CKn}{\ensuremath{\mathrm{CK}(n)}}

\newcommand{\Cob}{\ensuremath{\Omega}}
\newcommand{\Cobp}{\ensuremath{\Omega_{(p)}}}
\newcommand{\BP}{\ensuremath{\mathrm{BP}}}

\newcommand{\num}{\ensuremath{\mathrm{num}}}
\newcommand{\Knum}{\ensuremath{\Kn_{\num}}}
\newcommand{\SHk}{\mathrm{SH}(k)}

\newcommand{\iso}{\ensuremath{\mathrm{iso}}}
\renewcommand{\int}{\ensuremath{\mathrm{int}}}

\newcommand{\At}[1]{\ensuremath{{\A}\!\left(#1\right)}}   
\newcommand{\Bt}[1]{\ensuremath{{\B}\!\left(#1\right)}}
\newcommand{\CHt}[1]{\ensuremath{{\CH}\!\left(#1\right)}}
\newcommand{\Cht}[1]{\ensuremath{{\Ch}\!\left(#1\right)}}
\newcommand{\Knt}[1]{\ensuremath{{\Kn}\!\left(#1\right)}}

\newcommand{\CKnt}[1]{\ensuremath{{\CKn}\!\left(#1\right)}}

\newcommand{\BPt}[1]{\ensuremath{{\BP}\!\left(#1\right)}}

\newcommand{\Af}{\ensuremath{\A}}   

\newcommand{\Afnum}{\ensuremath{\A_{\num}}}
\newcommand{\Bf}{\ensuremath{{\B}}}
\newcommand{\Chf}{\ensuremath{{\Ch}}}
\newcommand{\Cobf}{\ensuremath{{\Cob}}}
\newcommand{\Cobpf}{\ensuremath{\Cobp}}
\newcommand{\BPf}{\ensuremath{{\BP}}}
\newcommand{\Knf}{\ensuremath{{\Kn}}}

\newcommand{\Kgr}{\ensuremath{{\mathrm{K}_0}}}
\newcommand{\gr}{\ensuremath{{\mathrm{gr}}}}

\newcommand{\mottech}{\ensuremath{\mathrm{M}}}   
\newcommand{\Mot}[2]{\ensuremath{\mottech_{#1}\!\left(#2\right)}}   
\newcommand{\MA}[1]{\ensuremath{\Mot{\A}{#1}}}
\newcommand{\MCH}[1]{\ensuremath{\Mot{\CH}{#1}}}
\newcommand{\MCh}[1]{\ensuremath{\Mot{\Ch}{#1}}}
\newcommand{\MKn}[1]{\ensuremath{\Mot{\Kn}{#1}}}

\newcommand{\MCKn}[1]{\ensuremath{\Mot{\CKn}{#1}}}

\newcommand{\Mker}[1]{\ensuremath{\widetilde{\,\mottech}_{\Kn}\!\left(#1\right)}} 
\newcommand{\Mkertwo}[1]{\ensuremath{\widetilde{\,\mottech}_{\K{2}}\!\left(#1\right)}} 

\newcommand{\oCh}{\ensuremath{\overline{\Ch}}}
\newcommand{\oKn}{\ensuremath{\overline\Kn}}
\newcommand{\oCKn}{\ensuremath{\overline\CKn}}

\newcommand{\oBP}{\ensuremath{\overline\BP}}

\newcommand{\oCht}[1]{\ensuremath{{\oCh}\left(#1\right)}}
\newcommand{\oKnt}[1]{\ensuremath{{\oKn}\left(#1\right)}}

\newcommand{\oBPt}[1]{\ensuremath{{\oBP}\left(#1\right)}}

\newcommand{\pttech}{\ensuremath{k}}
\newcommand{\Apt}{\ensuremath{\At{\pttech}}}   
\newcommand{\Bpt}{\ensuremath{\Bt{\pttech}}}

\newcommand{\Knpt}{\ensuremath{\Knt{\pttech}}}

\newcommand{\BPpt}{\ensuremath{\BPt{\pttech}}}

\newcommand{\BPptaug}{\ensuremath{\BP^{<0}(\pttech)}}

\newcommand{\unA}{\ensuremath{\un}}   
\newcommand{\unCH}{\ensuremath{\un_{\CH}}}   
\newcommand{\unCh}{\ensuremath{\un_{\Ch}}}   
\newcommand{\unKn}{\ensuremath{\un_{\Kn}}}   

\newcommand{\Enda}[1]{\ensuremath{\End\!\left(#1\right)}}   

\newcommand{\FGL}[1]{\ensuremath{F_{#1}}}   

\newcommand{\Laz}{\ensuremath{\mathbb L}}   
\newcommand{\Lazp}{\ensuremath{\mathbb L_{(p)}}}

\newcommand{\rk}{\ensuremath{\mathrm{rk\ }}}

\newcommand{\St}[1]{\ensuremath{\mathrm{St}^{#1}}}   
\newcommand{\StCh}{\ensuremath{\St{\Ch}}} 
\newcommand{\StCob}{\ensuremath{\St{\Cob}}}
\newcommand{\StBP}{\ensuremath{\St{\BP}}}
\newcommand{\Ph}[1]{\ensuremath{\Phi^{#1}}}   
\newcommand{\PhCob}{\ensuremath{\Ph{\Cob}}}
\newcommand{\PhBP}{\ensuremath{\Ph{\BP}}}

\newcommand{\Smk}{\ensuremath{\mathcal S\mathrm m_k}}   
\newcommand{\codim}[1]{\ensuremath{\mathrm{codim}\!\left(#1\right)}}   

\newcommand{\Lamn}[1]{\ensuremath{\widetilde{\Lambda\,}\!\left(#1\right)}}   

\newcommand{\rost}{\ensuremath{\mathfrak r}}

\newtheorem{lm}{Lemma}

\newtheorem{Def}{Definition}[section]

\newtheorem{Rk}[Def]{Remark}
\newtheorem{Ex}[Def]{Example}

\newtheorem{Th}[Def]{Theorem}
\newtheorem{Prop}[Def]{Proposition}
\newtheorem{Constr}[Def]{Construction}
\newtheorem{Cr}[Def]{Corollary}
\newtheorem{Lm}[Def]{Lemma}
\newtheorem{Qu}[Def]{Question}
\newtheorem{Conj}[Def]{Conjecture}

\newtheorem{LmA}{Lemma A.\ignorespaces}
\newtheorem{PropA}{Proposition A.\ignorespaces}
\newtheorem{LmB}{Lemma B.\ignorespaces}

\newtheorem*{Th-intro}{Theorem}
\newtheorem*{Cr-intro}{Corollary}
\newtheorem*{Def-intro}{Definition}
\newtheorem*{Prop-intro}{Proposition}
\newtheorem*{Conj-intro}{Conjecture}

\newcommand{\KnInt}{\ensuremath{\mathrm{K(n)}^{\mathrm{int}}}}

\title{Invertible Morava motives in quadrics}
\author{Andrei Lavrenov, Pavel Sechin}

\begin{document}

\date{}

\begin{abstract}
We associate to any element in the Milnor K-theory of a field $k$ modulo 2 an invertible Morava K-theory motive over $k$.
Specifically, for
$\alpha$ in $\km$ 
we construct an invertible $\Kn$-motive $L_\alpha$ 
in a way that is natural in the base field and additive in $\alpha$.
This can be seen as categorification of $\km$ in motives.

The motives $L_\alpha$ are constructed as direct summands of the $\Kn$-motives of quadrics, 
and we develop the necessary framework for the study of the latter.
We show that passing to the field of functions of quadrics of dimension greater than or equal to $2^{n+1}-1$ 
does not lose any information about the structure of $\Kn$-motives.
This is based on the study of  ``decomposition of the diagonal'' in Morava K-theory of quadrics. 

For quadrics of dimension less than $2^{n+1}-1$,
we show that their Chow motives can be ``reconstructed'' from their $\Kn$-motives,
although the latter appear structurally simpler. 
Our proof of this result relies on the use of the unstable symmetric operations of Vishik on algebraic cobordism.

The occurrence of the motive $L_\alpha$ as a direct summand of the $\Kn$-motive of $X$
can be seen as evidence that $\alpha$ is a cohomological invariant of $X$. 
We study this occurrence for quadrics and relate it 
to Kahn's Descent conjecture.
\end{abstract}

\maketitle

In the 1960s Grothendieck has initiated the study of algebraic varieties
via the theory of motives. 
Motives linearize algebraic geometry
and make it possible to decompose smooth projective
varieties into smaller pieces that contain information about cohomology.

Grothendieck's approach to motives
requires a choice of an algebro-geometric cohomology theory.
Classically, cohomology theories with rational coefficients are of central interest.
Yet for some questions of algebraic geometry the motives coming from theories 
with $\QQ$-coefficients are not so useful.
For example, the motives with $\QQ$-coefficients of a conic
and of $\mathbb{P}^1$ are isomorphic, even if the conic lacks a rational point;
thus, all arithmetic information is lost.

One way out of this predicament is to consider cohomology theories 
with coefficients in fields of positive characteristic. 
One could, similar to $\QQ$-coefficients, consider different versions of 
ordinary cohomology theory with $\mathbb{F}_p$-coefficients. 
However, as topology tells us, the situation in the mod-$p$-world is much more complicated
than rationally: ordinary cohomology theories are not enough. 
What is missing are precisely
Morava K-theories,
which play the central role in the chromatic homotopy theory.

In this paper, we study Grothendieck motives
defined for the algebraic Morava K-theory.
By examining quadrics,
 we demonstrate how these motives are well-suited for 
 the study of properties of algebraic varieties that become trivial over the algebraic closure of the base field.
 In particular, they allow us to associate elements in the Galois cohomology of the base field
to smooth projective varieties.

\section*{Introduction}
\addtocontents{toc}{\protect\setcounter{tocdepth}{1}}  

Let $k$ be a field of characteristic $0$.
One of the standard tools for the study of quadratic forms over $k$ 
is the investigation of algebraic cycles on the corresponding quadrics.
To relate cycles on different quadrics, one considers
the category of Chow motives $\CM_{\CH}(k)$ with integral or $\F{2}$-coefficients, 
where Chow groups, i.e.\  algebraic cycles modulo rational equivalence,
are representable. 

Quite a few major results about quadratic forms have been obtained using this tool.
Among these are the computations of the essential dimension of a quadric  \cite{KarpMer}, 
 the possible values of the Witt index \cite{KarpWitt},
and the construction of fields with new values of the $u$-invariant \cite{Vish-u-inv}.

In this paper we study quadrics in the category of motives $\CM_{\Kn}(k)$, Morava motives for short,
where the role of Chow groups is replaced by algebraic Morava K-theory $\Kn$ at the prime $2$.
The latter was introduced by Voevodsky \cite{Voe} 
as a crucial (then conjectural) aid in proving the Milnor conjecture.

The results of this paper 
yield a categorification of Milnor K-theory modulo $2$, 
describe a new approach to cohomological invariants
as well as 
provide a new tool for the study of quadratic forms and of algebraic cycles on quadrics.

\subsection*{Overview of the results.}

One of the central discoveries of this paper 
are the Morava motives that act as splitting objects for elements in the Milnor K-theory modulo 2.
Let $\Pic(-)$ denote
the group of isomorphism classes of invertible objects in a tensor additive category. 

\begin{Th-intro}[{Theorem~\ref{th:nat_transform_milnor_picard}}]
Let $k$ be a field of characteristic 0.

There exists a unique injective natural transformation 
$$ \kma{-}{n+1} \xhookrightarrow{\quad} \Pic\!\left(\CM_{\Kn}(-)\right) $$

between functors of abelian groups on the category of field extensions of $k$.
\end{Th-intro}

Thus, Morava motives provide a framework 
in which smooth projective varieties and elements of Milnor K-theory modulo $2$ 
appear on the same footing. This result is also  a first step
towards the motivic categorification of Milnor K-theory,
i.e.\ the study of elements of these groups as {\sl objects} in a 
category 
where morphisms have algebro-geometric meaning.

The connection between $\km$ and quadratic forms is via the celebrated Milnor conjectures 
that were proven by Voevodsky~\cite{Voe_Z2} and Orlov--Vishik--Voevodsky~\cite{OVV}.
Recall that anisotropic quadratic forms over $k$ are classified by the Witt ring $W(k)$.
There is a filtration on the Witt ring by the powers of the fundamental ideal $I(k)\subset W(k)$,
and the Milnor conjectures assert that the canonical maps $\km\rarr I^{n+1}(k)/I^{n+2}(k)$ 
and  $\km\rarr \HH^{n+1}_{\et}(k,\ZZ/2)$ are isomorphisms.
Thus, the theorem above can be formulated with any of these groups instead of the Milnor K-theory.

To construct the motive $L_\alpha$ associated to $\alpha\in \km$ 
we lift $\alpha$ to a quadratic form $q\in I^{n+1}(k)$,
and investigate the Morava motive of the corresponding quadric $Q$.
To study decompositions of $\Kn$-motives we develop tools 
that allow to relate $\Kn$-motives over $k$ and over some field extensions of $k$.

To every quadratic form $q$
one associates a sequence of field extensions $\{K_j\}_{0\le j\le h}$,
the generic splitting tower of $q$, 
which captures information about the anisotropic part of $q_F$ over all field extensions $F$ of $k$.
If we look at the objects of interest associated to this tower,
i.e.\ the anisotropic part of $q_{K_j}$ or the corresponding Chow motive, 
then these are, in general, not defined over $k$.
However, 
working with Morava motives avoids this issue:
we show that the information contained in the $\Kn$-motive of $Q$ {\sl over the base field $k$} 
is precisely the information contained in the Chow motive of $Q_{K_j}$ 
over a specific field extension $K_j$ in the generic splitting tower.
In other words, Morava motives of quadrics allow us to
look into the depth of the generic splitting tower without changing the base field.

In particular, if 
$q\in I^{n+1}(k)$,
then over the specific field $K_j$ in the splitting tower of $q$, for $j$ depending on $q$,
the anisotropic part of $q_{K_j}$ becomes a  $(n+1)$-Pfister form.
By using the results about the motives of Pfister quadrics,
we can then show that the $\Kn$-motive of $Q$ decomposes
into invertible summands, and thus we can define $L_{[q]}$
where $[q]\in I^{n+1}(k)/I^{n+2}(k)$.

However, the motives $L_\alpha$ should not be seen as something related only to quadrics.
We posit to view the occurrence of the motive $L_\alpha$ as a direct summand of the Morava motive $\MKn X$
as $\alpha$ being a cohomological invariant of $X$.
Our expectation is that in this way one can associate a cohomological invariant 
to every 
Chow motive that is rationally split into Tate motives and that is split over $\overline{k}$.
In this paper we describe how this works out for quadrics,
relating the expectation to Kahn's Descent conjectures 
and proving it unconditionally for $n=1,2,3$.

One of the new features of the Morava motives 
is that they are adjusted to work with splitting fields. 
To an element $\alpha$ in $\km$
one associates a class of ``$\Kn$-universal'' splitting fields $k(\alpha)$.
Even though the motive $L_\alpha$ trivializes over $k(\alpha)$,
we show that we still have 
 control of motivic decompositions and motivic isomorphisms over $k(\alpha)$.
 For example, we show that $L_\alpha$ is the unique non-trivial motive 
 that becomes isomorphic to the unit motive over $k(\alpha)$.

In the remainder of the introduction we describe our results in more detail, 
simplifying some statements and omitting historical remarks.
The latter can be found in the extensive introductions to the relevant sections of the article.

\subsection*{Recollection on algebraic Morava K-theory.}

Algebraic Morava K-theory $\Kn$ at a prime $p$ is a ring-valued oriented cohomology theory
on smooth varieties over $k$, constructed as  
a free theory by Levine--Morel~\cite{LevMor}. 
In other words, $\Kn$ is the universal oriented theory with the given formal group law of height $n$
over the coefficient ring $\mathbb{F}_p[v_n, v_n^{-1}]$.
In this paper we are  working with $\Kn$ at the prime $2$,
unless explicitly stated otherwise.

Our motivation in studying $\Kn$ is two-fold.
On the one hand, Morava K-theories  interpolate
 between algebraic K-theory $\Kgr$ modulo $p$, i.e.\ $\mathrm{K}(1)$,
and Chow groups modulo $p$, i.e.\ $\mathrm{K}(\infty)$.
 This has been made precise in~\cite{Sech} 
where it was shown that the associated graded ring $\gr^{i\,}\Kn$ of the topological filtration on $\Kn$
is isomorphic to $\CH^i\otimes \F{p}[v_n, v_n^{-1}]$ for $i\le p^n$.

On the other hand, it was first observed in \cite{SechSem}
 by Geldhauser and the second author that $\Kn$-motives
tend to be {\sl simpler} than Chow motives, at least for some projective homogeneous varieties.
For quadrics in particular, it was shown that 
if a quadratic form $q$ lies in $I^{n+2}(k)$, 
then the $\Kn$-motive of $Q$ is the same as if $q$ were hyperbolic. 
This has led to the speculation about the relation between cohomological invariants and the behaviour of Morava motives
under the name of ``guiding principle'', to which we return later.

Moreover, numerical versions of Morava K-theories $\Kn$
appear in the study of the isotropic motivic world of Vishik~\cite{VishIso, VishIso2},
especially in constructing points in the Balmer spectrum of $\mathrm{SH}(k)$ \cite{DuVish}. 
For a flexible field $k$ there are ``Morava weight cohomology'' functors on compact objects of $\SHk$
that take values in the category of numerical $\Kn$-motives $\CM_{\Knum}(k)$ \cite[Def.~6.3, Rem.~6.4]{DuVish}.
The invertible motives $L_\alpha$ that we construct can also be seen as numerical $\Kn$-motives,
and thus might appear in the Morava weight cohomology of an arbitrary smooth variety over $k$, not necessarily projective.

\subsection*{Universally surjective field extensions and motivic decompositions.}

The first result of this paper allows to reduce the study of $\Kn$-motives of arbitrary quadrics
to the case of quadrics of  dimensions less than $2^{n+1}-1$.
This is one of the features that makes Morava motives easier to work with than Chow motives.

In order to achieve this, we investigate the following more general
property. For an oriented theory $\A$,
we say that the field extension $K/k$ 
is $\A$-universally surjective, resp. bijective,
if the pullback morphism 
$$\A(Y)\rarr \A ({Y_{K}})$$
 is surjective, resp. bijective, for every smooth variety $Y$ over $k$.
The main case of interest is when $K=k(X)$, for $X$ an irreducible $k$-smooth projective variety.

For the case when $\A$ is Chow theory,
surjectivity and bijectivity above are equivalent,
and $k(X)/k$ is $\CH$-universally bijective  if and only if
$X$ is $\CH_0$-universally trivial (see~Corollary~\ref{cr:CH0_univ}).
Such field extensions are $\A$-universally bijective for any oriented $\A$.
However, in our cases of interest, e.g.\ for quadrics, $k(X)/k$ can be $\A$-universally surjective for some $\A$
even if $X$ has no $0$-cycle of degree $1$ and therefore cannot be $\CH_0$-universally trivial.

We develop criteria to characterize both $\A$-universally surjective and bijective
field extensions $k(X)/k$ using  decomposition of the diagonal techniques (Lemma~\ref{lm:A}, Proposition~\ref{prop:univ_bij}). 
These results are based on the unpublished work of Shinder joint with the second author.

For the purposes of this paper the main application of these results 
concerns $\oKn$, a version of $\Kn$ of ``rational elements'',
i.e.\ a quotient of $\Kn$ obtained by killing the elements that vanish over an algebraic closure of the base field.

\begin{Prop-intro}[{Example~\ref{ex:overline-Kn-univ-surj-quad}}] 
Let $Q$ be a quadric of dimension at least $2^{n+1}-1$.

Then $k(Q)/k$ is $\oKn$-universally bijective.
\end{Prop-intro}

As a tool for the study of motives, the result above can be applied as follows.

\begin{Prop-intro}[{Example~\ref{ex:k(Q)/k_reflects_MD_Kn}}] 
Let $Q$ be a quadric of dimension at least $2^{n+1}-1$.

The base change functor
$$ \CM_{\Kn}(k) \rarr \CM_{\Kn}(k(Q))$$
reflects motivic decompositions of
the motives of projective homogeneous varieties, 
i.e.\ any decomposition or any isomorphism of such motives over $k(Q)$ can be lifted to $k$.
\end{Prop-intro}

Recall that to a quadratic form one can associate a sequence of field extensions and quadratic forms over them,
called the generic splitting tower of $q$:
$$ k=K_0 \subset K_1 \subset \cdots \subset K_h$$
where $K_i:= K_{i-1}(Q_{i-1})$ for $Q_{i-1}$ the  projective quadric corresponding to
 the anisotropic part of $q$ over $K_{i-1}$.

If $K_{j_n}$ is the first field in this sequence over which the dimension of the corresponding quadric $Q_{j_n}$ is less than $2^{n+1}-1$, 
then using the proposition above one can reduce the study of the $\Kn$-motive of $Q$ 
to the $\Kn$-motive of $Q_{j_n}$. 

\subsection*{Morava and Chow motives of quadrics} 
Now let $Q$ be a quadric of dimension less than $2^{n+1}-1$. 
The $\Kn$-motive of $Q$ 
is structurally
simpler than its Chow motive:
it canonically splits off $\dim Q - 2^n+2$ Tate motives, while the Chow motive of $Q$ can be indecomposable.
Let $\Mker Q$ denote the complement to these Tate summands.
We show that the decomposition of the Chow motive of $Q$ and isomorphism classes  of its direct summands
 can be nevertheless reconstructed from $\Mker Q$.

\begin{Th-intro}[{for a more precise statement see~Theorem~\ref{th:prestableMDT}}]
Let $Q$ be an anisotropic quadric of dimension less than $2^{n+1}-1$.
To every direct summand $\widetilde{N}$ in $\Mker Q$ 
we associate an isomorphism class of a motive $N$ which is a direct summand of $\Mot{\CH}{Q}$,
naturally under base field extensions and direct sums.
Given a motivic decomposition  $\Mker Q\cong \oplus_\lambda \widetilde{N}^\lambda$ 
into indecomposable directs summands 
we thus obtain a motivic decomposition $\Mot{\CH}{Q} \cong \oplus_\lambda N^\lambda$,
providing a bijection between the corresponding summands:

$$ \widetilde{N}^\lambda\xhookrightarrow{\oplus} \Mker Q \mbox{\ \ \ } \xleftrightarrow{\,1-1\,} \mbox{\ \ \ } N^\lambda\xhookrightarrow{\oplus} \Mot{\CH}{Q}.$$

For two quadrics $Q_1$, $Q_2$ of dimension less than $2^{n+1}-1$, 
 summands $\widetilde{N_i}\xhookrightarrow{\oplus} \Mker{Q_i}$, $i=1,2$, 
are isomorphic up to a Tate twist, 
if and only if the corresponding Chow motives $N_1$, $N_2$ are also isomorphic up to a Tate twist.
\end{Th-intro}

We emphasize that this result is specific to quadrics
and does not follow from the general properties of Morava K-theories.
Together with propositions above, this theorem 
shows that the study of the $\Kn$-motives of quadrics over a field $k$
is equivalent to the study of the Chow motives of quadrics of dimension less than $2^{n+1}-1$
but over a different field. 

\subsection*{Morava motives over universal splitting fields for $L_\alpha$} 
The key observation above that $\Kn$-motives ``ignore'' base change 
to the field of functions of a high-dimensional quadric can be strengthened even further.
Let $\alpha$ be a symbol in $\km$, and $k(\alpha)$ be the field of functions of the corresponding Pfister quadric.
Then we show that by passing from $k$ to $k(\alpha)$
we retain some control over motivic decompositions and isomorphisms between $\Kn$-motives.

\begin{Th-intro}[{see~Theorem~\ref{th:iso_over_k(alpha)}, Proposition~\ref{prop:decomposition_over_k(alpha)} for more general statements}]
Let $\alpha \in \km$ be a symbol,
let $M, N\in \CM_{\Kn}(k)$ be indecomposable direct summands of the $\Kn$-motives of projective homogeneous varieties.

\begin{enumerate}
\item (isomorphisms) If $M_{k(\alpha)}\cong N_{k(\alpha)}$, 
then either $M\cong N$ or $M\cong N\otimes L_\alpha$.

\item (decompositions) 
Either $M_{k(\alpha)}$ is indecomposable,
or it decomposes into two indecomposable isomorphic summands.
\end{enumerate}
\end{Th-intro}

In fact, our result is more general. First, in the statement of the theorem one can take 
$\alpha$  to be any element in $\km$, and for it there exists a field extension $k(\alpha)$ that splits $\alpha$
and is suited for the study of $\Kn$-motives (Definition~\ref{def:k(alpha)}).
Second, assumptions on the motives above can be weakened, 
and, for example, $M,N$ can be any invertible motives in the first claim above.
In particular, it proves the following.

\begin{Cr-intro}[{Proposition~\ref{prop:k(alpha)-kernel-picard}}]
The kernel of the base change map from $\Pic\left(\CM_{\Kn}(k)\right)$
to $\Pic\left(\CM_{\Kn}(k(\alpha))\right)$ is isomorphic to $\ZZ/2$ and is generated by $L_\alpha$.
\end{Cr-intro}

\subsection*{Detecting the invertible Morava motives $L_\alpha$}

As we have already stated above, 
to every $\alpha\in\km $ we associate the invertible motive $L_\alpha$
that acts as a splitting object for $\alpha$,
i.e.\ it trivializes over $K/k$ if and only if $\alpha_K=0$.
We show that $L_\alpha$
can be detected just by counting Tate motives over $k$ and over $k(\alpha)$,
a splitting field of $\alpha$ that is suited for the study of $\Kn$-motives.

\begin{Th-intro}[{see~Theorem~\ref{th:detect_L_alpha}}]
Let $M$ be a $\Kn$-motive.  
Assume that 
$M_{k(\alpha)}$ splits off a Tate summand $\un$.

Then $M$ splits off either $\un$ or $L_\alpha$ as a direct summand.
\end{Th-intro}

Recall that 
for a {\sl symbol} $\alpha$ there is already a splitting object in the Chow motives, the Rost motive $R_\alpha$.
For Chow motives, however, the analogue of the above theorem does not hold.

Moreover, we can also consider the semi-simple abelian category of $\Knum$-motives,
where the objects $(L_\alpha)_{\num}$ behave in a similar way as $L_\alpha$ in $\Kn$-motives.
One can detect occurrence of $L_\alpha$ in the $\Kn$-motive $M$
by showing that $(L_\alpha)_{\num}$ is a direct summand in $M_{\num}$ (Proposition~\ref{prop:lifting_numerical_decompositions}),
i.e.\ ``just'' by computing $\Knum$ over $k$ and $k(\alpha)$ (see Remark~\ref{rk:detect_tate_count}).

Semi-simplicity of $\Knum$-motives makes them a convenient target 
for motivic measures on the Grothendieck ring of varieties.
Thus, the occurrence of $L_\alpha$ in the $\Kn$-motive of $X$
provides non-trivial information about the class of $X$ in $\mathrm{K}_0(\mathrm{Var}_k)$.
This will be investigated elsewhere.

\subsection*{Cohomological invariants and the guiding principle}

It was an idea of Serre~\cite{GMS}
that many algebraic objects defined over a field $k$ 
and that trivialize over $\overline{k}$, e.g.\ torsors for an algebraic group,
could be classified by the elements of Galois cohomology of $k$
associated to them, their cohomological invariants.
In this paper we propose a way to associate cohomological invariants
to some Chow motives using Morava K-theory.

\subsubsection*{The guiding principle.}

The first connection between $\Kn$-motives and cohomological invariants was observed in~\cite{SechSem} 
and embodied in the following ``guiding principle'': 

``Let $X$ be a projective homogeneous variety,
and let $\Kn$ be the algebraic Morava K-theory at a prime $p$. 
Then the vanishing of cohomological invariants of $X$ with $p$-torsion coefficients in degrees no
greater than $n+1$ should correspond to the splitting of the $\Kn$-motive of $X$.''

Note that this is merely a principle, and not a conjecture,
since it is not clear how to specify the definition 
of a cohomological invariant of $X$ to include even all known cases when the principle works.
And although the principle is useful in studying split $\Kn$-motives and vanishing of cohomological invariants,
it does not predict how to find non-trivial cohomological invariants when Morava motives are not split. 

In this paper we improve on this guiding principle in two ways.
First, we show that one of the directions of it holds for the weakest possible definition of a cohomological invariant. 
Second, we explain how one can use $\Kn$-motives
in order to {\sl define} cohomological invariants associated to a smooth projective variety or even a Chow motive.
This allows us to formulate a conjecture predicting that 
to every direct summand of the Chow motive of a projective homogeneous variety 
one can associate a cohomological invariant. 

\subsubsection*{One direction of the guiding principle}

Recall that by the Bloch--Kato conjecture
one can identify $\kmtech_{m}(k)/p^r$ with $\HH_{\et}^m(k,\,\mu_{p^r}^{\otimes m})$.
Also, if $k$ contains a primitive $p^r$-th root of unity,
then $\ZZ/p^r\cong \mu_{p^r}$ as \'etale sheaves.

\begin{Th-intro}[{Theorem~\ref{prop:guiding_principle}}]
Let $k$ be a field containing a primitive $p^r$-th root of unity.
Let $\Kn$ denote the algebraic Morava K-theory at a prime $p$. 
Let $X$ be a projective homogeneous variety for a semi-simple algebraic group of inner type. 

Assume that $\MKn X$ is a sum of Tate motives.

Then there exists a field extension $K/k$ such that $X_K$ is split
and 
$$
\HH^{m}(k,\,\ZZ/p^s)\hookrightarrow \HH^{m}(K,\,\ZZ/p^s) 
$$
is injective for all $s\le r$ and $m\le n+1$. 
In particular, $X$ has no 
(normalized) 
cohomological invariants in $\HH^{m}(k,\,\ZZ/p^r)$ of degree $m\le n+1$ 
in the weakest possible sense of cohomological invariant (see Definition~\ref{def:weak_coh_inf}).
\end{Th-intro}

\subsubsection*{Cohomological invariants from Morava motives}

For the purpose of our conjecture on cohomological invariants 
we assume 
that for $\Kn$ at a prime $p$, for all $r\in \NN$ and for every $\alpha\in \HH^{n+1}(k,\,\ZZ/p^r)$
there exists an invertible $\Kn$-motive $L_\alpha$ 
generalizing the ones for $p^r=2$ described above.
In fact, construction of these motives will appear in the forthcoming paper of the second author~\cite{SechInv}
for a field $k$ containing all $p^r$-th roots of unity.

Recall that by the result of Vishik--Yagita~\cite{VishYag} 
one can ``specialize'' any Chow motive $M$ 
to an $A$-motive for any oriented theory $\A$, in particular, to a $\Kn$-motive $M_{\Kn}$.

\begin{Def-intro}
Let $M$ be a Chow motive.
We say that $\alpha\in \HH^{n+1}(k,\,\ZZ/p^r)$
is a cohomological invariant of $M$ if $L_\alpha\sh j$ is a direct summand of $M_{\Kn}$ for some $j$.
\end{Def-intro}

For the following conjecture we need to consider the zeroth Morava K-theory $\mathrm{K}(0)$
that is defined just as $\CH^*\otimes \QQ$ and does not depend on the prime $p$.
In what follows we consider Chow motives $M$ 
such that $M_{\K0}$, i.e.\  $M$ with rational coefficients, is split
and such that $M_{\overline{k}}$ is split. 
These are, for example, direct summands of the motives of projective homogeneous varieties 
for a semi-simple algebraic group of inner type.

\begin{Conj-intro}[{see Conjecture~\ref{conj:coh_inv_morava}}]
Let $k$ contain $p^r$-th roots of unity for all $r$.
Let $n\ge 1$, and for $m\ge 0$ let $\K{m}$ be the $m$-th algebraic Morava K-theory at a prime $p$. 
Let $M$ be a Chow motive such that $M_{\overline{k}}$ is split.
Assume that $M_{\K{m}}$ splits for every $m$ such that $0\le m< n$,
i.e.\ it is isomorphic to a direct sum of Tate motives.

Then $M_{\Kn}$ decomposes into a direct sum of the motives $L_{\alpha}\sh j$ 
for some $j\ge 0$, $r\in\NN$ and $\alpha \in \HH^{n+1}(k,\,\ZZ/p^r)$.
\end{Conj-intro}

There are plenty of examples of motives $M$ coming from projective homogeneous varieties for which 
 the decomposition of $M_{\Kn}$ can be computed and the conjecture can be confirmed.
The case of $n=1$ and $M$ coming from projective homogeneous varieties 
is essentially due to Panin~\cite{Panin} (see~Section~\ref{sec:tits_k0_motives}).
The case $M=\MCH Q$ for $Q$ a quadric follows from \cite{SechSem} and the results of this paper. 
For some motives coming from
exceptional groups it is also known (see Section~\ref{sec:coh_inv_phv}): for $\mathrm G_2$, $\mathrm F_4$, $\mathrm E_6$ it follows from~\cite{PSZ}, 
and for $\mathrm{E}_7$ and $\mathrm{E}_8$ it follows from the ongoing work of Wohlschlager~\cite{Alois}.

If the conjecture above is true,
then one can associate a cohomological invariant to  
every Chow motive $M$ such that $M\otimes \QQ$ and $M_{\overline{k}}$ are split.
Indeed, one looks at the smallest possible number $n$ 
such that $M_{\Kn}$ is non-split, and then $M_{\Kn}$ contains $L_\alpha$ for non-trivial $\alpha$. 
If such $n$ does not exist, i.e.\ $M_{\Kn}$ is split for all $n$, then $M$ is also split, and thus has no invariants.

\subsection*{Invertible motives in quadrics.}

We confirm the above conjecture for the motives of quadrics $\MCH Q$ for all $n$, 
and for arbitrary direct summands $M$ in $\MCH Q$ for $n=1,2,3$ 
using the known cases of the Kahn--Rost--Sujatha conjecture~\cite{KRS}. 
Moreover, 
we show that for every such $M$ 
there exists $n\ge 0$ such that $M_{\K{n}}$ decomposes into invertible summands, 
some of which are non-trivial (Proposition~\ref{prop:from_Kn+1_to_Kn}),
and the problem lies in identifying these summands as $L_\alpha$.

We first describe the known cases where $L_\alpha$ appears in the $\Kn$-motive of a quadric.

\begin{Prop-intro}[{for a more precise statement see~Proposition~\ref{prop:Kn-motive-q-small-dim_n}}]
Let $Q$ be a projective quadric defined by an anisotropic quadratic form $q$  over $k$ of dimension at least $2^n+1$
such that there exist $q'$ of $\dim q'<2^n$ 
and $w:=q\perp q' \in I^{n+1}(k)$.

Then all non-trivial invertible summands of $\MKn Q$ are the Tate twists of the motive $L_{[w]}$
for $[w]$ the class of $w$ in $\HH^{n+1}(k,\ZZ/2)$.   
\end{Prop-intro}

We expect that the converse holds as well (see Conjecture~\ref{conj:inv_summand_quadrics})
and show that it follows from 
Kahn's conjecture about the descent of quadratic forms~\cite{Kahn-descent} (Proposition~\ref{kahn-unary}).
The validity of the latter 
thus would confirm our conjecture on cohomological invariants of Chow motives 
for direct summands in the Chow motives of quadrics.

\subsection*{Relation to other work.}

Although the results of this paper provide a logical sequel to~\cite{SechSem},
the methods developed here are different and we do not rely on~\cite{SechSem} 
in our proofs.

Our results on the Morava motives of quadrics are conceptually related to the computation of $\Kn$ of 
orthogonal groups previously obtained  in~\cite{LPSS, LPSS2} 
by Geldhauser, Petrov and the authors of this paper. 
Among other things, it was shown in loc.\ cit.\ that the sequence of pullback maps
$$ 
\ldots \rarr \Kn ({\mathrm{SO}_m}) \rarr \Kn ({\mathrm{SO}_{m-2}}) \rarr \ldots 
$$ 
stabilizes for $m\ge 2^{n+1}-1$, i.e.\ 
$\Kn$ ``does not see'' any difference between higher orthogonal groups.
What we show in this paper, is that $\Kn$ also treats
torsors under these orthogonal groups  
``without much difference''.

We also remark that invertible objects in different motivic categories
that are related to quadrics have been studied by Bachmann, Hu and Vishik, 
see~\cite{Hu, BachInv, BachVish, VishPic}.
In particular, Bachmann and Vishik have constructed in~\cite{BachVish}
an embedding of {\sl sets}
$$\mathrm{GW}(k) \rarr \Pic\big(\mathrm{DM}(k,\,\ZZ/2)\big).$$ 
As this map is not additive even when restricted to $I^n(k)$, 
and does not factor through $I^n(k)/I^{n+1}(k)$,
it does not provide an analogue of Theorem~\ref{th:nat_transform_milnor_picard}.

\subsection*{Structure of the paper.}

Although we are mainly interested in quadrics, many of the results of this paper may have wider applications.
Section~\ref{sec:univ_surj} studies the question of universally surjective and bijective field extensions
for arbitrary oriented cohomology theories. Applications 
to the motives and to the Rost Nilpotence Property
are investigated in Section~\ref{sec:mot_dec_base_change}.
Section~\ref{sec:morava_quadrics} contains the main results 
about the decomposition of the Morava motives of quadrics, it is one of the technical hearts of the paper. 
The other one is Section~\ref{sec:morava_motives_over_function_fields} 
where we investigate motivic decompositions and isomorphisms over function fields.
Section~\ref{sec:milnor_k-theory_picard} contains the construction and the properties of the invertible motives~$L_\alpha$.
In Section~\ref{sec:inv_summand_quadrics} we investigate invertible motives occurring as direct summands of the Morava motives of quadrics.
Finally, Section~\ref{sec:coh_inv} contains an explanation of the new approach 
to cohomological invariants, its application to quadrics and the proof of one direction of the guiding principle.
Appendix~\ref{app:a} contains the proof of geometric Rost Nilpotency Property extending the results of Vishik--Zainoulline \cite{VishZai}
to arbitrary oriented theories. 
Appendix~\ref{k2motives} contains description of all possible $\K2$-motivic decompositions of quadrics.

\subsection*{Acknowledgments}

We are very grateful to Evgeny Shinder for kindly allowing us to include the unpublished joint results with the second author.

We would like to thank Alexey Ananyevskiy, Nikita Geldhauser, Stefan Gille, Marc Hoyois, Victor Petrov, Alexander Vishik,
Stephen Scully, Evgeny Shinder, and Maksim Zhykhovich for useful discussion on the topics of this paper.

\newpage

\tableofcontents

\newpage

\section{Preliminaries}

In this section, 
we recall standard definitions and facts 
about oriented cohomology theories, algebraic Morava K-theory, 
motives, in particular, the Rost motives, quadratic forms, and Vishik's symmetric operations
that will be used in the main body of the paper.

Throughout the article the base field $k$ has characteristic 0,
$\overline{k}$ is its algebraic closure,
and
$\Sm_k$, $\SmProj_k$ (resp. $\Sm_S$, $\SmProj_S$)
denote the categories of smooth quasi-projective, resp., projective varieties over a field $k$ (resp. over a scheme $S$).
For $X\in\SmProj_S$ we usually denote by $\pi_X\colon X\rarr S$ the structural morphism of $X$.
For a scheme $X$ over $k$ we denote by $\overline{X}$ the scheme $X\times_k \overline{k}$.

For a field $F$ we denote by $\HH^m(F,\,A)$ the \'etale cohomology of $F$ with coefficients
in the \'etale sheaf $A$.

\subsection{Quadratic forms}
\label{sec:prelim_quad_forms}

All quadratic forms in this paper are assumed to be non-degenerate. We refer the reader to~\cite{EKM} for basic facts about quadratic form theory.
For a quadratic form $q$ we denote by $q_{\an}$ the anisotropic part of it, 
$\iw(q)$ denotes the Witt index $q$,
i.e. $q\cong q_{\an} \perp \hyp^{\perp\iw(q)}$.
A quadratic form $q$ is called {\sl split}, if $\dim q_{\an}\le 1$.

\subsubsection{Generic splitting tower}\label{sec:prelim_split_tower}

Let $q$ be a quadratic form over $k$.
One associates to is a sequence of field extensions \cite{Knebusch, Knebusch2}:

$$ k=:K_0 \subset K_1:=k(q_{\an}) \subset K_2 \subset \cdots \subset K_h,$$
together with quadratic forms $q_i$ over $K_i$, which are defined inductively:  
 $K_{i+1} := K_{i}(q_{i})$  and $q_i:= (q_{K_{i}})_{\an}$.
The field $K_{h-1}$ in this tower is called the {\sl leading field} of $q$,
and the field $K_h$ is called the {\sl generic splitting field} of $q$.
Note that if $q$ is even-dimensional, then $q_{h-1}$ is a general Pfister form,
since it becomes split over its function field $K_h$.

We denote by $i_s(q)$ the higher Witt indices of $q$,
although only $i_1$ will appear later: for anisotropic $q$ over $k$, we have
$i_1(q)=\iw(q_{k(q)})$.

\subsubsection{$I$-adic filtration on the Witt ring under base change}
\label{sec:milnor_krs}

Let $W(k)$ denote the Witt ring of quadratic forms over $k$,
and let $I(k)$ be the fundamental ideal in it. 
The Milnor conjecture asserts that $I^n(k)/I^{n+1}(k)$ is canonically isomorphic to the Milnor K-theory modulo 2,
it was proved by Orlov--Vishik--Voevodsky~\cite{OVV}; see also \cite{Morel} for a different approach.
We will need the following application of this proof, from loc.\ cit.

\begin{Th}[{\cite[Th.~4.2]{OVV}}]
\label{KRS}
Let $q$ be a quadratic form over $k$, 
and $n$ be an integer such that $2^n<\mathrm{dim}(q)$. 
Then
$$
\mathrm{Ker}\left(\frac{I^n(k)}{I^{n+1}(k)}\xrightarrow{\mathrm{res}}\frac{I^n(k(q))}{I^{n+1}(k(q))}\right)=0.
$$
\end{Th}

\begin{Cr}
\label{cr:KRS-Witt_mod_In+1_injectivity}
Let $q$ be a quadratic form over $k$, 
and $n$ be an integer such that $2^n<\mathrm{dim}(q)$. 
Then
$$
\mathrm{Ker}\left(\frac{W(k)}{I^{n+1}(k)}\xrightarrow{\mathrm{res}}\frac{W(k(q))}{I^{n+1}(k(q))}\right)=0.
$$
\end{Cr}
Another corollary of the Milnor conjecture is that the leading form of a quadratic form from $I^n(k)$ 
is a general $n$-fold Pfister form~[loc.\ cit, Th.~4.3].

\subsection{Oriented cohomology theories}

The notion of an oriented cohomology theory was introduced in \cite{LevMor, PanSmi},
and we use a version of it with the localization axiom as in~\cite[Def.~2.1]{Vish1}. 

Recall that for an oriented cohomology theory $\A$ and $X\in\Sm_k$,
$\A(X)$ is a commutative ring. For any $f\colon X\rarr Y$ in $\Sm_k$,
$f^*\colon \A(Y)\rarr \A(X)$ is a pullback morphism which is a ring a homomorphism;
for any projective $g\colon Z\rarr W$ in $\Sm_k$, $g_*\colon \A(Z)\rarr \A(W)$ is a push-forward morphism
which is a group homomorphism.
Often it also assumed that $\A(X)$ is a graded ring, $f^*$ preserves grading and $g_*$ sends the $i$-th graded component
to the $(i+\dim Z-\dim W)$-th one,
however, we will also consider ungraded version of theories, such as $\KK$.
In the former case we write $\A^i(X)$ to indicate the $i$-th graded component 
and $\A^*(X)$ or $\A(X)$ to indicate $\bigoplus_{i \in \ZZ} A^i(X)$.
Note that $\A(X)$ is a module over the ring $\A(\Spec k)$, which we denote by $\Apt$.

We will often write $[Z\xrarr{g} W]_A$ or just $[Z]_A$, where $g$ is a projective morphism,
for the element $g_*(1_Z)$ in $\A(W)$. If $X\in \SmProj_k$,
then we denote by $\deg_X$ or just $\deg$ the morphism $(\pi_X)_*$.

A morphism of theories $\phi$ from an oriented theory $A$ to an oriented theory $B$
is a collection of maps of commutative rings $A(X)\xrarr{\phi_X} B(X)$
such that $\phi_X$ commutes with push-forward and pullback morphisms.

\subsubsection{Localization axiom}
\label{sec:LOC}
Let $Y\in \Sm_k$ be an irreducible variety.
Let $Z$ be a closed subscheme of $Y$, 
and let $U$ be its open complement.
Then the localization axiom for $A$, which is a part of our definition of an oriented cohomology theory,
is the exactness of the following sequence of abelian groups
\begin{equation}\label{eq:loc}
\tag{LOC}
 \A(Z) \rarr \A(Y) \rarr \A(U) \rarr 0,
\end{equation}
where $A(Z)$ is defined as a Borel--Moore extension of $A$ to schemes of finite type over $k$ (see~e.g.\ \cite[Ch.~5]{LevMor}),
namely, $A(Z) = \mathrm{colim}_W A(W)$ for $W\rarr Z$ a projective morphism, $W\in\Sm_k$.
The morphism $\A(Z)\rarr \A(Y)$ is induced by the push-forward morphisms $A(W)\rarr A(Y)$.
The morphism $\A(Y)\rarr \A(U)$ is the pullback morphism.

\subsubsection{Values of the Borel--Moore theory via resolution of singularities}
\label{sec:borel_moore_theory_via_resolution_singularities}

Although $\A(Z)$ is defined as a colimit, it admits a concrete and finite presentation.
Namely, Vishik shows in  \cite[Prop. 7.7]{Vish1} that 
if $\tilde{Z}\rarr Z$ is a resolution of singularities obtained by blowing-up smooth centers $R_i$,
then the following is a short exact sequence:

$$ \oplus_i A(E_i) \rarr \oplus_i A(R_i) \oplus A(\tilde{Z}) \rarr A(Z) \rarr 0$$

where $E_i$ are exceptional divisors of blow-ups. 

\subsubsection{Extended and coherent oriented cohomology theories}
\label{sec:extended_coherent}

Throughout the paper we will be interested in what happens to cohomology theories
when we pass from $k$ to a not necessarily finite field extension $K$.
Thus the cohomology theories need to be extended to smooth varieties over $K$ (which are not of finite type over $k$, and thus do not lie in $\Sm_k$).

One of the approaches to achieve this is the notion of an {\sl extended} cohomology theory due to Gille--Vishik \cite[Def.~2.4]{GilleVishik}.
One considers a family of oriented cohomology theories $A_F$ 
 defined over all finitely generated field extensions $F$ of $k$
together with comparison maps $A_F(Y)\rarr A_{F(X)}(Y_{F(X)})$ 
for every smooth morphism $Y\rarr X$ in $\Sm_F$ and $X$ irreducible.

Moreover, one calls an extended cohomology theory {\sl coherent} [loc.\ cit., Def. 2.6]
 if the canonical homomorphism
from $\mathrm{colim}_{\,U\subset X} A(Y_U) \rarr A(Y_{F(X)})$ is an isomorphism in the situation above, 
where the colimit is taken over all non-empty open subschemes $U$ of $X$.
For the sake of brevity, we will use the term coherent theory 
to refer to coherent extended cohomology theory. 
We warn the reader that coherent theories are not related to coherent sheaves.

\subsubsection{Transversal base change property}
\label{sec:transversal_bc}

Recall that morphisms $f\colon X\rarr Y$, $g\colon Z\rarr Y$ between schemes are called {\sl transversal} (or $\mathrm{Tor}$-independent)
if for any $x\in X, z\in Z$ mapping to the same point $y\in Y$
we have $\mathrm{Tor}_i^{\OO_{Y,y}}(\OO_{X,x}, \OO_{Z,z})=0$ for $i\ge 1$.
A commutative square of schemes
\begin{center}
    \begin{tikzcd}
        W \arrow[r, "f'"] \arrow[d, "g'"] & Y \arrow[d, "g"] \\
        X \arrow[r, "f"] & Z \\
    \end{tikzcd}
\end{center}
is called transversal, if it is a pullback square, and $f$ and $g$ are transversal.

Consider a transversal square above where $X,Y,Z\in \Sm_k$. 
Then $W$ is also smooth quasi-projective over $k$.
Assume moreover that $f$ is projective, then so is $f'$.
The transversal base change property for an oriented theory $\A$ is the following equality:
\begin{equation}
\label{eq:bc}
    \tag{BC} g^* \circ f_* = f'_* \circ (g')^*.
\end{equation}

Let $\A$ be a coherent theory,
$g:Y\rarr X$ is a projective morphism in $\Sm_k$, $U\in \Sm_k$, 
and consider the following cartesian square:
\begin{center}
    \begin{tikzcd}
        Y\times_k U \arrow[d, "g\times\mathrm{id}"] & Y_{k(U)} \arrow[d, "g_{k(U)}"] \arrow[l, "\mathrm{id}_Y \times j "]\\
        X\times_k U                            & X_{k(U)} \arrow[l, "\mathrm{id}_X \times j "]
    \end{tikzcd}
\end{center}
where $j\colon \Spec k(U)\hookrightarrow U$ is the inclusion of the generic point.
Then one can show that the following equation holds, we call it extended transversal base change (cf.~\cite[Lm.~2.15~(iii)]{GilleVishik}):
\begin{equation}
\label{eq:ext_bc}
    \tag{extended BC} (\mathrm{id}_X \times j)^* \circ (g\times\mathrm{id})_*  = (g_{k(U)})_*\circ (\mathrm{id}_Y \times j)^*.
\end{equation}

\subsubsection{Algebraic cobordism and free theories}
\label{sec:alg_cob_free_theories}

The universal oriented cohomology theory is called algebraic cobordism $\Omega^*$,
and it was constructed in the seminal book of Levine--Morel~\cite{LevMor} (see also \cite{LevPand}).
In fact, it is universal for a much weaker notion of oriented cohomology theory than the one that we consider (see~\cite[Ch.~7]{LevMor}).

Given a formal group law $F_A$ over a commutative ring $A$,
let $\LL\rarr A$ be the corresponding morphism from the Lazard ring to $A$.
Recall that  $X\mapsto \Omega^*(X)\otimes_\LL A$
can be endowed with the structure of an oriented cohomology theory such that $\Omega^* \rarr \Omega^*\ot_{\LL} A$ is a morphism of theories. 
Oriented theories that have the form $\Omega^*(X)\otimes_\LL A$ for some $A$ and $F_A$
are called free theories \cite[2.4.14~(2)]{LevMor}
(and they satisfy the localization axiom, since $\Omega^*$ does).
By \cite[Th.~1.2.7]{LevMor} the Lazard ring is isomorphic to $\Omega^*(k)$,
and hence the free theory constructed above has ring $A$ as its ring of coefficients.
Free theories can be made into coherent theories, see \cite[Ex. 2.7]{GilleVishik}.
Moreover, every oriented theory $\A$ carries a topological filtration $\tau^\bullet A$,
which satisfies various desirable properties in the case of free theories \cite[Sec.~1.8]{Sech2}.

In particular, Chow groups $\CH$ and the Grothendieck K-theory $\KK$ are free 
theories by \cite[Th.~1.2.18,~1.2.19]{LevMor}.
If the prime $p$ is clear from the context, 
we denote by $\Ch:=\CH/p$, which is also a free theory.

\subsubsection{Other versions of algebraic cobordism and restriction to $\mathrm{char}\, k=0$.}

Recently, in the works of Annala \cite{Annala1, Annala2, Annala3, Annala4}
another construction of algebraic cobordism has appeared 
that works in a greater generality of derived schemes over some base,
and many of the fundamental properties of it are proved.
However, the most crucial for us, the localization axiom,
is not known to hold for this theory at the moment,

There is also a theory of algebraic cobordism, represented by a motivic spectrum $\mathrm{MGL}$
in $\mathrm{SH}(F)$ over any field $F$, introduced by Voevodsky~\cite{VoeICM}. 
If $\mathrm{char\ } F=0$, 
then $\mathrm{MGL}^{2*,*}(X)\cong \Omega^*(X)$
by the result of Levine~\cite{LevCobordism}.
In general, one can consider $\mathrm{MGL}^{2*,*}(X)$ as an oriented cohomology theory 
that satisfies the localization axiom, also over fields of positive characteristic.
However, 
the Rost Nilpotence Property for projective homogeneous varieties (see Section~\ref{sec:mot_dec_base_change})
and 
Vishik's symmetric operations used in Section~\ref{sec:morava_quadrics}
are only available in characteristic 0, 
as their proof and construction, respectively, rely on resolution of singularities.

The lack of a theory of algebraic cobordism that has all the expected properties 
is one of the main reasons why the results of this paper cannot be extended to base fields of positive characteristic.

\subsubsection{Algebraic Morava K-theory}
\label{sec:prelim_morava}

Let $p$ be a prime number, $n\ge 1$.
Let $F_{\KnInt}$ be a graded formal group law over $\Z{p}[v_n, v_n^{-1}]$, $\deg v_n=1-p^n$,
with the logarithm $\sum_{k\ge 0} p^{-k} v_n^{\frac{p^{nk}-1}{p^n-1}} t^{p^{nk}}$
(see e.g.\ \cite[App.~2]{Rav}). Note that the induced formal group law over $\F{p}[v_n,v_n^{-1}]$, denoted $F_{\Kn}$, has height $n$.

\begin{Def}
\label{def:morava_k-theory}
Free theory $\Omega^*\otimes_{\LL} \Z{p}[v_n, v_n^{-1}]$ associated to $F_{\KnInt}$ above
is called the $n$-th integral algebraic Morava K-theory at the prime $p$ and is denoted $\KnInt$.

Free theory $\KnInt/p=: \Kn$ is called the $n$-th algebraic Morava K-theory at the prime $p$.
\end{Def}

Note that for $p=2$, the Artin--Hasse exponent gives an isomorphism
between the multiplicative formal group law and $F_{\K{1}^{\mathrm{int}}}$,
which yields an isomorphism of presheaves of rings 
between $\K{1}^{\mathrm{int}}$ and $\KK\otimes\Z{2}[v_1, v_1^{-1}]$,
as well as $\K{1}$ and $\KK\otimes \F{2}[v_1,v_1^{-1}]$.

Recall also that there exists the notion of $p$-typical formal group law due to Cartier \cite{Cart},
the corresponding free universal $p$-typical oriented cohomology theory is denoted $\BP$. 
Since $\Kn$ has $p$-typical formal group law,
it can be also defined as $\BP\otimes_{\BPpt} \F{p}[v_n, v_n^{-1}]$ (similarly, for $\KnInt$).

We also consider the {\sl $n$-th connective Morava K-theory} $\CKn$ at the prime $p$,
which can be defined as $\BP\otimes_{\BPpt} \F{p}[v_n]$
(note that the map $\BPpt \rarr \F{p}[v_n, v_n^{-1}]$ used above factors through $\F{p}[v_n]$).

We defined algebraic Morava K-theory as a free theory for the specific choice of the formal group law 
(corresponding to Hazewinkel's choice of the generator $v_n$ in the sense of~\cite[App.~2]{Rav}).
This choice plays the role in the computation of $\Kn$ of a split quadric as a ring,
and hence in the description of the rational projectors for the summands of the $\Kn$-motive of a quadric.
In turn, the formulas for these projectors are required for computing the action of symmetric operations
used in Section~\ref{sec:morava_quadrics}. In all the other sections one can take $\Kn$ to be a free theory
with any graded formal group law over  $\F{p}[v_n,v_n^{-1}]$
that has height $n$. It is also clear to the authors that the results of 
Section~\ref{sec:morava_quadrics} remain valid for such $\Kn$.

Also, in some sections of the paper (e.g.\ \ref{sec:morava_motives_over_function_fields}) for simplicity of computations 
we consider free theory $\Kn/(v_n-1)$ with the coefficient ring $\F{p}$ still denoting it as $\Kn$.
Note that this theory has $\ZZ/(p^n-1)$-grading
instead of $\ZZ$-grading (cf.~\cite[4.1]{Sech}), and $\left(\Kn/(v_n-1)\right)^{[j]}(X)\cong \Kn^j(X)$ for $j\in \ZZ$
(where in the LHS $[j]$ denotes the residue class in $\ZZ/(p^n-1)$). Thus, hardly any information is lost by setting $v_n=1$.

Note that one can construct a motivic spectrum of algebraic Morava K-theory $\mathcal{K}(n)$ from $\mathrm{MGL}$,
following the same procedure as in topology which is applied to $\mathrm{MU}$,
see~\cite{LevTri}.
By the results of loc.\ cit., in characteristic 0, $\Kn^*(X)$ that we define 
is isomorphic to  $\mathcal{K}(n)^{2*,*}(X)$, the $(2*,*)$-diagonal of the bigraded cohomology theory represented by $\mathcal{K}(n)$.
In this paper we do not work with $\mathcal{K}(n)$,
although it can be useful in related questions: 
see e.g.\ \cite{YagSpin} for some computations of Spin groups 
that involve the Atiyah--Hirzebruch spectral sequence 
for $\mathcal{K}(n)$.

\subsubsection{Oriented cohomology theory of a split quadric}
\label{sec:prelim_A_split_quadric}

If $q$ is a quadratic form over $k$, we usually denote by $Q$ the corresponding projective quadric.
Since $q$ is always assumed to be non-degenerate, $Q$ is smooth. We will say that $Q$ is split, if $q$ is split (see Section~\ref{sec:prelim_quad_forms}).

Let $\A$ be an oriented cohomology theory, let $Q$ be a quadric of dimension $D$.
Let $d:=[\frac{D}{2}]$, i.e.\ $D=2d$ if $D$ is even and $D=2d+1$ if $D$ is odd.
Let $h\in \A(Q)$ denote the class $[H\hookrightarrow Q]$ of a smooth hyperplane section $H$ of $Q$, it does not depend on the choice of $H$.

Assume now that $Q$ is split.
Then for every $i<d$, there exists a linear embedding $\mathbb{P}^i \hookrightarrow Q$,
its class in $\A(Q)$ does not depend on the choice of this embedding and is denoted by $l_i\in \A(Q)$.
If $D$ is odd, one can similarly define $l_d$,
however, if $D$ is even, then the Grassmannian of $d$-dimensional linear subspaces on $Q$ 
has two components, and there are two linearly independent classes in $\A(Q)$
defined by the representatives of $\mathbb{P}^d\hookrightarrow Q$ for these two components.
We denote these classes by $l_d$ and $\widetilde{l_d}$, they are related by the identity
$$
h^d = l_d+\widetilde{l_d}+\sum_{i=1}^db_{i+1}l_{d-i}
$$
where $b_k$ denote the coefficients of the series $[2]_{\FGL{\A}}(t)=\FGL{\A}(t,\,t)=\sum_{k\geq1}b_kt^k$, 
and $\FGL{\A}$ is the formal group law of $\Af$, cf.~\cite[Proof of Prop.~4.1]{LPSS}.

The elements $l_i$ for $0\leqslant i\leqslant d$ and the powers $h^k$ of $h$ for $0\leqslant k\leqslant d$ form a free basis of $\At Q$ over $\Apt$,
we call these elements {\sl the standard basis}.
The multiplication in $\At Q$ satisfies the following identities: $h_{}\cdot l^{}_i=l_{i-1}$, $h_{}\cdot \widetilde{l_d}=l_{d-1}$, 
$l_d^2=l_0$ for $D\equiv0\,\ \mathrm{mod}\ 4$, $l_d^2=0$ otherwise.
The class $h^{k}$ for $k>d$ can be expressed in terms of classes $l_j$ using the formal group law of $A$,
as explained in \cite[Prop.~4.1]{LPSS} for the universal oriented theory, algebraic cobordism.
We will need the following corollaries of this computation for the case of algebraic Morava K-theory at the prime $2$.

\begin{Prop} 
\label{prop:KnInt_power_of_h_split_quadric}
Let $Q$ be a split quadric over $k$, $D=\dim Q\ge 2^{n+1}-1$.

Then $h^{D-(2^n-1)} = 2l_{2^n-1}+ cv_nl_0$ in $\KnInt(Q)$, where $c\in \Z{2}^{\times}$.
\end{Prop}
\begin{proof}
    Note that $2\cdot_{\KnInt} t \equiv 2t +cv_n t^{2^n} \mod t^{2^n+1}$ for some $c\in \Z{2}^{\times}$.
    Thus, by~\cite[Prop.~4.1]{LPSS} we get the claim, since the coefficients of $t^j$, $j\ge 2^n+1$, do not appear in the formula for $h^{D-(2^n-1)}$.
\end{proof}

\begin{Cr}[{\cite[Prop.~4.3]{LPSS}}]
\label{prop:l_0_rational_K(n)}
Let $Q$ be a smooth quadric, $\dim Q \ge 2^{n+1}-1$,
then $l_0$ is $k$-rational,
i.e.\ there exists an element $z\in \Kn(Q)$ that maps to $l_0$ in $\Kn(\overline{Q})$.
\end{Cr}

\begin{proof}
Take $z:=h^{D-(2^n-1)}\in \Kn(Q)$ and take the reduction of the formula of Proposition~\ref{prop:KnInt_power_of_h_split_quadric} modulo $2$. 
\end{proof}

\subsubsection{Classes of closed points under pullbacks}

Let $X\in \Sm_k$, let $\A$ be an oriented cohomology theory over $k$.

Every closed point $x$ in $X$ defines an element $[x]:=x_*(1)\in \A(X)$
where we consider $x$ as the morphism $\Spec k(x) \rarr X$.

\begin{Lm}
\label{lm:restriction_closed_points}
Let $f\colon W\rarr X$ be a non-surjective projective morphism in $\Sm_k$.
Then $f^*([x])=0$ for every closed point $x\in X$.
\end{Lm}
\begin{proof}
 It suffices to prove the claim for $\A=\Omega$, the universal oriented theory.    
Note that the classes of closed points lie in the group $\Omega^{\dim Y}(Y)$,
which is isomorphic to $\CH^{\dim Y}(Y)$, the group of 0-cycles, by \cite[Th.~4.5.1]{LevMor}.
However, using the moving lemma for Chow groups \cite[Sec.~11.4]{Ful},
there exists a $0$-cycle on $Y$ with the same class as $[x]$ in $\Omega^{\dim Y}(Y)$
whose support does not intersect the closed subset $f(W)\subset Y$.
By (BC) one gets that $f^*(y)=f^*([x])=0$.
\end{proof}

\subsection{Correspondences and motives}
\label{sec:prelim_motives}

\subsubsection{Action of correspondences}
\label{sec:corr_action}

Let $\A$ be an oriented cohomology theory over $k$.
Let $X\in\SmProj_k$ and $Y\in \Sm_k$.

Then $\A(X\times X)$ has the structure of (not necessarily commutative) ring,
with the class of the diagonal $\Delta_X:= [X\xrarr{\Delta} X\times X]_{\A}$
being the unit (see e.g. \cite[Sec.~62]{EKM} for $\A=\CH$).
Moreover, $\A(Y\times X)$ has the structure of a left module over $\A(X\times X)$:
 for $\alpha\in \A(X\times X)$ and $x\in A(Y\times X)$,
we define $\alpha\circ x$ as $(p_{1,3})_*(p_{1,2}^*(x)p_{2,3}^*(\alpha))$
where $p_{i,j}$ are projection maps from $Y\times X \times X$ either to $Y\times X$ or to $X\times X$.
Similarly,  $\A(X\times Y)$ 
has the structure of  a right module over $\A(X\times X)$
with $x\circ \alpha$  defined as 
 $(p_{1,3})_*(p_{1,2}^*(\alpha )p_{2,3}^*(x))$.

\subsubsection{Categories of motives}
\label{sec:motives_over_base}

Let $S\in\Sm_k$, let $\A$ be an oriented cohomology theory over $k$.
In this paper we consider the category of $\A$-motives over the base $S$, denoted $\CM_{\A}(S)$.
In particular, $\CM_{\CH}(\Spec k)$ is the category of Chow motives over a field $k$ (see e.g.\ \cite[Ch.~XII]{EKM}).
The category of relative Chow motives appeared in \cite[Sec.~1]{DenMurre} (with $\QQ$-coefficients), as well as in \cite[Sec.~2]{VishZai}.
The construction of $\CM_{\A}(S)$ is the same as for Chow motives, with $\A$ replacing $\CH$,
and goes back to Grothendieck (see~\cite{Manin}). 

The category of motives comes together with the motive functor:
$$ \SmProj_S \xrarr{\MA{-/S}} \CM_{\A}(S),$$
and the morphisms between $\MA{X/S}$ and $\MA{Y/S}$ for $X,Y\in \SmProj_S$  equidimensional
are given by the group $\A^{\dim Y}(X\times_S Y)$.
In the case when $S=\Spec F$ for some field $F$, we denote the category of motives as $\CM_{\A}(F)$
and the motive functor as $\MA{-}$.

The category $\CM_{\A}(S)$ is a rigid tensor additive category with the unit object $\un_S:=\MA{S/S}$.
The motive $\MA{\PP^1_S/S}$ splits off $\un_S$, and the complement to it is denoted by $\un_S(1)$.
One defines for $i\in \ZZ$, as usual, $\un_S(i)$ as $\left(\un_S(1)\right)^{\otimes i}$, 
where $L^{\otimes i}$ denotes $(L^{\vee})^{\otimes -i}$ if $i<0$.  
We refer to $\un_S(i)$ for any $i\in \ZZ$ as a {\sl Tate motive}; 
for $M\in \CM_{\A}(S)$ the motive $M(i):=M\otimes \un_S(i)$ is called a {\sl Tate twist} of $M$.
If the base of the motive is clear from the context, we may write $\un(i)$ instead of $\un_S(i)$. 
A motive is called  {\sl split}, if it is isomorphic to a direct sum of Tate motives. 

We will use the following functoriality of $\CM_{\A}(-)$:
if $f\colon S\rarr T$ is any morphism in $\Sm_k$,
then there exists a (monoidal) pullback functor $\mathrm{res}_{T/S}\colon\CM_{\A}(T)\rarr \CM_{\A}(S)$
that sends $\MA{X/S}$ to $\MA{X\times_S T/T}$;
if $g\colon Z\rarr W$ is smooth projective, then
there exists a push-forward functor $\CM_{\A}(Z)\rarr \CM_{\A}(W)$
that sends $\MA{X/Z}$ to $\MA{X/W}$.
Moreover, if $\A$ is an extended cohomology theory,
there are e.g.\ pullback functors $\CM_{\A}(S)\rarr \CM_{\A}(k(S))$ for irreducible $S$.

\subsubsection{Specialization from Chow motives to $\A$-motives}
\label{sec:prelim_vishik_yagita}

Let $\phi\colon\A\rarr \B$ be a morphism of oriented cohomology theories over $k$,
it induces a functor on motives $\Phi\colon\CM_{\A}(k)\rarr \CM_{\B}(k)$
that sends $\MA{X}$ to $\Mot{\B}{X}$, $X\in\SmProj_k$.
It was observed by Vishik--Yagita in \cite{VishYag}
that for some $\A, \B$ this functor induces a 1-to-1 correspondence
between isomorphism classes of objects. 
This is usually obtained via {\sl lifting of idempotents}:
if the homomorphism 
$$\Phi_X\colon\End(\MA{X})\rarr \End(\Mot{\B}{X})$$ 
is surjective, and its kernel is nilpotent, 
then one can lift idempotents along it by [loc.\ cit., Prop.~2.3].

We use it mostly in the following situation: 
when $\A$, $\B$ are free theories and $\A(k)\rarr \B(k)$ is a surjective morphism of graded rings
with the kernel concentrated in negative degrees,
then the nilpotence of the kernel of $\Phi_X$ is checked in [loc.\ cit., Prop.~2.7].
In particular, let $\A = \Omega$, $\B = \CH$ and $\phi\colon\Omega\rarr \CH$ be
the unique morphism of theories (given by the universality of $\Omega$).
Then the kernel of $\Omega(k)\cong \LL\rarr \CH(k)=\ZZ$ consists of the elements 
of the negative degree in $\LL$, and thus one gets 
that there is a 1-to-1 bijection between isomorphism classes of $\Omega$-
and Chow motives, [loc.\ cit., Cr.~2.8].
Similarly, if $\A=\BP$, $\B=\CH\otimes \Z{p}$,
or $\A=\CKn$ and $\B=\Ch$, the analogous statement holds.

If $M$ is a Chow motive, let $M_{\Omega}$ be the unique (up to isomorphism) lift
of $M$ to $\Omega$-motives.
For any oriented theory $\A$ there exists a unique morphism of theories $\phi\colon\Omega\rarr \A$
and the corresponding functor on motives sends $M_{\Omega}$ to $M_{\A}$.
The motive $M_{\A}$ is called the {\sl $\A$-specialization} of the Chow motive $M$.

\subsubsection{Rost motives}
\label{sec:prelim_rost_motives}

The Rost motives were constructed in \cite{Rost}
and played the crucial role in the proof of the Milnor conjectures.
We will work with $\Kn$-specializations of them,
the computation of which is based on \cite{VishYag}.

\begin{Prop}[{\cite[Prop.~6.2(2)]{SechSem}}]
\label{prop:prelim_L_alpha_Rost}
Let $R_\alpha$ be the Rost motive for a symbol $\alpha \in \km$.

Then the specialization $(R_\alpha)_{\Kn}$ splits off a Tate summand $\un_k$, 
let $L_\alpha$ be its complement.

Motive $L_\alpha$ satisfies the following properties:
\begin{enumerate}
    \item $(L_\alpha)_K \cong \un_K$ for a field extension $K/k$ if and only if $\alpha_K=0$;
    \item if $\alpha\neq \beta$, then $L_\alpha \ncong L_\beta$.
\end{enumerate}
\end{Prop}
\begin{proof}
In \cite[Prop.~6.2(2)]{SechSem} (1) is proved, and (2) follows from it,
since if $\alpha \neq \beta$, there exists a field extension $K$ such that $\alpha_K=0, \beta_K\neq 0$ 
($K$ can be taken to be e.g.\ the function field of the Pfister quadric corresponding to $\alpha$).
\end{proof}
\begin{Rk}
    We will show in Proposition~\ref{prop:additivity_L_alpha} that $L_\alpha^{\otimes 2}\cong \un$,
    which can also be seen by using the multiplicative properties of the Rost motive (cf. Lemma~\ref{lm:product_with_rost_motive}).
\end{Rk}

\begin{Prop}[{Rost, \cite{Rost}}]
\label{prop:prelim_motive_pfister}
Let $\alpha = \{a_1, \ldots, a_{n+1}\}\in \km$ be a symbol.
Let $p:=\bigotimes_{i=1}^{n+1} \langle1,-a_i\rangle$ be the corresponding $(n+1)$-Pfister form.

Then the Chow motive of $P$ decomposes as a direct sum of Rost motives:
$$ \MCH P \cong \bigoplus_{i=0}^{2^n-1} R_\alpha(i).$$

The $\Kn$-motive of $P$ decomposes as a direct sum of Tate motives and Tate twists of $L_\alpha$:

$$ \MKn P \cong \bigoplus_{i=0}^{2^n-1} \un\sh i \oplus  \bigoplus_{i=0}^{2^n-1} L_\alpha\sh i $$
\end{Prop}
\begin{proof}
    The $\Kn$-specialization result follows from the Chow motivic decomposition due to Rost
    by Proposition~\ref{prop:prelim_L_alpha_Rost}.
\end{proof}

\subsubsection{Invertible summands in $\Kn$-motives and $\Kn$-isotropic varieties}

In our study of invertible motives we will need the following easy lemma.

\begin{Lm}
\label{lm:Kn-isotropic_invertible_direct_summands}
    Let $X\in\SmProj_k$. Then the following are equivalent:
    \begin{enumerate}
        \item $\Mot{\Kn}{X}$ has $\un$ as a direct summand;
        \item $\Mot{\Kn}{X}$ has an invertible direct summand;
        \item there exists an element $a\in \Kn(X)$ such that $(\pi_X)_*(a)=1$;
        \item there exists a morphism $W\rarr X$ in $\SmProj_k$ with $[W]_{\Kn}=v_n^r$ for some $r$.
    \end{enumerate}
\end{Lm}
\begin{Rk}
Recall that Voevodsky defines $\nu_n$-varieties in~\cite[Sec.~4]{Voe_Zl}:
these are those $X\in \SmProj_k$ such that $[X]_{\Kn}=u\cdot v_n$ for $u\in\mathbb F_p^\times$, 
as shown in \cite[Prop.~4.4.22]{LevMor}.
\end{Rk}
\begin{proof}
Note that if we have morphisms of motives: $\un\sh i \xrarr{\alpha} \Mot{\Kn}{Y} \xrarr{\beta} \un\sh i$ for some $i$
and $\alpha, \beta \in \Kn(Y)$, then their composition is the identity (and hence defines the splitting off of $\un\sh i$) 
if and only if $(\pi_Y)_*(\alpha\cdot \beta)$ equals $1$.
This directly implies (1) $\Rightarrow$ (3).

(2) $\Rightarrow$ (3): let $L$ be an invertible direct summand of $\Mot{\Kn}{X}$,
then $L\otimes L^{\vee}\cong \un$ is a direct summand of $\Mot{\Kn}{X\times X}\sh{-\dim X}$.
By the above observation there exists an element $b\in \Kn(X\times X)$
such that $(\pi_{X\times X})_*(b)=1$. 
Therefore one can take $a=(\pra 1)_*(b)$, where $\pra 1\colon X\times X\rarr X$
is the canonical projection.

(3) $\Rightarrow$ (1): $\un \xrarr{a} \Mot{\Kn}{X} \xrarr{\pi_X} \un$ defines the splitting off of $\un$.

(1) $\Rightarrow$ (2): trivial.

Finally, recall that $\Omega(X)$ is generated as an abelian group by classes $[W\rarr X]_\Omega$ 
of projective morphisms, $W\in\Sm_k$,
by \cite[Prop.~3.3.1]{LevMor}.
Therefore $\Kn(X)$ is generated as $\F{p}[v_n,v_n^{-1}]$-module by classes $[W\rarr X]_{\Kn}$. 
This explains (3) $\Leftrightarrow$ (4).
\end{proof}

If one of the equivalent conditions of this lemma holds, 
$X$ is called {\sl $\Kn$-isotropic},
and otherwise {\sl $\Kn$-anisotropic} (following Vishik~\cite{VishIso, VishIso2}, cf. a general definition in Example~\ref{ex:systems_of_relations}).

\subsubsection{Semi-simple category of numerical $\Kn$-motives}
\label{sec:prelim_num_kn-motives}

We will also consider categories of motives associated to the numerical Morava K-theory.
Recall that for $X\in\SmProj_k$ 
an element $x\in \Kn(X)$ is called numerically trivial, denoted $x\sim_{\num} 0$, 
if $(\pi_X)_*(x\cdot y) =0$ for all $y\in \Kn(X)$. 
One defines $\Kn_{\num}(X)$ as the quotient of $\Kn(X)$
by the ideal of numerically trivial elements.
This allows to define the category $\CM_{\Kn_{\num}}(k)$,
since its definition only uses $\Kn_{\num}$ on smooth projective varieties. 
By~\cite[Prop.~6.1]{DuVish} this category is semi-simple.
In Section~\ref{sec:quotients_oriented_theories} we explain how $\Kn_{\num}$ can be extended to an oriented cohomology theory over $k$
(note, however, that $\Kn_{\num}$ is not an extended cohomology theory).

\begin{Lm}
\label{lm:numerical_tate_summands_Kn}
Let $X\in \SmProj_k$.
Then the number of Tate summands $\un(i)$ in $\Mot{\Kn_{\num}}{X}$
equals the dimension of $\Kn_{\num}^{\dim X-i}(X)$ and the dimension of $\Kn_{\num}^{i}(X)$ .
\end{Lm}
\begin{proof}
    By semi-simplicity of the category of the numerical $\Kn$-motives,
    every non-trivial morphism $\un(i)\rarr \Mot{\Kn_{\num}}{X}$ has a splitting,
    and similarly, for $\Mot{\Kn_{\num}}{X} \rarr \un(i)$, from which the claim follows. 
\end{proof}

\subsection{Steenrod and symmetric operations on cobordism}\label{sec:prelim_symm}

Recall that for a prime $p$, a set of representatives $\ibar=\{i_1,\ldots,i_{p-1}\}$ of all non-zero cosets modulo $p$, and a smooth projective $U$ of positive dimension, Vishik defines in~\cite[Definition~7.5]{VishSymAll} an element
$$
\eta_{p,\ibar}(U)=-\cfrac{\mathrm{deg}\big(c(-\TB U;\,\ibar)\big)}{p},
$$
where $\TB U$ is the tangent bundle of $U$, and 
$c(-\TB U;\,\ibar)=\prod_{r=1}^{p-1}c(-\TB U;\,i_r)$ is the product of Chern polynomials $c(-\TB U;\,t)\in\CHt{U}[t]$ evaluated at $t=i_r$.

He shows that $\eta_{p,\ibar}(U)\in\ZZ[\mathbf{i}^{-1}]$ for $\mathbf{i}=\prod_{r=1}^{p-1}i_r$, 
and only depends on the class $[U]\in\Laz$~\cite[Prop.~7.6]{VishSymAll},
and that the class $\overline\eta_p(U)$ of $\eta_{p,\ibar}(U)$ in $\ZZ[\mathbf{i}^{-1}]/p=\ZZ/p$ 
does not depend on the choice of $\ibar$~\cite[Prop.~7.8]{VishSymAll}. In other words, assuming additionally $\overline\eta_p([\mathrm{pt}])=1$, we obtain a homomorphism of abelian groups $\overline\eta_{p}\colon\Laz\rightarrow\ZZ/p$. For a fixed $p$, we will denote 
$$
\eta\colon\BPpt\rightarrow\Lazp\xrightarrow{\overline\eta_p}\ZZ/p
$$ 
the $\BP$-version of this map. Recall that $\BPpt\cong\Zp[v_1,\ldots, v_n, \ldots ]$ where $v_n$ can be identified with the class of a $\nu_n$-variety.

\begin{Prop}
\label{prop:eta} 
$\ $
\begin{enumerate}
\item
Let $U$ be a product of smooth projective varieties of positive dimension, then $\overline\eta_{p}(U)=0$.
\item
The kernel of $\eta\colon\BPpt\rightarrow\ZZ/p$ is generated over $\Zp$ by $p$, $p\cdot v_r$, $r>0$, and all monomials in $v_r$, $r>0$, that contain at least two factors.
\end{enumerate}
\end{Prop}
\begin{proof}$\ $
\begin{enumerate}
\item
Let $U=U_1\times U_2$ with  projections $\pi_i\colon U\rightarrow U_i$. Since  
$
c(-\TB U;\,t)=\pi_1^*c(-\TB{U_1};\,t)\cdot\pi_2^*c(-\TB{U_2};\,t),
$
and $\mathrm{deg}(\pi_1^*u_1\cdot\pi_2^*u_2)=\mathrm{deg}(u_1)\cdot\mathrm{deg}(u_2)$ for $u_i\in\CHt{U_i}$, 
we get $\eta_{p,\ibar}(U)=-p\cdot\eta_{p,\ibar}(U_1)\cdot\eta_{p,\ibar}(U_2)$.
\item
For $r>0$ one has $\eta(v_r)\neq0$ by~\cite[Prop.~7.8]{VishSymAll}, therefore the claim follows from (1).
\end{enumerate}
\end{proof}

The numbers $\eta_{p,\ibar}(U)$ appear in computations of the ``Chow traces'' of Vishik's symmetric operations. 
Recall that an operation (resp., additive, multiplicative) between theories $A^*\xrightarrow{G}B^*$ is a natural transformation of the corresponding presheaves of sets (resp., abelian groups, rings). 

To each {\sl multiplicative} operation $\Af\xrightarrow{G}\Bf$ one assigns a morphism of the corresponding formal group laws $(G_{\mathrm{pt}},\,\gamma_G)\colon (\Apt,\,\FGL A)\rightarrow(\Bpt,\,\FGL B)$, where $\gamma_G(x)\in x\Bpt[[x]]$ describes the action of $G$ on the Chern class of a line bundle $\LB$: 
$G\big(c_1^A(\LB)\big) = \gamma_G\big(c_1^B(\LB)\big)$. Vishik proves in~\cite[Th.~6.9]{Vish1} that the above assignment induces an equivalence of the category of {\sl free} theories and multiplicative operations and the category of formal group laws.

For instance, the total Steenrod operation $\StCh\colon\Chf\rightarrow\Chf[[t]]$ corresponds to $\gamma(x)=-t^{p-1}x+x^p$. 
Steenrod operation were defined by Voevodsky in the case of motivic cohomology~\cite{VoevSteen}, and by Brosnan by a different construction using $\mathbb Z/p$-equivariant Chow groups~\cite{Brosnan}. A simplified construction for $p=2$ is also given in~\cite{EKM}. 
The individual Steenrod operations $S^r|_{\Ch^m}$ are the coefficients of $\StCh|_{\Ch^m}$ at $t^{(m-r)(p-1)}$ (see~\cite[Sec.~6.4]{Vish1} for more details).

Taking $\gamma(x)=x\prod_{k=1}^{p-1}\big(x+_\Omega (i_k\cdot_\Omega t)\big)$, we obtain a lift $\StCob$ of the total Steenrod operation to algebraic cobordism:
\begin{center}
\begin{tikzcd}
  \Cobpf \arrow[r, "\StCob"] \arrow[d, "\mathrm{pr}"] & \Cobpf[[t]][t^{-1}]\arrow[d, "\mathrm{pr}"]\\
  \Chf\arrow[r, "\StCh"] &   \Chf[[t]][t^{-1}], \\
\end{tikzcd}
\end{center}
where the vertical arrows $\mathrm{pr}$ are the morphisms of theories.

Let $\square^p$ denote the (non-additive) operation of the $p$-th power. 
Vishik proves in~\cite[Th.~7.1]{VishSymAll} (see also \cite[Sec.~6]{Vish2}), that the ``part'' of $\square^p-\StCob$ corresponding to {\sl non-positive} powers of $t$ is canonically divisible by $[p]:=p\cdot_\Omega t/t$, i.e., that there exists a unique operation $\PhCob\colon\Cobpf\rightarrow\Cobpf[t^{-1}]$, such that 
$$
\left(\square^p-\StCob-[p]\cdot\PhCob\right)(\Cobf)\subseteq t\,\Cobf[[t]].
$$
The operation $\PhCob$ is called the {\sl symmetric operation} (on algebraic cobordism). It is not multiplicative, but it becomes additive if one cuts off the component of degree zero~\cite[Prop.~7.12]{VishSymAll}:
\begin{equation}
\label{eq:phi-add}
\PhCob(u+v)=\PhCob(u)+\PhCob(v)+\sum_{j=1}^{p-1}\frac{1}{p}{p\choose j}u^jv^{p-j}.
\end{equation}

The above operations can be extended from $\Cobpf$ to $\BPf$ since the latter is obtained by a multiplicative projector: 
$$
\StBP\colon\BPf\rightarrow\Cobpf\xrightarrow{\StCob}\Cobpf[[t]][t^{-1}]\rightarrow\BPf[[t]][t^{-1}],
$$
and similarly for $\Phi$ (see~\cite[Sec.~3]{VishCob} for more details).

We denote by $\phi$ the $\mathrm{Ch}$-trace of $\PhBP$, that is, the composition
$$
\phi\colon\BPf\xrightarrow{\PhBP}\BPf[t^{-1}]\xrightarrow{\mathrm{pr}}\Chf[t^{-1}],
$$
and similarly, by $\mathrm{st}$ 
the $\Ch$-trace of $\StBP$. Following~\cite{VishSymAll}, we also denote by $\phi^{t^r}$, $\mathrm{st}^{t^r}$ the coefficients of these operations at $t^{-r}$. The following result is proved in~\cite[Prop.~3.15]{VishSym2} for $p=2$, and in~\cite[Prop.~7.9]{VishSymAll} for an arbitrary $p$.
\begin{Prop}
\label{prop:symm_divide}
Let $X\in\Smk$, $r\geq0$, $v\in\BPt X$ and $u=[U]\in\BP^{-d}(\mathrm{pt})$, $d>0$. Then
$$
\phi^{t^r}(u\cdot v)=\eta(u)\cdot\mathrm{st}^{t^{r-pd}}(v).
$$
In particular, $\phi^{t^{k}}(u\cdot v)=\eta(u)\cdot\mathrm{pr}(v)$ for $k=pd-(p-1)\mathrm{codim}(v)\geq0$.
\end{Prop}

In other words, the symmetric operations allow to ``divide'' $u\cdot v$ by $u$ as long as $\eta(u)$ is not zero (cf.~\cite[Prop.~3.5]{SechCob}). 
We also summarize some elementary consequences of~\eqref{eq:phi-add}. 

\begin{Prop}
\label{prop:phi-add}
Let $X\in\Smk$, $u$, $v\in\BPt X$, and $r,\,m>0$. Then one has:
\begin{enumerate}
\item
$\phi^{t^r}(u+v)=\phi^{t^r}(u)+\phi^{t^r}(v)$;
\item
$\phi^{t^0}(u+v_mv)=\phi^{t^0}(u)+\phi^{t^0}(v_mv)$;
\item
$\phi^{t^0}(u+pv)=\phi^{t^0}(u)-\mathrm{pr}(v)^p$.
\end{enumerate}
\end{Prop}

To use Proposition~\ref{prop:symm_divide} for quadrics, we also recall the action of the Steenrod operation~\cite[Cor.~78.5]{EKM}.

\begin{Prop}\label{prop:St-Ch-quad}
Let $Q$ be a split smooth projective quadric of dimension $D$, and $h^i$, $l_i$ 
form the standard basis of $\Cht Q$ of Section~\ref{sec:prelim_A_split_quadric}, cf.~\cite[Prop.~68.1]{EKM}.
Then the Steenrod operation acts as follows:
$$
\StCh(h^i) = h^i(t+h)^i, \quad  \StCh(l_i) = l_i(t+h)^{D+1-i}t^{-1}. 
$$
\end{Prop}

\subsection{Motives of quadrics}

Let $q$ be a quadratic form and let $Q$ be the corresponding projective quadric. 
Let $D$ be the dimension of $Q$, $d=\lfloor\frac{D}{2}\rfloor$, i.e.\ $D=2d$ if $D$ is even and $D=2d+1$ if $D$ is odd.

\subsubsection{Motives of isotropic quadrics}\label{sec:prelim_mot_iso_quad}

Let $\A$ be an oriented theory.

\begin{Lm}
\label{lm:motive_of_isotropic_quadric}
Let $q=q'\perp\hyp$.
The $\A$-motive of $Q$ is decomposed as follows:
\begin{equation}
\label{eq:iso-quad}
\MA Q\cong \unA\oplus\MA{Q'}\sh 1\oplus\unA\sh {D},
\end{equation}
where $Q'$ is the projective quadric defined by $q'$.
\end{Lm}
\begin{proof}
    For $\A=\CH$ this result is due to Rost~\cite[Prop.~2]{Rost}, see also~\cite[Prop.~2.1]{Vish-quad}.
    For an arbitrary $\A$ it follows by specialization of motives, see Section~\ref{sec:prelim_vishik_yagita}.
\end{proof}

Applying~ Lemma~\ref{lm:motive_of_isotropic_quadric} inductively we get that the $\A$-motive of a split quadric $Q$ is split:
\begin{equation}
\label{eq:hyper-quad}
\MA Q\cong
\begin{cases}
\bigoplus_{i=0}^{\dim Q}\unA\sh i,&\text{for $D$ odd},\\
\unA\sh {d}\oplus\bigoplus_{i=0}^{\dim Q}\unA\sh i,&\text{for $D$ even}.
\end{cases}
\end{equation}

\subsubsection{Chow MDT and connections}\label{sec:ch-mdt-conn}

In the present section let $\A=\Ch=\CH/2$.  
Assume first that $Q$ is split. 
Following~\cite[Sec.~4]{Vish-quad}, we can choose the decomposition~\eqref{eq:hyper-quad} in a certain fixed way. 

Recall that $l_i\times h^i$ for $0\leq i\leq d$, $h^i\times l_i$ for $0\leq i<\frac{D}{2}$, 
and $h^d\times l_d+\mathrm{deg}(l_d^2)\,h^d\times h^d$ for $D$ even,
determine the decomposition of the diagonal $\Delta\in\Ch^D(Q\times Q)$ 
into a sum of mutually orthogonal projectors~\cite[Lm.~73.1, Th.~94.1]{EKM}. 
This decomposition of $\Delta$ determines the decomposition of $\MCh Q$ of the form~\eqref{eq:hyper-quad}.
If $D$ is even, following~\cite[Sec.~4]{Vish-quad}, we call the Tate summand $\unCh\sh d$, corresponding to $l_d\times h^d$, {\sl upper}, and the one, corresponding to $h^d\times l_d+\mathrm{deg}(l_d^2)\,h^d\times h^d$, {\sl lower}.

Let now $Q$ be non-split, and denote by $\Lambda(Q)$ the set of Tate summands of $\MCh{\overline Q}$ specified above. 
For an arbitrary direct summand $N$ of $\MCh Q$, 
there exists a direct summand $N'\cong N$ 
and a subset $\Lambda(N)\subseteq\Lambda(Q)$ such that $\overline{N'}$ is a sum of Tate motives from $\Lambda(N)$~\cite[Th.~5.6]{Vish-quad},
moreover, $\Lambda(N)$ does not depend on the choice of $N'$~\cite[Lm.~4.1]{Vish-quad}. 
The number of Tate motives in $\Lambda(N)$ is called the rank of $N$. 

We say that $L$, $L'\in\Lambda(Q)$ are {\sl connected} if there exists an indecomposable summand $N$ of $Q$
such that both $L$, $L'$ belong to $\Lambda(N)$. 
The set of connected components of $\Lambda(Q)$ is called the {\sl motivic decomposition type} of the Chow motive of $Q$ (Chow MDT for short).

One usually visualizes Chow MDT by denoting each Tate motive from $\Lambda(Q)$ by a 
``\,\begin{tikzpicture}
\filldraw [white] (0,-0) circle (1pt);
\filldraw [black] (0,0.05) circle (1pt);
\end{tikzpicture}\,'', 
and connecting the 
\,\begin{tikzpicture}
\filldraw [white] (0,-0) circle (1pt);
\filldraw [black] (0,0.05) circle (1pt);
\end{tikzpicture}\,’s 
corresponding to connected Tate motives. The twists of Tate motives are increasing from left to right, and if $D$ is even, one puts the upper $\unCh\sh d$ above the lower one, see~\cite[Sec.~4]{Vish-quad} for more details.

\subsubsection{Rational projectors in $\Kn$-motives of quadrics} \label{sec:rat-proj-quad}

Let us now return to the case $\A= \Kn$, where $n>1$, and denote:
\begin{align*}
D'&=D-2^n+1,\\
d'&=D'-d.
\end{align*}
In particular, we get $h^d=l_d+\widetilde{l_d}+v_nl_{d'}$, see Section~\ref{sec:prelim_A_split_quadric}.

Recall from~\cite[Sec.~7.1]{LPSS}, that the projectors $\pi_i=v_n^{-1}\cdot h^i\times h^{D'-i}$ for $0\leq i\leq D'$, together with
$$
\varpi_j=(h^j+v_nl_{D'-j})\times(l_j+v_n^{-1}h^{D'-j})
$$ 
for $d'\leq j\leq d-1$ or $j=d$, $D\equiv 2\mod4$, and
$$
\varpi_d=(h^d+v_nl_{d'})\times(\widetilde{l_d}+v_n^{-1}h^{d'})
$$
for $D\equiv 0\mod 4$, define a decomposition of the diagonal $\Delta\in\Kn^D(\overline{Q\times Q})$ into a sum of $2d+2$ mutually orthogonal Tate projectors. 
Note also that $\unKn\sh{2^n-1}$ is canonically isomorphic to $\unKn$ by the ``periodicity'' of $\Kn$.

If there exists an element in $\Kn^{D}(Q\times Q)$ that maps to a projector $\rho$ in $\Kn^{\,D}(\overline{Q\times Q})$,
then by the Rost Nilpotence Property (\cite{GilleVishik}, cf.~Section~\ref{sec:mot_dec_base_change}),
there exists a projector in $\Kn^{D}(Q\times Q)$ that maps to $\rho$. 
We will call $\rho$ a rational projector, and 
denote by $\MKn{Q,\,\rho}$ the corresponding motivic summand of $\MKn Q$. 
It follows also from the Rost Nilpotence Property 
that the isomorphism class of $\MKn{Q,\,\rho}$ is uniquely defined. Over $\overline k$ it  becomes isomorphic to a sum of Tate motives, and the number of Tate motives in this decomposition is called the rank of $\MKn{Q,\,\rho}$. 

The projectors $\pi_i$, $0\leq i\leq D'$ together with the complement projector
$$
\Pi=\Delta-\sum_{i=0}^{D'}\pi_i=\sum_{i=d'}^d\,\varpi_i
$$
in $\Knt{\overline Q\times\overline Q}$ are rational. Let 
$$
\Mker Q=\MKn{Q,\,\Pi}
$$ 
denote the motivic summand of $\MKn Q$ corresponding to $\Pi$. We will call $\Mker Q$ {\sl the kernel summand} of $\MKn Q$.
It is proved in~\cite[Th.~1.2]{LPSS} that for a general quadric $\Mker Q$ is indecomposable.

We remark that for $n=1$ elements $v_1^{-1}\cdot h^i\times h^{D-1-i}$ are also projectors, but not mutually orthogonal~\cite[Prop.~4.2]{LPSS}, therefore we assume $n>1$ in this section. However, 
$\Mot{\K1}{Q}$ always has $D$ Tate summands, and we denote by $\widetilde{\,\mathrm M}_{\K1}(Q)$ the complementary summand, 
cf. Section~\ref{sec:K0-motive-quadric} below.

Observe that the $\Kn$-kernel of $\MKn{Q}$ becomes isomorphic to 
$\oplus_{i=0}^{m} \un\sh i$ over $\overline{k}$ for $m=\min\{D,2^{n}-2\}$ 
if $\dim Q'$ is odd,
or to 
$\un\sh d \oplus \bigoplus_{i=0}^{m} \un\sh i$ 
if $\dim Q'$ is even. In particular, all Tate summands of $\Mker{\overline Q}$ are different in the odd-dimensional case,
and $\un\sh d$ is the only one with multiplicity $2$ in the even-dimensional case.

Let $\tau\in\Enda{\MKn Q}=\Knt{Q\times Q}$ denote the morphism induced by the graph of a reflection 
(with any $k$-rational center in $\mathbb P^{D+1}\setminus Q$). We will need it only when $D$ is even.
Note that $\tau$ stabilizes $h^i$, and $\overline\tau$ stabilizes $l_i$ for $i<d$, 
and interchanges $l_d$ and $\widetilde{l_d}$. In particular, we can consider $\tau$ as an endomorphism of $\Mker Q$.

\begin{Prop}
\label{prop:normal}
For any indecomposable summand $N$ of $\Mker Q$ there exist $I\subseteq\{d',\ldots,d\}$ such that $\sum_{i\in I}\varpi_i$ is rational and 
$$
N\cong\MKn{Q,\,\sum_{i\in I}\varpi_i}.
$$
\end{Prop}

\begin{proof}
Recall that $\mathrm{Hom}\big(\unKn\sh i,\,\unKn\sh j\big)=\Kn^{j-i}(k)$. Clearly, if $D$ is odd, then $\varpi_i$ define non-isomorphic Tate summands. In other words, $\Enda{{\Mker Q}_{\,\overline k}}\cong \mathbb F_2\cdot\varpi_{d'}\times\ldots\times\mathbb F_2\cdot\varpi_d$, and every idempotent in this ring has form $\sum_{i\in I}\varpi_i$ for some subset $I\subseteq\{d',\ldots,d\}$.

Next, consider the case when $D=2d$ is even, and let $\rho$ denote the rational idempotent corresponding to $N$. Now 
\begin{align}
\label{D=2mod4}
\Enda{{\Mker Q}_{\,\overline k}}\cong \mathbb F_2\cdot\varpi_{d'+1}\times\ldots\times\mathbb F_2\cdot\varpi_{d-1}\times\Enda{\MKn{\overline Q,\,\varpi_{d'}}\oplus\MKn{\overline Q,\,\varpi_d}},
\end{align}
where the last factor is isomorphic to $\mathrm M_{2\times2}(\mathbb F_2)$. Let $p$ denote the restriction of $\rho$ to this factor. If $p=0$, $\varpi_{d'}$, $\varpi_{d}$ or $\varpi_{d'}+\varpi_{d}$, the claim follows. 
If not,  
consider a graph of a reflection $\tau\in\Enda{\MKn Q}$, and 
for a projector $\sigma\in\Knt{\overline Q\times\overline Q}$ let us denote by $\widetilde\sigma$ its image under $\overline\tau$. Clearly, if $\widetilde\sigma$ is rational, $\sigma$ is rational as well. In particular, if $p=\widetilde\varpi_{d'}$ or $\widetilde\varpi_{d}$, then $N\cong\MKn{Q,\,\widetilde\rho}$ via $\tau$, and the claim follows.

The only two remaining projectors of $\mathrm M_{2\times2}(\mathbb F_2)$ correspond to $\sigma=(l_d+v_n^{-1}h^{d'})\times(\widetilde l_d+v_n^{-1}h^{d'})$ and $\widetilde\sigma$ in the case $D\equiv2\mod4$, or $\sigma=(l_d+v_n^{-1}h^{d'})\times(l_d+v_n^{-1}h^{d'})$ and $\widetilde\sigma$ in the case $D\equiv0\mod4$. If $p=\sigma$ or $\widetilde\sigma$, note that $\rho+\widetilde\rho=\varpi_{d'}+\varpi_d$ is rational, and therefore $\rho=p$ using that $N$ is indecomposable. In this situation, $\Mker Q$ has two isomorphic unary summands, which are twisted forms of $\unKn\sh d$. 

It remains to observe that $\overline\tau+\sigma$ is rational and defines an isomorphism between $\sigma$ and $\varpi_{d}$, therefore $N\cong\MKn{Q,\,\varpi_{d}}$.
\end{proof}

\begin{Cr}
Decomposition of $\Mker Q$ into a sum of indecomposable motives can be chosen in such a way that all summands have form $\MKn{Q,\,\sum_{i\in I}\varpi_i}$ for some $I\subseteq\{d',\ldots,d\}$.
\end{Cr}
\begin{proof}
In a decomposition of $\Mker Q$ choose an indecomposable $\MKn{Q,\,\rho}$ such that $\rho$ is {\sl not} orthogonal to $\varpi_{d'}+\varpi_d$, 
and choose $\rho'=\sum_{i\in I}\varpi_i$ such that $\MKn{Q,\,\rho'}$ is isomorphic to $\MKn{Q,\,\rho}$ by Proposition~\ref{prop:normal}. Decompose
$$
\Mker Q\cong\MKn{Q,\,\rho'}\oplus\bigoplus_k N_k
$$
for indecomposable $ N_k=\MKn{Q,\,\rho_k}$. Since $\rho_k$ are orthogonal to $\rho'$, they also have the required form, cf.~\eqref{D=2mod4}.
\end{proof}

\subsubsection{Rational isomorphisms between $\Kn$-motives of quadrics}
\label{sec:prelim_rat_iso_kn_quadric}
We now describe a simple form of  rational isomorphisms between direct summands of the $\Kn$-motives of two quadrics, $n>1$. 
For $j=1,2$ let $Q_j$ be a projective quadric of dimension $D_j$,
let $I_{j}\subseteq\{d_j',\ldots,d_j\}$
be such that $\sum_{i\in I_j}\varpi_i$ is a rational $\Kn$-projector on $Q_j$,
denote $N_j:=\MKn{Q_j,\,\sum_{i\in I_j}\varpi_i}$.
We assume for simplicity $D_j\not\equiv 0\mod 4$ for $j=1,2$, since that is sufficient for our applications.
Let us denote 
by $a_i$ 
the element $h^i + v_n l_{D_1'-i}$ in $\Kn(\overline{Q_1})$
for $i$ from $d'_1$ to $d_1$, 
in particular, 
the projector $\varpi_i$ on $\overline{Q_1}$ 
is equal to $v_n^{-1}\, a_i \times a_{D_1'-i}$. 
Similarly, 
let $b_i$ denote the element $h^i + v_n l_{D_2'-i}$ in $\Kn(\overline{Q_2})$, 
and the projector $\varpi_i$ on $\overline{Q_2}$  equals to $v_n^{-1}\, b_i \times b_{D_2'-i}$,
for 
$i$ from $d'_2$ to $d_2$. 

\begin{Prop}
\label{prop:quadric_kn-iso_normal_form}
In the notation above assume that for some 
$s$ such that $0\leq s<2^n-1$ 
there is an isomorphism $N_1\xrightarrow{\sim} N_2\sh s$.

\begin{enumerate}
    \item If $D_1, D_2$ are even, $s=0$, and $\{d_j', d_j\}\subseteq I_j$ for $j=1,2$,
    then there exists an isomorphism $N_1\cong N_2$ that over $\overline{k}$ has the following form:
        $$
        \psi+\sum_{i\in I_1\setminus\{d_1', d_1\}} v_n^{-i} a_i \times b_{D_2'-i},
        $$
    where $\psi$ either equals $v_n^{-1}(a_{d'}\times b_{d}+a_{d}\times b_{d'})$ or $v_n^{-2}\,a_{d'}\times b_{d'}+a_{d}\times b_{d}$. 

    \item Otherwise, an isomorphism $N_1\xrightarrow{\sim} N_2\sh s$ over $\overline{k}$ always has the form:
    $$\sum_{i\in I_1} v_n^{-1}\, a_i \times b_{D_2'-f(i)}, $$
    where $f\colon I_1\rarr I_2$ is a bijection such that $f(i)\equiv i-s \mod (2^n-1)$ for all $i\in I_1$.
\end{enumerate}
\end{Prop}
\begin{proof}
    The group of morphisms between two Tate motives $\MKn{\overline Q_1,\, \varpi_i}$ and $\MKn{\overline Q_2,\, \varpi_j}\sh{s}$ 
    is either zero if $j-s\neq i \mod 2^{n}-1$ or $\ZZ/2$ with the non-zero element being
    the isomorphism $a_i\times b_{D_2'-j}$. 
    Thus, if the multiplicity of $\un(t)$ in $\overline{N_j}$ is at most 1 for $j=1,2$ and for all $t\colon 0\le t\le 2^n-1$,
    then an isomorphism between $(N_1)_{\overline{k}}$ and $(N_2)_{\overline{k}}\sh{s}$ is unique
    and has the form as in (2). 
    Moreover, the condition that the multiplicity of $\un(s)$ in $\overline{N_j}$ is more than 1
    can be met only in the case (1). 
    In other words, it only remains to prove the proposition in this case.

    It follows from the explanation above that an isomorphism $(N_1)_{\overline{k}}\xrarr{\sim} (N_2)_{\overline{k}}$
    has the form as in the claim, 
    where $\psi$ is an isomorphism between $\MKn{\overline Q_1,\, \varpi_{d_1'}+ \varpi_{d_1}}$ and $\MKn{\overline Q_2,\, \varpi_{d_2'}+\varpi_{d_2}}$.
    The group of morphisms between these motives is isomorphic to $\mathrm{M}_{2\times 2}(\F{2})$ and contains $6$ isomorphisms.
   
    We can reduce the possible values of $\psi$ by acting on this group by the reflections of $\overline Q_1$, $\overline Q_2$
    (see Section~\ref{sec:rat-proj-quad} above).
    For example, if $\psi$ has the form $v_n^{-1}\,a_{d'}\times b_d+a_d\times b_d+v_n^{-2}\,a_{d'}\times b_{d'}$,
    then by composing it with $\overline{\tau_1}$ we get  $v_n^{-2}\,a_{d'}\times b_{d'}+a_{d}\times b_{d}$.
    We leave the computational details to the reader.
\end{proof}

\subsubsection{$\KK$-motives of quadrics}
\label{sec:K0-motive-quadric}

To every quadratic form $q$ one can associate its Clifford invariant $e_2(q) \in \mathrm{H}^2(k,\,\ZZ/2)$:
in the case of odd-dimensional quadratic forms it's the class of the even part of the Clifford algebra $C_0(q)$,
and in the case of even-dimensional quadratic forms it's the class of the whole Clifford algebra $C(q)$ (where we
forget the grading). Moreover, to an element $\alpha$ in $\mathrm{H}^2(k,\,\ZZ/2)$ 
one can associate an invertible $\KK$-motive $L_\alpha$. If $\alpha = [D]$ for $D$ a CSA over $k$,
then $\KK(L_\alpha) \cong \KK(D)$ and $L_\alpha$ is a direct summand of the $\KK$-motive of the Severi--Brauer variety
corresponding to $D$ by \cite{Quillen, Panin}.

One can describe the $\KK$-motive of $Q$ due to the work of Swan \cite{Swan}, cf.~\cite{Panin}.
If $Q$ has odd dimension $D$, then $\Mot{\KK}{Q} \cong \un^{\oplus D} \oplus L_{e_2(q)}$.
If $Q$ has even dimension $D$, let $E:=k(\sqrt{\det_{\pm}(q)})$ (where $\det_{\pm}(q)$ is the discriminant of $q$,
i.e.\ $E$ is either a quadratic field extension of $k$ or $k\times k$),
then $\Mot{\KK}{Q} \cong \un^{\oplus D} \oplus \Mot{\KK}{\Spec E}\ot L_{e_2(q)}$.

\subsubsection{Binary summands of Chow motives of quadrics}
\label{sec:prelim:izh-vish}

To study elements in the Galois cohomology associated with the motives of quadrics, we will need the following result.

\begin{Th}[Izhboldin--Vishik]
\label{th:binary_chow_motives}
Let $Q$ be a projective quadric over $k$, 
and assume that its Chow motive $\MCH Q$ has a binary summand $B$ of length $2^n-1$. 
Then there exists an element $\alpha\in\mathrm H^{n+1}(k,\,\mathbb Z/2)$
 such that for any field extension $K/k$ the following are equivalent:
\begin{enumerate}
\item
$\alpha_K=0$;
\item
$B_K$ is split.
\end{enumerate}
\end{Th}

The proof of this theorem in the case when $B$ is an upper summand of $\MCH Q$
can be found in \cite[Th.~6.9]{IzhVish}.
It is based on studying the action of the Milnor operations on the motivic cohomology of \v{C}ech schemes -- 
a technique introduced by Voevodsky for proving the Bloch--Kato conjecture.
An application of these techniques for quadrics 
was developed 
by Vishik in his thesis~\cite{VishInt},
and the proof of \cite[Th.~6.9]{IzhVish} is based on it.
Moreover, a proof of Theorem~\ref{th:binary_chow_motives}
in the given generality was known to Vishik, although, it did not appear in the literature.
Here we collect published results sufficient to establish the claim, without asserting any originality.

Let  $\mathrm{DM}(k,\,R)$ denote the Voevodsky triangulated category of motives with coefficients in a ring $R$ \cite{MVW},
it contains $\CM_{\CH(-,\,R)}(k)$ as a full subcategory [loc.\ cit., Prop.~20.1].
For  $X\in \Sm_k$ let $\mathcal X_X$ denote the motive of the standard simplicial scheme of $X$, see e.g.\ \cite[App.~B]{Voe_Z2}. 
We warn the reader that the Tate twist which we denote $\sh i$ in this paper corresponds to the $(i)[2i]$-twist in  Voevodsky's category.
Until the end of the next section we will only use the latter notation. 

A way to axiomatize the above mentioned machinery needed to prove results similar to Theorem~\ref{th:binary_chow_motives}
was developed in~\cite{Sem-coh-inv}.

\begin{Th}[{\cite[Th.~6.1~(b)]{Sem-coh-inv}}]
\label{th:nikita}
Let $X$ be a smooth projective irreducible variety over $k$ with no zero-cycles of odd degree. 
Assume the following:
\begin{enumerate}
    \item 
    there exists a direct summand  $M$ of $\Mot{\CH(-,\,\ZZ_2)}{X}$ such that there is an exact triangle
\begin{equation*}
\mathcal X_X(2^n-1)[2^{n+1}-2]\rightarrow M\rightarrow \mathcal X_X\rightarrow\mathcal X_X(2^n-1)[2^{n+1}-1]
\end{equation*}
in the category $\mathrm{DM}(k,\,\mathbb Z_2)$;
\item 
$M(\dim X-2^n+1)[2\dim X-2^{n+1}+2]$ is also a direct summand of $\Mot{\CH(-,\,\ZZ_2)}{X}$;
\item 
there exists a $\nu_n$-variety $Y$ with a morphism $Y\rightarrow X$. 
\end{enumerate}

Then there exists an element $\alpha\in\HH^{n+1}(k,\,\ZZ/2)$ 
such that for any field extension $K/k$ we have 
$\alpha_K=0$ if and only if $X_K$ has a zero-cycle of odd degree.
\end{Th}

To apply this theorem for a binary summand $B$ of the Chow motive of a quadric $Q$,
one shows that $B$ is isomorphic to a Tate twist of the upper motive of the appropriate orthogonal Grassmannian $X$ of $Q$.
The existence of the exact triangles required for (1) of Theorem~\ref{th:nikita} is shown for the orthogonal Grassmannians 
in~\cite{VishInt}. Condition (2) follows from a version of Poincar\'e duality for upper motives,
and (3) is ultimately related to the length of $B$ being $2^n-1$.

\begin{proof}[Proof of Theorem~\ref{th:binary_chow_motives}]
Let $\pi\in\Ch\left(\overline Q\times\overline Q\right)$ denote a rational projector corresponding to $B_{\Ch}$. 
We may assume that $\pi$ has the form $h^i\times l_i+l_j\times h^j$, 
in particular, $B_{\,\overline k}\cong\un(i)[2i]\oplus\un(D-j)[2D-2j]$. 
Since the length of $B$ equals $2^n-1$,
we have $D-j=i+2^n-1$.

We may assume that $B$ is indecomposable. 
By~\cite[Th.~3.5]{Karp-upper} 
$B_{\Ch}$ is isomorphic to a Tate twist of the upper motive of a certain projective homogeneous variety $X$ 
for the group $\mathrm{SO}(q)$, 
i.e.\
$B_{\Ch}\cong U_{\Ch}(X)(i)[2i]$. 
The restriction to $\ZZ/2$-coefficients is not essential for our purposes:
the splitting of $B_K$ is equivalent to the splitting of $(B_{\Ch})_K$  
for all field extensions $K/k$ (cf.~\cite{Haution}).
In other words, $X_K$ has a zero-cycle of odd degree
iff $U_{\Ch}(X)_K$ is split iff $B_K$ is split for any $K/k$.
We now check the conditions (1)-(3) of Theorem~\ref{th:nikita} 
to show that this splitting is controlled by an element $\alpha\in\HH^{n+1}(k,\,\ZZ/2)$, 
which finishes the proof.

\begin{enumerate}
    \item 
    Denote by $M$ the motivic summand of $\Mot{\CH(-,\,\ZZ_2)}{X}$ that lifts $U_{\Ch}(X)$~\cite{CherMer}. 
Let $Q^s$ denote the variety of $s$-planes on a quadric $Q$, i.e.\ the $s$-th orthogonal Grassmannian.
Observe that over a field extension $K/k$, the motive $B_K$ is split, iff $l_i$ is $K$-rational, iff $l_j$ is $K$-rational.
Note that rationality of the mod-$2$ cycle $l_s$ is equivalent to the existence
of a rational point on $Q^s$ by the Springer Theorem~\cite[Lm.~72.1]{EKM}, and similarly the splitting of the upper motive of $X$ over $K$
is detected by the existence of a $K$-rational point on $X$.
Therefore we have 
$\mathcal X_X\cong\mathcal X_{Q^i}\cong \mathcal X_{Q^j}$ in $\mathrm{DM}(k,\,\mathbb Z)$ by~\cite[Th.~2.3.4]{VishInt}. 
In fact, it follows that we could take $X=Q^i$ from the very beginning~\cite[Prop.~2.15]{Karp-upper}. 

An exact triangle with $B$ and $\mathcal X_{Q^i}$, $\mathcal X_{Q^j}$ is obtained
in~\cite[Lm.~3.23]{VishInt}. One checks that it has the required form, i.e.\ the shifts and Tate twists are as in (1),
using that $D-j=i+2^n-1$.

\item 
For $\ZZ/2$-coefficients the claim follows from the Poincar\'e duality for $U_{\Ch}(X)$~\cite[Prop~5.2]{Karp-poincare} 
and to lift coefficients
to $\ZZ_2$ one can appeal e.g.\ to~\cite[Prop.~4.10]{SemZhykh}. 

\item
By~\cite[Lm.~7.5]{GPS} there exists $Y$ of $\dim Y=2^n-1$ such that $U_{\Ch}(X)$ is isomorphic to its upper direct summand. 
Then $Y$ is a $\nu_n$-variety by~\cite[Lm.~6.2]{Sem-coh-inv}.
\end{enumerate}
\end{proof}

\subsection{Tensor multiplication by Rost motives}

The following result is well-known to the experts,
but since we could not find a reference, we include a proof for the sake of completeness. 
We work with $\mathrm{DM}(k,\,\ZZ)$ as in the previous section and use the notation introduced there. 

\begin{Lm}\label{lm:product_with_rost_motive}
Let $Y$ be a geometrically connected projective homogeneous variety,  
let $\alpha \in \HH^{n+1}(k,\,\ZZ/2)$ be a non-zero pure symbol,  
and $R_\alpha$ the corresponding Rost motive. 

If $\alpha_{k(Y)} = 0$, 
then 
for any direct summand $M$ of $\MCH Y$ one has 
$M\otimes R_\alpha \cong M \oplus M(2^n-1)[2^{n+1}-2]$.
\end{Lm}

\begin{proof}

Let $Q_\alpha$ denote the projective {\sl norm} quadric of dimension $2^n-1$, see e.g.\ \cite[Sec~4]{Voe_Z2}. 
Since $Q_\alpha$ has a rational point over $k(Y)$, we conclude that 
$\mathcal X_{Q_\alpha}\otimes\MCH Y\cong\MCH Y$ 
in $\mathrm{DM}(k,\,\ZZ)$ by~\cite[Th.~2.3.6]{VishInt},
where the isomorphism is induced by $\mathcal X_{Q_\alpha}\rightarrow\ZZ$.
This implies $\mathcal X_{Q_\alpha}\otimes M\cong M$ for any direct summand $M$ of $\MCH Y$. 
Therefore, after tensor multiplication by $M$, the exact triangle of~\cite[Th.~4.4]{Voe_Z2} yields an exact triangle 
$$
M(2^n-1)[2^{n+1}-2]\rightarrow M\otimes R_\alpha\rightarrow M\xrightarrow{\delta} M(2^n-1)[2^{n+1}-1],
$$
where the map $M\otimes R_\alpha\rightarrow M$ is induced by $Q_\alpha\rightarrow\Spec(k)$ (cf.~\cite[Th.~4.3]{Voe_Z2}). 
The claim of the lemma is equivalent to $\delta$ being $0$, and to prove the latter we may assume $M=\MCH Y$. 

Finally, an isomorphism $\MCH Y \oplus \MCH Y(2^n-1)[2^{n+1}-2]\xrightarrow{\cong}\MCH Y\otimes R_\alpha$ is constructed in~\cite[Cor.~3.4]{PSZ} 
in such a way that the map $\MCH Y\otimes R_\alpha \rightarrow \MCH Y$
is induced by $Q_\alpha\rightarrow\Spec(k)$ defines a projection on the first component, cf.~\cite[Proof of Lm.~3.2]{PSZ}. 
In particular, the exact triangle associated with this map is isomorphic to the one coming from the direct sum decomposition.
\end{proof}

\begin{Rk}
In fact, the assumption that $Y$ is geometrically connected and homogeneous is only used in the last paragraph.
One can drop these assumptions (i.e.\ $Y\in \SmProj_k$ can be any irreducible) 
by using instead the methods appearing in Section~\ref{sec:morava_motives_over_function_fields}.
\end{Rk}

\section{$\A$-universal surjectivity, injectivity and bijectivity of field extensions}
\label{sec:univ_surj}

In this section we present the notion of $\A$-universal surjectivity, injectivity and bijectivity
 of a field extension
for a coherent oriented cohomology theory $\A$. 
Our presentation is based on the unpublished joint work of Shinder and the second author 
and is included here with the kind permission of the former.

If $\A$ is the Chow ring theory, 
then the notions of $\A$-universal surjective and bijective field extensions 
coincide and are not new.
If $K=k(X)$ for a smooth proper $X$, 
then this property is satisfied if and only if 
$X$ is universally $\CH_0$-trivial, and the latter notion was introduced in~\cite{CH0_univ}.

For other oriented theories, e.g.\ for Morava K-theories,
the notions of $A$-universally surjective and bijective field extensions do not coincide (Example~\ref{ex:univ_surj_not_bij}).
Moreover, we are interested in this property in application to $k(Q)/k$ for $Q$ an anisotropic quadric over $k$.
Note that $Q$ does not possess a 0-cycle of degree 1, and thus cannot be $\CH_0$-universally trivial,
but, nevertheless, $k(Q)/k$ can be $\Kn$-universally surjective (Prop.~\ref{prop:Kn-univ-surj-quad}).

For $X\in \SmProj_k$ we relate 
the notion of $\A$-universally surjective field extension $k(X)/k$ 
to the $\A$-decomposition of the diagonal (Lemma~\ref{lm:A}). 
Introduced in the work of Bloch and Srinivas~\cite{BloSri},
decomposition of the diagonal for $\A=\CH$ appeared 
to be effective in applications to birational geometry~\cite{Voisin_dec, Voisin_CH2, CTP} (see~\cite[Sec.~7]{Sch-survey} for a survey).
Although in the case of quadrics the rationality problem is easy (a quadric is rational if and only if it has a rational point),
these technique allow us to find many examples of $\A$-universally surjective field extensions
for $\A$ being K-theory and Morava K-theory (Section~\ref{sec:examples_A-univ-surj}).
We also describe a criterion for $\A$-universal bijectivity (Prop.~\ref{prop:univ_bij}).

Finally, we would like to note that we develop the above notions 
in the setting of quotients of coherent theories. 
This is mainly needed to include numerical theories, e.g.\ $\Kn_{\num}$.
These are not extended oriented cohomology theories 
and do not admit pullback maps $\Kn_{\num}(X)\rarr \Kn_{\num}(X_K)$ 
for arbitrary field extensions $K/k$, so we adjust the exposition to take care of it. 
Although the full generality of these results is not strictly necessary for this paper, we plan to use them in future work,  e.g.\ in \cite{SechInv}.

\subsection{Quotients of oriented theories}
\label{sec:quotients_oriented_theories}

In~\cite[Example 4.1]{VishIso} Vishik has introduced the notion of an oriented cohomology theory with relations 
that we recall.

\begin{Constr}[Vishik]\label{constr:gamma}
 Let $A$ be an oriented cohomology theory over $k$. 
Let $\Gamma=\{(Q_\lambda, \gamma_\lambda)\}_{\lambda\in I}$ 
be a collection of smooth projective varieties $Q_\lambda$ together with elements $\gamma_\lambda \in \A(Q_\lambda)$.	
We will call it a {\sl system of relations}.

Then one defines $A_\Gamma(X)$ for $X\in \Sm_k$ as the quotient of $A(X)$ modulo the subgroup generated by the elements of the form 
$(\pi_2)_*(\pi_1^*(\gamma_\lambda)\cdot x)$ where $\pi_j$, $j=1,2$, are the canonical projections from $Q_\lambda\times X$ 
and $x\in \A(Q_\lambda\times X)$.
The functorial pullback and push-forward structures of $A$ induce the same ones for $A_\Gamma$,
thus $A_\Gamma$ becomes an oriented cohomology theory over $k$
and the projection $A\rarr A_\Gamma$ is a morphism of oriented theories.
\end{Constr} 

\begin{Ex}
\label{ex:systems_of_relations}
Let $\Gamma_{\overline{k}}$, $\Gamma_{\iso}$, $\Gamma_{\num}$
 consist of all pairs $(X, \gamma)$ where $X$ is any smooth projective variety over $k$
and elements $\gamma \in A(X)$ are defined respectively as follows: 

\begin{enumerate}
\item (elements vanishing over finite field extensions of $k$) element $\gamma$ lies in the kernel of the map 
$$\A(X)\rarr \mathop{\mathrm{colim}}_{\substack{k\subset L\subset \overline{k}\\ [L:k]<\infty}}\ \A(X\times_k L),$$
where the colimit is taken for subfields $L$ in a chosen algebraic closure of $k$.
The colimit group on the right will be by abuse of notation denoted $\A(\overline{X})$, cf. Remark~\ref{rk:overline_proj}\,(1) below.

\item (anisotropic elements, introduced by Vishik \cite{VishIso, VishIso2}) 
element $\gamma$ lies in the image of the push-forward map $A(Y)\rarr A(X)$
 for a smooth projective $A$-anisotropic variety $Y$ over $k$, i.e.\ such that $(\pi_Y)_*\colon A(Y)\rarr A(k)$ is zero;
\item (numerically trivial elements) element $\gamma$ is numerically trivial, $\gamma\sim_{\num} 0$,
i.e.\ it satisfies $(\pi_X)_*(\gamma \cdot \eta)=0$ 
for all $\eta \in A(X)$, $\pi_X:X\rarr \Spec k$ is the structure morphism.
\end{enumerate}

The corresponding quotient cohomology theories $A_\Gamma$ (over $k$) 
will be denoted by $\overline{A}, A_{\iso}, A_{\num}$
for the three cases above, respectively.
Note that if $\A$ is coherent and generically constant (\cite[Def.~4.4.1]{LevMor}), e.g.\ if $\A$ is free,
then there are canonical surjective morphisms of theories:
$$ \A \twoheadrightarrow \overline{A} \twoheadrightarrow \A_{\num}, \quad  \A \twoheadrightarrow \A_{\iso} \twoheadrightarrow \A_{\num}.$$
\end{Ex}

\begin{Rk}\label{rk:overline_proj}
\phantom{a}

\begin{enumerate}
    \item If $X\in \SmProj_k$,
    then $\overline{A}(X)\cong \mathrm{Im}\left( A(X)\rarr A(\overline{X})\right)$.
    However, in general, if $X$ is not projective, this isomorphism does not hold.    
    Moreover, if $\A$ is a free theory, then $\A(X\times_k \overline{k})$ is defined 
    (note that if $\overline{k}$ is not finite over $k$, then 
    by definition of extended oriented cohomology theory this group is not apriori defined),
    and it is indeed isomorphic to 
    $\mathop{\mathrm{colim}}_{\substack{k\subset L\subset \overline{k}\\ [L:k]<\infty}}\ \A(X\times_k L)=:\A(\overline{X})$,
    i.e.\ the abuse of notation introduced in Example~\ref{ex:systems_of_relations}, (1), 
    does not lead to ambiguity in this case.

    \item If $X\in \SmProj_k$, then $\A_{\num}(X)=\A(X)/(a\mid a\sim_{\num} 0)$. 
    Note that for non-projective $X$ the RHS is not defined.
\end{enumerate}
\end{Rk}

\begin{Ex}
Another example of a theory with relations is the conormed Chow ring,
defined in \cite{Fino}, and also studied in \cite{GeldZhykh, AnanGeld}.
\end{Ex}

Assume that $A$ is an extended cohomology theory, $K/k$ is a f.g.\ field extension and
we have two system of relations $\Gamma(k)$ and $\Gamma(K)$ for the restriction of $A$ 
to $\Sm_k$ and $\Sm_K$, respectively.
If for all $X\in\Sm_k$ we get canonical morphisms  
$$p^*_X\colon A_{\Gamma(k)}(X) \rarr A_{\Gamma(K)}(X_K)$$ 
induced from $A$,
then we will say that $\Gamma(k)$ and $\Gamma(K)$ are {\sl compatible systems of relations} in $A$.

In particular, if for every $(X,\gamma)\in\Gamma(k)$ we have that $\gamma_K \in \mathrm{Ker}\left(A(X_K)\rarr A_{\Gamma(K)}(X_K)\right)$
(e.g.\ if $(X_K, \gamma_K)\in \Gamma(K)$),
then $\Gamma(k)$ and $\Gamma(K)$ are compatible.

\begin{Prop}\label{prop:comp_sys_res}
Let $X, Y \in \Sm_k$, $Y$ is irreducible. 
Let $\A$ be a coherent cohomology theory
and $\Gamma(k), \Gamma(k(Y))$ are compatible systems of relations in $A$.

Then the pullback map $A(X\times Y) \rarr A(X_{k(Y)})$ 
induces the map $A_{\Gamma(k)}(X\times Y) \rarr A_{\Gamma(k(Y))}(X_{k(Y)})$.
\end{Prop}
\begin{proof}
Let $(Q, \gamma) \in \Gamma(k)$. By the assumption on compatibility, $\gamma_{k(Y)}$ vanishes in $\A_{\Gamma(k(Y))}(Q_{k(Y)})$.
Therefore, by extending $\Gamma(k(Y))$ if necessary without changing the quotient theory,
we may assume 
that $(Q_{k(Y)}, \gamma_{k(Y)})\in \Gamma(k(Y))$.

Consider the following  commutative diagram with a transversal square on the right:
\begin{center}
\begin{tikzcd}
	& Q \times_k X \times_k Y  \arrow[dl, "p_Q"] \arrow{dr} &			& Q_{k(Y)} \times_{k(Y)} X_{k(Y)} \arrow[ll, rightarrow] \arrow{dr} & \\
Q	&  					& X\times_k Y  &											& X_{k(Y)} \arrow[ll, rightarrow]. \\
\end{tikzcd}
\end{center}
Let $b\in A(Q \times_k X \times_k Y)$, then $p_Q^*(\gamma)\cdot b$ is an element that goes to zero in~$A_\Gamma (X\times Y)$.
When we restrict it to $Q_{k(Y)} \times_{k(Y)} X_{k(Y)}$ it becomes of the form $p_{Q_{k(Y)}}^*(\gamma_{k(Y)}) \cdot b_{k(Y)}$,
and therefore is sent to zero in $A_{\Gamma(k(Y)}(X_{k(Y)})$.
By~\ref{eq:ext_bc}
we get that an element in the kernel of $A(X\times_k Y)\rarr A_{\Gamma(k)}(X\times_k Y)$
maps to zero in $A_{\Gamma(k(Y))}(X_{k(Y)})$.
\end{proof}

The systems of relations $\Gamma_{\overline{K}}$ are compatible for all field extensions $K/k$,
and one can show that $\overline{A}$ becomes an extended cohomology theory.
However, analogous statements do not hold for $\Gamma_{\iso}$, $\Gamma_{\num}$ and $A_{\iso}$, $A_{\num}$.
For the numerical equivalence we show the following.

\begin{Prop}\label{prop:num_base_change}
Let $A$ be a free theory such that $A(k)$ is an integral domain.

Let $Y\in \SmProj_k$ be such that
there exist $a\in A(Y)$ with $\deg (a)\neq 0$, i.e.\ $Y$ is $\A$-isotropic.

Then the systems or relations $\Gamma_{\num}(k)$ and $\Gamma_{\num}(k(Y))$
are compatible.
\end{Prop}
\begin{proof}
It suffices to check that if $x\in A(X)$ is numerically trivial over $k$,
then its pullback $x_{k(Y)} \in A(X_{k(Y)})$ is numerically trivial over $k(Y)$.

Assume the contrary, i.e.\ there exists $\zeta \in A(X_{k(Y)})$
such that $c:=\deg (x_{k(Y)} \cdot \zeta) \neq 0$.
Then by (LOC) and coherence of $A$ there exists $z \in A(X\times Y)$ that maps to $\zeta$,
and by~\ref{eq:ext_bc} we have that 
$$(p_Y)_*(p^*_X(x)\cdot z) \equiv c\cdot 1_Y \mod \tau^1 A(Y)$$
where $p_X, p_Y$ denote the projections from $X\times Y$ on the components $X,Y$, resp.

Let $a\in A(Y)$ be
 an element that lies in the highest topological filtration subgroup $\tau^i A(Y)$
among the elements with $\deg (a)\neq 0$.
Then using the projection formula we have 
$$\deg_Y((p_Y)_*(p_X^*(x)\cdot z \cdot p_Y^*(a)) = \deg(a)\cdot c$$
and, in particular, it is not zero by the assumption that $A(k)$ is a domain.

Therefore using the projection formula we get that
$$\deg_X ( x \cdot (p_X)_*(p_Y^*(a)\cdot z) ) = \deg(a)\cdot c \neq 0$$
and hence $x$ is not numerically trivial.
\end{proof}
\begin{Cr}
\label{cr:num_generic_restriction}
Under the assumptions of the proposition there are natural pullback maps for every $X\in\Sm_k$

$$A_{\num}(X\times Y) \twoheadrightarrow A_{\num}(X_{k(Y)}).$$
\end{Cr}
\begin{proof}
    Apply Proposition~\ref{prop:comp_sys_res}.
\end{proof}

\subsection{Criterion for universally surjective field extensions}

Let $\A$ be a coherent oriented cohomology theory, $K/k$ be a f.g. field extension.
Let $\Gamma(k)$, $\Gamma(K)$ be compatible systems of relations.

\begin{Def}
A finitely generated field extension $K/k$ is called $A_\Gamma$-universally surjective (resp., injective, bijective)
if the map $p_X^*\colon \A_{\Gamma(k)}(X) \rarr \A_{\Gamma(K)}(X_K)$ 
is surjective (resp., injective, bijective) for all $X\in\Sm_k$.
\end{Def}

\begin{lm}[S.--Shinder, unpublished, cf.\  \cite{BloSri}]\label{lm:A}

Let $Y\in \SmProj_k$ be irreducible, 
let $A$ be a coherent theory, 
and let $\Gamma(k)$, $\Gamma(k(Y))$ be compatible systems of relations in $A$. 

Then TFAE:
\begin{enumerate}
    \item  $k(Y)/k$ is $A_\Gamma$-unviersally surjective;
    \item  the map $A_{\Gamma(k)}(Y)\rarr A_{\Gamma(k(Y))}(Y_{k(Y)})$ is surjective;
    \item  the image of $A_{\Gamma(k)}(Y) \rarr A_{\Gamma(k(Y))}(Y_{k(Y)})$ contains the class of the ``diagonal'' point 
    $\delta_Y$  (i.e.\ $[\Spec k(Y) \xrarr{(j,\mathrm{id})} Y\times_k \Spec k(Y)]_{A_\Gamma}$
     where $j$ is the inclusion of the generic point).
\end{enumerate}
\end{lm}

\begin{proof}
We prove the last implication in the sequence $(1)\Rightarrow (2) \Rightarrow (3)\Rightarrow (1)$, as the others are trivial.
Below we drop the field of definition of the relation writing just $\Gamma$.
Assume that there exists an element $a\in A_\Gamma(Y)$ 
 such that $\delta_Y = p^*(a)$
 where $p^*\colon A_\Gamma(Y)\rarr A_\Gamma(Y_{k(Y)})$ 
is the pullback map.
Define $b\in A_\Gamma(Y\times Y)$ to be $\Delta_Y-p_1^*(a)$, where $p_1\colon Y\times Y \rarr Y$ is the projection on the first component.
Recall that $A_\Gamma(Y\times Y)$ possesses the ring structure 
and acts on the left on $\A_\Gamma(X\times Y)$ (see~Section~\ref{sec:corr_action}).
Thus, for every $x\in A_\Gamma(X\times Y)$
we have $x = \Delta \circ x = p_1^*(a) \circ x + b\circ x$.

Let $\widetilde{x}\in A_\Gamma(X_{k(Y)})$, 
and let $x\in A_\Gamma(X\times Y)$ be an element that maps to it along the pullback map;
it exists by the coherence property of $\A$ and the compatibility assumption. 
The following square is transversal in $\Sm_k$
\begin{center}
\begin{tikzcd}
X \times Y \times Y \arrow[r, "p_{1,3}"] \arrow[d, "p_{1,2}"] & X\times Y \arrow[d, "p_X"] \\
X \times Y 			\arrow[r, "p_X"] 					   & X 
\end{tikzcd}
\end{center}
and therefore we get by~\ref{eq:bc} and by the definition of the action the following equalities 
$$ p_1^*(a) \circ x = (p_{1,3})_*(p_{1,2}^*(x)\cdot p_{2}^*(a)) = (p_{1,3})_*(p_{1,2}^*(x\cdot p_Y^*(a) )) = p_X^*(p_{X,*}(x\cdot p_Y^*(a))),$$
where $p_Y\colon X\times Y\rarr Y$, $p_2\colon X\times Y\times Y\rarr Y$ are the projection morphisms. 
On the other hand, the image of the element $b\circ x = p_{1,3}^*(p_{1,2}^*(x)\cdot p_{2,3}^*(b))$ in $A_\Gamma(X_{k(Y)})$
can be obtained as the image in the right lower corner of the following commutative diagram by~\ref{eq:ext_bc}:

\begin{center}
\begin{tikzcd}
A_\Gamma(X\times Y \times Y) \arrow{r} \arrow[d, "(p_{1,3})_*"] & A_\Gamma(X\times Y \times \Spec k(Y)) \arrow[d, "(p_{1,3})_*"] \\
A_\Gamma(X\times Y) \arrow{r} & A_\Gamma(X\times \Spec k(Y)) 
\end{tikzcd}
\end{center}

However, the horizontal pullback maps preserve multiplication,
and the restriction of $p_{2,3}^*(b)$ to the right upper corner is zero by construction.
Therefore, $b\circ x$ maps to $0$ in $A_\Gamma(X_{k(Y)})$.
Thus, the element $\widetilde{x}$ 
is the image of $p_{X,*}(x\cdot p_Y^*(a))$ along the pullback map $A_\Gamma(X)\rarr A_\Gamma(X_{k(Y)})$.
\end{proof}

\subsection{Criterion for universal bijectivity}

In this section we formulate a criterion for universal bijectivity
for a coherent oriented theory.
In the case of Chow groups it turns out that universal bijectivity 
is a corollary of universal surjectivity, which is not true for other free theories.

The results of this section are not required later in the paper; 
however, we hope they provide insight into why the notion of universal surjectivity is more subtle for cohomology theories other than Chow groups.

\begin{Lm}[sufficient condition for universal injectivity, S.--Shinder, unpublished]\label{lm:univ_inj}
Let $\A$ be a coherent oriented theory. Let $Y$ be an irreducible smooth projective variety over $k$.

Assume that there exists an element $a\in A(Y)$ such that 
\begin{enumerate}
\item $(\pi_Y)_*(a)=1\in A(k)$;

\item for every $W \in \Sm_k$, $\dim W<\dim Y$ and projective morphism $f\colon W\rarr Y$ 
we have $f^*(a) =0$ in $A(W)$.
\end{enumerate}

Then $k(Y)/k$ is $\A$-universally injective field extension.
\end{Lm}
\begin{proof}
Let $U\subset Y$ be an open subscheme with closed complement $Z\hookrightarrow Y$.
For $X\in\Smk$, the pullback of an element from $A(X)$
vanishes in $A(X\times U)$ if and only if its pullback to $A(X\times Y)$ is a push-forward of an element in $A(X\times Z)$ by the localization axiom.
By the coherence property of $A$ we have an isomorphism $A(X_{k(Y)})\cong \mathrm{colim}_U A(X\times U)$,
and therefore to prove the proposition it suffices 
to check
that the image of $p_1^*\colon A(X)\rarr A(X\times Y)$ does not intersect the image of $A(X\times Z)$ for all closed subschemes $Z\subsetneq Y$.
Note that the pullback map $p_1^*\colon A(X)\rarr A(X\times Y)$ is split
by the map $x\mapsto (p_1)_*(x\cdot p_2^*(a))$, and it suffices to show that $(p_1)_*(x\cdot p_2^*(a))$ 
is $0$ 
for every element $x$ that is a push-forward from $\A(X\times Z)$.

Recall that $\A(X\times Z)$ can be described using the resolution of singularities $\pi\colon\widetilde{Z}\rarr Z$
(see Section~\ref{sec:borel_moore_theory_via_resolution_singularities}).
It follows from this description that 
every element in $\A(X\times Z)$
is the push-forward of an element from $\A(X\times W)$ along $\id_X\times f$ 
where $W\in\Sm_k$, $\dim W<\dim Y$, 
and $f\colon W\rarr Y$ is a projective morphism.

Then for $x\in A(X\times W)$
we have that $(\id_X\times f)_*(x) \cdot p_2^*(a)= (\id_X\times f)_*(x \cdot p_2^* f^*(a)) = 0$ as $f^*(a)=0$ by assumption.  
Thus, $(p_1)_*((\id_X\times f)_*(x) \cdot p_2^*(a))=0$.
\end{proof}

\begin{Ex}\label{ex:inj_alg_clos}
Let $X$ be an irreducible variety over $k$ that has a rational point.

Then $k(X)/k$ is universally $A$-injective for any coherent $A$:
use Lemma~\ref{lm:restriction_closed_points} to apply Lemma~\ref{lm:univ_inj} for the class of a rational point.
Thus, if $k=\overline{k}$, then $k(X)/k$ is universally $A$-injective for all irreducible $X\in\SmProj_k$.
In particular, every f.g.\ field extension $K/k$ is $\overline{\A}$-universally injective.
\end{Ex}

\begin{Prop}[criterion for universal bijectivity, S.--Shinder, unpublished]\label{prop:univ_bij}
Let $A$ be a coherent oriented theory. Let $Y$ be an irreducible smooth projective variety over $k$.

Assume that the pullback map $A(\Spec k) \rarr A(\Spec k(Y))$ is injective
(e.g.\ $A$ is a free theory).

Then the following are equivalent:
\begin{enumerate}
\item field extension $k(Y)/k$ is $A$-universally bijective;
\item there exists an element $a\in A(Y)$ such that $f^*(a)=0$ for every projective morphism $f\colon W\rarr Y$ with $W\in\Sm_k$, $\dim W<\dim Y$,
and the image of $a$ in $A(Y_{k(Y)})$ equals the class of the 'diagonal' point~$\delta_Y$.
\end{enumerate}
\end{Prop}

\begin{proof}
First, we show necessity. If $k(Y)/k$ is $A$-universally surjective,
then by Lemma~\ref{lm:A} there exist $a\in A(Y)$ that maps to the class of the diagonal.
Since we have that $(\pi_{Y_{k(Y)}})_*(a_{k(Y)})=1$ and $A(\Spec k) \rarr A(\Spec k(Y))$ is injective,
we get that $(\pi_Y)_*(a)=1$ by applying~\ref{eq:ext_bc}.

Note that the following diagram commutes for every $f\colon W\rarr Y$: 

\begin{center}
\begin{tikzcd}
A(Y) \arrow[r, "f^*"] \arrow{d} &  A(W) \arrow{d} \\
A(Y_{k(Y)}) \arrow[r, "f_{k(Y)}^*"] & A(W_{k(Y)}).
\end{tikzcd}
\end{center}

The class of any rational point in $A(Y_{k(Y)})$ restricts to $0$ along $f_{k(Y)}^*$ by~Lemma~\ref{lm:restriction_closed_points}.
Thus, $a_{W_{k(Y)}}$ is zero. However, if $k(Y)/k$ is $A$-universally bijective, then also $a_{W}$ must be 0,
and the necessity is proved.

Second, assuming that there exist $a$ as in the statement we get $A$-universally surjectivity by Lemma~\ref{lm:A}
and $A$-universal injectivity by Lemma~\ref{lm:univ_inj}.
\end{proof}

In the case $\A=\CH^*$, the notions that we have defined here
have been studied before under a different name.

\begin{Def}[{\cite[1.2]{CH0_univ})}]
Let $Y$ be an irreducible smooth proper variety over $k$.
Then $\CH_0(Y)$ is said to be universally trivial
if the pullback map $\CH_0(Y)\rarr \CH_0(Y_L)$ is an isomorphism for all field extensions $L/k$.
\end{Def}

\begin{Cr}\label{cr:CH0_univ}
For a smooth proper variety $Y/k$ TFAE:
\begin{enumerate}
\item $k(Y)/k$ is $\CH^*$-universally bijective;
\item $k(Y)/k$ is $\CH^*$-universally surjective;
\item $\CH_0(Y)$ is universally trivial.
\end{enumerate}
\end{Cr}
\begin{proof}
By \cite[Lm.~1.3]{CH0_univ} and Lemma~\ref{lm:A} the conditions (2) and (3) are equivalent.

It suffices to show that $\CH_0$-universally trivial implies $\CH^*$-universally bijective.
By definition we have an isomorphism $\CH_0(Y)\xrarr{\sim}\CH_0(Y_{k(Y)})$ 
and therefore we can lift the class of the diagonal point to the class $a$ in $\CH_0(Y)$.
Moreover, the map $f^*\colon \CH_0(Y)\rarr \CH_0(W)$ is zero 
for any projective non-surjective morphism $f\colon \tilde{Z}\rarr Y$ in $\Sm_k$.
Thus the conditions of Proposition~\ref{prop:univ_bij} are satisfied.
\end{proof}

\begin{Rk}
It is not hard to see that $\CH^*$-universally bijective implies $\A$-universally bijective for any oriented theory $\A$.
\end{Rk}

\begin{Rk}
For other cohomology theories universal surjectivity does not, in general, imply universal bijectivity
as we show in Example~\ref{ex:univ_surj_not_bij} for Morava K-theory $\K2$.
\end{Rk}

We finish this section by the following well-known example.

\begin{Ex}
\label{ex:univ-bij_transcendental}
Let $K/k$ be a purely transcendental field extension, $\mathrm{tr.\,deg.\,}K/k<\infty$.

Then it is $A$-universally bijective for any coherent oriented theories $A$.
\end{Ex}
\begin{proof}
It suffices to treat the case $K=k(t)$,
then we can take $Y=\mathbb{P}_k^1$, and apply Proposition~\ref{prop:univ_bij} with $a$ being the class of any rational point.
\end{proof}

\subsection{Examples of $\A$-universally surjective field extensions}
\label{sec:examples_A-univ-surj}

\subsubsection{Smooth projective varieties with split motives}

One way to construct $\A$-universally surjective field extensions
is by considering smooth projective varieties $Y$ such that $\A$-motive of $Y$
decomposes into Tate motives. Here is a version of this statement for theories with relations.

\begin{Lm}
\label{lm:Tate_motive_universally_surjective}
Let $A$ be a coherent cohomology theory.
Let $Y$ be an irreducible smooth projective variety over $k$
and let $\Gamma(k)$, $\Gamma(k(Y))$ be compatible systems of relations
such that $\A_{\Gamma(k)}(k)\rarr \A_{\Gamma(k(Y))}(k(Y))$ is surjective.
Assume that $M_{A_\Gamma}(Y)$ is a sum of Tate motives.

Then $k(Y)/k$ is $A_\Gamma$-universally surjective.
\end{Lm}
\begin{Rk}
If $\A$ is a free theory, then $\A(k)\xrarr{\cong}\A(K)$ for any field extension $K/k$,
and the surjectivity assumption of the Lemma holds for any quotient of $\A$.
\end{Rk}
\begin{proof}
By the compatibility of systems of relations
there exists a pullback functor from $A_{\Gamma(k)}$-motives over $k$ 
to $A_{\Gamma(k(Y))}$-motives over $k(Y)$ (Proposition~\ref{prop:comp_sys_res}).
Thus, if $\mathrm M_{A_\Gamma}(Y)$ is a sum of $N$ Tate motives,
then so is $\mathrm M_{A_\Gamma}(Y_{k(Y)})$.
Therefore the map $A_{\Gamma}(Y)\rarr A_{\Gamma}(Y_{k(Y)})$ is a direct sum of $N$ copies 
of the map $\A_{\Gamma(k)}(k)\rarr \A_{\Gamma(k(Y))}(k(Y))$ which is surjective by the assumption.
The claim now follows from Lemma~\ref{lm:A}.
\end{proof}

\begin{Ex}\label{ex:Kn_split_univ_surj}
Let $q\in I^{n+2}(k)$ or $q\in \langle c\rangle+I^{n+2}(k)$ for some $c\in k^{\times}$,
then by \cite[Prop. 6.18, 6.21]{SechSem} the motive of the corresponding quadric $\MKn{Q}$ is a sum of Tate motives
(this claim will be revisited in Corollary~\ref{cr:Kn-split-quad}).

Therefore $k(Q)/k$ is a $\Kn$-universally surjective field extension by Lemma~\ref{lm:Tate_motive_universally_surjective}.

In particular, let $q$ be a anisotropic $6$-dimensional Pfister quadratic form.
Then the $\KK$-motive of the corresponding quadric is a sum of Tate motives (cf.~Section~\ref{sec:K0-motive-quadric}).
Thus, $k(Q)/k$ is a $\KK$-universally surjective field extension. 
It is an open question, whether $k(Q)/k$ is also $\KK$-universally injective.
\end{Ex}

\begin{Ex}
Let $X$ be a projective homogeneous variety for a simple group of type $\mathsf G_2$, $\mathsf F_4$ or $\mathsf E_8$ over $k$. 
Then the $\KK$-motive of $X$ is split, as follows from the following more general claim.

Let $G$ be a split semi-simple group, and $E\in\HH^1(k,\,G)$ such that all Tits algebras of $E$ 
are trivial, cf.~\cite[Sec.~3.1]{Panin}. 
In particular, if $G$ is simply connected, this condition is satisfied for any $E$. 
Let $X_0$ be a projective homogeneous variety for $G$, and $X$ be the twisted form of $X_0$ defined by $E$. 
Then the $\KK$-motive of $X$ is split by~\cite[Th.~4.2]{Panin}, cf.~\cite[Lm.~7.8]{SechSem}, therefore $k(X)/k$ is a $\KK$-universally surjective field extension.
\end{Ex}

\subsubsection{Norm varieties}

Recall that Rost motives $R_\alpha$ for prime $p$
exist for a ``symbol'' $\alpha$ in $\HH^{m+1}(k,\,\mu_p^{\otimes m})$,
i.e.\ for $\alpha$ a product of $m$ elements of $\HH^1(k,\,\mu_p)$ and one element of $\HH^1(k,\,\ZZ/p)$, see \cite[Sec.~4]{KarpMer}. 
Moreover, motive $R_\alpha$ is a direct summand of the Chow motive of the 
standard norm
variety $X_\alpha$ [loc.\ cit., Sec.~5d].

\begin{Prop}
\label{prop:split_variety_univ_surj}

Let $p$ be a prime, let $1\le m\le n$,
and let $\mathrm{K}(m)^{\mathrm{int}}$ be an $m$-th integral Morava K-theory at prime $p$.
Let $\alpha\in \HH^{n+2}(k,\,\mu_p^{\otimes (n+1)})$ be a symbol, 
and let $X=X_\alpha$ be the 
standard norm 
variety for $\alpha$.

Then the field extension $k(X)/k$ is $\mathrm{K}(m)^{\mathrm{int}}$-universally surjective.
\end{Prop}
\begin{proof}
It was shown by Karpenko--Merkurjev \cite[Th.~5.8]{KarpMer_norm} 
that the degree map 
$\CH_0(X_{k(X)})_{(p)}\rarr \Zp$ is an isomorphism.
It follows that the class of any closed point in $\CH_0(X_{k(X)})_{(p)}$ is equal to the class of the diagonal point.
Since there is an isomorphism $\CH_0(X)_{(p)}\cong \Omega^{\dim X} (X) \otimes_\ZZ \Zp$~\cite[Th.~1.2.19]{LevMor}, the last statement is also true
for any oriented cohomology theory.

The Chow motive of $X$ splits as $R_\alpha\oplus M$ where $R_\alpha$ is the Rost motive of $\alpha$.
The specialization of $R_\alpha$ to the $\K{m}^{\mathrm{int}}$-motives (see Section~\ref{sec:prelim_vishik_yagita})
can be computed by the results of \cite{Yag}:
$(R_\alpha)_{\K{m}^{\mathrm{int}}}$ splits as a sum of $p$ Tate motives, see~\cite[Prop.~6.2]{SechSem}.

In order to apply Lemma~\ref{lm:A} it remains to notice that the class of the diagonal point in $\K{m}^{\mathrm{int}}(X_{k(X)})$
lies in the subgroup $\K{m}((R_{\K{m}})_{k(X)})$ to which $\K{m}(R_{\K{m}})$ maps isomorphically.
Indeed, the class of a $0$-cycle $z$ 
can be 
seen as the morphism of motives $\un\sh{\dim X} \rarr \Mot{\Omega}{X}$
which factors through $R_\Omega$ (as it factors after specializing to Chow motives).
Then 
specializing 
this morphism to $\K{m}^{\mathrm{int}}$-motives we obtain the claim.
\end{proof}

\begin{Cr}
Assume that $k$ contains a $p$-th primitive root of unity. 

Then there exists a $\mathrm{K}(n)^{\mathrm{int}}$-universally surjective field extension $K/k$
such that 
\begin{enumerate}
    \item $\HH^{n+2}(K,\,\ZZ/p) = 0$;
    \item $\HH^{m}(k,\,\ZZ/p)  \rarr 
    \HH^{m}(K,\,\ZZ/p)$
    is injective for all $m$ such that $1\le m\le n+1$.
\end{enumerate}
\end{Cr}
\begin{proof}
    Note that by the Bloch--Kato conjectures and after identification $\mathrm{K}^\mathrm{M}_{n+2}(k)/p$ with $\HH^{n+2}(k,\,\ZZ/p)$,
    the latter is generated by symbols.
    The field $K$ is constructed as a colimit of a tower of extensions $K_j$
    where $K_0=k$, $K_{j+1}$ is obtained as the composite of the functional fields of 
    standard norm 
    varieties 
    for symbols in $\HH^{n+2}(K_j,\,\ZZ/p)$. The vanishing of $\HH^{n+2}(K,\,\mathbb Z/p)$ then follows by the construction.

    The injectivity on the $m$-th cohomology groups, $m<n+2$, is well-known for functional fields of splitting varieties,
    however, 
    cf.\ Proof of Corollary~\ref{cr:injectivity_galois_cohomology}.
\end{proof}

\subsubsection{Quadrics}

We describe all quadrics $Q$
such that $k(Q)/k$ is $\KK$-universally surjective field extension.

Recall that the even component $C_0(q)$ of the Clifford algebra of $q$ is one of the following:
a central simple algebra over $k$ if $\dim q$ is odd, 
a central simple algebra over $k(\sqrt{\det_{\pm}(q)})$ if $\dim q$ is even and $\det_{\pm}(q)$ is not a square in $k$,
and a product $C_0(q)\cong C_0'(q)\times C_0'(q)$ of two central simple algebras $C_0'(q)$ over $k$ 
if $\dim q$ is even and $\det_{\pm}(q)$ is a square, see e.g.\ \cite[Ch.~V, Sec.~2]{Lam}.

\begin{Prop}\label{prop:K_0_univ_surj_quad}
Let $Q$ be a smooth quadric of dimension $D>0$ corresponding to a quadratic form $q$ over $k$. 
\begin{enumerate}
    \item 
    If $\dim q$ is odd, or if $\dim q$ is even and $\det_{\pm}(q)$ is not a square, 
    then $k(Q)/k$ is $\KK$-universally surjective if and only if 
    $C_0(q)$ is not a division algebra. 
        \item 
        If $\dim q$ is even and $\det_{\pm}(q)$ is a square, then 
    $k(Q)/k$ is $\KK$-universally surjective if and only if 
    $C_0'(q)$ is not a division algebra. 
\end{enumerate}
\end{Prop}
\begin{Rk}
Using the notion of the Schur splitting index $i_{\mathrm S}(q)$ from~\cite[Sec.~1]{Izh} one can reformulate
the above proposition as follows: $k(Q)/k$ is $\KK$-universally surjective if and only if $i_{\mathrm S}(q)>0$. 
\end{Rk}
\begin{proof} 
If $Q$ has even dimension $D$ and $\det_{\pm}(q)$ is not a square, let $E:=k(\sqrt{\det_{\pm}(q)})$, and otherwise let $E:=k$. 
By~\cite[Th.~1]{Swan}, $\KK(Q)$ is isomorphic to $\ZZ^D\oplus\KK(C_0(q))$ where $C_0(q)$ is the even part of the Clifford algebra of $q$. 
To apply Lemma~\ref{lm:A} we thus need to check that the map induced by the base change from $k$ to $k(Q)$
yields a surjection on $\KK$ of 
$C_0(q)$. Note that $\KK(C_0(q)\otimes_k{k(Q)})$ coincides with $\KK(C_0(q)\otimes_E{E(Q_E)})$.  

For a central simple algebra $D$ over $E$ the map $\KK(D)\rarr\KK(D_K)$ 
is injective for all field extensions $K/E$, and for $K$ algebraically closed 
$\KK(D_K)\cong \ZZ$ with the image of 
$\KK(D)$ in it given by the ideal generated  by the index of the algebra $D$~\cite[Th.~3.3]{Lam-non}. 
Thus, 
$\KK(C_0(q))\rarr \KK(C_0(q)\otimes_E{E(Q_E)})$ is surjective if and only if it
the corresponding index does not change over $E(Q_E)$.

If $D$ is odd, by 
the index reduction formula of Merkurjev~\cite[Th.~1]{Mer} 
the index of $C_0(q)$ changes over $k(Q)$ 
iff $C_0(q)$ is a division algebra. 
If $D$ is even and $\det_{\pm}(q)$ is a square, $C_0(q)\cong C_0'(q)\times C_0'(q)$ for a central simple algebra $C_0'(q)$,
and then again by the index reduction formula~\cite[Th.~3]{Mer}
the index of $C_0'(q)$ changes over $k(Q)$ 
iff
$C_0'(q)$ is a division algebra. 
If $D$ is even and $\det_{\pm}(q)$ is not a square, then $C_0(q_E)\cong C_0(q)\times C_0(q)$, 
and the surjectivity of $\KK(C_0(q))\rarr \KK(C_0(q)\otimes_E{E(Q_E)})$ is equivalent
to the surjectivity of $\KK(C_0(q_E))\rarr \KK(C_0(q_E)\otimes_E{E(Q_E)})$, i.e.\ 
we are reduced to the 
case of trivial discriminant treated above. 
\end{proof}

We also provide a sufficient condition for $k(Q)/k$ to be $\KnInt$-universally surjective field extension
in terms of the first Witt index of $Q$.

\begin{Prop}\label{prop:Kn-univ-surj-quad}
Let $n\ge 1$, let $Q$ be a smooth anisotropic quadric such that $i_1(Q)\ge 2^n$ and $D=\dim Q\ge 2^{n+1}-1$. 

Then $k(Q)/k$ is $\KnInt$-universally surjective and $\Kn$-universally surjective.
\end{Prop}
\begin{proof}
The claim about $\Kn=\KnInt/2$ follows from the claim about $\KnInt$, we show the latter.

Let $h\in \KnInt(Q)$ be the class of a smooth hyperplane section of $Q$, let $E$ be a splitting field of $q$.
Then  the class $h_E^{D-2^n+1}$ equals $2l_{2^n-1}+cv_nl_0$ in the group $\KnInt(Q_E)$ by Proposition~\ref{prop:KnInt_power_of_h_split_quadric},
where $c\in\Z{2}^\times$.

Let $L/k$ be any quadratic field extension over which $Q$ becomes isotropic. 
As $i_1(Q)\ge 2^n$ it follows that there is a linear subspace $\mathbb{P}_L^{2^n-1}\subset Q_L$,
and denote $l^L_{2^n-1}$ the class of it in $\KnInt(Q_L)$.
Denote by $y$ the push-forward of $l^L_{2^n-1}$ to $\Kn(Q)$ along the finite map $Q_L\rarr Q$.
Clearly, $y_{E}=2l_{2^n-1}$ for a splitting field $E$ of $Q$.

To apply Lemma~\ref{lm:A} it suffices to show that $c^{-1}(h^{D-2^n+1}-y)_{k(Q)}$ equals the class of the diagonal point
in $\Kn(Q_{k(Q)})$. Note that $\CH_0(Q_{k(Q)})=\ZZ\cdot l_0$ and 
therefore the classes of any two rational points are the same in any oriented cohomology theory, in particular, in $\KnInt$.
Note also that $\CH_i(Q_{k(Q)})=\ZZ\cdot l_i$ for $i\le 2^n-1$. Moreover, for the classes $h^{\CH}$, $y^{\CH}$ in $\CH(Q)$
defined in the same way as for $\KnInt$ above, we have $\left((h^{\CH})^{D-2^n+1}-y^{\CH}\right)_{k(Q)}=0$ 
(since $\CH_i(Q_{k(Q)})\xrarr{\sim} \CH_i(Q_{\overline{k(Q)}})$ is an isomorphism).

Let $\gr^*_\tau \KnInt$ be the associated graded of the topological filtration on $\KnInt$ 
and consider the canonical morphism of theories $\rho\colon\CH^* \rarr \gr^*_\tau \Kn^{\int}$ (see \cite[Cor.~4.5.8]{LevMor}, cf.~\cite[Prop.~1.17]{Sech2}).
The map $\rho$ sends $(h^{\CH})^{D-2^n+1}$ to $h^{D-2^n+1} \mod \tau^{D-2^n+2}$ and $y^{\CH}$ to $y \mod \tau^{D-2^n+2}$.
We thus obtain that $c^{-1}(h^{D-2^n+1}-y)_{k(Q)}$ lies in $\tau^{D-2^n+2} \Kn^{\int}(Q_{k(Q)})$. 
However, the topological filtration on each graded component of $\KnInt$ changes
only every $2^n-1$ steps (see \cite[Prop.~3.15]{Sech2}),
and therefore we have $(h^{D-2^n+1}-y)_{k(Q)} \in \tau_0 \Kn^{\int}(Q_{k(Q)})$.

Recall that the map 
$$\CH_0(Q_L)\otimes \mathbb{Z}_{(2)}\rarr \tau_0 \Kn^{\int}(Q_L)=\gr^D_\tau \Kn^{\int}(Q)$$
 is an isomorphism for every field extension $L/k$,
since $\rho$ is surjective with the torsion kernel for every variety and $\CH_0(Q_L)$ is torsion-free.
Combining with the observations above, we get that $c^{-1}(h^{D-2^n+1}-y)_{k(Q)}$ equals $a\cdot l_0$ for some $a\in \ZZ_{(2)}$.
However, if we pass to the splitting field $E$, 
then by construction we will get that $a=1$.
Thus, $c^{-1}(h^{D-2^n+1}-y)$ is a class in $\KnInt(Q)$ that becomes the class $l_0$ of the diagonal point in $\KnInt(Q_{k(Q)})$.
\end{proof}

\begin{Rk}
In the case $n=1$ the Morava K-theory $\K1$ is isomorphic to $\KK\otimes\ZZ_{(2)}$.
Note that if $i_1(Q)\ge 2$, then the index of the Clifford algebra is not maximal,
in accordance with Prop.~\ref{prop:K_0_univ_surj_quad}.
\end{Rk}

\begin{Cr}
\label{cr:Kn-univ-surj-kill-all-(n+2)-pfister}
Let $L$ be the composite of the function fields of $(n+2)$-Pfister quadrics defined over $k$.

Then $L/k$ is $\KnInt$-universally surjective field extension.
\end{Cr}
\begin{Rk}
We will show in Corollary~\ref{cr:Kn-split-quad}
that for any $q\in I^{n+2}(k)$ the field of functions of $Q$ and its generic splitting field
are $\KnInt$-universally surjective field extensions of $k$, which does not directly follow from this corollary.
\end{Rk}

\begin{proof}
    We have $L\cong \mathrm{colim}_{I\subset J, |I|<\infty}\  k(\prod_{\alpha \in I} Q_\alpha)$
    where $J$ is the set of isomorphism classes of all $(n+2)$-Pfister forms over $k$.
    Therefore it suffices to prove that $k(\prod_{\alpha \in I} Q_\alpha)/k$
    is $\KnInt$-universally surjective field extension. 
    Moreover, this claim follows from case $|I|=1$, 
    i.e.\ $k(Q_\alpha)/k$.
    However, if a Pfister quadric $Q_\alpha$ is isotropic, then the claim is trivial (e.g.\ by Example~\ref{ex:univ-bij_transcendental}),
    and if it is anisotropic, then $i_1(Q_\alpha) = 2^{n+1}$
    and the claim follows by  Proposition~\ref{prop:Kn-univ-surj-quad}.
\end{proof}

We finish this section by providing an example of $\K{2}^{\mathrm{int}}$-universally surjective field 
extension that is not $\K{2}^{\mathrm{int}}$-universally injective.

\begin{Ex}\label{ex:univ_surj_not_bij}
Let $q\in I^4(k)$ be an anisotropic 4-Pfister form, $Q$ is the corresponding quadric.
From Proposition~\ref{prop:Kn-univ-surj-quad}
it follows that $k(Q)/k$ is $\mathrm{K}(2)^{\mathrm{int}}$-universally surjective. 

Recall that $\gr^i_\tau \K{2}^{\mathrm{int}}(X)$ is naturally isomorphic to $\CH^i(X)\otimes \Z{2}$ for $i\le 4$ for any $X\in\Sm_k$
by \cite[Prop.~6.2, (v)]{Sech2}.
Moreover, if $\dim X < 7$, then $\gr^4_\tau \K{2}^{\mathrm{int}}(X) = \tau^4 (\K{2}^{\mathrm{int}})^1(X)$
by \textup{[loc.\ cit., Prop.~3.15]}.
Therefore if the extension $k(Q)/k$ were $\K{2}^{\mathrm{int}}$-universally bijective,
then the morphism $\CH^4(X)\otimes \Z{2} \rarr \CH^4(X_{k(Q)})\otimes \Z{2}$
would be injective (in fact, an isomorphism, but we will not need that). 

Let now $k$ be a field containing the square root of $-1$ 
and elements $a_1, a_2, b_1, b_2\in k^\times$
such that the Pfister form $q:=\langle\langle a_1, a_2, b_1, b_2\rangle\rangle$ is anisotropic. 
Then Karpenko--Merkurjev in~\cite[Th.~6.5]{KarpMerChow} prove that 
there exists a field extension $k_A/k$ (which does not split any quadratic forms)
and a $5$-dimensional quadric $X$ over $k$ 
such that $2$-torsion in $\CH^4(X_{k_A})$ has arbitrarily large cardinality.
However, by construction this torsion vanishes over the splitting field of $q$,
and hence $k(Q)/k$ is {\sl not} $\K{2}^{\mathrm{int}}$-universally bijective.
\end{Ex}

\section{Base change of motivic decompositions}
\label{sec:mot_dec_base_change}

In this section we apply the results of Section~\ref{sec:univ_surj} to motives (see Section~\ref{sec:prelim_motives}).
If $K/k$ is $\A$-universally bijective field extension,
then the base change functor from $\A$-motives over $k$
to $\A$-motives over $K$ is fully faithful. 
However, it turns out that the motivic decomposition types
do not change under base change from $k$ to $K$,
if $K$ is just $\A$-universally surjective field extension. 

The reason for this is the the so-called Rost Nilpotence Property (or RNP, for short).
It was introduced in~\cite{Rost} and proved there for Chow motives of quadrics.
For an extended oriented theory $\A$, $X\in \SmProj_k$ and a field extension $K/k$ 
we say that $\A$-RNP holds for $X$ and $K/k$,
if the kernel of the pullback map 
 $$
 \mathrm{End}(\MA X)\rarr \mathrm{End}(\MA{X_K})
 $$
consists of nilpotents. More generally, one replaces $\MA X$
with any motive in $\CM_{\A}(k)$ and asks if it satisfies RNP for $K/k$.

The known cases of RNP are the following:
\begin{itemize}
    \item $X$ is a projective homogeneous variety, $\A= \CH$, $K/k$ is any \cite{Rost, BrosnanRost, BrosnanMot, CherGilMer};
    \item $X$ is a projective homogeneous variety, $\A$ is a coherent theory, $K/k$ is any \cite{GilleVishik};
    \item $X$ is a smooth projective variety such that its Chow motive is generically split, $\A=\CH$, $K/k$ is any \cite{VishZai};
    \item $X$ is a projective surface or a birationally ruled threefold, $\A= \CH$, 
    $K/k$ is any \cite{GilleSurfaces, GilleThreefolds, RosSawant}.
\end{itemize}

Apart from these cases, the validity of RNP is widely open.
In this paper we change the direction of attacking this question.
Instead of proving $\A$-RNP for one $X$ and all field extensions $K/k$,
we fix $K/k$ and prove $\A$-RNP for all $X$, albeit for specific $K$.
More precisely, we show that if for $Y\in \SmProj_k$ the image of $(\pi_Y)_*\colon\A(Y)\rarr \Apt$ contains $1$,
then every $X\in\SmProj_k$ satisfies $\A$-RNP for $k(Y)/k$ (Proposition~\ref{prop:RNP_from_geometricRNP}).
In particular, these properties hold if $Y\in\SmProj_k$ and $k(Y)/k$ is $\A$-universally surjective field extension (Corollary~\ref{cr:rnp_univ_surj}).
We should note, however, that $\A$-RNP for any field extension $K/k$
follows from $\A$-RNP for $\overline{k}/k$, which is therefore the main case of this property.
Thus, addressing the functional field extension
is an approximation to RNP in its full generality.

The main application of RNP for our purposes is the following.
If $K/k$ is $\A$-universally surjective field extension,
then the functor $\mathrm{res}\colon\CM_A(k)\rarr \CM_A(K)$ reflects
motivic decompositions (see Definition~\ref{def:reflects_MD}, Proposition~\ref{prop:A-univ-surj_reflects_MD}).
Thus, one can use examples of Section~\ref{sec:examples_A-univ-surj}
to reduce the study of $\Kn$-motives over $k$ to some field extensions that appear to be simpler.

If one is interested only in the motivic decompositions of $\A$-motives over $k$, for which RNP is known to hold for all field extensions,
then one can replace the study of $\A$-motives with $\overline{A}$-motives (Proposition~\ref{prop:reflects_MD}).
To pursue this, we formulate the criterion for $\overline{A}$-universally bijectivity of field of functions 
of projective homogeneous varieties, see Lemma~\ref{lm:B}.
The main example that we are going to use throughout the rest of the paper is that 
 for every quadric $Q$ of dimension greater or equal than $2^{n+1}-1$
the field extension $k(Q)/k$ is $\oKn$-universally bijective (Example~\ref{ex:overline-Kn-univ-surj-quad}).
We also introduce the notion of $\oKn$-equivalence of field extensions
and define the class of extensions $k(\alpha)$ 
for $\alpha\in \mathrm{H}^{n+1}(k,\,\ZZ/2)$ 
that play the role of universal splitting fields of $\alpha$ (Definition~\ref{def:k(alpha)}).

\subsection{Functors reflecting motivic decompositions}

\begin{Def}
\label{def:reflects_MD}
Let $F\colon \PM_A(X)\rarr \PM_B(Y)$ be an additive functor between categories of motives, 
where $X\in\Smk$, $Y\in\Sm_L$, $L/k$, and 
let $S$ be a class of objects in $\PM_A(X)$.

We say that $F$ {\sl reflects motivic decompositions} of objects in $S$
if the following conditions are satisfied:
\begin{enumerate}
\item if $N$ is an indecomposable object in $S$, then so if $F(N)$;
\item if for two indecomposable objects $N_1$ and $N_2$ in $S$, there exist an isomorphism $F(N_1)\cong F(N_2)$,
then there exists an isomorphism $N_1\cong N_2$. 
\end{enumerate}

If we do not specify $S$, we mean that $F$ reflects motivic decompositions for all objects of $\CM_{\A}(k)$. 
\end{Def}

\begin{Prop}
\label{prop:A-univ-surj_reflects_MD}
Let $\A$ be a free theory.
Let $K/k$ be a f.g.\ field extension that is $\A$-universally surjective.

Then the restriction functor
$$ 
\mathrm{res}\colon \CM_{\A}(k) \rarr \CM_{\A}(K)
$$
reflects motivic decompositions.
\end{Prop}
\begin{proof}
    The claim follows from the results of Vishik--Yagita~\cite[Prop.~2.2, 2.3, 2.5]{VishYag}:
    the conditions (0), (1) of ($\ast$) of [loc.\ cit., p.~587] are the assumptions of the Proposition,
    and the condition~(2) is $\A$-RNP for all $X\in\SmProj_k$ and $K/k$.
    The latter is shown in Corollary~\ref{cr:rnp_univ_surj} in the next section (independently of this Proposition).    
\end{proof}

\subsection{Rost Nilpotence Property}
\label{sec:mot_over_base}

\subsubsection{Geometric Rost Nilpotence Property}
\label{sec:geometric_RNP}

The following statement  was discovered by Vishik--Zainoulline for the case of Chow groups in~\cite{VishZai}, 
however, it also holds in a greater generality as we observe,
and more interestingly it can be applied to deduce
the usual RNP in some cases (see~Corollary~\ref{cr:rnp_k(n)}).
We call it ``geometric Rost Nilpotency Property''.

\begin{Prop}[{cf. \cite[Lm.~3.2]{VishZai}}]
\label{prop:geometric_RNP}
Let $X\in \Sm_k$ be irreducible.
Let $\Af$ be a coherent theory.

Then for every motive $M\in \CM_A(X)$ 
the restriction functor $\eta^*\colon \CM_A(X)\rarr \CM_A(k(X))$
induces a surjective homomorphism of rings
$$ \End(M) \rarr \End(M_{k(X)}) $$
with nilpotent kernel.

In particular, $\eta^*$ reflects motivic decompositions.
\end{Prop}
\begin{proof}
    It is shown in Appendix, Prop.~A.\ref{prop:app_geometric_RNP}, that
    the kernel of $\End(M) \rarr \End(M_{U})$ is nilpotent for any non-empty open $U$ in $X$.
    The nilpotence of the kernel of the restriction to the generic point follows by using the coherence property of $\Af$. 
    Finally, the claim about the reflection of motivic decompositions follows from~\cite{VishYag}:
    lifting idempotents is [loc.\ cit., Lm.~2.4] and lifting isomorphisms is [loc.\ cit., Lm.~2.1].
\end{proof}

\subsubsection{RNP and universally surjective field extensions}

Geometric RNP allows to deduce RNP for $k(X)/k$ in the case 
when the pullback map $\A(Y)\rarr \A(Y\times X)$ is injective for all $Y$.

\begin{Prop}
\label{prop:RNP_from_geometricRNP}
\label{cr:rnp_k(n)}
    Let $\A$ be a coherent theory.
    Let $X\in\SmProj_k$ such that the image of $(\pi_X)_*\colon \A(X)\rarr \Apt$ contains 1.
    Then every motive in $\CM_A(k)$ satisfies RNP for $k(X)/k$.

    In particular, if $\A(k)$ is a field and $X$ is $\A$-isotropic,
    then every motive in $\CM_A(k)$ satisfies RNP for $k(X)/k$.
\end{Prop}
\begin{Rk}
    Note that if $\A=\CH\otimes R$ for some ring $R$,
    and $X$ satisfies the assumption of the Corollary,
    then $\CH(Y)\otimes R\rarr \CH(Y_{k(X)})\otimes R$ is injective for all $Y$,
    e.g.\ by Lemma~\ref{lm:univ_inj}. 
    Thus, for Chow groups this statement is not particularly useful, 
    however, we can still get non-trivial implications for $\CH$-RNP, 
    see Section~\ref{sec:CH-RNP_from_Kn-RNP}.

    Similarly, if $X$ has a $0$-cycle of degree invertible in $\Apt$,
    then $\A(Y)\rarr \A(Y_{k(X)})$ is injective for all $Y\in \Sm_k$,
    and the claim is trivial. 
\end{Rk}
\begin{proof}
    Let $a\in \A(X)$ be such that $(\pi_X)_*(a)=1$.
    Then for every $Y\in \SmProj_k$ the pullback morphism  $(\pi_X)^*\colon\A(Y) \rarr \A(Y\times X)$ is split injective,
    with the retraction map given by $(\pi_X)_*(a\cdot -)$.
    In particular, for every  $M \in \CM_{A}(k)$ the ring map $\End(M)\rarr \End(M_{X})$ is injective.
    However, by Proposition~\ref{prop:geometric_RNP} the kernel of $\End(M_X) \rarr \End(M_{k(X)})$ is nilpotent,
    and hence also the kernel of the composition $\End(M)\rarr \End(M_{k(X)})$.
\end{proof}
\begin{Rk}
Although we show in the proof that $\A(Y)\rarr \A(Y\times X)$ is injective,
it is not true, in general, that $\A(Y)\rarr \A(Y_{k(X)})$ is injective.

For example, if $Y=X=Q$ is an excellent quadric of dimension $2^n-1$,
then one can compute $\Kn(Q)$ and $\Kn(Q_{k(Q)})$ using the computations of algebraic cobordism
of Rost motives \cite{VishYag} to see that the kernel of the map 
$\Kn(Q)\rarr \Kn(Q_{k(Q)})$ is non-trivial.
\end{Rk}

\begin{Ex}
\label{ex:K0-RNP}
Let $X\in\SmProj_k$ be such that there exists a virtual vector bundle $[V]\in\KK(X)$
with $\chi([V])=1$
(for example, if $X$ is geometrically rational,
then one can take $V=\OO_X$, since $\chi(\OO_X)=1$).

Then every motive in $\CM_{\KK}(k)$ satisfies RNP for $k(X)/k$.
\end{Ex}

\begin{Cr}
\label{cr:rnp_univ_surj}
Let $\A$ be a free theory.
Let $K/k$ be a f.g.\ field extension that is $\A$-universally surjective.

Then every $X\in\SmProj_k$ satisfies $\A$-RNP for $K/k$.
\end{Cr}
\begin{proof}
    Since we work in characteristic 0, there exists smooth projective $Y$ over $k$, such that $K\cong k(Y)$.
    By Lemma~\ref{lm:A} there exists an element $a\in \A(Y)$ such that $a_{k(Y)}=\delta_Y$.
    However, $\A(k)\cong \A(k(Y))$, and therefore $(\pi_Y)_*(a)=(\pi_{Y_{k(Y)}})_*(\delta_Y)=1$ 
    and Proposition~\ref{prop:RNP_from_geometricRNP}
    can be applied.
\end{proof}

\begin{Prop}
\label{prop:Kn-RNP_hypersurface}
Let $p$ be a prime, let $\Kn$ be the $n$-th Morava K-theory at prime $p$.
Let $H$ be a smooth projective hypersurface of degree $dp$, $p\nmid d$, and dimension greater or equal than $p^n-1$.

Then every motive in $\CM_{\Kn}(k)$ satisfies RNP for the field extension $k(H)/k$.

In particular, if $p=2$, $Q$ is a smooth quadric of dimension greater or equal than $2^n-1$,
then  every motive in $\CM_{\Kn}(k)$ satisfies RNP for the field extension $k(Q)/k$.
\end{Prop}
\begin{proof}
    If there is a rational map $X \dasharr Y$, 
    then $\A$-RNP for $k(X)/k$ implies $\A$-RNP for $k(Y)/k$.
    Indeed, $\A(Z_{k(X)})\rarr \A(Z_{k(X\times Y)})$ is injective, e.g.\ by Lemma~\ref{lm:univ_inj},
    and thus the kernel of $\A(Z)\rarr \A(Z_{k(Y)})$ is contained in the kernel
    $\A(Z)\rarr \A(Z_{k(X)})$ for all $Z\in\Sm_k$.
    Therefore if $H$ is a hypersurface of dimension $N>p^n-1$,
    we can consider a smooth hypersurface of dimension $p^n-1$ of degree $pd$ that is a subvariety of $H$,
    and prove the claim for it. Thus, we assume $\dim H = p^n-1$.
        
    By Proposition~\ref{prop:RNP_from_geometricRNP} it suffices to check that $[H]_{\Kn} = (\pi_H)_* (1_H)$ is invertible in $\Knpt=\F{p}[v_n, v_n^{-1}]$.
    If $j\colon H\rarr \mathbb{P}^{p^n}$ is an embedding of degree $dp$,
    then $j_*(1_H) = (dp)\cdot_{\Kn} h$ where $h$ is the class of the hyperplane in $\Kn(\mathbb{P}^{p^n})$.
    By construction of Morava K-theory we have $p\cdot_{\Kn} h = v_n h^{p^n}$, 
    since $h^N=0$ for $N$ greater than the dimension of the projective space,
    and $d\cdot_{\Kn}h^{p^n} = d^{p^n} v_n h^{p^n}$. Finally, $[H]_{\Kn} = (\pi_{\mathbb{P}^{p^n}})_* j_* (1_H) = d^{p^n} v_n$
    is invertible in $\F{p}[v_n,v_n^{-1}]$ by assumption on $d$.
\end{proof}

\subsubsection{From $\A$-RNP to $\CH\otimes\Apt$-RNP}
\label{sec:CH-RNP_from_Kn-RNP}

    Recall that for every free theory $\A$ 
    there exists a surjective morphism of theories $\rho_{\A}\colon\CH^*\otimes \Apt\rarr \mathrm{gr}^*_\tau A$ (see e.g.\ \cite[Prop.~1.17]{Sech2}).
    For a particular $\A$ it may happen that $\rho$ is an isomorphism in some degrees,
    and in this case one can deduce $\CH\otimes \Apt$-RNP from $\A$-RNP.

\begin{Prop}
\label{prop:A-RNP_implies-CH-RNP}
Let $\A$ be a free theory such that $\rho_{\A}$ is an isomorphism in degree $d$.
Let $X\in\SmProj_k$ be of dimension $d$. 
Let $K/k$ be a field extension.

If $\MA X$ satisfies RNP for $K/k$,
then $\Mot{\CH\otimes \Apt}{X}$ satisfies RNP for $K/k$.
\end{Prop}
\begin{proof}
    Let $a\in \CH^{d}(X\times X)\otimes\Apt$ be such that $a_K$ is zero.
    Let $\alpha$ be a lift of $\rho(a)$ to $\tau^{d}\A(X\times X)$.
    Since $\rho(a)_K=0$, we get that $\alpha_K$ lies in $\tau^{d+1}\A(X\times X\times K)$.
    However, the composition of correspondences is compatible with the filtration $\tau$ 
    in the following sense:
    if $u\in\tau^{\dim Y+i} \A(Y\times Y)$, 
    $v\in\tau^{\dim Y+j} \A (Y\times Y)$,
    then the composition $v\circ u$ lies in $\tau^{\dim Y+i+j} \A(Y\times Y)$ 
    for every $Y\in\SmProj_k$ 
    (cf.~\cite[Lm.~1.18]{Sech2}).
    Since $\tau$ is a finite filtration, we get that $(\alpha_K)^{\circ N}=0$ for some $N$.
    By $\A$-RNP for $K/k$ we get that $\alpha$ is nilpotent, and hence $\rho(a)=\alpha \mod \tau^{d+1}$ is also nilpotent
    in $\mathrm{gr}_{\tau}^{d} \A(X\times X)$. 
    Since $\rho$ is a morphism of theories and an isomorphism in degree $d$,
    we get that $a$ is nilpotent.
\end{proof}

\begin{Ex}
Let $X\in \SmProj_k$  be such that there exists a virtual vector bundle $[V]\in\KK(X)$
with $\chi([V])=1$
(for example, if $X$ is geometrically rational,
then one can take $V=\OO_X$, since $\chi(\OO_X)=1$).

Then for every $Y\in\SmProj_k$ of dimension at most 2,
its Chow motive $\MCH{Y}$ satisfies RNP for $k(X)/k$.
\end{Ex}
\begin{proof}
    Recall that $\rho_{\KK}$ is an isomorphism in degrees at most 2.
    Thus, we can apply Proposition~\ref{prop:A-RNP_implies-CH-RNP} to Example~\ref{ex:K0-RNP}.
\end{proof}

\begin{Cr}
\label{cr:Kn-RNP_implies_CH-RNP}
Let $\KnInt$ be the $n$-th Morava K-theory at prime $p$.
Let $Y\in\SmProj_k$, $\dim Y\le p^n$.

Let $K/k$ be a field extension
such that $\KnInt$-RNP holds for $Y$.

Then $\Mot{\CH\otimes\Zp}{Y}$ 
satisfies RNP for the field extension $K/k$.
\end{Cr}
\begin{proof}
    By \cite[Prop.~6.2]{Sech2} $\rho$ is an isomorphism for $\KnInt$ in degrees up to $p^n$.
\end{proof}
\begin{Ex}
\label{ex:Chow-RNP-from-Kn-RNP}
    Let $X\in\SmProj_k$ be $\Kn$-isotropic (e.g.\ see Proposition~\ref{prop:Kn-RNP_hypersurface}). 

    Then for every $Y\in\SmProj_k$ with $\dim Y\le p^n$ its $p$-local Chow motive $\Mot{\CH\otimes\Zp}{Y}$
    satisfies RNP for the field extension $k(X)/k$.
\end{Ex}

\subsection{$\overline{A}$-universal bijectivity}
\label{sec:criterion_overline_A}

\subsubsection{$\overline{A}$-universal surjectivity and bijectivity}

\begin{Lm}\label{lm:overline_inj}
Let $A$ be an extended oriented cohomology theory, let $X\in \Sm_k$, $X$ is irreducible.

Then for every projective $Y\in \Sm_k$
the canonical map $\overline{A}(Y)\rarr \overline{A}(Y_{k(X)})$ is injective.

In particular, if $k(X)/k$ is $\overline{A}$-universally surjective,
then this map is an isomorphism.
\end{Lm}
\begin{proof}
We have the following commutative diagram:

\begin{center}
	\begin{tikzcd}
		\overline{A}(Y) \arrow{d} \arrow{r} & \overline{A}(Y_{k(X)}) \arrow{d} \\
        A(\overline{Y}) \arrow{r} & A(\overline{Y_{k(X)}})
	\end{tikzcd}
\end{center}

The vertical maps are inclusions by Remark~\ref{rk:overline_proj}, 
and the lower horizontal map is injective by Example~\ref{ex:inj_alg_clos}.
The claim follows.
\end{proof}
\begin{Prop}
\label{prop:overline_surj_implies_overline_bij}

Let $\A$ be a coherent theory. Let $Y\in\Sm_k$
such that $k(Y)/k$ is $\overline{A}$-universally surjective field extension.
Then it is a $\overline{A}$-universally bijective field extension.
\end{Prop}
\begin{proof}
    Let $U\in\Sm_k$, and let $X$ be its smooth projective compactification
    with the closed complement $Z$.
    Using the localization axiom (see e.g.\ Section~\ref{sec:LOC}) 
    for $\overline{A}$ we get a commutative diagram
    of right exact rows:
    \begin{center}
        \begin{tikzcd}
            \A(Z_{k(Y)}) \arrow[r] & \A(X_{k(Y)})\arrow[r] & \A(U_{k(Y)}) \arrow[r] & 0 \\
            \A(Z) \arrow[u] \arrow[r] & \A(X)\arrow[r] \arrow[u] & \A(U) \arrow[u] \arrow[r] & 0 
        \end{tikzcd}
    \end{center}
    The central upper arrow is an isomorphism by Lemma~\ref{lm:overline_inj}. 
    Similarly, the left upper arrow is an isomorphism even in the case when $Z$ is not smooth 
    by definition of $\A(Z)$ (see e.g.\ Section~\ref{sec:borel_moore_theory_via_resolution_singularities}).
    Hence the right upper arrow is a bijection.
\end{proof}

\subsubsection{Criterion for $\overline{A}$-universal bijectivity for projective homogeneous varieties}

Recall that if $Y$ is a projective homogeneous variety over $k$,
then there exists a finite field extension of $k$ splitting $Y$, over which $Y$ becomes cellular (see e.g.\ \cite[Sec.~66]{EKM}).
Therefore $\overline{A}(Y)$ is a subgroup of a free f.g.\ $A$-module $A(\overline{Y})$. 
Moreover, $A(Y_F)\xrarr{\sim} A(Y_L)$ for any two fields $F,L\colon k\subset F\subset L$ splitting $Y$,
and hence we have canonical inclusions $\overline{A}(Y_K) \subset \overline{A}(Y_{E})$
for every f.g.\ field extensions $k\subset K \subset E$.  

For a f.g.\ splitting field $K$ of $Y$ the classes of rational points on $Y_{K}$ generate a free direct summand of $A(\overline{Y})$
of rank equal to the number of connected components of $Y_{K}$,
and on every component any two rational points yield the same class.
As we have an isomorphism $A(\overline{Y})\xrarr{\sim} A(\overline{Y_{k(Y)}})$
we can speak about the class of the diagonal point $\delta_Y$ in $A(\overline{Y})$
which is just $\sum_i x_i$ where the sum runs over different components of $Y_{\overline{k}}$
and $x_i$ is any rational point on the $i$-th component.
We call this element $l_0$, to emphasize that 
in the case when $Y$ is geometrically connected, any rational point 
has the class $l_0$ in $\overline{Y}$.
We will say that $l_0\in \overline{A}(Y)$, if
$l_0$ lies in the image of the canonical injective map $\overline{A}(Y)\hookrightarrow \A(\overline{Y})$ (cf.~Remark~\ref{rk:overline_proj}).

We can now adopt Lemma~\ref{lm:A} for our purposes.

\begin{lm}\label{lm:B}
Let $Y$ be an irreducible projective homogeneous variety. 

Then TFAE:

\begin{enumerate}
\item the field extension $k(Y)/k$ is $\overline{A}$-universally surjective;
\item the field extension $k(Y)/k$ is $\overline{A}$-universally bijective;
\item the class $l_0$ lies in $\overline{A}(Y)$.
\end{enumerate}
\end{lm}
\begin{proof}
Universal surjectivity and bijectivity are equivalent by Proposition~\ref{prop:overline_surj_implies_overline_bij}.
If one of these holds, then clearly $l_0$ is rational.
If $l_0$ is rational, 
then universal surjectivity follows from Lemma~\ref{lm:A} by recalling that $l_0=\delta_Y$ in $\overline{A}(Y_{k(Y)})$.
\end{proof}

\begin{Rk}
If $Y$ is a projective homogeneous variety and has a rational point over $k$, then $Y$ is rational (see e.g.~\cite[Th.~21.20]{Borel}).
Therefore $k(Y)$ is $A$-universally bijective for all oriented theories $A$ (cf. Example~\ref{ex:univ-bij_transcendental}).

Thus, interesting examples of $\overline{A}$-universally surjective extensions
come only from anisotropic projective homogeneous varieties.
\end{Rk}

\begin{Ex}
\label{ex:overline-Kn-univ-surj-quad}
\label{ex:quadric_kn_univ_surj}
Let $Q$ be a $k$-smooth projective quadric of dimension at least $2^{n+1}-1$.

Then $k(Q)/k$ is $\oKn$-universally bijective.
In particular, if $\dim Q\ge 3$,
then $k(Q)/k$ is $\overline{\KK/2}$-universally bijective.
\end{Ex}
\begin{proof}
Class $l_0$ lies in $\oKn(Q)$ by Corollary~\ref{prop:l_0_rational_K(n)}.
\end{proof}

\subsubsection{Universal splitting field extension for an element in Galois cohomology}

We can also compare two field extensions of $k$
by looking at what happens with $\overline{A}$-motivic decompositions over them.

\begin{Def}
Let $K_1, K_2$ be f.g.\ field extension of $k$.
They
are called $\overline{A}$-equivalent if for every projective $X\in \Sm_k$
the groups $\overline{A}(X_{K_1})$ and $\overline{A}(X_{K_2})$ (as subgroups of $A(X_{\overline{K_1\cdot K_2}})$) coincide.
\end{Def}

From the previous results we can deduce the following criterion to decide the $\overline{A}$-equivalence of field extensions.

\begin{Prop}\label{prop:A-eq-fields}
Let $Y_1, Y_2$ be smooth projective homogeneous varieties.

Then $k(Y_1)$ and $k(Y_2)$ are $\overline{\A}$-equivalent
iff $l_0^{Y_2}$ lies in $\bar{A}((Y_2)_{k(Y_1)})$ and $l_0^{Y_1}$ lies in in $\overline{A}((Y_1)_{k(Y_2)})$. 
\end{Prop}
\begin{proof} The necessity is clear.
For sufficiency note that by Lemmata~\ref{lm:B} and~\ref{lm:overline_inj} 
the map $\overline{A}(X_{k(Y_1)})\rarr \overline{A}(X_{k(Y_1)(Y_2)})$
is an isomorphism for every projective $X$, because $l_0^{Y_2}$ is rational over $k(Y_1)$.
And similarly for $Y_1$ and $Y_2$ interchanged, thus $\overline{A}(X_{k(Y_1)})=\overline{A}(X_{k(Y_1)(Y_2)})=\overline{A}(X_{k(Y_2)})$.
\end{proof}

\begin{Ex}
Let $\alpha \in \mathrm{H}^{n+1}(k,\,\ZZ/2)$,
and let $q_1, q_2 \in I^{n+1}(k)$ two of the lifts of $\alpha$ in the Witt ring 
(which exist by the validity of the Milnor conjectures).
Let $K_i$ be the generic splitting field extension of $q_i$.

Then $K_1$ and $K_2$ are $\oKn$-equivalent.
\end{Ex}

\begin{proof}
By our construction $(q_2)_{K_1}\in I^{n+2}(K_1)$
and therefore $\MKn{(Q_2)_{K_1}}$ is split by \cite[Prop.~6.18]{SechSem} (cf.~Example~\ref{ex:Kn_split_univ_surj}). 
In particular, $l_0 \in \oKn((Q_2)_{K_1})$. 
By symmetry we obtain the other assumption of Proposition~\ref{prop:A-eq-fields} and the claim follows.
\end{proof}

\begin{Rk}
Note that the fields $K_1$, $K_2$ 
do not have to be stably birational.

For example, let $q_1$ be an $(n+1)$-Pfister form, $q_2 = q_1 \perp \widetilde{q}$
where $\widetilde{q}$ is a $(n+k)$-Pfister form that does not split over $k(q_1)$, $k\ge 2$ (for example, one can take a ``generic'' $\widetilde{q}$).
Then $[\widetilde{q}]\in \HH^{n+k}(k,\,\ZZ/2)$ vanishes over $K_2$, but not over $K_1$.
\end{Rk}

The fields described in this example will play 
a special role in the study of the Picard group of $\Kn$-motives 
(see Sections~\ref{sec:iso-motives-kalpha},~\ref{sec:MDT_over_k(alpha)},~\ref{sec:milnor_k-theory_picard})
which motivates the following.

\begin{Def}\label{def:k(alpha)}
Let $\alpha \in \mathrm{H}^{n+1}(k,\,\ZZ/2)$.

We denote by $k(\alpha)$ the class of $\oKn$-equivalent field extensions of $k$
equivalent to the generic splitting field of any quadric $q\in I^{n+1}(k)$ 
such that $[q]=\alpha$ in $I^{n+1}(k)/I^{n+2}(k)\cong \mathrm{H}^{n+1}(k,\,\ZZ/2)$.  
\end{Def}

\subsection{Base change for $\overline{A}$-motives}

\begin{Prop}\label{prop:reflects_MD}
Let $A$ be an extended oriented theory, and $k(Y)/k$ be a finitely generated field extension. 
Denote by $S$ the class of $A$-motives over $k$
such that they satisfy RNP for the field extension $\overline{k}/k$ 
and their restriction to $k(Y)$ satisfies RNP for $\overline{k(Y)}/k(Y)$. 

If $k(Y)/k$ is $\overline{A}$-universally surjective,
then $\mathrm{res}_{k(Y)/k}\colon\PM_A(k)\rarr \PM_A(k(Y))$ reflects motivic decompositions of objects in $S$.

In particular, this holds if $A$ is a free theory, $Y$ is a projective homogeneous variety
such that $l_0 \in \overline{A}(Y)$,
 and $S$ consists of direct summands of projective homogeneous $G$-varieties for all semi-simple algebraic groups $G$.
\end{Prop}

\begin{proof}

We have the following 
diagram of restriction functors:

\begin{center}
	\begin{tikzcd}
			\PM_A(k) \arrow{r} \arrow{d} & \PM_A(k(Y)) \arrow{d} \\
			\PM_{\overline{A}}(k) \arrow{r} & \PM_{\overline{A}}(k(Y))
	\end{tikzcd}
\end{center}

The restrictions of vertical functors to to the full subcategories of objects in $S$, resp. $S_{k(Y)}$, 
satisfy the conditions ($\ast$) of \cite[p.~587]{VishYag}, 
and therefore they induce an isomorphism on the set of isomorphism classes of objects by [loc.\ cit., Prop.~2.2].
In particular, they reflect motivic decompositions of objects in $S$, resp. $S_{k(Y)}$.

The lower horizontal functor induces isomorphisms on morphisms between objects in $S$
by our assumption and Proposition~\ref{prop:overline_surj_implies_overline_bij}. 
The claim about the upper horizontal functor then follows.

Finally, if $A$ is a free theory, then the class of motives of projective homogeneous varieties (as described in the Proposition)
satisfies RNP by the results of \cite{GilleVishik}. The condition $l_0\in \overline{A}(Y)$ is 
equivalent to $k(Y)/k$ being $\overline{A}$-universally surjective by Lemma~\ref{lm:B},
\end{proof}

\begin{Ex}
Let $A:=\KK/2$ and $Q$ be a smooth quadric of dimension at least $3$.
Then $\mathrm{res}_{k(Q)/k}$ reflects motivic decompositions of projective homogeneous varieties.
\end{Ex}

The following is the main property that we will use in the next section
to study $\Kn$-motivic decompositions of quadrics.

\begin{Ex}
\label{ex:k(Q)/k_reflects_MD_Kn}
Let $Q$ be a quadric of dimension at least $2^{n+1}-1$.

Then $\mathrm{res}_{k(Q)/k}\colon\CM_{\Kn}(k)\rarr \CM_{\Kn}(k(Q))$ 
reflects motivic decompositions of projective homogeneous varieties.
\end{Ex}

We finish this section by the following statement
that we will use in Section~\ref{sec:one_direction_guiding_principle}
to relate splitting of $\Kn$-motives to the existence of cohomological invariants.

\begin{Cr}
\label{cr:injectivity_galois_cohomology}
Let $X$ be a projective homogeneous variety over $k$ 
such that $l_0$ is $\oKn$-rational in $X$. 

Then the canonical morphism 
$$ \HH^{m} (k,\mu_p^{\otimes s}) \rarr  \HH^{m} (k(X), \mu_p^{\otimes s})  $$
is injective for $m$ such that $2\le m\le n+1$ and for all $s$. 

If, moreover, $\mu_{p^r}\subset k$, 
then the canonical morphism
$$ \HH^{m} (k,\ZZ/p^r) \rarr  \HH^{m} (k(X), \ZZ/p^r)  $$
is injective for 
all $m$ such that $2\le m\le n+1$, and for all $s$. 
\end{Cr}
\begin{proof}
    Note that the degree of the field extension $k(\mu_p)/k$ is prime to $p$,
    and therefore by the projection formula the base change of cohomology groups with $p$-torsion coefficients 
    from $k$ to $k(\mu_p)$ is injective. Thus, to prove the first claim we can assume $\mu_p\subset k$
    and ignore different Tate twists of the $p$-th roots of unity.

    In this case, let $R_\alpha$ be the generalized Rost motive
    that corresponds to  a symbol $\alpha$ in $\HH^m(k,\,\ZZ/p)$, where $m$ is such that $2\le m\le n+1$ (see \cite{Voe_Zl, SusJou, KarpMer_norm}).
    By its properties, $R_\alpha$ splits over a field extension $K$
    if and only if $\alpha_K=0$. Moreover, the same is true about the $\Kn$-specialization of the Rost motive by \cite[Prop.~6.2]{SechSem}.
    The motive $R_\alpha$ as well as its specializations satisfy Rost Nilpotency Property by~\cite[4.2]{GilleVishik},
    and therefore by applying Proposition~\ref{prop:reflects_MD} we get that $\left((R_\alpha)_{\Kn}\right)_{k(X)}$ is not split, 
    and hence $\alpha_{k(X)}\neq 0$. In other words, there are no symbols in the kernel of the morphism
    $\HH^m(k,\,\ZZ/p)\rarr \HH^m(k(X),\,\ZZ/p)$.

    Then it follows from the Bloch--Kato conjecture~\cite{Voe_Zl} 
    that the kernel is trivial. 
    For every element $\beta\in \HH^m(k,\,\ZZ/p)$ 
    there exists a field extension $K/k$ such that $\beta_K$ is a non-trivial symbol:
    see~\cite[Th.~2.10]{OVV} for $p=2$;
    a similar argument works for odd $p$ as well, once the motivic cohomology of the \v{C}ech scheme corresponding to the norm varieties $X_\alpha$
    are computed, cf. \cite[Sec.~6]{Voe_Zl}, see e.g. \cite[Th.~5.1]{Sem-coh-inv}; alternatively, 
    one can also use \cite[Th.~2.1]{MerSus-rost} in the argument of \cite[Th.~2.10]{OVV}.
    By applying the above argument over $K$
    we get that $\beta_{K(X)}$ does not vanish, and hence $\beta_{k(X)}\neq 0$.

    Finally, to show the claim for $\ZZ/p^r$-coefficients under the assumption that $\mu_{p^r}\subset k$
    we again use the Bloch--Kato conjecture. It follows from its validity that the natural sequence
    $$ 0\rarr \HH^m(k,\,\ZZ/p^s) \rarr \HH^m(k,\,\ZZ/p^{s+1}) \rarr \HH^m(k,\,\ZZ/p) \rarr 0$$
    is exact for $s\le r-1$. Thus, the claim about $\ZZ/p^r$-torsion coefficients can be deduced inductively
    from the case $\ZZ/p$ treated above.
\end{proof}

\section{Morava motivic decompositions of quadrics}\label{sec:morava_quadrics}

In this section, we give a complete description of the $\Kn$-motivic decomposition of a quadric $Q$
 in terms of the Chow motivic decomposition of $Q$ over a certain field extension
 in the generic splitting tower of $Q$. 
 
Let us recall the previously known results about $\Kn$-motives of quadrics for $n\ge 2$.
 In \cite[Prop.~6.18,~6.21]{SechSem} 
 it was shown 
 that $\MKn{Q}$ is split if and only if $q\in I^{n+2}(k)$ or $q \in \langle c \rangle+I^{n+2}(k)$ for some $c\in k^\times$. 
 In \cite[Prop.~7.1]{LPSS} it was shown that for any quadric $Q$ of dimension at least $2^n-1$
 there are $\dim Q-2^n+2$ Tate summands that split off from $\MKn{Q}$, and the complement of them
 is indecomposable in the case when $Q$ is a generic quadric [loc.\ cit., Th.~7.8]. 
 Also, the $\Kn$-specialization of the Rost motives $R_\alpha$
 was discussed in \cite[Prop.~6.2]{SechSem}, thus, providing decompositions of the Morava motives of excellent quadrics.
 All of these statements follow from the results of this section, and the proofs here do not depend on the cited papers.

We show that the study 
of the $\Kn$-motivic decompositions of quadrics 
breaks into three cases\footnote{The terminology is inspired by the stabilization of $\Kn(\mathrm{SO}_m)$
 discovered in \cite{LPSS}.}:

\begin{center}
    \begin{tabular}{l l}
         stable & $\dim Q > 2^{n+1}-2$ \\
         pre-stable &  $2^n-1 \le \dim Q \le 2^{n+1}-2$ \\
         unstable & $\dim Q < 2^n-1$
    \end{tabular}
\end{center}

The stable case (Section~\ref{sec:stable}) can be reduced to the pre-stable one using the results of the previous section.
The functor of base change from $k$ to $k(Q)$, where $\dim Q\ge 2^{n+1}-1$, reflects motivic decompositions (Example~\ref{ex:k(Q)/k_reflects_MD_Kn}).
Thus, one can base change along the generic splitting tower of quadric until the anisotropic dimension of the quadric is 
in the pre-stable case.

On the other hand, for all smooth projective varieties of dimension less than $2^n-1$ 
the endomorphisms of their $\Kn$-motives are identified
with the endomorphisms of their Chow motive by the results of \cite{Sech2}.
In particular, the decomposition of the $\Kn$-motive of a quadric ``repeats''
the decomposition its Chow motive in the unstable case (Section~\ref{sec:unstable}).

Thus, the most interesting case is the pre-stable one (Section~\ref{sec:prestable}).
For every such quadric $Q$ there are always $\dim Q-2^n+2\ge 1$ Tate summands in $\Mot{\Kn}{Q}$,
therefore the substantial information of the motive is concentrated in the complement to these Tate motives,
which we call the $\Kn$-kernel motive of $Q$ and denote $\Mker{Q}$. 
Despite the fact that $\Mker{Q}$ is more ``compact'' than $\Mot{\CH}{Q}$,
 we show in Theorem~\ref{th:prestableMDT}  that it carries essentially the same information,
  and this is the main result of this section.
In particular, the decomposition type of the $\Kn$-motive of $Q$
is obtained from the Chow-MDT by ``cutting out'' the central part
of length $2^n-1$ (Section~\ref{sec:morava-mdt}),
and the Chow-MDT can be recovered 
by adding excellent connections of length $2^n-1$ (see Proposition~\ref{prop:outer-excel})
to the decomposition of this central part.

\subsection{Unstable case}\label{sec:unstable}

The easiest case of comparison between Chow and Morava motives
is when the dimensions of the varieties are smaller than $2^n-1$.
Note, however, that Theorem~\ref{th:prestableMDT}, proven independently of this result, covers this case as well.

\begin{Prop}[{\cite[Cor.~6.6]{Sech2}}]\label{prop:unstable_quadric}
Let $Q_1, Q_2$ be smooth anisotropic quadrics of dimensions less than $2^n-1$.

Then $\mathrm{End}(\MKn{Q_i})$ is canonically isomorphic to $\mathrm{End}(\MCh{Q_i})$ for $i=1,2$,
and so there is a $1$-to-$1$ correspondence between summands of $\MKn{Q_i}$
and $\MCh{Q_i}$.

Moreover, if $N_i$ is a summand of $\MKn{Q_i}$, $i=1,2$
and $N_1\cong N_2$, then the corresponding Chow summands are also isomorphic.
\end{Prop}
\begin{proof}
It is proven in~\cite{Sech2} that under these assumptions $(\Kn^{\mathrm{int}})^{\dim Q_j}(Q_i\times Q_j)$
is canonically isomorphic to $\CH^{\dim Q_j}(Q_i\times Q_j)\otimes \ZZ_{(2)}$,
and hence the same is true for mod $2$ coefficients.
\end{proof}

\subsection{Stable case}\label{sec:stable}

Let $q$ be a non-degenerate quadratic form over $k$,
consider its generic splitting tower (see Section~\ref{sec:prelim_split_tower}):

$$ k=K_0 \subset K_1 \subset \cdots \subset K_h,$$

and let $q_i:=(q_{K_i})_{\an}$ be a quadratic form over $K_i$.

\begin{Def}\label{def:Kn-kernel}
In the notation above, let $j$ be the minimal integer such that $\dim q_j \le 2^{n+1}$.
The quadratic form $q_j$ over $K_j$ is called the $\Kn$-kernel form of $q$.
\end{Def}

Note that in the above definition the field extensions $K_i/K_{i-1}$ for $i\le j$
are function fields of quadrics of dimension at least $2^{n+1}-1$.

\begin{Prop}\label{prop:quad_MDT_stable}
Let $q$ be a quadratic form of dimension $D+2$ that is greater or equal than $2^{n+1}+1$.
Let $q'$ be the $\Kn$-kernel form of $q$, let $K$ be its field of definition and
let $r:=\frac{\dim q - \dim q_j}{2}$.

Then there exists an isomorphism of $\Kn$-motives:
$$ \Mot{\Kn}{Q} \cong N\oplus \bigoplus_{i=0}^r \un(r) \oplus \bigoplus_{j=D-r+1}^D \un(j) $$
where $N$ is uniquely determined up to isomorphism by the property that
 $N_{K}$ is isomorphic to $\MKn{Q'}\sh r$.
\end{Prop}
\begin{proof}
By definition, the functor $\mathrm{res}_{K/k}$ is a composition of the functors $\mathrm{res}_{K_i(Q_i)/K_i}$, 
where  $Q_i$ are quadrics in the generic splitting tower of $Q$ of dimension
greater or equal than $2^{n+1}-1$. 
By Example~\ref{ex:k(Q)/k_reflects_MD_Kn}
these functors reflect motivic decompositions.

The motive $\Mot{\Kn}{Q_{K}}$ has motivic decomposition as stated by Lemma~\ref{lm:motive_of_isotropic_quadric}.
\end{proof}

In particular, we can reprove one of the main results of \cite{SechSem} with the methods of this paper.

\begin{Cr}[{\cite[Prop.~6.18,~6.21]{SechSem}}]\label{cr:Kn-split-quad}
Let $q$ be a quadratic form s.t.\ $q\in I^{n+2}(k)$ or $q \in \langle c \rangle+I^{n+2}(k)$ for some $c\in k^\times$.
Then $\Mot{\Kn}{Q}$ and $\Mot{\KnInt}{Q}$ are split.

Let $K$ be the generic splitting field of $q$.
Then $k(Q)/k$ and $K/k$ are $\KnInt$-universally surjective field extensions,
and hence $\oKn$-universally bijective field extensions.
\end{Cr}
\begin{proof}
For such $q$ its $\Kn$-kernel form is either $0$, or one-dimensional.
Hence $\Mot{\Kn}{Q}$ is split by Proposition~\ref{prop:quad_MDT_stable}.
The claim about $\KnInt$-motive follows, cf.\ \cite[Th.~4.1]{LavPet}.
Universal surjectivity property now follows from Lemma~\ref{lm:Tate_motive_universally_surjective}.
\end{proof}

\subsection{Pre-stable case}\label{sec:prestable}

Let $X\in \Sm_k$. 
Since  $\Kn(X) \cong \BP(X)\ot_{\BP(k)} \mathbb{F}_2[v_n, v_n^{-1}]$,
any element $\beta \in \Kn(X)$ can be written as $\sum_{s\in \ZZ} x_s \otimes v_n^{-s}$
where $x_s \in \BP(X)$ and only finitely many of $x_s$ are non-zero.
In the following lemma we show that we can bound powers $v_n^{-s}$ that appear in this decomposition
for quotients $\oKn$ and $\oBP$ instead of 
$\Kn$ and $\BP$.
The statement is used in the proofs of Proposition~\ref{prop:l_0_rational_isotropic} and of Theorem~\ref{th:prestableMDT}.

\begin{Lm}
\label{lm:from_kn_to_bp}
Let $X\in\Sm_k$  
and let $\alpha \in \oKn(X)$.
Then there exists an element $\beta \in  \oBP(X)$ of non-negative codimension  
such that $\beta \otimes v_n^{-r} \in \oBP(X)\otimes_{\BPpt}  \mathbb{F}_2[v_n, v_n^{-1}]$
maps to $\alpha$.

If, moreover, the restriction of $\alpha$ to the generic points of $X$ 
vanishes,
then $\beta \otimes v_n^{-r}$ maps to $\alpha$ 
for some $\beta \in \oBPt{X}$ of positive codimension. 
\end{Lm}
\begin{proof}
    As explained before the statement of the lemma,
    there exist elements $x_s\in \overline{\BP}^{a-s(2^n-1)}(X)$, $s\in \ZZ$,
    finitely of which are non-zero,
    and such that $\alpha = \sum_s x_s \otimes v_n^{-s}$.
    Therefore $\alpha = \beta \otimes v_n^{-t}$
    where $t:= \max\{s|x_s \neq 0\}$ and $\beta = \sum_s  v_n^{t-s} x_s$.
    
    Assume that $\alpha = \beta \otimes v_n^{-t}$
    where $t$ is the minimal possible.        
    If $\beta\in \BP^{<0}(k)\cdot \overline{\BP}(X)$,
    i.e.\ it equals a linear combination $\sum_i \mu_i y_i$,
    where $\mu_i\in \BPpt^{<0}$, $y_i\in \overline{\BP}(X)$,
    then one may write  $\beta \otimes v_n^{-s}$ as a linear combination of elements $y_i \otimes v_n^{-i}$
    where $i< t$. 
    Hence, if $t$ is minimal, then $\beta \notin \BP^{<0}(k)\cdot \overline{\BP}(X)$.
    
    By a theorem of Levine--Morel~\cite[Th.~1.2.14]{LevMor} we have
    $\BP(X)/\left(\BP^{<0}(k)\cdot \BP(X)\right)\cong \CH(X)\otimes \Z{2}$, which is concentrated in non-negative degrees.
    Since $\overline{\BP}(X)$ is the quotient of $\BP(X)$,
    we get that   
    $\overline{BP}^{<0}(X)\subset \BP^{<0} \overline{\BP}(X)$.
    Moreover, an element of $\oBP^0(X)$ that projects to 0 in $\CH^0(X)\otimes \Z{2}$, 
    i.e.\ whose restriction to the generic points of $X$ vanishes, lies in $\BP^{<0}(k)\cdot \overline{\BP}(X)$.
    Thus, if $\codim \beta<0$, then $t$ is not minimal. 
    This contradiction proves the first claim.

    Assume now that  $\codim \beta =0$. 
    For simplicity, assume that $X$ is irreducible. Let $\beta_{k(X)}$, the restriction of $\beta$
    to the generic point of $X$, equal $m\cdot 1_{k(X)}$ in $\BPt{k(X)}$, $m\in\Z{2}$,
    then $\alpha_{k(X)} = m v_n^{-t}\cdot 1_{k(X)}$ in $\Knt{k(X)}$.
    By our assumption it equals to zero, i.e.\ $m\in 2\Z{2}$. 
    However, then $\alpha = (\beta-m\cdot 1_X)\otimes v_n^{-t}$, where $\beta -m\cdot 1_X \in \oBP(X)$.
    Thus, we may assume that $\beta$ restricts to 0 in the generic point of $X$,
    and by the result  of Levine-Morel above, it lies in $\BP^{<0}(k)\cdot \overline{\BP}(X)$.
    Therefore $t$ is not minimal. This contradiction proves the second claim.    
\end{proof}

\subsubsection{Morava isotropic quadrics}

The results of Proposition~\ref{prop:quad_MDT_stable} are obtained  
with the help of the fact that $l_0$ is rational in $\Kn(Q)$ 
when the dimension $D$ of the quadric is at least than $2^{n+1}-1$ (Corollary~\ref{prop:l_0_rational_K(n)}). 
We explain now what happens below this dimension.
	
\begin{Prop}\label{prop:l_0_rational_isotropic}
Let $Q$ be a smooth projective quadric of dimension $D\le 2^{n+1}-2$. 
Then $l_0$ is rational in $\Knt{\overline{Q}}$ if and only if $Q$ is isotropic.
\end{Prop}

We will need the following lemma in the proof.
\begin{Lm}
\label{lm:chow_steenrod}
Let $Q$ be a split smooth projective quadric of dimension $D=2^n-1+r$, $0<r<2^n$.
Let $\StCh\colon\Chf\rightarrow\Chf[[t]]$ denote the total Steenrod operation {\rm(}see Section~\ref{sec:prelim_symm}{\rm)},
and let $h^i$, $l_i$ be the standard basis of $\Cht Q$ {\rm(}see Section~\ref{sec:prelim_A_split_quadric}{\rm)}. Then the following holds. 
\begin{enumerate}
\item
$\StCh(l_0)=l_{0\,}t^{2^n-1+r}$.
\item
For $r=2^n-1$ and $m<n$, $\StCh(l_{2^n-2^m})$ does not have a summand of the form $l_{0\,}t^{c}$.
\item
For $r<2^n-1$ there exists a unique $m<n$ such that $\StCh(l_{2^n-2^m})$ has a summand of the form $l_{0\,}t^{c}$.

Number $m$ is determined by the equation $r = 2^{n-1}+\ldots +2^{m+1} + x$, $x<2^m$,
i.e., $m$ is the position of the ``highest'' $0$ digit among the $n$ lower digits in the base $2$ expansion of $r$.
\item
For $m$, $x$ as above, $\StCh(l_{2^n-2^m})$ has a summand $l_{2^m-x\,}t^{2^m-1}$.
\item
For $m$, $x$ as above, and any $q\neq m$, $\StCh(l_{2^n-2^q})$ does {\sl not} have a summand of the form $l_{2^m-x\,}t^{c}$.
\end{enumerate}
\end{Lm}

\begin{proof}[Proof of Lemma~\ref{lm:chow_steenrod}]
By Proposition~\ref{prop:St-Ch-quad} we get that 
\begin{equation}
\label{eq:steenrod_l} 
\StCh(l_{2^n-2^m})
= l_{2^n-2^m} (t+h)^{2^m+r} t^{-1}.
\end{equation}
In particular, $(1)$ is clear. 
Note that the coefficient at $l_{0\,}t^{2^{m+1}-2^n+r-1}$ in the right hand side is  $\binom{2^m+r}{2^n-2^m}$. 
Since $2^n-2^m = 2^{n-1} + 2^{n-2} +\ldots + 2^{m}$, in order for $\binom{2^m+r}{2^n-2^m}$
to be odd we must have by Lucas's Theorem (cf.\ also~\cite[Lm.~78.6]{EKM}) that 
$$
2^m+r = \varepsilon\cdot 2^n+2^{n-1}+\ldots +2^m + x
$$ 
where $\varepsilon\in\{0,1\}$ and $0\le x < 2^m$,
i.e., $r = \varepsilon\cdot 2^n+2^{n-1}+\ldots +2^{m+1} + x$.
In particular, $\varepsilon=0$ since $r<2^n$. This implies $(2)$ and $(3)$.

Moreover, the coefficient at $l_{2^n-2^m-r\,}t^{2^m-1}$ in the right hand side of~\eqref{eq:steenrod_l} equals 
$$
\binom{2^m+r}{r}=\binom{2^{n-1}+\ldots +2^m + x}{2^{n-1}+\ldots +2^{m+1} + x}=1,
$$
and $2^n-2^m-r=2^m-x$, which implies $(4)$.

Observe that $D=2^{n+1}-2^{m+1}+x-1$, and since $x\leq 2^m-1$, 
we have $d=\lfloor\frac{D}{2}\rfloor\leq2^n-2^{m-1}-1$.
On the other hand, $2^n-2^q > 2^n -2^{m-1}-1$ for $q<m$.
In other words, for $l_{2^n-2^q}\in\Cht Q$ we have $q\geq m$.

Clearly, for $q=n$, $(5)$ follows from $(1)$. 
Finally, observe that for $m<q<n$, 
$$ 
\StCh(l_{2^n-2^q}) =  l_{2^n-2^q} (t+h)^{2^q+2^n-2^{m+1}+x\,}t^{-1} = l_{2^n-2^q} (t+h)^{2^q-2^{m+1}+x\,} t^{2^n-1},
$$
since $h^{2^n}=0$. 
Thus, the minimal $i$ such that a summand of the form $l_{i\,}t^c$ appears in the right hand side
is $i=2^n-2^{q+1}+2^{m+1}-x$. 
However, $i>2^m-x$ since $ 2^n > 2^{q+1}-2^m$, which implies $(5)$.
\end{proof}

\begin{proof}[Proof of Proposition~\ref{prop:l_0_rational_isotropic}]
If $Q$ is isotropic, then the class of a rational point gives rationality of $l_0$ (for any theory).
Assume that $l_0$ is rational in $\Knt{\overline{Q}}$.  
It suffices to show that $l_0$ is rational in $\Cht Q$ due to the Springer theorem~\cite[Lm.~72.1]{EKM}, \cite{Springer}.

By Lemma~\ref{lm:from_kn_to_bp} there exists an element $z\in\oBPt Q \subset \BP(\overline{Q})$, $s\in \ZZ$ 
such that $l_0 = z\otimes v_n^{-s}$  in $\Knt{\overline{Q}}=\BPt{\overline Q}\otimes_{\BPpt}\mathbb F_2[v_n^{\pm1}]$,
and $\codim z>0$. 
Since $\codim z = \codim {l_0} - s(2^n-1) = \dim Q - s(2^n-1) \le 2(2^n-1)$ and $\codim z$ is bounded above by $\dim Q$,
we get that either $s=0$ or $s=1$.
If $s=0$, then $z\in \oBP^{\dim Q}(Q)\cong \CH_0(Q)\otimes \Z{2}$ and the claim follows.
Therefore we may assume that $s=1$.

Let $r$ equal $D-(2^n-1)$, $0<r<2^n$. 
To prove that $Q$ is isotropic, we will apply the $\Ch$-trace $\phi$ of 
Vishik's symmetric operation $\PhBP$ to $z$ to get that some $l_i$, $i\geq0$, is rational in $\Cht{\overline Q}$ (see Section~\ref{sec:prelim_symm}).

The element $z$ can be decomposed in the basis $h^i$, $l_i$. Since $h^i$ are always rational,
we may assume that $z$ is a $\BPpt$-linear combination of $l_i$, containing by the assumption a summand $\lambda^{\,} v_n l_0$ for $\lambda\in\ZZ_{(2)}^\times$. 

First, let us consider the case $r<2^n-1$. Then all the summands of this linear combination
are of the form $\mu\,l_i$ where $\mu\in \BPptaug$. We are going to apply $\phi^{t^e}$ to it
with $e>0$, which is an additive operation (Proposition~\ref{prop:phi-add}).
Applying $\phi^{t^e}$ to $uv$ for $u=[U]\in\BPptaug$, $v\in\BPt{\overline Q}$, 
we obtain a multiple of $\eta(u)$ by Proposition~\ref{prop:symm_divide},
and $\eta(u)$ is zero whenever $u$ is a product of at least two $v_i$, $i>0$  by Proposition~\ref{prop:eta}.
Thus we are interested only in the values $\phi (v_m l_a)$ where $\deg v_m l_a = \deg v_n l_0$, 
i.e.\ $a=2^n-2^m$.

By Lemma~\ref{lm:chow_steenrod}, there exists at most one $m<n$ such that $\StCh(l_{2^n-2^m})$ has a summand of the form $l_{0\,}t^{c}$.
If the coefficient $\nu$ at $v_m l_{2^n-2^m}$ in the decomposition of $z$ lies in $2\ZZ_{(2)}$, we conclude that 
$$
\phi^{t^{e}}(z)=\phi^{t^{e}}(\lambda\,v_nl_0)=l_0
$$ 
by Propositions~\ref{prop:phi-add},~\ref{prop:symm_divide},~\ref{prop:eta} for $e=2^n-1-r$, and therefore $l_0\in\Cht{\overline Q}$ is rational.

If the coefficient $\nu$ lies in $\ZZ_{(2)}^\times$,
we know that $\StCh(l_{2^n-2^m})$ has a summand $l_{2^m-x\,}t^{2^m-1}$ for $x=2^{m+1}+r-2^n<2^m$ by Lemma~\ref{lm:chow_steenrod}, 
and $\StCh(l_{2^n-2^q})$ does not have such a summand for $q\neq m$. Then we conclude that 
$$
\phi^{t^e}(z)=\phi^{t^e}(\nu\,v_ml_{2^n-2^m})=l_{2^m-x\,}
$$ 
for $e=2^m-1$, therefore $l_{2^m-x}\in\Cht{\overline Q}$ is rational.

Second, in the case $r=2^n-1$ the above proof works with minor modifications.
On the one hand, the linear combination that equals the decomposition of $z$ 
might contain $\mu\,l_r$ for some $\mu\in\ZZ_{(2)}$.
However, $z$ is mapped to $\mu\,l_{r}$ in $\Cht{\overline Q}$ and hence $\mu\in2\ZZ_{(2)}$ or $Q$ is isotropic.
We thus can assume the former.
On the other hand, we might need to use the operation $\phi^{t^0}$,
which is in general non-additive. However, it acts additively on the linear combination of interest
by Proposition~\ref{prop:phi-add}(2,3). 
Part (3) of it is only needed for the summand $\mu\, l_r$ considered above with $\mu\in 2\Z{2}$
(note that since $D\equiv 2\mod 4$ in this case, 
we have $l_r^2=0$). 
All other $l_i$ appearing in the linear combination 
have coefficients from $\BPptaug$ by codimensional reasons, and part (2) of~\ref{prop:phi-add} applies.
Thus, the computations above remain valid in this case as well.
\end{proof}

\begin{Cr}
\label{cr:li_rational_small_dim}
Let $Q$ be a quadric of dimension at most $2^{n+1}-2$.
Then the following conditions are equivalent:
\begin{enumerate}
\item
$l_i\in\oKnt Q$;
\item
$\iw(q)>i$.
\end{enumerate}
\end{Cr}

\begin{Rk}
Note that for an arbitrary quadratic form $q$, rationality of $l_i$ for $i> d-2^{n-1}$ in $\Knt Q$ detects a ``hole'' in the splitting pattern: 
if $q_s$ is the $\Kn$-kernel form of $q$, and $q_{s-1}$ is the previous kernel form in the generic splitting tower, one has $\dim q_{s-1}>2^{n+1}$ and $\dim q_s<2^{n}$. 
%
%
Conjecturally, $q_{s-1}$ is a Pfister neighbour in this case, and $q$ has a ``small distance to $I^{n+2}$''. We return to these questions in Section~\ref{sec:appearance_invertible_summands_quadrics}.
\end{Rk}

\subsubsection{Outer excellent connections, revisited}
\label{sec:outer-ex-rev}

In this section, we show that 
the existence of {\sl outer excellent connections} of Vishik\footnote{We refer the reader to~\cite{VishExc} and references therein for the history 
and different approaches to the existence of excellent connections.} 
is an easy corollary of Proposition~\ref{prop:l_0_rational_isotropic}. 
These connections were discovered by Vishik with the use of the symmetric operations,
and as we used them in the proof of~\ref{prop:l_0_rational_isotropic},
our proof is, arguably, not new. 

However, the outer excellent connections 
are precisely the ones needed to reconstruct Chow MDT from $\Kn$-MDT
for the quadrics of dimensions between $2^n-1$ and $2(2^{n}-1)$ (see Theorem~\ref{th:prestableMDT}). 
We thus include the proof of Proposition~\ref{prop:outer-excel} in this text 
to make it more self-contained, and to emphasize the interrelationship between the Chow and Morava motives. 
We will assume $n\geq2$, but for $n=1$ an analogous result is clear from Section~\ref{sec:K0-motive-quadric}.

Let $Q$ be a smooth projective quadric of dimension $D$, and let $n$ be such that $2^n-1 \le D \le 2^{n+1}-2$.
We use the notation $D'$, $d'$ from Section~\ref{sec:rat-proj-quad}. 
We denote 
$$
\omega^{\Ch}_j=h^j\times l_j+l_{D'-j}\times h^{D'-j}\in\Ch^D(\overline Q\times\overline Q)
$$
for $d'\leq j\leq d-1$ and for $j=d$ if $D\equiv 2\mod 4$, and
$$
\omega^{\Ch}_d=h^d\times\widetilde l_d+l_{d'}\times h^{d'}
$$
if $D\equiv0\mod4$ (where $\widetilde l_d=l_d+h^d\in\Cht{\overline Q}$ and $h^i=0=l_i$ for $i<0$). 
For a rational projector $\rho\in\oCh^{\,D}(Q\times Q)$ we denote $\MCh{Q,\,\rho}$ the corresponding motivic summand of $\MCh{Q}$.
We will also use the notations of Section~\ref{sec:ch-mdt-conn}.
\begin{Prop}[Outer excellent connections]
\label{prop:outer-excel}
In the notation above, if $Q$ is anisotropic, 
and $N$ is an indecomposable summand of $\MCh Q$, 
there exist $I\subseteq\{d',\ldots,d\}$ such that $\sum_{i\in I}\omega^{\Ch}_i$ is rational and 
$$
N\cong\MCh{Q,\,\sum_{i\in I}\omega^{\Ch}_i}.
$$
In other words, if the motives $\unCh\sh j$ and $\unCh\sh{j+2^n-1}$ belong to $\Lambda(Q)$, 
then they are connected. 
In the case $D=2^{n+1}-2$ and $j=0$ or $j=2^n-1$ the claim is that the upper ``middle'' Tate motive $\unCh\sh{2^n-1}$ is connected to $\unCh\sh 0$ and the lower is connected to $\unCh\sh{2^{n+1}-2}$. 
\end{Prop}

\begin{proof}

Assume that $\unCh\sh j$ and $\unCh\sh{j+2^n-1}$ belong to $\Lambda(Q)$ and are not connected, 
in other words, there exists an indecomposable summand $N$ of $\MCh{Q}$ such that $\unCh\sh j\in\Lambda(N)$, 
and $\unCh\sh{j+2^n-1}\not\in\Lambda(N)$. We may assume that $N$ has a ``normal form'', i.e., $\overline N=\oplus_{L\in\Lambda(N)}L$, and let $\pi^{\Ch}\in\oCh^{\,D}(Q\times Q)$ denote the corresponding rational projector. Then decomposing $\pi^{\Ch}$ as a sum of $l_i\times h^i$ and $h^i\times l_i$ (and maybe $h^d\times\widetilde{l_d}$ for $D\equiv 0\mod 4$) corresponding to Tate summands from $\Lambda(N)$, we conclude that $\pi^{\Ch}$ has a summand $h^j\times l_j$ 
and does not have a summand $l_{D'-j}\times h^{D'-j}$.

Since $Q$ is anisotropic, we get that by Proposition~\ref{prop:l_0_rational_isotropic} element $l_0$ is not rational in $\Knt{\overline{Q}}$.
We are going to arrive to a contradiction with this claim.

By \cite[Cor.~2.8]{VishYag} 
there exists a rational projector $\pi^{\CKn}\in\oCKn^{\,D}(Q\times Q)$ lifting $\pi^{\Ch}$. Decomposing $\pi^{\CKn}$ in the $\mathbb F_2$-basis of $\CKn^{D}(\overline Q\times\overline Q)$ (and using that it lifts $\pi^{\Ch}$), 
we conclude that $\pi^{\CKn}$ equals the sum of the same $l_i\times h^i$ and $h^i\times l_i$ (and maybe $h^d\times\widetilde{l_d}$ for $D\equiv 0\mod 4$) as in the decomposition of $\pi^{\Ch}$, some of $v_n\,l_i\times l_{D'-i}$ and maybe $v_n^2\,l_0\times l_0$ in the case $D=2^{n+1}-2$. Let us also denote by $\pi^{\Kn}$ the image of $\pi^{\CKn}$ in $\oKn^{\,D}(Q\times Q)$.

If $\pi^{\CKn}$ has a summand $v_n\,l_{D'-j}\times l_j$ in the decomposition above, then the composition
$$
(v_n^{-1} h^j\times h^{D'-j}) \circ \pi^{\Kn}
$$ 
is rational, and either equal to $(v_n^{-1}h^j+l_{D'-j})\times h^{D'-j}$, or to $(v_n^{-1}+l_{D'}+v_n\,l_0)\times h^{D'}$ (if $j=0$, $D=2^{n+1}-2$, and $\pi^{\CKn}$ has a summand $v_n^2\,l_0\times l_0$). 
In any case, pulling it back along the diagonal we obtain that $l_0$ is rational in $\Knt Q$, 
which leads to a contradiction.

If $\pi^{\CKn}$ does not have a summand $v_n l_{D'-j}\times l_j$, then the composition
$$
\pi^{\Kn}\circ (v_n^{-1} h^j\times h^{D'-j}) 
$$ 
is either equal to $h^j\times l_j$, or to $h^{D'}\times(l_{D'}+v_n\,l_0)$ (if $j=D'$ and $\pi^{\CKn}$ has a summand $v_n^2\,l_0\times l_0$),
and pulling it back along the diagonal we again obtain that $l_0$ is rational in $\Kn(\overline{Q})$, 
which leads to a contradiction.
\end{proof}

\subsubsection{Projectors in connective Morava K-theory}
\label{sec:proj-ckn}

As shown in \cite{VishYag} (see Section~\ref{sec:prelim_vishik_yagita}),
one can lift projectors defining $\Ch$-motives to projectors in  $\mathrm{CK}(n)$-motives.
We make this lifting explicit for the projectors on quadrics below,
which will be used in the computations of the following sections. 
We assume $n\geq2$ in this section.

Let $Q$ be a smooth projective quadric of dimension $2^n-1\leq D\le 2^{n+1}-2$, $n\geq2$. 
Denote by
$$
\omega^{\CKn}_j=h^j\times l_j+l_{D'-j}\times h^{D'-j}+v_n\,l_{D'-j}\times l_j\in\CKn^D(\overline Q\times\overline Q)
$$
for $d'\leq j\leq d-1$ and $j=d$ for $D\equiv 2\mod 4$, and by
$$
\omega^{\CKn}_d=h^d\times\widetilde l_d+l_{d'}\times h^{d'}+v_n\,l_{d'}\times\widetilde l_d
$$
for $D\equiv0\mod4$, where $\widetilde l_d=l_d+h^d+v_n\,l_{d'}\in\CKnt{\overline Q}$. Straightforward calculation shows that $\omega^{\CKn}_j$ determine the decomposition of $\Delta\in\CKn^{D}(\overline Q\times\overline Q)$ as a sum of mutually orthogonal projectors (cf.~\cite[Prop.~7.2]{LPSS}).
Clearly,  $\omega^{\CKn}_j$ is a lift of the projector $\omega^{\Ch}_j$ from Section~\ref{sec:outer-ex-rev}

\begin{Prop}
In the notation above, let $I\subseteq\{d',\ldots,d\}$. Then $\sum_{i\in I}\omega^{\Ch}_i\in\Ch^D(\overline Q\times\overline Q)$ is rational if and only if $\sum_{i\in I}\omega^{\CKn}_i\in\CKn^D(\overline Q\times\overline Q)$ is rational.
\end{Prop}

\begin{proof}
The ``if'' part of the proposition is clear. Let now $\pi^{\Ch}=\sum_{i\in I}\omega^{\Ch}_i$ be rational 
and denote by $\pi^{\CKn}$ a rational projector in $\CKn^D(\overline Q\times\overline Q)$ lifting $\pi^{\Ch}$ (that exists by~\cite[Cor.~2.8]{VishYag}).

Observe that for $D$ odd and for $d'+1\leq j\leq d-1$ for $D$ even, the only idempotent in 
$$
\Enda{\MCKn{\overline Q,\,\omega^{\CKn}_j}}
$$ 
lifting $\omega^{\Ch}_j$ is $\omega^{\CKn}_j$. 
Since $\pi^{\CKn}$ restricted to $\MCKn{\overline Q,\,\omega^{\CKn}_j}$ 
is an idempotent, we conclude that $\pi^{\CKn}=\sum_{i\in I}\omega^{\CKn}_i$ in the odd-dimensional case. Next, we study the restriction of $\pi^{\CKn}$ to $$\MCKn{\overline Q,\,\omega^{\CKn}_{d'}}\oplus\MCKn{\overline Q,\,\omega^{\CKn}_d}$$ for $D$ even.

Observe also that for even $D<2^{n+1}-2$ in fact $\MCh{\overline Q,\,\omega^{\Ch}_{d'}}$ and $\MCh{\overline Q,\,\omega^{\Ch}_{d}}$ are Tate motives $\unCh\sh d$, and similarly for $\CKn$, and therefore the natural map 
$$
\Enda{\MCKn{\overline Q,\,\omega^{\CKn}_{d'}}\oplus\MCKn{\overline Q,\,\omega^{\CKn}_d}}\rightarrow\Enda{\MCh{\overline Q,\,\omega^{\Ch}_{d'}}\oplus\MCh{\overline Q,\,\omega^{\Ch}_d}}
$$ 
is an isomorphism. This implies that $\pi^{\CKn}=\sum_{i\in I}\omega^{\CKn}_i$ in this case as well.

It remains to consider the case $D=2^{n+1}-2$ (where $d'=0$). 
If, say, $0\in I$, $d\not\in I$, idempotents lifting $\omega^{\Ch}_{0}$ have form
\begin{equation}
\label{eq:additional-ckn-proj}
\omega^{\CKn}_{0}+a_1\cdot v_n\,l_{0}\times h^d+a_2\cdot v_n\,h^d\times l_{0}+a_3\cdot v_n^2\,l_{0}\times l_{0}
\end{equation}
for $a_i\in\mathbb F_2$. 
Clearly, if all $a_i=0$, then $\pi^{\CKn}$ has the required form. If all $a_i=1$, consider the morphism $\tau\in\Enda{\MCKn Q}$ given by the graph of a reflection (see Section~\ref{sec:rat-proj-quad}).  Then the image of $\pi^{\CKn}$ under $\tau$ is rational and equals $h^d\times h^d+\sum_{i\in I}\omega^{\CKn}_i$.

In the remaining cases we can push $\pi^{\CKn}$ along one of the projections $Q\times Q\rightarrow Q$ and obtain that $v_n^2\,l_{0}$ is rational. However, in this case $Q$ is isotropic by Proposition~\ref{prop:l_0_rational_isotropic}, and therefore the ``additional'' summands of~\eqref{eq:additional-ckn-proj} can be subtracted from $\pi^{\CKn}$. 
\end{proof}

\begin{Rk}
In fact, Proposition~\ref{prop:l_0_rational_isotropic} is not essential for the above argument, one can use the theorem of Levine--Morel~\cite[Th.~1.2.14]{LevMor} instead.
\end{Rk}

Let us also denote by $\omega^{\Kn}_i$ the images of $\omega^{\CKn}_i$ in $\Kn^{D}(\overline Q\times\overline Q)$. Clearly, for $I\subseteq\{d',\ldots,d\}$ an element $\sum_{i\in I}\omega^{\Kn}_i$ is rational if and only if $\sum_{i\in I}\varpi_i$ is rational (see Section~\ref{sec:rat-proj-quad}). This implies

\begin{Cr}
\label{cr:proj-ch-k}
In the notation above, let $I\subseteq\{d',\ldots,d\}$ be such that $\sum_{i\in I}\omega^{\Ch}_i\in\Ch^D(\overline Q\times\overline Q)$ is rational. Then $\sum_{i\in I}\varpi_i\in\Kn^D(\overline Q\times\overline Q)$ is also rational.
\end{Cr}

The inverse to Corollary~\ref{cr:proj-ch-k} is one of the main results of the next section.

\subsubsection{Motivic decompositions of quadrics in the pre-stable case}
\label{sec:prestable_quad}

In this section we show that 
the $\Kn$-motivic decomposition of a quadric
of dimension less than $2^{n+1}-1$ 
determines uniquely its Chow motivic decomposition. 
In this section we assume $n\geq2$, and the analogous statement for $n=1$ is clear from 
Section~\ref{sec:K0-motive-quadric}, see
Remark~\ref{rk:th414-n1} below. 

\begin{Th}\label{th:prestableMDT}\ 
\begin{enumerate}
\item 
Let $Q$ be a projective quadric of dimension $D< 2^{n+1}-1$, and $I\subseteq\{d',\ldots,d\}$.
Then the projector $\sum_{i\in I} {\omega}^{\Ch}_i$ in $\Ch^{D}(\overline Q\times\overline Q)$ is rational 
if and only if the projector $\sum_{i\in I} \varpi_i$ in $\Kn^{D}(\overline Q\times\overline Q)$ is rational.
\item
For $j=1,2$ let $Q_j$ be a projective quadric of dimension $D_j<2^{n+1}-1$,
let $I_{j}\subseteq\{d_j',\ldots,d_j\}$
be such that $\sum_{i\in I_j}\varpi_i$ is a rational $\Kn$-projector on $Q_j$ and 
 $\sum_{i\in I_j} {\omega}^{\Ch}_i$ is a rational $\Ch$-projector on $Q_j$,
and denote 
$$
N^{\Ch}_j=\MCh{Q_j,\,\sum_{i\in I_j}\omega^\Ch_i}\quad\text{ and }\quad N^{\Kn}_j=\MKn{Q_j,\,\sum_{i\in I_j}\varpi_i}.
$$ 
If $N^{\Kn}_1$ is isomorphic to a Tate twist  of $N^{\Kn}_2$, then $N^\Ch_1$ is isomorphic to a Tate twist of $N^\Ch_2$.
\end{enumerate}
\end{Th}
\begin{Rk}
The Tate twist in (2) needed for the isomorphism 
of the Chow motives is uniquely determined,
and can be deduced from the sets $I_j$, as follows from the proof. 
\end{Rk}

We will use the following lemmata in the proof. The next lemma for $Q_1=Q_2$ is~\cite[Lm.~73.2]{EKM}. 
The proof works verbatim for the following slight generalization.
\begin{Lm}
\label{lm:ldxld}
Let $Q_1$, $Q_2$ be projective quadrics of dimensions $D_1$, $D_2$, let $d_j=\left\lfloor\frac{D_j}{2}\right\rfloor$, 
and assume that $x\in\oCht{Q_2\times Q_1}$ contains  $l_i\times l_{d_1}$
in its decomposition as a sum of standard basis vectors with non-zero coefficient for some $0\leq i\leq d_2$. Then $Q_1$ is hyperbolic.
\end{Lm}

The next Lemma describes the action of the Steenrod operations on $\Cht{\overline Q_2\times\overline Q_1}$.
\begin{Lm}
\label{lm:steenrod-QxQ}
Let $Q_1$, $Q_2$ be projective quadrics of dimensions $D_j$ equal to $2^{n+1}-2$ or $2^{n+1}-3$, 
let $d_j=\left\lfloor\frac{D_j}{2}\right\rfloor$, 
 $0\leq m<n$ and $\,0\leq s\leq2^n-1$. Then
$$
\StCh(h^a \times l_b)\in\Ch^{<D_1-s}(\overline Q_2\times\overline Q_1)[[t]]
$$
for $0\leq a\leq d_2$, $b=a+s+2^n-2^m$, and
$$
\StCh(l_b \times h^a)\in\Ch^{<D_1-s}(\overline Q_2\times\overline Q_1)[[t]]
$$
for $0\leq a\leq d_1$, $b=a+D_2-D_1+s+2^n-2^m$.
\end{Lm}
\begin{proof}[Proof of Lemma~\ref{lm:steenrod-QxQ}]
By Proposition~\ref{prop:St-Ch-quad} (and multiplicativity of $\StCh$) we have
$$
\StCh(h^a \times l_b) = \StCh(h^a)\times\StCh(l_{b}) = h^a(t+h)^a \times l_{b}(t+h)^{D_1+1-b\,}t^{-1}.
$$

Note that 
$D_1+1=2^n+d_1$ under our assumptions, and 
$b\le d_1$, 
thus $(t+h)^{D_1+1-b}=(t+h)^{2^n+d_1-b}=t^{2^n}(t+h)^{d_1-b}$ 
since $h^{2^n}=0$.
Therefore the minimal index $i$ such that $l_i$ appears in the right hand side of the above equation is $i=b-(d_1-b)$.
On the other hand the maximal index $j$ such that $h^j$ appears in the right hand side of the equation is $j=2a$. 
Thus, the maximal codimension of the element 
that appears in the right hand side equals 
$$
2a+D_1-(2b-d_1)<D_1-s.
$$
The computation 
for $\StCh(l_b \times h^a)$ 
is similar.
\end{proof}

We have described in Section~\ref{sec:prelim_rat_iso_kn_quadric}
possible rational forms of isomorphisms between summands of the $\Kn$-kernels of quadrics.
In the following lemma we use the notation $a_i, b_j$ from there.

\begin{Lm}
\label{lm:simple-shift}
In the notation and the assumptions of Theorem~\ref{th:prestableMDT} (2), 
for $j=1,2$ 
assume moreover that $D_j$ is either $2^{n+1}-2$ or $2^{n+1}-3$, 
and that $|I_1|=|I_2|\ge 2$. 

Then there exists an isomorphism $N^{\Kn}_1\xrarr{\simeq} N^{\Kn}_2\sh{s}$ for $0\leq s<2^n-1$ 
that over $\overline{k}$ has the following form:
\begin{equation}
\label{eq:form-of-psi}
\Psi = \sum_{i\in I_1} v_n^{-1}\, a_i \times b_{D_2'-f(i)},
\end{equation}
where 
$f\colon I_1\rarr I_2$ is a bijection of one of the following types:
\begin{enumerate}
    \item $f(i) = i - s$ for all $i\in I_1$,
    \item $f(i) = i-s +(2^n-1)$  for all $i\in I_1$.
\end{enumerate}
\end{Lm}
\begin{proof}
    We will need to distinguish the two cases as in Proposition~\ref{prop:quadric_kn-iso_normal_form}.

    {\bf Case~1.} Assume that $D_1=2^{n+1}-2$  and $\{d_1',d_1\}\subseteq I_1$. 
    In this case we must have $D_2=D_1$, $s=0$, $\{d_2', d_2\} \subseteq I_2$
    by counting the multiplicity of the Tate motives in $\overline N^{\,\Kn}_1$, $\overline N^{\,\Kn}_2$.
    By Proposition~\ref{prop:quadric_kn-iso_normal_form}, 
    $\Psi$ equals 
    $$
    \psi+\sum_{i\in I_1\setminus\{d_1', d_1\}} v_n^{-1}\, a_i \times b_{D_2'-i},
    $$
    where $\psi$ can have one of the two 
    forms.  
    The first form of $\psi$ is precisely the one claimed in~\eqref{eq:form-of-psi} with $f(i)=i$ for all $i\in I_1$. 
    The second form of $\psi$ is $v_n^{-2}\,a_{d'}\times b_{d'}+a_{d}\times b_{d}$. 
    However, in this case 
    $v_n^2\,h^{d}\times h^{d}\cdot\,\Psi=a_{{d}}\times b_{d}$ is rational,
    and the composition $\Psi^{-1}\circ (a_{{d}}\times b_{d})$ equals $v_n^{-1}\,a_{d}\times a_{d'}$.
    This is a rational projector, and its existence contradicts indecomposability of $N_1$. 

    {\bf Case 2.}
    Assume that $D_j$, $I_j$ are not as in the Case~1. 
    Then $\Psi$ has the form~\eqref{eq:form-of-psi} by Proposition~\ref{prop:quadric_kn-iso_normal_form} 
    and we need to show that the bijection $f\colon I_1\rarr I_2$ is either a shift by $-s$ or by $-s+(2^n-1)$. 
    
    Assume the contrary.
    Note that $I_1, I_2 \subseteq \{0, \ldots, 2^n-1\}$,
    and both of the elements $0$ and $2^n-1$ cannot appear in  $I_1$ or $I_2$ by the assumption. 
    Therefore, for $i\in I_1$ we may have either $f(i)=i-s$, or $f(i)=i-s+(2^n-1)$,
    or if $i=2^n-1$ and $s=0$, we may have $f(2^n-1)=0$. In the latter case, using $s=0$,
    we exchange $N_1$ and $N_2$, and replace $f$ by $f^{-1}$. 
    Thus, we may assume that for all $i\in I_1$, $f(i)\in \{i-s, i-s+(2^n-1)\}$.
    Moreover, $f(i)$ is uniquely determined by $i$ and $s$: only one of the numbers $i-s$, $i-s+2^n-1$ lies in $I_2$ by our assumptions.

    Let $I_1 = \{i_1,\ldots, i_m\}$ where $i_1<\ldots<i_m$, 
    and since $f$ is not a shift by our assumption, we have $f(i_m) = i_m -s$ 
    and $f(i_1) = i_1-s+2^n-1$. 
    Then 
    there exists $k$, $1\le k<m$, such that $f(i_t)=i_t-s+2^n-1$ for $t\le k$
    and $f(i_t) = i_t - s$ for $t>k$. Then 
    $$
    f(i_{k+1})<\ldots<f(i_m)<f(i_1)<\ldots<f(i_k),
    $$
    and let $x=f(i_k)-f(i_{k+1})> 0$. 
    Using the explicit formulas for $\Psi, \Psi^{-1}$ from Proposition~\ref{prop:quadric_kn-iso_normal_form},
    we compute that the rational endomorphism $\Psi^{-1}\circ (\Psi \cdot (1\times h^x))$ of $N_1$
    has the following form: 
    $$
    \sum_{(D_2'-f(i)+x) + f(j) =D_2'} v_n^{-1}\, a_i \times a_{D_1'-j}.
    $$
    However, $f(j) - f(i) = x$ is only possible for $j=i_k$, $i=i_{k+1}$ and the above sum has only one element.
    By multiplying this element with $h^y\times 1$ for $y=i_{k+1}-i_k$, 
    we get that the projector $v_n^{-1}\,a_{i_{k+1}} \times a_{D_1'-i_{k+1}}$ is rational, 
    contradicting indecomposability of $N_1$ (and $|I_1|\geq2$). 
    
    Recall that in the case $s=0$ and $f(2^n-1)=0$ 
    we exchanged $f$ by $f^{-1}$. Thus, we proved that $f^{-1}$ is the shift of one of the required forms,
    i.e.\ $f(i)$ is either $i$ or $i-(2^n-1)$ for all $i \in I_1$.
    Since $f(2^n-1)=0$, the latter must be the case. But for $i-(2^n-1)$ to lie in $I_2\subseteq \{0,\ldots, 2^n-1\}$ for all $i\in I_1$,
    we need to have $I_1=\{2^n-1\}$, which contradicts the assumption $|I_1|\geq2$. Therefore this case, in fact, does not happen.
    \end{proof}

\begin{proof}[Proof of Theorem~\ref{th:prestableMDT}]
By Corollary~\ref{cr:proj-ch-k}, to get $(1)$
we only have to show that if $\varpi:=\sum_{i\in I} \varpi_i$ is rational,
then $\omega:=\sum_{i\in I} {\omega}^\Ch_i$ is also rational. 
Similarly to prove $(2)$, 
given a rational $\Psi^{\Kn}\in\Knt{\overline{Q}_1\times\overline{Q}_2}$ that determines an isomorphism between
$\overline{N}^{\,\Kn}_1$ and a Tate twist of $\overline{N}^{\,\Kn}_2$, 
we need to find a rational element $\Psi^{\Ch}$ in $\Cht{\overline{Q}_1 \times \overline{Q}_2}$
that determines an isomorphism between $\overline N^{\,\Ch}_1$ and a Tate twist of $\overline N^{\,\Ch}_2$.  
We prove $(2)$ below, and the proof of $(1)$ can be obtained
repeating the same argument for $Q_1=Q_2$. 

As we do not assume that the corresponding quadratic forms $q_j$ are isotropic,
we can always change $q_j$ by $q_j\perp \mathbb{H}^{\oplus k_j}$ using Lemma~\ref{lm:motive_of_isotropic_quadric} and Proposition~\ref{prop:normal}.
Thus, we can assume that the dimensions of $Q_j$ are  either $2^{n+1}-2$ or $2^{n+1}-3$. 
Clearly, we can also assume that $N^{\Kn}_j$ are indecomposable. 

Consider first the case when $|I_1|=|I_2|\ge 2$.
The isomorphism $N^{\Kn}_1\cong N^{\Kn}_2\sh s$ for $0\leq s<2^n-1$ is determined by a rational element
$\Psi^{\Kn} \in \Kn^{D_2+s}(\overline{Q}_1\times\overline{Q}_2)$, 
and we may assume by Lemma~\ref{lm:simple-shift} that  
$$
\Psi^{\Kn} = \sum_{i\in I_1} v_n^{-1}\,(h^i + v_n l_{D_1'-i})\times(h^{D_2'-f(i)} + v_n l_{f(i)}),
$$
where $f(i) = i - s$ for all $i\in I_1$. 
Indeed, if $f(i)=i+(2^n-1-s)$ one can just exchange $Q_1$ and $Q_2$, and $s$ by $2^n-1-s$.
Note that after this exchange $s$ could become $2^n-1$, 
however, 
due to the condition $I_1,\,I_2\subseteq\{0,\ldots,2^n-1\}$, this would contradict the assumption $|I_1|=|I_2|\ge 2$
and therefore does not happen.

Assume now $|I_1|=|I_2|=1$, let $I_1=\{i\}$, and $f\colon I_1\rightarrow I_2$ 
denote the bijection induced by isomorphism $\Psi^{\Kn}$ between $N^{\Kn}_1$ and a Tate twist of $N^{\Kn}_2$. 
If $I_1=\{0\}$, $I_2=\{2^n-1\}$ or vice versa, 
we set $s=2^n-1$. 
Exchanging $Q_1$ and $Q_2$, if necessary, we may write $f(i)=i-s$. 
In all other cases, due to the condition $I_1,\,I_2\subseteq\{0,\ldots,2^n-1\}$,
we can assume that $f(i)=i-s$ for some $0\leq s<2^n-1$, again exchanging $Q_1$ and $Q_2$, if necessary. 
Then $\Psi^{\Kn}$ has the form~\eqref{eq:form-of-psi} by Proposition~\ref{prop:quadric_kn-iso_normal_form}\,(2). 

It will be more convenient
for us to work with the composition inverse $\left(\Psi^{\Kn}\right)^{-1}\in\Kn^{D_1-s}(\overline{Q}_2\times\overline{Q}_1)$ below, because $\codim{\left(\Psi^{\Kn}\right)^{-1}}\leq D_1\leq 2(2^n-1)$. Clearly, 
$$
\left(\Psi^{\Kn}\right)^{-1}= \sum_{i\in I_1} v_n^{-1}\,(h^{f(i)} + v_n l_{D_2'-f(i)})\times(h^{D_1'-i} + v_n l_{i}),
$$ 
and the
rationality of $\left(\Psi^{\Kn}\right)^{-1}$ 
is equivalent to the rationality of the element 

$$
\Theta= \sum_{i\in I_1} \left(h^{i-s}\times l_{i} + l_{D_2'-i+s}\times h^{D_1'-i} + v_n l_{D_2'-i+s} \times l_{i}\right). 
$$

The strategy of the rest of the proof is as follows. We lift the element $\Theta$ 
to some rational element in $\BPt{\overline{Q}_2\times\overline{Q}_1}$, 
and then apply the traces of symmetric operations $\phi^{t^k}$ to it to obtain a rational element 
$$
\Psi^\Ch:=\sum_{i\in I_1} \left(h^{i-s}\times l_{i} + l_{D_2'-i+s}\times h^{D_1'-i} \right)
$$
in $\Ch^{D_1-s}(\overline{Q}_2\times\overline{Q}_1)$. This element defines a rational isomorphism
between $\overline{N}^{\,\Ch}_2\sh s$ and $\overline{N}^{\,\Ch}_1$.

By Lemma~\ref{lm:from_kn_to_bp} there exists an element $z\in \oBPt{Q_2\times Q_1} \subseteq \BPt{\overline{Q_2}\times \overline{Q_1}}$ of $\codim z>0$ 
such that $\Theta = z\otimes v_n^{-t}$. 
Since $\codim z = \codim \Theta -t(2^n-1) = D_1-s -t(2^n-1)$, we have $t<2$. 
Let us denote $\mathrm{pr}\colon\BP\rightarrow\Ch$ the canonical morphism of theories. 
If $t\le 0$, we have $\Psi^\Ch=\mathrm{pr}(v_n^{-t} z)$, and the claim follows.
Therefore we may assume that $t=1$.

Clearly, $\mathrm{pr}(z)=0\in\mathrm{Ch}(\overline Q_2\times\overline Q_1)$ and 
the image of $z$ in $\CKnt{\overline Q_2\times\overline Q_1}$ coincides with $v_n\Theta$, 
therefore we can rewrite $z$ as follows:
$$
z=2z_0+\sum_{m>0}v_mz_m+\sum_{i,j>0}v_iv_jz_{ij}
$$
for $z_m$, $z_{ij}\in\BPt{\overline Q_2\times\overline Q_1}$, where 
$\mathrm{pr}(z_n)=\Psi^\Ch\in\Cht{\overline Q_2\times\overline Q_1}$. 
Then 
$
\phi^{t^0}(z)=\mathrm{pr}(z_0)^2+\sum_{m>0}\phi^{t^0}(v_mz_m)
$
and $\phi^{t^k}(z)=\sum_{m>0}\phi^{t^k}(v_mz_m)$ for $k>0$ by Propositions~\ref{prop:phi-add},~\ref{prop:symm_divide},~\ref{prop:eta}. 

Observe that
a square in $\Cht{\overline Q_2\times\overline Q_1}$ is always equal to a sum of $h^{2a}\times h^{2b}$, in particular, it is always a rational element and we may subtract it from the result of the operation if necessary.

Observe also that  
$k:=2(2^n-1)-\mathrm{codim}(z_n)=2(2^n-1)-D_1+s\geq0$, and therefore 
$$
\phi^{t^k}(v_nz_n)=\mathrm{pr}(z_n)=\Psi^\Ch\in\Ch^{D_1-s}(\overline Q_2\times\overline Q_1)
$$
by Proposition~\ref{prop:symm_divide}. It remains to show that $\phi^{t^k}(v_mz_m)=0$ for $m\neq n$.

First, for $m>n$ 
one has $z_m\in \tau^{>D_1-s\,}\BPt{\overline{Q}_2\times\overline Q_1}$, 
and since operations preserve topological filtration, 
$\phi^{t^k}(v_mz_m)$ vanishes in $\Ch^{D_1-s}(\overline Q_2\times\overline Q_1)$.

We now compute $\phi^{t^k} (v_m z_m)$ for $m<n$. 
Then $z_m$ has  codimension $D_1-s+2^m-2^n$, and
it can be written as a sum of the elements $h^a \times l_b$ where  
$b=a+s-2^m+2^n$, 
the elements $l_b\times h^a$ where $b=a+D_2-D_1+s+2^n-2^m$,
the elements $l_a \times l_b$, 
and the elements $h^a \times h^b$. To compute the value of the traces of the symmetric operation we use Propositions~\ref{prop:symm_divide},~\ref{prop:St-Ch-quad} and Lemma~\ref{lm:steenrod-QxQ}.

Since $\phi^{t^k}(v_m\,h^a \times h^b)$ is always rational, we can subtract it from 
from the result, 
so we do not need to consider it. 
By Proposition~\ref{prop:symm_divide} and Lemma~\ref{lm:steenrod-QxQ}, $\phi^{t^k}(h^a \times l_b)$ and $\phi^{t^k}(l_b\times h^a)$ cannot equal to a non-zero element of codimension $D_1-s$.

Finally, the value $\phi^{t^k}(v_m\, l_a \times l_b)$ is a linear combination
of the elements of the form $l_u\times l_v$ in Chow groups.
However, none of these elements lie in $\Ch^{D_1-s}(\overline Q_2\times\overline Q_1)$ except for $l_{d_2}\times l_{d_1}$ (in the case $s=0$). But the rational element $\phi^{t^k}(z)$ cannot contain a summand of this form by Lemma~\ref{lm:ldxld}.

This shows that $\Psi^\Ch=\phi^{t^k}(z)$ is rational, 
and it defines an isomorphism between $\overline{N}^{\,\Ch}_2\sh s$ and $\overline{N}^{\,\Ch}_1$.
This implies that ${N}^{\Ch}_1$ and ${N}^{\Ch}_2\sh s$ are isomorphic as well, see e.g.~\cite[Cor.~92.7]{EKM}.
\end{proof}

In Theorem~\ref{th:prestableMDT} we have established a correspondence 
between the direct summands of $\Mker Q$ and the direct summands of $\MCh Q$ 
that uses the data of rational projectors defining these summands. 
Any bijection between {\sl isomorphism classes} of summands of $\MCh Q$ 
and summands of $\Mker Q$ has to depend on some choices, as the following example shows. 
Consider  $Q$, an $(n+1)$-fold Pfister quadric corresponding to a pure symbol $\alpha$,
its $\Kn-$ and Chow motivic decompositions are described in Proposition~\ref{prop:prelim_motive_pfister}.
In particular, the multiplicity of $L_\alpha\sh0$ in the $\Kn$-motive of $Q$ is 2,
and  the ``corresponding'' summands of $\MCh Q$ are $R_\alpha$ and $R_\alpha\sh{2^n-1}$,
which are isomorphic only up to a Tate twist.

However, it follows from the proof of Theorem~\ref{th:prestableMDT} together with Lemma~\ref{lm:kn-normal-form-well-def} below, 
that the above example is, in fact, the only one case where isomorphic direct summands of  $\Mker Q$
correspond to non-isomorphic summands of $\MCh Q$.

\begin{Rk}
\label{rk:th414-n1}
An analogue of Theorem~\ref{th:prestableMDT} for $n=1$ concerns $\K1$ and $\Ch$-motives of anisotropic quadrics of dimension at most $2$,
among which only the case of dimension $2$ is not straightforward. 
In this case, if the discriminant is trivial, i.e., for a $2$-fold Pfister quadric,
the claim amounts to~Propositions~\ref{prop:prelim_L_alpha_Rost},~\ref{prop:prelim_motive_pfister}. 
If the discriminant is non-trivial, the Chow motive is indecomposable, which reduces the claim to the isomorphisms of motives of quadrics, rather than their direct summands. 
However, the isomorphism class of the quadric of dimension $2$
is determined by its cohomological invariants: the discriminant and the Clifford invariant, see e.g.~\cite[Ch.~XII, Prop.~2.4, Th.~2.1]{Lam}. 
Since these can be read off from the $\KK$-motive (and also from the $\K1$-motive) of the quadric, see Section~\ref{sec:K0-motive-quadric},
the claim also follows.
\end{Rk}

\subsection{Morava MDT}
\label{sec:morava-mdt}

One can talk about the Morava K-theory decomposition type of $Q$ 
in analogy with Section~\ref{sec:ch-mdt-conn}. This allows to visualize the structure of motivic summands of $\MKn Q$ as a picture,
similarly to the case of Chow motives. 
In fact, to describe the MDT of the $\Kn$-motive one can always pass to the pre-stable case by Proposition~\ref{prop:quad_MDT_stable} and Lemma~\ref{lm:motive_of_isotropic_quadric}, where motivic summands of Morava and Chow motives are in $1$-to-$1$ correspondence by Theorem~\ref{th:prestableMDT}. 

Denote by $\Lamn Q$ the set of Tate summands of $\Mker{\overline Q}$ 
corresponding to $\varpi_i$, $d'\leq i\leq d$, 
then for an indecomposable summand $N$ of $\Mker Q$ 
there exists an isomorphic summand $N'\cong N$ and a set $\Lamn N\subseteq\Lamn Q$ 
such that $\overline{N'}$ is a sum of Tate motives from $\Lamn N$ by Proposition~\ref{prop:normal}. 
\begin{Lm}
\label{lm:kn-normal-form-well-def}
In the notation above, assume that $q\not\in I^{n+1}(k)$. Then the set $\Lamn N$ does not depend on the choice of $N'$.
\end{Lm}
\begin{proof}
We can assume that $D\leq 2^{n+1}-2$ by Proposition~\ref{prop:quad_MDT_stable}. 
If we have two distinct subsets $I_1\neq I_2\subseteq\{d',\ldots,d\}$ such that the corresponding motives $N^{\Kn}_j$ in the notations of Theorem~\ref{th:prestableMDT} are isomorphic, then the motives $N^{\Ch}_j$ are also isomorphic up to a Tate twist by $\un(s)$.
Clearly, $s=0$ is impossible since Chow MDT is well defined,
then by dimensional reasons $s=\pm(2^n-1)$ and $N^{\Ch}_j$ are binary summands of length $2^n-1$. However, in this case $Q$ is split over its field of functions, therefore $q$ is a general $(n+1)$-Pfister form, which contradicts the assumption $q\not\in I^{n+1}(k)$.
\end{proof}

For $q\in I^{n+1}(k)$ the situation is different: by Proposition~\ref{prop:quad_MDT_stable} 
we can assume that $q$ is an $(n+1)$-general Pfister form, 
and therefore $\Mker Q$ decomposes as a direct sum of invertible motives corresponding to the Chow Rost motives $R_\alpha$. 
Since the ``middle'' invertible summands (corresponding to $R_\alpha$ and $R_\alpha\sh{2^n-1}$) are isomorphic, $\widetilde\Lambda$ {\sl does} depend on the choice we make in this case. However, irrespective of whether $q$ lies in $I^{n+1}(k)$ or not, one can give the following definition, similarly to $\Ch$-case.

\begin{Def}
We say that $L$, $L'\in\Lamn Q$ are {\sl connected} if there exists an indecomposable summand $N$ of $\Mker Q$ such that $L$, $L'\in\Lamn N$. The set of connected components of $\Lamn Q$ is called the {\sl Morava K-theory motivic decomposition type} of $Q$ (Morava MDT for short).
\end{Def}

Observe also that the summand of $\MCh{\overline Q}$ corresponding to $\omega^\Ch_{d'}$ contains the upper ``middle'' Tate $\unCh\sh d$ of $\Lambda(Q)$ as a direct summand, and the summand, corresponding to $\omega^\Ch_{d}$, contains the lower  $\unCh\sh d$.  
Due to the correspondence between $\omega^\Ch_{i}$ and $\varpi_{i}$ of Theorem~\ref{th:prestableMDT}, it is natural to use a similar terminology for the Tate summands of $\Mker{\overline Q}$ corresponding to $\varpi_{d'}$ and $\varpi_d$, i.e., call $\MKn{\overline Q,\,\varpi_{d'}}$ the upper Tate $\unKn\sh d$ of $\Lamn{\overline Q}$, and $\MKn{\overline Q,\,\varpi_{d}}$ the lower one.

To visualize Morava MDT we can denote the Tate motives from $\Lamn Q$ by 
dots 
``\,\begin{tikzpicture}
\filldraw [white] (0,-0) circle (1pt);
\filldraw [black] (0,0.05) circle (1pt);
\end{tikzpicture}\,''
as in the Chow case, however, due to the fact that $\unKn\cong\unKn\sh{2^n-1}$ there is no immediate choice, which of the Tate motives to draw as the most left one. To make the correspondence between $\omega^\Ch_{i}$ and $\varpi_{i}$ of Theorem~\ref{th:prestableMDT} more transparent, it is natural to place the upper and the lower Tate motives in the middle for $D$ even. 

\begin{Ex}
\label{ex:chow-morava-mdt}
Let $q=\langle\langle t_0\rangle\rangle\otimes\langle 1,\,-t_1,\,-t_2,\,-t_3\rangle$ over $\mathbb Q(t_0,\,t_1,\,t_2,\,t_3)$. By~\cite[Th.~6.1,\,Ex.\,(1)]{Vish-quad} its Chow MDT has the following form, where outer excellent connections are drawn in red. According to Theorem~\ref{th:prestableMDT}, these connections should be \textsc{contracted} to obtain Morava MDT of $Q$ for $n=2$. We also depict with grey dots $4$ Tate summands complementary to the $\K2$-kernel summand of $Q$. This gives us the alternative way to switch from one picture to the other: just \textsc{deleting} outer excellent connections.

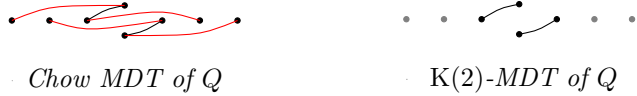
\begin{center}
\begin{tikzpicture}
\filldraw [white] (-0.3,0.5) circle (1pt);
\filldraw [black] (0,0) circle (1pt);
\filldraw [black] (0.5,0) circle (1pt);
\filldraw [black] (1,0) circle (1pt);
\filldraw [black] (1.5,0.2) circle (1pt);
\filldraw [black] (1.5,-0.2) circle (1pt);
\filldraw [black] (2,0) circle (1pt);
\filldraw [black] (2.5,0) circle (1pt);
\filldraw [black] (3,0) circle (1pt);
\filldraw [white] (3.3,0) circle (1pt);
\draw (1,0) .. controls (1.25,0.15) ..(1.5,0.2);
\draw (2,0) .. controls (1.75,-0.15) ..(1.5,-0.2);
\draw[red] (0.5,0) .. controls (1,-0.15) .. (2,0);
\draw[red] (1,0) .. controls (2,0.15) .. (2.5,0);
\draw[red] (0,0) .. controls (0.5,0.2) .. (1.5,0.2);
\draw[red] (3,0) .. controls (2.5,-0.2) .. (1.5,-0.2);
\filldraw [black] (0,-0.8) circle (0pt) node[anchor=west]{\text{$\,$Chow MDT of $Q$}};
\end{tikzpicture}
\qquad\qquad
\begin{tikzpicture}
\filldraw [white] (-0.3,0.5) circle (1pt);
\filldraw [gray] (0,0) circle (1pt);
\filldraw [gray] (0.5,0) circle (1pt);
\filldraw [black] (1,0) circle (1pt);
\filldraw [black] (1.5,0.2) circle (1pt);
\filldraw [black] (1.5,-0.2) circle (1pt);
\filldraw [black] (2,0) circle (1pt);
\filldraw [gray] (2.5,0) circle (1pt);
\filldraw [gray] (3,0) circle (1pt);
\filldraw [white] (3.3,0) circle (1pt);
\draw (1,0) .. controls (1.25,0.15) ..(1.5,0.2);
\draw (2,0) .. controls (1.75,-0.15) ..(1.5,-0.2);
\filldraw [black] (0,-0.8) circle (0pt) node[anchor=west]{\text{$\ \,\K2$-MDT of $Q$}};
\end{tikzpicture}
\end{center}
We will determine the $\K2$-motive of $Q$ in Example~\ref{ex:morava-chow-mdt} independently of Vishik's results, in particular the Chow MDT of $Q$ can be recovered from its $\K2$-MDT by \textsc{adding} outer excellent connections. 
\end{Ex}

For $D$ odd one can switch between Chow and Morava MDT in the same way, but we warn the reader that the $\Kn$-kernel summand of the odd dimensional quadric $Q$ has an {\sl odd} number of complementary Tate motives in $\MKn Q$. We would also depict them as grey dots by both sides of the kernel summand, but we cannot distribute them equally.

\begin{Ex}
Let $Q$ be an isotropic non-hyperbolic $3$-dimensional quadric, its Chow MDT has the following form. The anisotropic part $Q'$ of $Q$ does not have outer excellent connections of length $3$, therefore $\K2$-MDT of $Q'$ coincides with its Chow MDT. The $\K2$-kernel summand of $Q$ has rank $3$ with one complementary Tate summand $\unKn\sh 0$, therefore $\widetilde{\,\mathrm{M}}_{\K2}(Q)$ is a sum of the Tate summand $\unKn\sh 0$ and the indecomposable binary summand $\Mot{\K2}{Q'}\sh 1$. We could choose one of the following option to visualize $\K2$-MDT of $Q$ in one of the following forms.

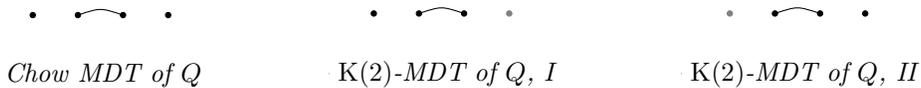
\begin{center}
\begin{tikzpicture}
\filldraw [white] (-0.3,0.5) circle (1pt);
\filldraw [black] (-0.1,0) circle (1pt);
\filldraw [black] (0.5,0) circle (1pt);
\filldraw [black] (1.1,0) circle (1pt);
\filldraw [black] (1.7,0) circle (1pt);
\filldraw [white] (1.9,0) circle (1pt);
\draw (0.5,0) .. controls (0.8,0.1) .. (1.1,0);
\filldraw [black] (-0.6,-0.8) circle (0pt) node[anchor=west]{\text{Chow MDT of $Q$}};
\end{tikzpicture}
\qquad\qquad
\begin{tikzpicture}
\filldraw [white] (-0.3,0.5) circle (1pt);
\filldraw [black] (-0.1,0) circle (1pt);
\filldraw [black] (0.5,0) circle (1pt);
\filldraw [black] (1.1,0) circle (1pt);
\filldraw [gray] (1.7,0) circle (1pt);
\filldraw [white] (1.9,0) circle (1pt);
\draw (0.5,0) .. controls (0.8,0.1) .. (1.1,0);
\filldraw [black] (-0.7,-0.8) circle (0pt) node[anchor=west]{\text{$\K2$-MDT of $Q$, I}};
\end{tikzpicture}
\qquad\qquad
\begin{tikzpicture}
\filldraw [white] (-0.3,0.5) circle (1pt);
\filldraw [gray] (-0.1,0) circle (1pt);
\filldraw [black] (0.5,0) circle (1pt);
\filldraw [black] (1.1,0) circle (1pt);
\filldraw [black] (1.7,0) circle (1pt);
\filldraw [white] (1.9,0) circle (1pt);
\draw (0.5,0) .. controls (0.8,0.1) .. (1.1,0);
\filldraw [black] (-0.75,-0.8) circle (0pt) node[anchor=west]{\text{$\K2$-MDT of $Q$, II}};
\end{tikzpicture}
\end{center}
\end{Ex}

When we do not compare Chow and Morava MDT, we omit grey dots.

\subsection{Application: criteria of isomorphism}

As an application of the techniques we develop, we obtain several Morava analogues of celebrated results on the Chow motives of quadrics. 
\subsubsection{Vishik's criterion of isomorphism}
We recall Vishik's criterion of isomorphism of Chow motives 
of quadrics~\cite[Prop.~5.1]{VishInt}, see also~\cite[Th.~4.18]{Vish-quad}, \cite[Th.~93.1]{EKM};
for generalizations see~e.g.~\cite{DC, DCQM}.

\begin{Th}[Vishik]
The Chow motives of projective quadrics $Q$ and $Q'$ over $k$ of the same dimension are isomorphic
if and only if  $\iw(q_K)=\iw(q'_K)$ for any field extension $K/k$.
\end{Th}

We can prove the following $\Kn$-analogue of this criterion. We use the notation of Section~\ref{sec:rat-proj-quad}.

\begin{Prop}
\label{kn-vishik-criterion}
The $\Kn$-motives of projective quadrics $Q$ and $Q'$ over $k$ of the same dimension $D$ are isomorphic 
if and only if for any field extension $K/k$ if $\iw(q_K)$, $\iw(q'_K)\geq d'$, then $\iw(q_K)=\iw(q'_K)$.
\end{Prop}

\begin{proof}
Let $\MKn Q\cong\MKn{Q'}$
and $K/k$ be such that $\iw(q_K)$, $\iw(q'_K)\geq d'$. Then 
$$
\mathrm{dim}\big((q_K)_{\mathrm{an}}\big),\,\mathrm{dim}\big((q'_K)_{\mathrm{an}}\big)\leq 2^{n+1},
$$
and therefore the Chow motives of $Q_K$ and $Q'_K$ 
are isomorphic by Theorem~\ref{th:prestableMDT} (recall the decomposition of the motive of an isotropic quadric
from Section~\ref{sec:rat-proj-quad}), 
in particular, $\iw(q_K)=\iw(q'_K)$.

To prove the ``if'' part, 
we can assume that $\mathrm{dim}(q),\,\mathrm{dim}(q')\leq 2^{n+1}$ 
by base change to the composite of the fields of definition of their $\Kn$-kernels.  
Then for any field extension $K/k$ one has $\iw(q_K)=\iw(q'_K)$, and by Vishik's Criterion we conclude 
that the Chow (and therefore $\Kn$-) motives of $Q$ and $Q'$ are isomorphic. 
\end{proof}

\begin{Rk}
This criterion can alternatively be stated in terms of simultaneous rationality of $l_j$;
we note, however, that it seems not to be possible to re-formulate this criterion
in terms of equal number of Tate summands over all field extensions, 
as it is done in~\cite{DC,DCQM} for Chow motives.
\end{Rk}

\subsubsection{Similarity of quadratic forms modulo powers of fundamental ideal}

We recall Izhboldin's criterion of isomorphism of $\CH$-motives of {\sl odd-dimensional} quadrics~\cite[Cor.~2.9]{Izh}, see also~\cite[Rem.~93.7]{EKM}. 

\begin{Th}[Izhboldin]
\label{odd-chow}
Chow motives of odd-dimensional projective quadrics $Q$ and $Q'$ over $k$ are isomorphic
if and only if $q$ and $q'$ are similar:  $q'\cong c\,q$ for some $c\in k^\times$.
\end{Th}

To state the $\Kn$-analogue of this result, we start with the following observation, cf.~\cite[Prop.~6.18]{SechSem}.

\begin{Prop}
\label{similar-motives}
Let $q$, $q'$ be quadratic forms over $k$ of the same dimension and similar modulo $I^{n+2}(k)$. Then the $\Kn$-motives of $Q$ and $Q'$ are isomorphic.
\end{Prop}
\begin{proof}
Let $p=q'\perp-c\,q\in I^{n+2}(k)$.
The $\Kn$-motives of projective quadrics defined by $p\perp c\,q$ and $\hyp^{\,\dim q}\perp c\,q$ are isomorphic, 
since this can be checked over the splitting field of $p$ by Corollary~\ref{cr:Kn-split-quad}. 
However, $p\perp c\,q=q'\perp \mathbb H^{\,\mathrm{dim}(q)}$, therefore the claim follows from Lemma~\ref{lm:motive_of_isotropic_quadric}.
\end{proof}

Omitting the condition on dimension in the above Proposition, one can still  
compare $\Kn$-motives of quadratic forms similar modulo $I^{n+2}(k)$ using the same type of argument. 
For simplicity, we state it only for forms of sufficiently big dimension.

\begin{Cr}
\label{cr:similar-kernel}
Let $q$, $q'$ be quadratic forms over $k$ similar modulo $I^{n+2}(k)$, 
such that $\dim(q')\ge \dim(q)\geq 2^n$.
Let $s:=\frac{1}{2}(\dim(q')-\dim(q))$. 
Then $\Mker{Q'}\cong\Mker Q\sh s$.
\end{Cr}

For odd-dimensional quadrics we can prove the ``inverse'' to Proposition~\ref{similar-motives}. 

\begin{Prop}
\label{odd-criterion}
$\Kn$-motives of projective quadrics $Q$ and $Q'$ over $k$ of the same odd dimension are isomorphic if and only if $q$ and $q'$ are similar modulo $I^{n+2}(k)$.
\end{Prop}

\begin{proof}
The ``if''-part is proven in Proposition~\ref{similar-motives}.
Assume that the $\Kn$-motives of $Q$ and $Q'$ be isomorphic. 
Let $K$ be the composite of the field of definitions of $\Kn$-kernels of $q,q'$.
Then we have $\dim (q_K)_{\mathrm{an}}$, $\dim (q'_K)_{\mathrm{an}}<2^{n+1}$. 
Using Proposition~\ref{kn-vishik-criterion}, we have $\dim (q_K)_{\mathrm{an}}=\dim (q'_K)_{\mathrm{an}}$ and 
the $\Kn$-motives of the corresponding quadrics are isomorphic by Lemma~\ref{lm:motive_of_isotropic_quadric},
therefore their Chow motives are isomorphic as well by Theorem~\ref{th:prestableMDT}. 

Using Lemma~\ref{lm:motive_of_isotropic_quadric} again,
we conclude that the Chow motives of $Q_K$ and $Q'_K$ are isomorphic, and using Proposition~\ref{odd-chow} we conclude that $q_K$ and $q'_K$ are similar, i.e., $q'_K\cong c\,q_K$ for some $c\in K^\times$. 
 In particular, $\mathrm{det}(q')=c^{\,\mathrm{dim}(q)}\mathrm{det}(q)$ as an element of $K^\times/K^{\times2}$, and since $q$ is odd-dimensional we can assume that 
$c\in k^\times$. 
Finally, by Corollary~\ref{cr:KRS-Witt_mod_In+1_injectivity} we conclude that $q'\perp-c\,q\in I^{n+2}(k)$.
\end{proof}

\begin{Rk}
Recall that Proposition~\ref{odd-chow} does not have an analogue for {\sl even-dimensional} quadrics. In fact, Izhboldin proved in~\cite[Cor.~4.8]{Izh} that 
for any $n\geq3$, there exist a field $k$ 
and quadratic forms $q,\,q'\in I^n(k)$ of $\,\mathrm{dim}(q)=\mathrm{dim}(q')=2^{n+1}-2$ such that 
the Chow motives of $Q$ and $Q'$ are isomorphic, 
but $q$ and $q'$ are not similar.
It follows
that the analogue of Proposition~\ref{odd-criterion} for even-dimensional quadrics does not hold.
\end{Rk}

\section{Morava motives over the fields of functions}\label{sec:morava_motives_over_function_fields}

In this section we investigate the behavior of $\A$-motives over function fields of some varieties
with the main applications to $\Kn$-motives over splitting fields $k(\alpha)$ of elements $\alpha \in \HH^{n+1}(k,\,\ZZ/2)$.
We show that
the isomorphism of motives $M, N \in \CM_{\Kn}(k)$ over $k(\alpha)$
implies that $M$ is either isomorphic to $N$, or to $N\otimes L_\alpha$ under mild assumptions on $M, N$ (Theorem~\ref{th:iso_over_k(alpha)}).
Under mild assumption on indecomposable motive $M\in \CM_{\Kn}(k)$ 
we show that it either stays indecomposable over $k(\alpha)$ or it splits into two isomorphic summands (Proposition~\ref{prop:decomposition_over_k(alpha)}).
For example, both of the results above can be applied to motives that are direct summands of the motives of projective homogeneous varieties.

The results above are obtained using the category of motives over a base
and the geometric Rost Nilpotence Property (Section~\ref{sec:geometric_RNP}).
Recall that the latter allows to lift decompositions and isomorphisms of motives 
from $\CM_{\A}(k(X))$ to $\CM_{\A}(X)$. However,
studying motivic decompositions over $X$ instead of over $k(X)$ has its benefits:
for example, there is a push-forward functor back to motives over $k$, if $X$ is smooth projective. 

Another question that we address in this section is the following:
when do Tate motives split off from a motive $M$ over some function field $k(X)$.
This splitting over $k(X)$ should, in some sense, be governed by the upper motive of $X$, 
although we are not aware of a precise formulation of this statement for Chow motives.
However, in Theorem~\ref{th:splitting_off_Tate}
we show that if $[X]_A$ is invertible
and some technical assumption on $\Mot{\A}{X}$ is fulfilled, 
then a Tate motive splits off from some $M\in\CM_{\A}(k)$
over $k(X)$ only if $M$ is a direct summand of the $\A$-outer motive of $X$.
Albeit not useful for Chow motives, this result is a powerful tool for $\Kn$-motives, when $X$ is a product of quadrics of dimension $2^n-1$.
Among other purposes, it is required for our study of cohomological invariants of quadrics in Section~\ref{sec:coh_inv_direct_summands_quadrics}.

In this section when working with Morava K-theory $\Knf$ we set $v_n=1$ for simplicity of computations.

\subsection{Upper and lower motives}

In this section $\Af$ denotes a coherent cohomology theory.

\begin{Def}
\label{def:upper-proj}
Let $X\in \SmProj_k$ be irreducible.

A projector $\pi \in A^{\dim X}(X\times X)$ is called {\sl upper} (resp. {\sl lower})
if its restriction to $A^{\dim X}(k(X)\times X)$ (resp. to $A^{\dim X}(X\times k(X))$) 
equals the class of the diagonal rational point $\delta_X$.

A projector is called {\sl outer} if it is both upper and lower.
\end{Def}

We will also call a motivic summand of $\Mot{A}{X}$ upper (resp., lower, outer) if it is defined by an upper (resp., lower, outer) projector.

\begin{Rk}
For a connected commutative ring $\Lambda$, and  $A=\CH\otimes \Lambda$, Karpenko defines 
the notions of upper and lower projectors in~\cite[Def.~2.10]{Karp-upper}.
For $X$ such that $\CH_0(X_{k(X)})\otimes \Lambda$ is a free module of rank $1$ the definition of loc.\ cit.\ and Definition~\ref{def:upper-proj} coincide.
To see that, note  
that the multiplicity of a Chow correspondence 
can be computed as the degree of the $0$-cycle obtained from its restriction to $k(X)\times X$.
\end{Rk}

\begin{Ex}
\label{ex:product_upper_motives}
If $X_1\times X_2$ is irreducible, 
and $U_i$ are upper (resp., lower) motives of $\Mot{A}{X_i}$ for $i=1,2$, then $U_1\otimes U_2$ is an upper (resp., lower) motive of $\Mot{A}{X_1\times X_2}$. 
\end{Ex}

\begin{Lm}\label{lm:upper_motive}
Let $X\in \SmProj_k$ be irreducible, $Y\in \Sm_k$.
Let $\alpha \in \At{Y\times X}$, let $\pi\in \At{X\times X}$ be a lower projector of $X$.
Then $\pi\circ \alpha$ and $\alpha$ have the same restriction to $\At{Y_{k(X)}}$.

Similarly, if $\pi$ is an upper projector, then
$\alpha\circ \pi$ and $\alpha$ have the same restriction to $\At{Y_{k(X)}}$.
\end{Lm}
\begin{proof}
Consider the following commutative diagram
with a transversal square on the right:
\begin{center}
\begin{tikzcd}
& Y\times X \times X \arrow[ld, "\id_Y \times p_1"'] \arrow[rd, "\id_Y\times p_2"'] &   Y \times X \times k(X) \arrow[rd,"p_{13}"] \arrow{l}& \\
Y\times X  & & Y\times X & Y\times k(X). \arrow{l}  \\
\end{tikzcd}
\end{center}
Let $\pi$ be a lower projector. Element $(\pi\circ \alpha)|_{Y_{k(X)}}$
can be computed as $(p_{13})_*(p_{12}^*(\alpha)\cdot p_{23}^*(\pi|_{X_{k(X)}}))$ 
where $p_{12}$, resp.~$p_{23}$, is a projection from $Y\times X\times k(X)$ onto $Y\times X$, resp. $X_{k(X)}$, 
by \ref{eq:ext_bc}.
By assumption, 
we have 
$$p_{23}^*(\pi|_{X_{k(X)}}) = [Y\times k(X) \xrarr{\id_Y \times \delta_X} Y\times X\times k(X)].$$
Thus, multiplying $p_{12}^*(\alpha)$ with it can be computed 
as the push-forward along $\id_Y \times \delta_X$ of $\alpha|_{Y\times k(X)}$. 
Therefore, $(\pi\circ \alpha)|_{Y_{k(X)}}$ equals $(p_{13})_*\circ (\id_Y \times \delta_X)_*$ applied to $\alpha|_{Y\times k(X)}$,
which is $\alpha|_{Y\times k(X)}$ by functoriality of push-forwards. 

The claim about the upper projector is obtained by transposition.
\end{proof}

The left (resp., right) action of the ring of correspondences $\At{X\times X}$ 
on 
$$
\At{Y\times Z \times X}\cong\Hom(\Mot{A}{Y}_X, \Mot{A}{Z}_X)
$$ 
(resp., $\At{X\times Y\times Z}\cong \Hom(\Mot{A}{Y}_X, \Mot{A}{Z}_X)$)
extends to the action on morphisms between motives:
if $M, N \in \CM_A(k)$,
then $\pi\in A^{\dim X}(X\times X)$
acts on $\Hom(M_X, N_X)$ (on the left and on the right) 
as follows from Lemma~\ref{lm:correspondence_action_on_motives} below.
However, this action is not compatible with the composition of morphisms,
i.e.\ for $\pi\in \At{X\times X}$, $\phi\colon L_X\rarr M_X$, $\psi\colon M_X\rarr N_X$
we have in general
$$ (\pi \circ \psi) \circ \phi \neq \pi \circ (\psi \circ \phi) \neq (\pi\circ \psi) \circ (\pi \circ \phi). $$

\begin{Lm}
\label{lm:correspondence_action_on_motives}
Let $X\in\SmProj_k$ and $\pi \in \At{X\times X}$. 

Let $\alpha\colon \Mot AY\rarr \Mot AZ$ be a morphism in $\CM_A(k)$, and 
 $\phi\colon \Mot AZ_X\rarr \Mot AW_X$ be a morphism in $\CM_A(X)$. 
Then $(\pi\circ \phi) \circ \alpha_X = \pi\circ (\phi\circ \alpha_X)$.

Similarly, for $\psi\colon \Mot AV_X\rarr \Mot AY_X$, 
we have $\pi \circ (\alpha_X \circ \psi) = (\pi \circ \alpha_X) \circ \psi$.
\end{Lm}
One can also reformulate these results for the right action of $\pi$.
\begin{proof}
    Direct computation.
\end{proof}

Now if $\pi$ is a lower projector, 
we can apply Lemma~\ref{lm:upper_motive} 
to motives over $X$.

\begin{Cr}\label{cr:lifting_isos_upper_projector}
Let $X \in \SmProj_k$, let $M,N \in \CM_A(k)$
and let $\pi$ be a lower $A$-projector of $X$.

Assume that there is an isomorphism $M_{k(X)}\xrarr{\varphi} N_{k(X)}$.

Then it can be lifted to an isomorphism $M_X \xrarr{\phi} N_X$
such that $\pi$ trivially acts on it.
\end{Cr}
\begin{proof}
    It follows from Proposition~\ref{prop:geometric_RNP} that any lift of $\varphi$
    to a morphism $M_X\xrarr{\phi} N_X$
    is an isomorphism.
    If we choose one such lift $\phi$,
    then by Lemma~\ref{lm:upper_motive} the restriction of $\pi\circ \phi$
    to $k(X)$ is still $\varphi$, and thus an isomorphism
    that we seek. 
\end{proof}

Note that $\Delta_X$ is always an outer projector of $X$.
Another example that we will use is the following.

\begin{Ex}\label{ex:outer_projector_quadric}
Let $Q$ be an anisotropic quadric of dimension $2^n-1$, $\Af=\Knf$.

Then there is a unique (up to isomorphism) indecomposable direct summand $\MKn{Q,\, \pi}$ of $\Mker Q$   
such that $\pi_{\,\overline{k}}$ contains $(1+l_0)\times (1+l_0)$.
This projector is neither upper nor lower. 

To see this, note that $\Kn^{2^n-1}(Q\times k(Q))$ is generated by $1, l_0$,
where $l_0$ equals the class of any rational point, in particular,
the class of the diagonal point.
The restriction of $\pi$ to this group can be computed on the level of rational elements,
and thus equals $1+l_0$. The restriction of $\pi^t$ has the same form.

The sum $1\times 1 + \pi$ is an outer projector (in fact, the $\Kn$-specialization of the indecomposable upper Chow projector), although $1\times 1$ 
is also neither upper nor lower projector.
In particular, an indecomposable upper or lower $\Kn$-projector 
does not always exist even for projective homogeneous varieties.
\end{Ex}

\subsection{Isomorphisms between Morava motives over $k(\alpha)$}
\label{sec:iso-motives-kalpha}

In this section we study the behavior of isomorphisms between motives over function field of quadrics.
In the course of proof there appears a subtle issue that one needs to compare two pullback maps
$A(Y_{k(X)}) \rarr  A(Y_{k(X\times X)})$ defined by two canonical inclusions $k(X)\hookrightarrow k(X\times X)$.
The next Proposition answers it in most cases of interest, 
note that it is also used in the next Section in the proof of Theorem~\ref{th:splitting_off_Tate}.

\begin{Prop}
\label{prop:pullbacks_equal_twisted_diagonal}
    Let $A$ be a coherent oriented theory. 
    Let $X\in \SmProj_k$ be geometrically irreducible,
    let $\sigma$ denote the transposition isomorphism of $k(X\times X)$.
    
    Assume that 
    the class of the diagonal point $\delta_{X\times X}$
    and the class 
    of the twisted diagonal point $\sigma\circ \delta_{X\times X}$ are the same
    in $\At{X\times X \times k(X\times X)}$.
    
    Then for every $Y\in \Sm_k$ the pullback maps $p_i^*\colon\At{Y_{k(X)}}\rarr \At{Y_{k(X\times X)}}$  
    along two canonical field inclusions $k(X)\hookrightarrow  k(X\times X)$, $i=1,2$,
    are equal.
\end{Prop}

\begin{proof}
    The pullback maps $p^*_1, p^*_2\colon\At{Y\times X}\rarr \At{Y\times X\times X}$
    are related by the action of the class of the graph of the transposition $[\Gamma_\sigma \rarr X^{\times 4}]\in \At{X^{\times 4}}$. 
    The pullback maps in which we are interested are induced by these maps,
    and thus if the class $[\Gamma_\sigma]$ restricts to the class of the diagonal point in   $\At{X\times X \times k(X\times X)}$,
    then we can apply Lemma~\ref{lm:upper_motive} and get the claim.
    However, this restriction is precisely the class of the twisted diagonal point as in the assumption of the proposition.
\end{proof}

\begin{Rk}
\label{rk:p1_equal_p2_num}
    Since the above proof uses only general properties of oriented theories, 
    it can also be applied in the case of the theory with relations $\A_\Gamma$. However, if $\A_\Gamma$ is not coherent, 
    one has to be careful that all the pullback maps that are used are defined (these include not only $p_1^*, p_2^*$
    appearing in the statement, but also e.g.\ $\A(Y\times X) \rarr \A(Y\times k(X))$ needed for Lemma~\ref{lm:upper_motive}).

    For example, if $\A$ is a free theory, its coefficient ring $\Apt$ is a domain, and $[X]_A$ is invertible in it, 
    then for the numerical version of $\A$ all these pullback maps are defined by Corollary~\ref{cr:num_generic_restriction},
    and therefore $p_1^*=p_2^*$ as the morphisms $\Afnum\left(Y_{k(X)}\right) \rarr \Afnum\left(Y_{ k(X\times X)}\right)$.    
\end{Rk}

\begin{Cr}\label{cr:pullbacks_k(Q_times_Q)}
Let $A$ be a coherent oriented theory, $Q$ be a quadric over $k$ of dimension greater than 0, $Y\in \Sm_k$.
 
Then pullback maps $p_i^*\colon A(Y_{k(Q)}) \rarr A(Y_{ k(Q\times Q)}) $, $i=1,2$,
along two canonical projections are equal.
\end{Cr}
\begin{proof}
    To apply Proposition above it suffices to note 
    that the classes of any two rational points on $Q\times Q\times k(Q\times Q)$
    are the same in $A$.
\end{proof}
\begin{Rk}
If one is interested only in motives of projective homogeneous varieties, 
then it suffices to work with $\overline{A}$ instead of $A$,
and the argument above can be simplified.

Indeed, if $Y$ is a geometrically cellular variety,
then both $A(Y\times \overline{k(Q)})$ and $A(Y\times \overline{k(Q\times Q)})$
can be identified canonically with $A(Y\times_k \overline{k})$.
Then the pullback maps $p_1^*, p_2^*$ from 
$A(Y\times \overline{k(Q)})$ to $A(Y\times \overline{k(Q\times Q)})$
are identities after this identification,
and therefore also the maps $p_1^*$, $p_2^*$ 
on $\overline{A}(Y\times k(Q))\subset A(Y\times \overline{k(Q)})$ are equal.
\end{Rk}

Let $Q$ be a quadric over $k$ and let $N,M \in \CM_{\Kn}(k)$ be motives.
Let $p_1, p_2$ denote the canonical projections from $Q_{k(Q)}$ to $Q$, resp. $\Spec k(Q)$,
and let $\pra 1, \pra 2$ denote the canonical projections from $\Spec k(Q\times Q)$ to $\Spec k(Q)$.
Consider the following 
diagram of categories.
\begin{center}
\begin{tikzcd}
    \CM_{\Kn}(Q_{k(Q)}) \arrow[r] & \CM_{\Kn}(k(Q\times Q)) \\
	\CM_{\Kn}(Q) \arrow[u, "p_1^*"] \arrow[r] & \CM_{\Kn}(k(Q)) \arrow[u, "\pra 1^*"]
\end{tikzcd}
\end{center}

Let $\alpha\colon M_{k(Q)}\rarr N_{k(Q)}$ be a morphism in $\CM_{\Kn}(k(Q))$,
let $\phi\colon M_Q\rarr N_Q$ be any lift of this morphism to $\CM_{\Kn}(Q)$. 
Then we can consider two morphisms, $p_1^*(\phi)$ and $p_2^*(\alpha)$, between $M_{Q_{k(Q)}}$ and $N_{Q_{k(Q)}}$ in $\CM_{\Kn}(Q_{k(Q)})$.
If we restrict these morphisms to the generic point of $Q_{k(Q)}$,
then we get $\pra 1^*(\alpha)$ and $\pra 2^*(\alpha)$, respectively.
However, by Corollary~\ref{cr:pullbacks_k(Q_times_Q)} these morphisms are equal.
Let $\alpha \times 1$
denote $p_2^*(\alpha)$ and $\overline{\phi}$ denote $p_1^*(\phi)$,
and thus $\overline{\phi} - \alpha\times 1$ vanishes in the generic point of $Q_{k(Q)}$.

Let $Q$ be now a quadric of dimension $2^n-1$ that has upper binary motive,
and thus there exist an outer projector $p\in \Knt{Q\times Q}$
such that $p_{k(Q)}$ is $1\times l_0 + l_0 \times 1 + l_0\times l_0$.
Assume that $\alpha\colon M_{k(Q)}\rightleftarrows N_{k(Q)}\colon\!\beta$ are inverse isomorphisms in $\CM_{\Kn}(k(Q))$.
By Corollary~\ref{cr:lifting_isos_upper_projector} we can lift them to isomorphisms $\phi\colon M_Q \rightleftarrows N_Q\colon\!\psi$ in $\CM_{\Kn}(Q)$
such that $p$ acts trivially on $\phi$. 

The action of the projector $p_{k(Q)}$ on $\Hom(M_{Q_{k(Q)}},\,N_{Q_{k(Q)}})$
can be computed: it projects onto the subgroup $\Hom(M_{k(Q)},\,N_{k(Q)})\times\langle 1,l_0\rangle$.
In view of the discussion above, we can write $\overline{\phi}$ as $\alpha\times 1 + u\times l_0$
where $u\in \Hom(M_{k(Q)},\,N_{k(Q)})$. 
It turns out that $\overline{\psi}$ also has such form.

\begin{Lm}\label{lm:lift_of_isos_from_k(Q)_to_Q}
In the notation above 
$\overline{\psi}$ has the form $\beta\times 1+v\times l_0$
where $\beta, v\in \Hom(N_{k(Q)},\,M_{k(Q)})$
and $\beta$ is the inverse to $\alpha$.
\end{Lm}
\begin{proof}
By construction we have that 
    $$\overline{\psi} = \beta\times 1 +v\times l_0 + x,$$
with  $\beta, v\in  \Hom(N_{k(Q)},\, M_{k(Q)})$, $\beta$ inverse to $\alpha$ by construction
and $p_{k(Q)}$ acts by $0$ on $x$.

Note that $q_{k(Q)}\cong \hyp\perp q'$, 
and for any smooth projective variety $Y$ over $k(Q)$
we have the splitting
$\Knt{Y\times Q}\cong \Knt Y\langle 1,l_0\rangle \oplus \Knt{Y\times Q'}$
where the first group on the right is obtained by acting with the outer projector $p_{k(Q)}$.
Thus, abusing notation we say that $x$ lies in $\Knt{N\times M\times Q'}$.

We claim that the composition $x \circ \overline{\phi}$
lies in $\Knt{M\times M \times Q'}$. 
Let $M$ be a direct summand of $\MKn Y$, let $N$ be a direct summand of $\MKn Z$
for some $Y,Z\in \SmProj_k$.
Then the composition $x\circ \overline{\phi}$ is computed
by taking the push-forward to $\Knt{Y\times Y\times Q_{k(Q)}}$
of the element $\pra{12}^*(\overline{\phi}) \cdot \pra{23}^*(x)$ in $\Knt{Y\times Z\times Y\times Q_{k(Q)}}$.
However, the product of $\pra{12}^*(u)\times l_0$ with $\pra{23}^*(x)$ is zero,
since  $\pra{23}^*(x)$ restricted to the generic point of $Q_{k(Q)}$ vanishes.
And the product of $\pra{12}^*(\alpha)\times 1$ with $\pra{23}^*(x)$ lies in $\Knt{Y\times Z\times Y\times Q'}$ 
and hence its push-forward lies in $\Knt{Y\times Y\times Q'}$.

On the other hand, since $\overline{\psi}\circ \overline{\phi}$ is the identity,
it lies in $\Hom(M_{k(Q)},\, M_{k(Q)})\cdot 1_{Q_{k(Q)}}$
and has no component in $\Knt{M\times M\times Q'}$,
and hence $x\circ \overline{\phi}$ is zero.
However, this means that
 $(\beta\times 1 +v\times l_0)\circ \overline{\phi}$ is the identity,
 and due to the uniqueness of the inverse 
 we get that $(\beta\times 1 +v\times l_0)$ is the inverse to $\overline{\phi}$
 and $x=0$.
\end{proof}

\begin{Th}\label{th:iso_over_k(alpha)}
Let $Q$ be an anisotropic quadric of dimension $2^n-1$ with an upper binary Chow motive $R$.
Let $L_Q$ be a non-trivial invertible summand of $R_{\,\Kn}$. 
Let $M, N$ be two indecomposable $\Kn$-motives over $k$
satisfying the following property: 
\begin{equation}
\label{eq:cond-on-motives}
\text{if an endomorphism of the motive is not an isomorphism, then it is nilpotent.}
\end{equation}

If $M_{k(Q)}\cong N_{k(Q)}$,
then either $M\cong N$ or $M\otimes L_Q \cong N$.
\end{Th}

\begin{Rk}\phantom{test}
\begin{enumerate}
    \item 
        If $M$, $N$ are indecomposable direct summands of the motives of projective homogeneous varieties,
    then they satisfy the assumption~\eqref{eq:cond-on-motives} by the Rost nilpotence principle, cf.~\cite[Lm.~2.1]{Karp-upper}.
    \item 
    Conjecturally, under the assumptions of Theorem~\ref{th:iso_over_k(alpha)}, $q$ is a Pfister neighbour and $R$ is the Rost motive~\cite[Conj.~4.21]{Vish-quad}. Moreover, for $M$, $N$ summands of projective homogeneous varieties, Theorem~\ref{th:binary_chow_motives} and Proposition~\ref{prop:reflects_MD} allows to reduce the result to this case. However, our argument works under weaker assumptions.
\end{enumerate}
\end{Rk}

\begin{proof}
By Lemma~\ref{lm:lift_of_isos_from_k(Q)_to_Q} and discussion before it,
we can lift the mutually inverse isomorphisms {$\alpha\colon M_{k(Q)} \rightleftarrows N_{k(Q)}\colon\!\beta$}
to mutually inverse isomorphisms $\phi\colon M_Q \rightleftarrows N_Q\colon\!\psi$ 
in the category $\CM_{\Kn}(Q_{k(Q)})$ 
such that $\overline{\phi}=\alpha\times 1+u\times l_0$, $\overline{\psi} = \beta\times 1 + v\times l_0$. 

By computing the compositions of $\overline{\phi}$ and $\overline{\psi}$
we obtain the following equations:
\begin{equation}\label{eq:lift_iso}
 \alpha\circ v+u\circ \beta = 0, \qquad  \beta\circ u + v\circ \alpha = 0.
\end{equation}
In what follows, we take push-forwards of some morphisms in $\CM_{\Kn}(Q)$
to $\CM(k)$, but compute their expressions after base change to $k(Q)$.
We will say then that the corresponding element in $\CM_{\Kn}(k(Q))$
is $k$-rational, since it is a pullback of an element from $\CM_{\Kn}(k)$.

For example, by taking the push-forwards of $\overline{\phi}, \overline{\psi}$
we get that $\alpha +v$ and $\beta + u$ are $k$-rational.
By subtracting $(\alpha+v)\times 1$ from $\overline{\phi}$
we thus also get that $u\times (1+l_0)$, and similarly $v\times (1+l_0)$ 
are $k$-rational. The compositions of these elements are
$(u\circ v) \times 1$ and $(v\circ u) \times 1$,
and therefore also $u\circ v, v\circ u$ are also $k$-rational as their push-forwards.

First, assume that one of the morphisms $u\circ v, v\circ u$ 
 is an isomorphism of motives over $k$.  Without loss of generality, let it be $u\circ v$.
For any 
$X$, $Y\in\SmProj_k$ 
the group $\Hom(\MKn X_{Q},\, \MKn Y_{Q})$ 
is $\Kn(X\times Y\times Q)$, 
and can be identified with $\Hom(\MKn X,\, \MKn Y\otimes \MKn Q)$.
Thus, we can view the $k$-rational element $u\times (1+l_0)$
as the morphism of motives $M_{k(Q)} \rarr \left(N\otimes \MKn Q\right)_{k(Q)}$ in $\CM_{\Kn}(k(Q))$. 
Moreover, it does not change under the composition with $(\id_N \otimes \pi)_{k(Q)}$
where $\pi$ is a projector in $\Knt{Q\times Q}$ onto $L_Q$ (recall that $\pi$ becomes $(1+l_0)\times (1+l_0)$ over $k(Q)$).
Therefore, it can be seen as a base change to $k(Q)$ of a morphism $M  \rarr  N\otimes L_Q$, where $L_Q = \MKn{Q,\,\pi}$.

Similarly, $v\times (1+l_0)$ can be seen as a base change to $k(Q)$ of a morphism $N\otimes L_Q \rarr M$. 
Direct computation shows that the composition of these two morphisms is $u\circ v$ over $k(Q)$
and by Rost Nilpotence Property (Corollary~\ref{cr:rnp_k(n)})
we get that $M$ is a direct summand of $N\otimes L_Q$.
However, since $N$ is indecomposable, then so is $N\otimes L_Q$,
and therefore $M$ is isomorphic to $N\otimes L_Q$.

Second, if none of $u\circ v$ or $v\circ u$ is an isomorphism,
then by the assumption they are nilpotents.
Note that we have the following identity between  $k$-rational elements:
$$(\alpha+u) \circ (\beta+v) = \id_N + u\circ v$$
where we have used the relations~(\ref{eq:lift_iso}).
Similarly, $(\beta+v) \circ (\alpha+u) $  
differs from identity by a nilpotent. 
Then one can modify these morphisms so that they become isomorphisms, i.e.\ $M\cong N$
(see e.g\ the proof of \cite[Lm.~2.1]{VishYag}).
\end{proof}

Theorem~\ref{th:iso_over_k(alpha)} also allows to identify 
some non-invertible summands of the $\Kn$-motives of quadrics.

\begin{Cr} Let $Q$ be a quadric of dimension less than $2^{n+1}-1$.
Let $N$ be a binary summand of $\MKn Q$
that is the specialization of a non-binary indecomposable summand of $\MCH Q$, and let $\alpha \in  \HH^{n+1}(k,\,\ZZ/2)$.

Assume that over the field $k(\alpha)$ 
the motive $N$ is isomorphic to the Rost motive $(R_{\beta_{k(\alpha)}})_{\Kn}\sh i$ for some $i$ 
and a pure symbol $\beta \in \HH^{m}(k,\,\ZZ/2)$, where  
$2\le m<n+1$.

Then $N\cong (R_\beta)_{\Kn} \otimes L_\alpha(i)$.
\end{Cr}
\begin{proof}
For simplicity, we assume that $i=0$.
By Theorem~\ref{th:iso_over_k(alpha)} we have either $N\cong (R_\beta)_{\Kn}$ or $N\cong (R_\beta)_{\Kn}\otimes L_\alpha$.
However, in the first case Theorem~\ref{th:prestableMDT} allows to extend this isomorphism
to an isomorphism of Chow motives 
leading to a contradiction with the assumption that the Chow motive corresponding to $N$ 
is not binary.
\end{proof}

\begin{Ex}[cf.~Section~\ref{sec:small_kahn_dim}]
Let $Q$ be a 5-dimensional anisotropic quadric, which is not excellent, and becomes completely split over some quadratic extension. 
Then the Chow motive of $Q$ has the following MDT~\cite[Prop.~6.10]{Vish-quad}. %
\begin{center}
\begin{tikzpicture}
\filldraw [white] (-0.3,0.5) circle (1pt);
\filldraw [black] (-0.1,0) circle (1pt);
\filldraw [black] (0.5,0) circle (1pt);
\filldraw [black] (1.1,0) circle (1pt);
\filldraw [black] (1.7,0) circle (1pt);
\filldraw [black] (2.3,0) circle (1pt);
\filldraw [black] (2.9,0) circle (1pt);
\filldraw [white] (3.1,0) circle (1pt);
\draw (1.1,0) .. controls (1.4,0.1) .. (1.7,0);
\draw (-0.1,0) .. controls (0.2,0.2) and (0.8,0.2) .. (1.1,0);
\draw (1.7,0) .. controls (2.0,0.2) and (2.6,0.2) .. (2.9,0);
\draw (0.5,0) .. controls (1.1,-0.3) and (1.7,-0.3) .. (2.3,0);
\filldraw [black] (0,-0.8) circle (0pt) node[anchor=west]{\text{Chow MDT of $Q$}};
\end{tikzpicture} 
\qquad\qquad\begin{tikzpicture}
\filldraw [white] (-0.3,0) circle (1pt);
\filldraw [gray] (-0.1,0) circle (1pt);
\filldraw [black] (0.5,0) circle (1pt);
\filldraw [black] (1.1,0) circle (1pt);
\filldraw [black] (1.7,0) circle (1pt);
\filldraw [gray] (2.3,0) circle (1pt);
\filldraw [gray] (2.9,0) circle (1pt);
\filldraw [white] (3.1,0) circle (1pt);
\draw (1.1,0) .. controls (1.4,0.1) .. (1.7,0);
\filldraw [black] (0,-0.8) circle (0pt) node[anchor=west]{\text{$\ \K2$-MDT of $Q$}};
\end{tikzpicture} 
\end{center}
The binary summand of the Chow motive 
is the Rost motive $R_\alpha$ for $\alpha \in \HH^3(k,\,\ZZ/2)$. 
The $\K2$-kernel motive of $Q$ is then $L_\alpha\sh 1 \oplus N$ where $N$ is a rank $2$ motive.

Over the field of functions of a Pfister quadric corresponding to $\alpha$ the anisotropic part of $q$
becomes a conic. Therefore $N_{k(\alpha)}\cong (R_\beta)_{\K2}\sh2$ 
where $R_\beta$ is the Rost motive of $\beta \in \HH^2(k(\alpha),\,\ZZ/2)$,
i.e.\ $\beta$ is the Clifford invariant $[C_0(q_{k(\alpha)})]$ and if it is a symbol over $k(\alpha)$,
then by Merkurjev's index reduction formula $[C_0(q)]$ is a symbol over $k$. 

Therefore $N\cong (R_{[C_0(q)]})_{\K2} \ot L_\alpha$.
\end{Ex}

\begin{Qu}
Let $N$ be a rank $2$ indecomposable summand of the $\Kn$-motive of some quadric $Q$.
Does $N$ necessarily have the form $(R_\beta)_{\Kn}\otimes L_\alpha$ for some $\alpha, \beta$?
\end{Qu}

\subsection{Splitting off Tate motives over function fields of quadrics}

\begin{Th}\label{prop:gen_splitting_Tate}\label{th:splitting_off_Tate}
Let $\Af$ be a free theory. 
Let $N \in \PM_{A}(k)$ be an indecomposable motive. 

Let $X\in \SmProj_k$ that satisfies the following conditions:
\begin{enumerate}
\item $[X]_A$ is invertible in $\Apt$;
\item if an endomorphism of $\Mot AX$ restricts to $[X]_A^{-1}\cdot 1$ 
in the generic point of $X\times X$, then its power is a non-trivial projector of $\Mot AX$.
\end{enumerate}

Let $U=\Mot A{X,p}$ be a summand of the motive of $X$, where $p$ is an outer projector.

Assume that $N_{k(X)}$ splits off a Tate summand $\un_{k(X)}$.

Then $N$ is a direct summand of $U$.
\end{Th}
\begin{Rk}
The main example of $X$ that
satisfies
the conditions of the Theorem 
for $\Af=\Knf$ 
is a product of quadrics of dimension $2^n-1$ (see~Lemma~\ref{lm:quadrics_endomorphisms}).
\end{Rk}
\begin{proof}
By Proposition~\ref{prop:geometric_RNP} we can lift the splitting of $N_{k(X)}$ to $N_X$,
i.e.\  we have morphisms $\un_X\xrarr{\alpha} N_X \xrarr{\beta} \un_X$
such that their composition is the identity of $\un_X$.
We can view the elements $\alpha, \beta$ 
as morphisms\footnote{ Note that this is just the adjointness of the pullback functor $\CM_{\A}(k)\rarr \CM_{\A}(X)$ 
to the pullback functor on the right
 and to the twisted pullback on the left.}
 $a\colon\Mot AX\rarr N$ and $b\colon N\rarr \Mot AX\sh{-\dim X}$.

The composition $b\circ a$ is an element of $\A^0(X\times X)$,
and we are going to show now that its restriction to $\A^0(k(X\times X))\cong \A^0(k)$ is $1+x$ for a nilpotent $x\in \A^0(k)$.
Let $N$ be a summand of $\Mot AY$ for some $Y\in\SmProj_k$. 
Then $(b\circ a)_{k(X\times X)}$ equals the push-forward of $\pra{1}^*(\alpha|_{Y_{k(X)}})\cdot \pra{2}^*(\beta|_{Y_{k(X)}})$ 
from $A(Y\times k(X\times X))$ to $A(k(X\times X))$. 
If we assume that $\pra{1}^*=\pra{2}^*$, 
then instead we can compute the projection to the point on $Y_{k(X)}$ of the product of $\alpha|_{Y_{k(X)}}$ and $\beta|_{Y_{k(X)}}$,
but this equals precisely $1$ as it is the same as the composition of the restrictions of $\alpha$ and $\beta$ to $\CM_A(k(X))$.
If $\Apt$ is an integral domain, then we may perform the same computation in the numerical version $A_{\num}$: 
the equality $\pra{1}^*=\pra{2}^*$  then holds by Remark~\ref{rk:p1_equal_p2_num}.
Otherwise, we do this computation for $(A/\mathfrak{p})_{\num}$ where $\mathfrak{p}$ is a prime ideal of $\Apt$. 
Thus, we get that $(b\circ a)_{k(X\times X)}\in A^0(k)$ is mapped to $1$ in $\Apt/\mathfrak{p}$ for all prime ideals, 
hence it equals $1+x$ where $x$ is nilpotent.

Note that the action $\alpha \circ p$ of the projector $p$ on $\alpha$ 
can be identified with the composition $a\circ p$, 
and, similarly, for $p\circ \beta$ and $p\circ b$.
By Lemma~\ref{lm:upper_motive} the composition of the restrictions of $\alpha\circ p, p\circ \beta$ to $k(X)$
is still the identity of $\un_{k(X)}$.

Thus, by replacing $a$ with $a\circ p$ and $b$ with $(1+x)^{-1}[X]_A^{-1}\cdot  p\circ b$ 
may assume that the composition $b\circ a$ is an endomorphism of $U$ 
such that its $n$-th power is a projector.
By replacing $a$ with $a\circ (b\circ a)^{n-1}$ we may also assume 
that $b\circ a$ is a projector. 

Then $a\circ b \circ a\circ b$ is a projector of $N$
and $a$ and $b$ yield inverse isomorphisms between $(N, a\circ b \circ a\circ b)$ and $\Mot{A}{X,\,  b\circ a}$.
By assumption $N$ is indecomposable, hence $(N,\, a\circ b \circ a\circ b)\cong N$ and 
it is a direct summand of $U$.
\end{proof}

In order to apply this Theorem for quadrics we need the following.

\begin{Lm}
\label{lm:quadrics_endomorphisms}
Let $Q_i$, $i\in I$, where $I$ a finite set, be a collection of quadrics of dimension $2^n-1$.

If $\alpha$ is an endomorphism of $\MKn{\prod Q_i}$ that restricts to $1$ in the generic point of $\prod Q_i$,
then a power of $\alpha$ is a non-trivial projector.
\end{Lm}
\begin{proof}
Since products of quadrics satisfy RNP for all field extensions~\cite{GilleVishik},
we can work with rational classes over $\overline{k}$.
As the motive of a split quadric is a sum of Tate motives, 
the endomorphism ring of the product of split quadrics is finite. The endomorphism ring of rational elements
is thus a subring of a finite ring, and hence finite as well.

We will show now that if an element $\alpha$ in $\End\left(\MKn{\prod Q_i}\right)$ 
restricts to $1$ in the generic point, then so do its 2-primary powers,
in particular, it cannot be nilpotent.
However, if an element of a finite ring is not a nilpotent, 
then its power is a non-trivial idempotent. 

Indeed, denote $X=\prod_i Q_i$. 
The restriction of the composition $\alpha^{\circ 2}$ to the generic point of $X \times X$
can be computed in the numerical version of $\Kn$. 
Moreover, using~\ref{eq:ext_bc} for $\Knf$ and the square 
\begin{center}
\begin{tikzcd}
 X\times X \times X \arrow[d, "\pra{13}"']  &   X \times k(X\times X) \arrow[d] \arrow{l} \\
X\times X  &  \Spec k(X\times X), \arrow{l} 
\end{tikzcd}
\end{center}
and the fact that $\pra1=\pra2$ for $\Kn_{\num}$ by Remark~\ref{rk:p1_equal_p2_num},
we may instead compute the class $\left[\left(\alpha|_{X_{k(X)}}\right)^2\right]_{\Knum}$ where the square is taken in the ring $\Kn_{\num}\left(X_{k(X)}\right)$.

Recall that $(q_i)_{k(X)} \cong \hyp\perp q_i'$ where $q_i'$ is anisotropic~\cite[Cor.~1]{Hoffmann-sep},
therefore $\Mot{\Kn_{\num}}{Q_i\times k(X)} \cong \un \oplus \un$. Then 
$\Knum\left(X_{k(X)}\right)$ is freely generated by the elements of the form $\bigtimes_i a_i$
where $a_i \in \Knum(Q_i\times k(X))$ is either $1$ or $l_0$. 
Since by assumption $\alpha$ restricts to $1$ in $k(X\times X)$, 
and the square of an element $\bigtimes_i a_i$ is zero as soon as at least one of $a_i$ is $l_0$,
we conclude that  $\left(\alpha|_{X_{k(X)}}\right)^2=1$ in $\Kn_{\num}\left(X_{k(X)}\right)$.
\end{proof}

\begin{Cr}
\label{cr:splitting_off_tate_over_product_quadrics}
Let $M \in \PM_{\Kn}(k)$ be an indecomposable motive, 
let $Q_i$, $i\in I$, be a collection of quadrics of dimension $2^n-1$,
 and let $K$ be the composite of the fields of functions of $Q_i$.

If $M_K$ splits off a Tate summand,
then there exist a finite subset $J$ of $I$
such that $M$ is a direct summand of an outer motive of $M_{\Kn}(\prod_{j\in J} Q_j)$.
\end{Cr}
\begin{proof}
    Direct application of Theorem~\ref{th:splitting_off_Tate} using Lemma~\ref{lm:quadrics_endomorphisms}.
\end{proof}

\begin{Prop}
\label{prop:splitting_over_quadrics_with_dim_2^n-1}
\label{cr:splitting_over_quadrics_with_dim_2^n-1}
Let $N$ be an indecomposable $\Kn$-motive over $k$ that is not invertible.
Let $K$ be a field extension of $k$
 obtained as the composite  of fields of functions of quadrics of dimension $2^n-1$ 
 that are defined over $k$.
 
Then motive $N_K$ is not split.
\end{Prop}
\begin{proof}
We can assume that there exists a finite number of anisotropic quadrics $Q_i$ of dimension $2^n-1$
such that $N_K$ is split for $K=k(\prod_i Q_i)$.
By Corollary~\ref{cr:splitting_off_tate_over_product_quadrics} we get that $N$ is a direct summand
of an outer motive of $\MKn{\prod_i Q_i}$.

Recall from Example~\ref{ex:outer_projector_quadric} that an outer $\Kn$-motive of $Q_i$
can be taken to be the specialization of the outer Chow motive of $Q_i$, but it has the form $\un\oplus \widetilde{U}(Q_i)$ 
where summand $\un$ is given by the projector $1\times 1$.
An outer motive of $\prod_i Q_i$ can be thus chosen as a direct sum of $\bigotimes_{i\in J} \widetilde{U}(Q_i)$ for some $J\subset I$,
see Example~\ref{ex:product_upper_motives}.

Over the field $K$ we have $(q_i)_K \cong \hyp\perp q_i'$ where $q_i'$ is anisotropic~\cite[Cor.~1]{Hoffmann-sep}.
Therefore $\widetilde{U}(Q_i)_K\cong \un \oplus P_i'$ where $P_i'$ is a direct summand of $\MKn{Q_i'}$.

On the other hand, motive $\bigotimes_{i\in I} \MKn{Q_i'}$ does not contain Tate summands, and hence $\bigotimes P_i'$ as well.
To see that, note that $Q_i'$ is $\Kn$-anisotropic for all $i$, and hence also $\prod_i Q_i'$,
as $\Knt{\prod Q_i}\rarr \Knpt$ factors through $\Knt{Q_i}$.
And if we have a splitting $\un \xrarr{a} \MKn X\xrarr{b} \un$ for some $X$, then $\pi_X(a\cdot b) =1$,
i.e.\ $X$ is $\Kn$-isotropic.

Therefore, motive $\left( \bigotimes_{i\in J} \widetilde{U}(Q_i)\right)_K$ has only one Tate summand
(the tensor product of the only Tate $\un$ splitting off of $\widetilde{U}(Q_i)_K$ for each $i$).
Hence $N_K$ is Tate of rank $1$, 
and hence $N$ was invertible (recall that $N$ satisfies RNP by Corollary~\ref{cr:rnp_k(n)}).
\end{proof}

We complement Proposition~\ref{prop:splitting_over_quadrics_with_dim_2^n-1} by the following splitting criterion for invertible motives.

\begin{Prop}
\label{prop:split_invertible_2n-1}
\label{cr:split_invertible_2n-1}
Let $Q$ be a quadric of dimension at least $2^n-1$. 
Let $L\in \PM_{\Kn}(k)$ be an non-Tate invertible motive.

Then 
$L_{k(Q)}\cong\un_{k(Q)}$ 
if and only if 
$Q$ has an indecomposable upper binary Chow motive $R$ of length $2^n-1$ 
such that $R_{\,\Kn}\cong\un\oplus L$.
\end{Prop}

\begin{proof}
The ``if'' part is trivial.

For the ``only if'' part recall that 
any invertible motive $L$ satisfies RNP for 
any field extensions, since $\End(L)\cong\End(\un)$. 
In particular, if the dimension of $Q$ is greater than $2^{n+1}-1$, then by Proposition~\ref{prop:reflects_MD} non-Tate motive $L$ cannot split over $k(Q)$.

We, thus, assume that $2^n-1\le \dim Q\le 2^{n+1}-2$.
By Theorem~\ref{prop:gen_splitting_Tate} we obtain that the 
Morava motive of $Q$ contains 
a summand that is isomorphic to a Tate twist of $L$. 
By Theorem~\ref{th:prestableMDT} this summand corresponds to a direct summand of the Chow motive of $Q$ 
that becomes completely split over $k(Q)$. It follows that this has to be a binary summand $R$ of length $2^n-1$
that is also an upper summand. 
\end{proof}

We finish this section with a 
 result that complements Theorem~\ref{prop:gen_splitting_Tate}.
This is well known to experts, but the statement does not seem to appear in this generality in the literature.

\begin{Prop}\label{prop:gen_splitting_Tate_res}
Let $A$ be a coherent theory,  $X\in \SmProj_k$, $N\in \PM_A(k)$ such that a Tate summand $\un_{k(X)}\sh i$ splits off of $N_{k(X)}$.
Then $\Mot AX\sh i$ splits off of $N\otimes \Mot AX$.

If $N_{k(X)}$ is completely split, i.e.\ isomorphic to $\oplus_{i\in I} \,\un_{k(X)}{\sh i}^{\oplus k_i}$,
 then $N\otimes \Mot AX$ is isomorphic to the direct sum $\oplus_{i\in I}\, \Mot AX{\sh i}^{\oplus k_i}$.
\end{Prop}
\begin{proof}
By Proposition~\ref{prop:geometric_RNP} we lift
the splitting to the category $\PM_A(X)$, i.e.\ $N_X\cong \un_X\sh i \oplus R$.

Consider the restriction functor $\mathrm{res}\colon\PM_A(X)\rarr \PM_A(k)$
that sends $\mathrm M_A(Y\rarr X)$ to $\mathrm M_A(Y)$. 
In particular, it sends $N_X$ to $N\otimes \Mot AX$ and $\un_X$ to $\Mot AX$,
so we get the claim by applying $\mathrm{res}$ to the splitting above.

Similarly, we obtain the claim for the total splitting of $N$.
\end{proof}

\begin{Cr}
Let $Q$ be a quadric of dimension $2^n-1$ with the Chow upper binary motive.

Let $L_Q$ denote the non-trivial invertible summand of $\MKn Q$ 
and $R$ the complement to $\un\oplus L_Q$.

Then $L_Q\otimes R\cong  R$.
\end{Cr}
\begin{proof}
Recall that $R$ does not have invertible summands (cf.~proof of Proposition~\ref{prop:splitting_over_quadrics_with_dim_2^n-1}). 
By applying Proposition~\ref{prop:gen_splitting_Tate_res} to 
$N=L_Q$ and $X=Q$, 
we get that $L_Q \otimes (\un \oplus L_Q \oplus R) \cong \un \oplus L_Q \oplus R$. 
Using the Krull--Schmidt property (\cite[Sec.~2.6]{LPSS}) we obtain that $L_Q\otimes R$ is isomorphic to $R$
(note also that $L_Q^{\otimes 2}\cong \un$).
\end{proof}

\begin{Rk}
Conjecturally~\cite[Conjecture~4.21]{Vish-quad}, the quadric as in the Corollary 
should be a neighbour of an $(n+1)$-fold Pfister quadric,
i.e.\ $q\perp q'$ is a general $(n+1)$-Pfister form.

Then $R\cong \MKn{Q'}\sh1$~\cite[Prop.~4]{Rost-new},~\cite[Th.~7.1]{KarpMer-excel} and the reason for the above isomorphism 
is due to the fact that over $k(Q')$ the motive $L_Q$ splits and Prop.~\ref{prop:gen_splitting_Tate_res} can be applied.
However, we have established this isomorphism without assuming the existence of $Q'$.
\end{Rk}

\begin{Qu}
If $\MKn{Q'}\otimes L_\alpha \cong \MKn{Q'}$  
for some quadric $Q'$ of dimension less than $2^n-1$, does it follow that $q'$ is a neighbour of the Pfister form 
corresponding to $\alpha$? 
\end{Qu}

\subsection{Motivic decompositions over $k(\alpha)$}\label{sec:MDT_over_k(alpha)}

In Definition~\ref{def:k(alpha)} we have defined a class of field extensions $k(\alpha)$ splitting an element $\alpha \in \HH^{n+1}(k,\,\ZZ/2)$.
The basic example of it is when $\alpha$ is a symbol
and the field of functions of the corresponding Pfister quadric is $k(\alpha)$.
In fact, for arbitrary $\alpha$
one can take $k(\alpha)$ to be a composition $k\subset K\subset k(\alpha)$
where $K/k$ is $\oKn$-universally bijective, and $k(\alpha)/K$ is the field of functions of a Pfister quadric.
Thus, in the situations when one can use RNP for $K/k$, by Proposition~\ref{prop:reflects_MD} 
one reduces to the function field of a Pfister quadric.

In this section we study the decompositions of $\Kn$-motives over fields $k(\alpha)$.
However, we formulate our main result (Proposition~\ref{prop:decomposition_over_k(alpha)})
slightly differently: we consider fields of functions $k(Q)$, where $Q$ is a quadric of dimension $2^n-1$
with the upper (also, outer) motive being binary. If $Q$ is a Pfister neighbour,
then this binary motive is the Rost motive $R_\alpha$ for a symbol $\alpha\in \HH^{n+1}(k,\ZZ/2)$
and $k(Q)$ is stably birational to the $k(Q_\alpha)$, where $Q_\alpha$ is the corresponding Pfister quadric.
In particular, in this case $k(Q)$ lies in the class $k(\alpha)$.
Moreover, every such $Q$ is conjectured to be a Pfister neighbour, but we do not rely on this in the proof.

\begin{Prop}
\label{prop:decomposition_over_k(alpha)}
Let $Q$ be an anisotropic quadric of dimension $2^n-1$ such that its upper Chow motive is binary.
 
Let $M\in \PM_{\Kn}(k)$ be an indecomposable motive satisfying the following condition:
\begin{equation}
\label{eq:property-motive}
\tag{$\dagger$}
    \text{if a $k$-rational endomorphism of $M_{k(Q)}$ is not an isomorphism, then it is nilpotent.}
\end{equation}

Then one of the following holds:
\begin{enumerate}
    \item $M_{k(Q)}$ is indecomposable,
    \item $M_{k(Q)}\cong N^{\oplus 2}$ for some motive $N$.  
    Moreover, if a direct summand of $M_{k(Q)}$ cannot be a proper direct summand of itself,
    then $N$ is indecomposable.    
\end{enumerate}
\end{Prop}
\begin{Rk}
An indecomposable direct summand of the motive of a projective homogeneous variety
satisfies~\eqref{eq:property-motive}.
Moreover, it cannot be a proper direct summand of itself.
\end{Rk}

\begin{proof}

Let $M=\MKn{Y,\,\pi}$ where $Y\in\SmProj_k$ and $\pi \in \Knt{Y\times Y}$ is a projector.
We assume $M_{k(Q)}$ is not indecomposable, and need to show (2). 
Thus, we have $\pi_{k(Q)} = \pi_1 + \pi_2$ in $\Kn(Y\times Y\times k(Q))$
for $\pi_1$, $\pi_2$ non-trivial orthogonal projectors, and denote by $M_1$, $M_2$ corresponding summands of $M_{k(Q)}$.
By Lemma~\ref{lm:upper_motive} we can lift $\pi_1$ to an endomorphism $\phi_1$ of $M_Q$ 
on which an outer projector of $Q$ acts trivially. 
Here and below we use the outer projector of $Q$ that becomes of the form $1\times 1 +1\times l_0 +l_0\times 1$ over $k(Q)$, cf. Section~\ref{sec:iso-motives-kalpha}.

Recall 
that $(\phi_1)_{Q_{k(Q)}}$ 
has the form $\pi_1 \times 1 + x \times l_0$
for some $x\in \End(M_{k(Q)})$
(see discussion before Lemma~\ref{lm:lift_of_isos_from_k(Q)_to_Q}). 
We claim that, moreover, we may assume that $\phi_1$ is a projector.

The restriction of $\phi_1$ to $k(Q)$ is a projector,
and since the kernel of this restriction map is nilpotent (by Proposition~\ref{prop:geometric_RNP}), there exists 
a polynomial in $\phi_1$ (in the ring $\Enda{M_Q}$) that has the same restriction to $k(Q)$ and is a projector by~\cite[Lm.~2.4]{VishYag}.
However, the composition of the elements of the form $a\times 1 + b\times l_0$ in $\End(M_{Q_{k(Q)}})$
has the same form. 
Therefore, a polynomial in $(\phi_1)_{Q_{k(Q)}}$
will have the same form as we seek.

For the element $x$ that appears in  $(\phi_1)_{Q_{k(Q)}}$ 
we get the following relation due to the fact that $\phi_1$ is a projector:
\begin{equation}\label{eq:anticommutant_relation}
\pi_1 \circ x + x\circ \pi_1 = x.
\end{equation}
This relation implies that $x$ is a sum of two morphisms $x_1:=\pi_1\circ x\colon M_1\rarr M_2$ and $x_2:=x\circ\pi_1\colon M_2\rarr M_1$ as an endomorphism of $M=M_1\oplus M_2$.
In particular, $x^{\circ 2}$ is a direct sum of two morphisms $x_2\circ x_1\colon M_1\rarr M_1$ and $x_1\circ x_2\colon M_2\rarr M_2$,
and hence 
$\pi_1$ commutes with $x^{\circ 2}$ in $\End(M_{k(Q)})$.

By 
taking the push-forward of the element $\phi_1$ 
from $\Knt{Y\times Y\times Q}$ to $\Knt{Y\times Y}$ 
and base change to $k(Q)$ 
we get that $\pi_1 + x$ is a $k$-rational element.
Hence $x\times (1+l_0)$ is also $k$-rational as we can subtract $(\pi_1+x)\times 1$ from $\pi_1 \times 1 + x\times l_0$
without leaving the $k$-rational subgroup.
By composing the element $x\times (1+l_0)$ with itself and taking its push-forward
we see that $x^{\circ 2}$ is $k$-rational. Consider two following possibilities.

1. If $x^{\circ 2}$ is an isomorphism, then so must be $x_1\circ x_2$ and $x_2\circ x_1$, and hence $M_1$ is isomorphic to $M_2$, 
as we wanted to show. 
Moreover, assume that $M_1\cong M_2$ is not indecomposable. 
Note that 
we can apply the above argument about $\pi_1$ to a projector on the direct summand $N_1$ of $M_1\cong N_1\oplus N_2$. 
Then we conclude that $N_1$ is isomorphic to $N_2\oplus M_2\cong N_1\oplus N_2^{\oplus 2}$,
i.e.\ $N_1$ is a non-trivial direct summand of 
itself.

2. If $x^{\circ 2}$ is not an isomorphism, then it has to be nilpotent by~\eqref{eq:property-motive}. 
Using the fact that $\pi_1$ and $x^{\circ 2}$ commute and the relation~\ref{eq:anticommutant_relation} 
one gets that $(\pi_1+x)^{\circ 2^m} = \pi_1+\sum_{i=0}^m x^{\circ2^i}$,
and since $x^{\circ2}$ is nilpotent, we see that $(\pi_1+x)^{\circ 2^m}$ is a $k$-rational projector for high enough $m$.
Since $M$ is indecomposable, and  $(\pi_1+x)^{\circ 2^m}$ is not zero,
it must be the identity element. 
This implies that the projector $\pi_2$ equals $\sum_{i=0}^m x^{\circ 2^i}$, a sum of nilpotents,
and hence is zero.
\end{proof}

\begin{Cr}
\label{cr:pfister_morava_mdt}
Let $Q'$ be any quadric,
and let $Q$ be a quadric of dimension $2^n-1$ with upper binary Chow motive.

Then every indecomposable summand of $\MKn{Q'}$ stays indecomposable over $k(Q)$. 
\end{Cr}
\begin{proof}
We assume $n>1$, as otherwise the claim is trivial. By passing to the 
$\Kn$-kernel form of $Q'$ 
we may assume that $\dim Q'<2^{n+1}-1$. Let $N$ be an indecomposable direct summand of $\MKn{Q'}$,
we may also assume that it is not-Tate, and in particular that $N$ lies in the $\Kn$-kernel motive of $Q'$.

If $\dim Q'$ is odd, 
then $N$ cannot decompose into two isomorphic summands over any field extension, because all Tate motives in $\Mker{Q'_{\,\overline k}}$ are different, see Section~\ref{sec:rat-proj-quad}.  
Thus, by Proposition~\ref{prop:decomposition_over_k(alpha)} the motive $N$ stays indecomposable over $k(Q)$.

If $\dim Q'$ is even and $N_{k(Q)}$ decomposes into two isomorphic summands,
then $N$ becomes $\un\sh d^{\oplus 2}$ over $\overline{k}$, again by Section~\ref{sec:rat-proj-quad}. 
If $N_{k(Q)}\cong \un(d)^{\oplus 2}$,
then by Theorem~\ref{th:splitting_off_Tate} 
$N$ has to be a direct summand of the outer motive of $Q$.
However, $\MKn Q$ does not contain an indecomposable summand that becomes $\un\sh d^{\oplus 2}$ over $\overline{k}$. 
Hence $N_{k(Q)}\cong L^{\oplus 2}$ where $L$ is not split.

Let $\widetilde{N}$ be the corresponding to $N$ by Theorem~\ref{th:prestableMDT} indecomposable direct summand 
of the Chow motive of $Q'$. 
Recall that the Chow MDT of $\widetilde{N}$ is obtained by adding excellent connections of length $2^n-1$ to $\unCH\sh d^{\oplus 2}$.
Thus, if $\dim Q'<2^{n+1}-2$, then there are no connections to add, and 
$(\widetilde{N})_{\,\overline{k}}\cong\unCH\sh d^{\oplus 2}$.
In other words, $\widetilde{N}$ is a binary motive of length $0$, 
and thus $\widetilde N\cong R_{\mathrm{disc}(q)}\sh d$.
In particular, $\widetilde{N}$ splits over a field $K$ iff $\mathrm{disc}(q)$ is trivial. 
Hence $\widetilde{N}$ remains indecomposable over $k(Q)$, and therefore also does $N$.

If $\dim Q'=2^{n+1}-2$, $d=2^n-1$, then $\widetilde{N}$ is a rank $4$ motive
 with the Tate summands $\un\sh0, \un\sh d^{\oplus 2}, \un\sh{2d}$ 
over $\overline{k}$. 
Since $N$ splits into two non-trivial invertible summands over $k(Q)$, 
we see thet $\widetilde{N}$ splits into two binary motives. 
Over the field $k(Q'\times Q)$ this binary motive is split, and hence $Q'$ becomes hyperbolic. 
This implies that $(Q')_{k(Q)}$ is a Pfister quadric~\cite[Cor.~23.4]{EKM}.
Therefore its $\Kn$-kernel motive is isomorphic to $L\oplus \bigoplus_{j=0}^{2^n-2} L\sh j$.
By applying Proposition~\ref{prop:decomposition_over_k(alpha)} to the complement summand to $N$ in the $\Kn$-kernel motive of $Q'$ (here we use $n>1$) 
we get that it is a direct sum of invertible motives $L'$ already over field $k$ where $(L')_{k(Q)}\cong L$.
Let $B$ be a corresponding to $L'$ binary motive as a summand of $\MCh{Q'}$.
By the Theorem of Vishik--Izhboldin~\ref{th:binary_chow_motives} there exists $\beta \in \HH^{n+1}(k,\,\ZZ/2)$ such that 
$B$ is split 
over the splitting field $k(\beta)$, 
and hence $\widetilde{N}_{k(\beta)}$ splits off two Tate motives $\un\sh0, \un\sh{2d}$, and the complement summand is a binary motive of length $0$, 
i.e.\ $R_{\mathrm{disc} \left(q'_{k(\beta)}\right) }$.
Since $Q'$ becomes a Pfister quadric over $k(Q)$, its discriminant is trivial,
and hence $\widetilde{N}$ is split over $k(\beta)$,
i.e.\ $N_{k(\beta)}$ is split, contradicting Prop.~\ref{cr:splitting_over_quadrics_with_dim_2^n-1}.
\end{proof}
\begin{Rk}
\label{rk:chow_connections_over_k_alpha}
Using Theorem~\ref{th:prestableMDT} one can reformulate this corollary as the claim
that the connections in the ``central part'' of the Chow-MDT of $Q'$ of dimension less than $2^{n+1}-1$ do not vanish over $k(Q)$.
\end{Rk}

\section{Milnor K-theory modulo 2 as invertible Morava motives}\label{sec:milnor_k-theory_picard}

In this section we present one of the main new constructions of this article:
invertible motives corresponding to elements in Milnor K-theory modulo 2.

In the course of the proof of the Bloch--Kato conjecture~\cite{MerNorm, MerSusBK2, MerSusBK3, RostBK3, Voe_Z2, Voe_Zl}
it became clear that one needs some geometric objects 
that control the vanishing of the elements in $\mathrm{K}^{\mathrm{M}}_{n+1}(k)/p$.
Although one knows since Nesterenko--Suslin~\cite{NesSus} and Totaro~\cite{TotaroMilnor}  
that Milnor K-theory is a part of motivic cohomology of the point, 
and thus has an algebro-geometric description,
one needs more tools to study what happens to it over different field extensions.
The motivic approach to achieve this is by the study of the splitting of the Rost motives reviewed below,
however, see \cite{DemFlo} for a different approach 
 via  existence of rational points on some ind-varieties.

For a symbol $\alpha \in\km$, 
Rost motives $R_\alpha$
that appear in Pfister quadrics (see~Prop.~\ref{prop:prelim_motive_pfister}) 
were these objects
that allowed to prove the Milnor conjectures \cite{Voe_Z2, OVV}.
Two properties of them are crucial for this purpose: 
\begin{enumerate}
    \item[($R_\alpha$-1)] $R_\alpha$ splits over field extension $K$  if and only if $\alpha_K$ vanishes;
    \item[($R_\alpha$-2)] $R_\alpha$ is the upper motive of a $\nu_n$-variety $Q$, i.e.\ a smooth projective variety such that $[Q]_{\Kn}=v_n$.
\end{enumerate}
In this way $\Knf$ shows up in the Voevodsky's proof and, in fact, 
it is one of the main insights that led to conjecturing the existence
of algebraic Morava K-theory in the first place~\cite{Voe}.

Nevertheless Rost motives admit two shortcomings. 
First, $R_\alpha$
is defined and exists only for a symbol $\alpha$,
and not for an arbitrary element in $\km$.
Second, at least to the authors knowledge, there is no criteria that allows 
to show that $R_\alpha$ is a direct summand of some Chow motive $M$.
This makes finding smooth projective varieties in which $R_\alpha$
appears a complicated task. 
However, a lot of progress has been made in this direction
for projective homogeneous varieties,
see e.g.\ \cite{PSZ, Macdonald}.
Invertible $\Kn$-motives $L_\alpha$ that we introduce
 avoid both of these handicaps.

 Recall that $\Pic(C)$ denotes the group of classes of isomorphisms of invertible objects 
 in a symmetric monoidal category $C$ with the group operation coming from the monoidal structure.

\begin{Th}\label{th:nat_transform_milnor_picard}
Let $k$ be a field of characteristic 0.

\begin{enumerate}
    \item 
There exists a unique injective natural transformation 

\[
\begin{aligned}
    &\mathrm{K}^{\mathrm{M}}_{n+1}(-)/2 &&\xhookrightarrow{\quad} &&\Pic\left(\CM_{\Kn}(-)\right) \\
    &\hspace{2em} \alpha &&\mapsto && \hspace{1em} L_\alpha
\end{aligned}
\]

between functors of abelian groups on the category of field extensions of $k$.

\item If $\alpha$ is a symbol, then $L_\alpha$ is 
isomorphic to the summand of the $\Kn$-specialization of the Rost motive $R_\alpha$
defined in Proposition~\ref{prop:prelim_L_alpha_Rost}.
\end{enumerate}
\end{Th}

Recall that by the validity of the Milnor conjectures (\cite{Voe_Z2, OVV})
we have the following functorial isomorphisms:
$$ \km \cong \HH^{n+1}(k,\,\ZZ/2) \cong I^{n+1}(k)/I^{n+2}(k). $$
Therefore we could also speak of any of these groups in place of Milnor K-theory modulo $2$ in Theorem~\ref{th:nat_transform_milnor_picard}.

The construction of Theorem~\ref{th:nat_transform_milnor_picard} 
can be seen as a first step towards the ``categorification'' of Milnor K-theory modulo $2$.
Further one could incorporate other functorialities of Milnor K-theory in the motivic world.
For example, 
for a finite extension $K/k$ there is an additive norm map $\mathrm{Nm}_{K/k}:\mathrm{K}^{\mathrm M}_{n+1}(K)\rarr \mathrm{K}^{\mathrm M}_{n+1}(k)$.
In a future paper of the second author~\cite{SechInv}
it will be shown that the construction $\alpha \mapsto L_\alpha$
is also functorial with respect to these norm maps. For the categories of Morava motives (and their Picard groups)
one considers norm functors coming from the Weil restriction of varieties (cf. \cite{KarpWeil, Jouk, BachHoy}). 

The injectivity property of the natural transformation $\alpha \mapsto L_\alpha$
is the analogue of the above property ($R_\alpha$-1) for the Rost motives. The fact that $L_\alpha$ is an invertible motive
can be seen as some sort of analogue of ($R_\alpha$-2). Indeed, if $\MKn X$ contains an invertible motive as a direct summand,
then 
there exists a morphism $W\rarr X$ with $[W]_{\Kn}=v_n^r$ for some $r$ (see Lemma~\ref{lm:Kn-isotropic_invertible_direct_summands}).

In Theorem~\ref{th:detect_L_alpha} we give a criterion of how to detect $L_\alpha$.
It turns out that one can do it by counting Tate motives that are summands of the given motive $M$:
over $k$ and over some $\Kn$-universal splitting field of~$\alpha$.
This detection criteria can be stated using $\Kn$-numerical motives,
which are arguably simpler.
We should note here that Rost motives $R_\alpha$ 
vanish in $(\CH/2)_{\num}$-motives and are split in $(\CH\ot\QQ)_{\num}$- or $(\CH/p)_{\num}$-motives for odd prime $p$, 
i.e.\ cannot be detected numerically.

In a future paper of the second author~\cite{SechInv} 
the construction of invertible motives $L_\alpha$ 
will be extended to other primes $p$ 
and to classes  $\alpha \in \HH^{n+1}(k,\,\ZZ/p^r(n))$ for $r\in\NN$.

\subsection{Construction and uniqueness.}\label{sec:construction}

\begin{proof}[Strategy of the proof of Theorem~\ref{th:nat_transform_milnor_picard}.] 
We provide the proof
modulo statements that are shown in the rest of this section.
In the proof we use the Milnor conjecture on quadratic forms
and identify $\km$ with $I^{n+1}(k)/I^{n+2}(k)$, see~\cite{OVV}.

{\bf Construction.}
First, we construct a functorial map $\alpha \mapsto L_\alpha$
from $I^{n+1}(k)/I^{n+2}(k)$ to $\Pic\left(\CM_{\Kn}(k)\right)$.
This is done in Proposition~\ref{prop:def_L_alpha} as follows:
we lift an element $\alpha$ in $I^{n+1}(k)/I^{n+2}(k)$ to a class of a quadratic form $q$,
and the motive $L_\alpha$ appears as a direct summand of the $\Kn$-motive 
of the corresponding quadric $Q$. 
If $\alpha$ is a symbol, then this construction yields $L_\alpha$ defined in Proposition~\ref{prop:prelim_L_alpha_Rost}.
We then check that this map is additive in Proposition~\ref{prop:additivity_L_alpha}.

{\bf Injectivity.}
If $\alpha$ is a symbol and is non-zero,
then $L_\alpha$ is non-split, as it is shown in  Proposition~\ref{prop:prelim_L_alpha_Rost}.
However, it follows from the proof of the Milnor Conjecture (see \cite[Th.~2.10]{OVV})
that for an arbitrary non-zero element $\alpha$ in $\km$
there exists a field extensions $K/k$ such that $\alpha_K$ is a non-zero symbol.
Injectivity then follows from the functoriality of our construction.

{\bf Uniqueness.} It is enough to show that
for every $\alpha \in \km$ 
there exist a field extension $K/k$
such that the kernel of $\km\rarr \mathrm K^{\mathrm M}_{n+1}(K)/2$
is generated by $\alpha$
and that there is a unique invertible motive $L$ up to isomorphism 
such that $L_K\cong \un_K$. This is shown in Proposition~\ref{prop:k(alpha)-kernel-picard}.
\end{proof}

\begin{Prop}\label{prop:def_L_alpha}
Let $q\in I^{n+1}(k)$ be an anisotropic quadratic form,
 let $\alpha$ denote $[q]\in I^{n+1}(k)/I^{n+2}(k)$.

Then the $\Kn$-kernel part of $\MKn Q$ decomposes 
as a direct sum $ L_\alpha\sh{\frac{\dim Q}{2}} \oplus \bigoplus_{i=0}^{2^n-2} L_\alpha \sh i$
where $L_\alpha$ is an invertible motive. 

This construction defines a functorial map 
$$ I^{n+1}(k)/I^{n+2}(k) \rarr \Pic\left(\CM_{\Kn} (k) \right).$$
\end{Prop}
\begin{proof}
Let $q'$ be the $\Kn$-kernel form of $q$ and $K$ be its field of definition.
Then by our assumption and by the Arason--Pfister Hauptsatz either $q'$ 
is split\footnote{If $\alpha\neq 0$, then this cannot happen by $J$-filtration conjecture \cite[Th.~4.3]{OVV}.},
or it is a general $(n+1)$-Pfister form over $K$ with $[q']=\alpha_K$ in $I^{n+1}(K)/I^{n+2}(K)$.
In both cases the $\Kn$-kernel of the $\Kn$-motive of $Q'$ has the form as in the statement, see Prop.~\ref{prop:prelim_motive_pfister}.
By Prop.~\ref{prop:quad_MDT_stable} the $\Kn$-kernel of the $\MKn Q$ 
also has the same form. We define $L_{\alpha }$ as the invertible motive in the 
$\Kn$-kernel of $\MKn Q$ that ``lives in the grading $0$'', 
i.e.\  becomes $\un$ after completely splitting $Q$.

We need to show that $L_\alpha$ depends on $\alpha$, not on $q$.
Let $\widetilde{q}\in I^{n+1}(k)$ be a quadratic form such that $q-\widetilde{q}\,$ lies in $I^{n+2}(k)$,
then the $\Kn$-kernel motives of $Q$ and of $Q'$ differ by a Tate twist (see Corollary~\ref{cr:similar-kernel}).
Therefore, $L_\alpha$ is well-defined up to an isomorphism.

Functoriality of the construction is straight-forward: if we lift $\alpha \in I^{n+1}(k)/I^{n+2}(k)$
 to $q \in I^{n+1}(k)$, then $q_K$ is a lift of $\alpha_K$ for a field extension $K/k$,
 $\MKn{Q_K}=\MKn{Q}_K$ and hence $L_{\alpha_K} \cong (L_\alpha)_K$.
 \end{proof}

\begin{Prop}\label{prop:additivity_L_alpha}
Let $\alpha, \beta\in I^{n+1}(k)/I^{n+2}(k)$.
Then 
\begin{equation}\label{eq:additivity_L}
L_\alpha \ot L_\beta \cong L_{\alpha +\beta}.
\end{equation}

In particular, $L_\alpha^{\otimes 2} \cong \un$ for any $\alpha$.
\end{Prop}
\begin{proof}
Let $q_\alpha, q_\beta \in I^{n+1}(k)$ be quadratic forms that lift $\alpha, \beta$, respectively.
Then $q_\gamma:=q_\alpha\perp q_\beta$ is a lift of $\gamma:=\alpha+\beta$. 
By construction the motives $L_\alpha, L_\beta, L_{\gamma}$ are direct summands
of the $\Kn$-kernel parts of the motives of $Q_\alpha, Q_\beta, Q_\gamma$, respectively. 

Let $K_\alpha, K_\beta, K_\gamma$ be field extension of $k$ over which 
$\Kn$-kernels of these forms are defined, and let $F$ be the composite of these field extensions over $k$. 
Then the field extension $F/k$ satisfies the assumptions of Prop.~\ref{prop:reflects_MD}
and therefore it suffices to prove the isomorphism~\eqref{eq:additivity_L} over $F$.
Moreover, over $F$ the quadratic forms $q_\alpha, q_\beta, q_\gamma$
become Witt-equivalent  
to (not necessarily anisotropic\footnote{as we have not assumed that $\alpha, \beta$ or $\gamma$ are non-zero.}) 
general Pfister forms $q'_\alpha, q'_\beta, q'_\gamma$, respectively.

Let $F_\beta$ be the field of functions of $Q'_\beta$,
then $\alpha_{F_\beta} = \gamma_{F_\beta}$, and therefore $(L_\alpha)_{F_\beta} = (L_\gamma)_{F_\beta}$.
It follows from Theorem~\ref{th:iso_over_k(alpha)} that either isomorphism (\ref{eq:additivity_L}) holds 
or $L_\alpha \cong L_\gamma$. 
However, the latter implies that $\alpha_F = \gamma_F$ by Proposition~\ref{prop:prelim_L_alpha_Rost},
and hence $\alpha=\gamma$ by Theorem~\ref{KRS},
i.e.\  $\beta = 0$ and (\ref{eq:additivity_L}) holds nevertheless.
\end{proof}

\begin{Prop}\label{prop:k(alpha)-kernel-picard}
Let $\alpha\in \HH^{n+1}(k,\,\ZZ/2)$ and let $k(\alpha)$ be a $\Kn$-universal splitting field of $\alpha$.

Then $\Ker\left(\Pic(\CM_{\Kn}(k)) \rarr \Pic(\CM_{\Kn}(k(\alpha)) \right)$
is isomorphic to $\ZZ/2$ and is generated by $L_\alpha$.
\end{Prop}
\begin{proof} Note that that for any invertible $\Kn$-motive $L$
its endomorphisms can be identified with the endomorphism of $\unKn$, which is a field $\Kn^0(k)$ and does not depend on $k$. 
In particular, $L$ satisfies Rost Nilpotence Property for any field extension,
and satisfies the condition of Theorem~\ref{th:iso_over_k(alpha)}.

The field $k(\alpha)$ is a composite of $k\subset K\subset K(Q)$ 
where $K/k$ is $\overline{\Kn}$-universally surjective
and $Q$ is a Pfister quadric over $K$. 
By Proposition~\ref{prop:reflects_MD} the kernel on the Picard groups of Morava motives
for the base change from $k$ to $K$ is trivial.
And by Theorem~\ref{th:iso_over_k(alpha)} the kernel for the base change $K(Q)/K$
is $\ZZ/2$ with the generator $L_{\alpha_K}$.
\end{proof}

\subsection{Detection.}\label{sec:detection}

We complement the construction of the motives $L_\alpha$
with the statement that allows to detect them as direct summands of the $\Kn$-motives.
Note that if $L_\alpha$ is a direct summand of $M$,
then over a splitting field of $\alpha$ there is a ``new'' Tate summand of $M$.
It turns out that the converse is true. 

Moreover, the motives $L_\alpha$ can be also seen 
as $\Kn$-numerical motives $L_\alpha^{\num}$
and it suffices to find $L_\alpha$ as direct summands numerically (Proposition~\ref{prop:lifting_numerical_decompositions}).

\begin{Th}
\label{th:detect_L_alpha}
Let $\alpha \in \km$ and let $k(\alpha)$ 
denote a representative of the class of $\Kn$-universal splitting fields of $\alpha$ (see~Definition~\ref{def:k(alpha)}).

Let $M$ be a $\Kn$-motive over $k$, let $N$ be a $\Kn$-numerical motive over $k$. 
\begin{enumerate}
    \item If $M_{k(\alpha)}$ contains a direct summand $\un\sh i$,
    then $M$ contains either $\un\sh i$ or $L_\alpha\sh i$ as a direct summand.
    \item If $N_{k(\alpha)}$ contains
     a direct summand $\un\sh i$, then $N$ contains either $\un\sh i$ or $L^{\num}_\alpha\sh i$ as a direct summand.
    \item If $M^{\num}_{k(\alpha)}$ contains a direct summand $\un\sh i$,
    then $M$ contains either $\un\sh i$ or $L_\alpha\sh i$ as a direct summand.
\end{enumerate}
\end{Th}

\begin{Rk}
\label{rk:detect_tate_count}
Note that for a $\Kn$-numerical motive $N$, the rank of $\Kn^j_{\num}(N)$ equals
the number of Tate summands $\un(j)$ in $M$ (see Lemma~\ref{lm:numerical_tate_summands_Kn}). 
Therefore one can reformulate Theorem~\ref{th:detect_L_alpha} as follows:
the number of summands $L_\alpha(j)$ in $M$ equals the difference
$$\rk\Kn^j_{\num}(N_{k(\alpha)}) - \rk \Kn^j_{\num}(N).$$
\end{Rk}

\begin{proof}[Proof of Theorem~\ref{th:detect_L_alpha}]
As (1) and (3) follow from (2) by Proposition~\ref{prop:lifting_numerical_decompositions} below, we restrict to the latter. 
We can assume that $N$ does not contain Tate motives as direct summands and that $i=0$.

By construction of $k(\alpha)$, 
it is a composition of the field extensions $k\subset K\subset K(Q)$
where $K/k$ is the $\oKn$-universally bijective, and $Q$ is a Pfister quadric corresponding to the symbol $\alpha_K$.
Note that if $K/k$ is $\oKn$-universally bijective field extension, 
then it is also $\Knum$-universally bijective. 
Since universally bijective field extensions do not affect motivic decompositions and isomorphisms,
it suffices to prove the claim over $K$, i.e.\ we can assume that $\alpha$ is a symbol.

Let $Q_\alpha$ be a norm quadric for symbol $\alpha$.
By using geometric RNP (Proposition~\ref{prop:geometric_RNP})
we can lift the splitting off of Tate from $N_{k(\alpha)}$ to motives over $Q_\alpha$.
By applying the push-forward functor we thus get:
\begin{equation}
\label{eq:mot_norm_split}
   N\otimes \mathrm M_{\Kn}^{\num}(Q_\alpha) \xhookleftarrow{\oplus} \mathrm M_{\Kn}^{\num}(Q_\alpha)(i). 
\end{equation}

However, $\mathrm M_{\Kn}^{\num}(Q_\alpha)\cong \un \oplus L_\alpha$. Indeed, $\mathrm M_{\Kn}(Q_\alpha)\cong \un \oplus L_\alpha \oplus \mathrm M_{\Kn}(Q')(1)$
where $Q'$ is an anisotropic quadric of dimension $2^n-3$ (see e.g.\ \cite[Th.~17]{Rost}).
The $\Kn$-numerical motive of $Q'$ is zero, since $Q'$ is $\Kn$-anisotrpoic.
Thus, the claim now follows from (\ref{eq:mot_norm_split}) by using the semi-simplicity 
of the category of $\Kn$-numerical motives (\cite[Prop.~6.1]{DuVish}).
\end{proof}

\begin{Prop}
\label{prop:lifting_numerical_decompositions}
Let $A$ be an oriented theory
and let $\Gamma$ be a system of relations over $k$
such that the map $A(\Spec k)\rarr A_\Gamma(\Spec k)$ has nilpotent kernel.

Then 
\begin{enumerate}
\item the canonical homomorphism 
$$ \Pic\left(\CM_A(k)\right) \rarr \Pic\left(\CM_{A_\Gamma}(k)\right)$$
is injective;

\item if for $M\in \CM_A(k)$, $L\in \Pic(\CM_A(k))$
for the corresponding $A_\Gamma$-motives $M_{A_\Gamma}$ contains $L_{A_\Gamma}$
as a direct summand,
then $M$ contains $L$ as a direct summand.
\end{enumerate}
\end{Prop}
\begin{proof}
Claim (1) is a particular case of the claim (2), so we restrict to the latter.
Assume that we have morphisms $\alpha, \beta$ of $A_\Gamma$-motives:
$$ L_{A_\Gamma} \xrarr{\alpha} M_{A_\Gamma} \xrarr{\beta} L_{A_\Gamma}$$
such that their composition $\beta\circ \alpha$ is the identity. 

We can lift these morphisms to morphisms $a, b$
between $L$ and $M$
and the composition $b\circ a \in \Hom(L,\,L)$ is a lift of the identity.
Note that the ring of endomorphisms of any invertible object 
is canonically isomorphic to the endomophisms of the unit, 
in our case $A(k)$ or $A_\Gamma(k)$.
Thus, $b\circ a$ is identified with an element in $A(k)$ 
that is mapped to $1$ in $A_\Gamma(k)$. By our assumption $b\circ a$ differs from $1$ by nilpotent,
and therefore $b\circ a$ is an isomorphism, and therefore $L$ is a direct summand of $M$.
\end{proof}

We complement the detection criterion for $L_\alpha$
with the following statement that relates the occurrence of $L_\alpha$ in $\mathrm M_{\Kn}(X)$
to the vanishing of $\alpha$ over $k(X)$.

Recall that if $X$ is $\Kn$-anisotropic, 
then $\mathrm M^{\num}_{\Kn}(X)$ is zero, and $\mathrm M_{\Kn}(X)$ cannot contain $L_\alpha$ as a direct summand
by Proposition~\ref{prop:lifting_numerical_decompositions}.
However, if $X$ is $\Kn$-isotropuc, then $\mathrm M_{\Kn}(X)$ contains a Tate summand (Lemma~\ref{lm:kn-isotropic-tate}).
We show that if in addition $\alpha_{k(X)}=0$, then for each Tate direct summand $\un(j)$ in $\mathrm M_{\Kn}(X)$
there is a direct summand $L_\alpha(j)$.

\begin{Prop}
Let $X\in\SmProj_k$, let $s\in \ZZ$, and let $J$ be a finite set.
Let $\alpha \in \km$ such that $\alpha_{k(X)}=0$. 

Assume that $\un(s)^{\oplus J}$ is a direct summand of $\mathrm M_{\Kn}(X)$.
Then $\left(L_\alpha(j)\right)^{\oplus J}$ 
is also a direct summand of $\mathrm M_{\Kn}(X)$.
\end{Prop}
\begin{proof}
    Since $\alpha_{k(X)}=0$, we have $(L_\alpha)_{k(X)}\cong \un_{k(X)}$ by the functoriality of motives $L_\alpha$.
    By Proposition~\ref{prop:geometric_RNP} we can lift this isomorphism to motives over $X$,
    i.e.\  $(L_\alpha)_X$ is isomorphic to $\un_X$.
    By taking the push-forward to the category of motives over $k$ we obtain:
    $$ L_\alpha \otimes \mathrm M_{\Kn}(X) \cong \mathrm M_{\Kn}(X).$$
    We can view this isomorphism in the category of numerical $\Kn$-motives, which is semi-simple by~\cite[Prop.~6.1]{DuVish}.
    Then $\mathrm M_{\Kn}^{\num}(X)$ contains $L_\alpha \otimes \un(s)^{\oplus J}$, 
    and the claim follows by Proposition~\ref{prop:lifting_numerical_decompositions}.
\end{proof}

\begin{Lm}
\label{lm:kn-isotropic-tate}
Let $X\in\SmProj_k$ be $\Kn$-isotropic, i.e.\ $\Kn(X)\xrarr{(\pi_X)_*} \Knpt$ is non-zero.

Then $\mathrm M_{\Kn}(X)$ contains $\un(0)$ and $\un(\dim X)$ as direct summands.
\end{Lm}
\begin{Rk}
If $\dim X \equiv 0 \mod (2^n-1)$, we have $\un(\dim X)\cong \un(0)$ by periodicity of Morava K-theory,
and it is possible that $\mathrm M_{\Kn}(X)$ contains only one Tate summand (e.g.\ $X$ is an anisotropic quadric of dimension $2^n-1$).
\end{Rk}
\begin{proof}
    By assumption there exists $u\in \Kn^{\dim X}(X)$ such that $(\pi_X)_*=1$. 
    Thus, the following compositions are identities:
    $$ \un(0) \xrarr{u} \mathrm M_{\Kn}(X) \xrarr{1_X} \un(0) $$
    $$ \un(\dim X)\xrarr{1_X} \mathrm M_{\Kn}(X) \xrarr{u} \un(\dim X),$$
    and provide direct summands that we look for.
\end{proof}

\section{Invertible summands of motives of quadrics}\label{sec:inv_summand_quadrics}

In Section~\ref{sec:milnor_k-theory_picard} 
we have defined invertible $\Kn$-motives $L_\alpha$ for $\alpha \in \km$
which are direct summands in the motives of quadrics $Q$ with $q\in I^{n+1}(k)$.
In this section we show that these invertible motives appear also
in quadrics that are 'not very far' from $I^{n+1}(k)$,
and conjecture that these are the only quadrics containing non-trivial invertible summands.
To provide evidence for our conjecture we deduce 
it from the Kahn's descent conjectures that are mostly open.
However, the known cases are sufficient to show our claim for $n=1,2,3$.

\subsection{Structure of invertible motives and Kahn--Rost--Sujatha conjecture}

\begin{Prop}\label{prop:inv_summand_becomes_L_alpha}
Let $Q$ be a quadric over $k$.
Let $L$ be an invertible motive in $\MKn Q$.
Let $K$ be the field of definition of the $\Kn$-kernel of $Q$.

Then there exists unique $\alpha \in \HH^{n+1}(K,\,\ZZ/2)$
such that $L_K\cong L_\alpha(i)$ for some $i\in \ZZ$. 

Moreover, $\alpha$ is unramified over $k$.
\end{Prop}
\begin{proof}
By Theorem~\ref{th:prestableMDT} the invertible motive $L_K$ in the $\Kn$-kernel of $\MKn{Q_K}$
corresponds to a binary summand $R$ of the Chow motive of $Q_K$. 
By the Vishik--Izhboldin's Theorem~\ref{th:binary_chow_motives}
there exists $\alpha \in \HH^{n+1}(K,\,\ZZ/2)$ such that 
$R_F$ is split iff $\alpha_F$ is zero for any field extension $F/K$.
By Theorem~\ref{th:prestableMDT} the motive $R_F$ is split iff $L_F$ is split,
and by the universal property of $L_\alpha$ we get that $L_F \cong L_\alpha(i)$ 
(where $i\in \ZZ$ is determined by $L_{\overline{k}} \cong \un_{\overline{k}}(i)$).

The claim that $\alpha$ is unramified follows by applying the following lemma
to $X$ that is a smooth projective model of the field $K/k$.
Recall that $K$ is a composite of the field of functions of quadrics of dimension greater than $2^{n+1}-2$,
so that the injectivity assumption of the lemma is satisfied for each such field extension.
\end{proof}

\begin{Lm}\label{lm:L_alpha_unramified}
Let $L$ be an invertible $\Kn$-motive over a field $k$, 
let $X$ be a smooth projective variety over $k$
such that for any divisor $D$ on $X$ the homomorphism
$$\HH^n(k(D),\,\ZZ/2) \rarr \HH^n(k(D\times X),\,\ZZ/2) $$
is injective.

If $L_{k(X)}\cong L_\alpha$ for some $\alpha \in \HH^{n+1}(k(X),\,\ZZ/2)$,
then $\alpha$ is unramified over $k$.
\end{Lm}
\begin{proof}
Consider two field embeddings $i_1, i_2\colon k(X) \rarr k(X\times X)$ coming from projections $p_1, p_2\colon X\times X\rarr X$.
We have that $L_{k(X\times X)} \cong i_1^*(L_\alpha) \cong i_2^*(L_\alpha)$,
and by Theorem~\ref{th:nat_transform_milnor_picard} it follows that $i_1^* \alpha = i_2^* \alpha$ in $\HH^{n+1}(k(X\times X),\,\ZZ/2)$.

Consider a valuation over $k(X)$ corresponding to a divisor $D$ on $X$,
and let $\partial_D$ denote the residue map on the Galois cohomology. 
This map fits into a commutative diagram ($\HH^m$ stands for the Galois cohomology with $\ZZ/2$-coefficients):
\begin{center}
	\begin{tikzcd}
		\HH^{n+1}(k(X\times X)) \arrow[r, "\partial_{D\times X}"] & \HH^n(k(X\times D))\\
		
		\HH^{n+1}(k(X)) \arrow[r, "\partial_D"] \arrow[u, "i_1^*"] & \HH^n(k(D)) \arrow[u] \\
	\end{tikzcd}
\end{center}
Since the right upper arrow is injective, $\alpha$ has zero residue along $D$
iff $i_1^*\alpha$ has zero residue along $D\times X$.
However, $i_1^*\alpha = i_2^*\alpha$ 
and $\partial_{D\times X} (i_2^*\alpha) = 0$ for trivial reasons. 
\end{proof}

\begin{Cr}
\label{cr:inv_order_2}
Let $Q$ be a quadric over $k$.
Let $L$ be an invertible direct summand of $\Mot{\Kn}{Q}$.

Then $L^{\otimes 2} \cong \un(j)$ for some $j$.
\end{Cr}
\begin{proof}
    We have that $(L^{\otimes 2})_K \cong \left(L_{\alpha}(i)\right)^{\otimes 2} \cong \un(j)$,
    and this isomorphism descends to $k$ by Proposition~\ref{prop:reflects_MD}.
\end{proof}

\begin{Prop}
\label{prop:invertible_summ_quadrics_iso}
Let $Q$ be a quadric, $L_1, L_2$ be two invertible direct summands of $\Mker{Q}$. 

Then $L_1 \cong L_2(s)$ for some $s\in \ZZ$.
\end{Prop}
\begin{proof}
    For $n=1$ the claim follows from the description of the $\KK$-motives of quadrics, see Section~\ref{sec:K0-motive-quadric}
    (see also Section~\ref{sec:rat-proj-quad} for the definition of $\widetilde{\,\mathrm M}_{\K1}(Q)$).

    Let $n\ge 2$. If $Q$ is even-dimensional and $\Mker{Q}$ contains an invertible motive in the 'middle', 
    then $q\in I^{n+1}$ (see the proof of Corollary~\ref{cr:pfister_morava_mdt}). 
    In this case,$\Mker{Q}$ decomposes into a sum of $L_{[q]}(j)$ as explained in Proposition~\ref{prop:def_L_alpha},
    and the claim is true.

    Otherwise, we may assume that $q\notin I^{n+1}$, and the rational projector that defines 
    an invertible summand of $\Mker{Q}$ is determined uniquely (see~Section~\ref{sec:morava-mdt}).
    Since we need not require $Q$ to be anisotropic, we may replace $q$ by $q\perp \hyp$, if needed,
    so that $\dim Q \neq 0 \mod 4$, and use Lemma~\ref{lm:motive_of_isotropic_quadric} to relate their motives.
    Therefore, we may assume that the rational projector for $L_1$ has the form $\varpi_i = a_i \times a_{D'-i}$,
    and $a_j\times a_{D'-j}$ for $L_2$,
    where $a_s = h^s +v_n l_{D'-s} \in \Kn(\overline{Q})$.
    By Corollary~\ref{cr:inv_order_2} $L_1^{\otimes 2}$ and $L_2^{\otimes 2}$ are split,
    and therefore $a_i \times a_i$, $a_{D'-i}\times a_{D'-i}, a_j\times a_j, a_{D'-j}\times a_{D'-j}$ are rational.
    However, $h\cdot a_i = a_{i+1}$, and by multiplying these rational elements 
    with $1\times h^b$ or $h^b\times 1$ we can get 
    that at least one of the elements $a_i \times a_{D'-j}$ or $a_j\times a_{D'-i}$ is rational.
    Since these elements determine isomorphisms between motives $L_1$ and $L_2(s)$ over $\overline{k}$ for appropriate $s$,
    we get the claim by the Rost Nilpotence Property.    
\end{proof}

Note that for the field extension $K/k$ in the Proposition~\ref{prop:inv_summand_becomes_L_alpha}
the base change map $\HH^{n+1}(k,\,\ZZ/2)\rarr \HH^{n+1}(K,\,\ZZ/2)$ is injective.
Thus, if $\alpha$ defined over $K$ lies in the image of this base change map, 
then it descends to $k$ uniquely.
We expect this to be the case.

\begin{Conj}\label{conj:inv_summand_descent}
Let $L$ be an invertible direct summand of $\mathrm M_{\Kn}(Q)$ of a quadric $Q$ over $k$.

Then $L \cong L_\alpha(i)$ for some $\alpha \in \mathrm{H}^{n+1}(k,\,\ZZ/2)$, $i\in \ZZ$.
\end{Conj}

Conjecture~\ref{conj:inv_summand_descent} follows from the 
validity of the Kahn--Rost--Sujatha conjecture 
about the unramified Galois cohomology (see below).
The latter is known in some cases, and so the validity of \ref{conj:inv_summand_descent} 
follows for $n=1,2,3$. However, in Section~\ref{sec:appearance_invertible_summands_quadrics} 
we deduce a stronger statement for these $n$ from a more general conjecture of Kahn.

\begin{Conj}[{Kahn--Rost--Sujatha, {\cite[p.845]{KRS}}; see also \cite{KahnSujatha, KahnSujatha2}}]\label{conj:KRS}
Let $Q$ be a quadric over $k$ of dimension greater or equal than $2^{n+1}-1$.

Then the base change homomorphism
$$ \eta^{n+1}_2\colon \HH^{n+1}(k,\,\ZZ/2) \rarr \HH^{n+1}_{\mathrm{nr}}(k(Q)/k,\,\ZZ/2) $$
is an isomorphism.
\end{Conj}

\begin{Prop}
\label{prop:KRS_implies_invertible_summands_quadrics}
Conjecture~\ref{conj:KRS} implies Conjecture~\ref{conj:inv_summand_descent}.
\end{Prop}
\begin{proof} 
Follows from Proposition~\ref{prop:inv_summand_becomes_L_alpha}.
\end{proof}

\subsection{Quadratic forms of small Kahn dimension}
\label{sec:small_kahn_dim}

Recall the following invariant that measures the distance from a quadratic form
 to the powers of the fundamental ideal in the Witt ring.

\begin{Def}[{Kahn, \cite[Def.~1.1]{Kahn-def}}; {Vishik, \cite[Def.~6.8]{Vish-quad}\footnote{Note that 
the notation of Vishik differs by taking $\dim_{n+1}$ as $\dim_n$ of Kahn.}}]\label{def:dimn}
Let $q$ be a quadratic form over $k$, then the $n$-th Kahn dimension of $q$ is 
$$\dim_n (q) := \min\{\dim q' \mid q\perp q'\in I^{n+1}(k)\}.$$

Note that if $\dim_n(q)<2^n$, 
then there exists unique $q'$ of dimension $\dim_n(q)$ such that $q\perp q'\in I^{n+1}(k)$.
In this case we will denote $q'$ as $r_n(q)$. Moreover, in this case we define 
 $$\omega_{n+1}(q)\in I^{n+1}(k)/I^{n+2}(k) \cong \HH^{n+1}(k,\,\ZZ/2)$$
 as the class of $q \perp r_n(q)$  modulo $I^{n+2}(k)$.
\end{Def}

\begin{Prop}\label{prop:Kn-motive-q-small-dim_n}
Let $q$ be an anisotropic quadratic form over $k$ of $\dim q\geq 2^n$. Assume that 
$\dim_nq\leq 2^n$, and let $q'$ be a form of $\dim q'=\dim_nq$ such that $q\perp q'\in I^{n+1}(k)$, 
and $\alpha\in\HH^{n+1}(k,\,\ZZ/2)$ denote 
the class of $q \perp q'$  modulo $I^{n+2}(k)$.

Then the $\Kn$-kernel summand $\Mker Q$ of $Q$ 
decomposes as a sum of 
a Tate twist of
$\MKn{Q'}\otimes L_\alpha$  
and 
$(2^n-\dim_nq)$ Tate twists of $L_{\alpha}$.
\end{Prop}

\begin{Rk}\phantom{a}
\begin{enumerate}
\item
In the notation of Proposition~\ref{prop:Kn-motive-q-small-dim_n}, if $\dim_n(q)<2^n$ one has $q'=r_n(q)$, $\alpha = \omega_{n+1}(q)$.
\item 
By Theorem~\ref{th:prestableMDT}, $\MKn{Q'}$ cannot have invertible summands, and therefore Proposition~\ref{prop:Kn-motive-q-small-dim_n} allows to describe all invertible summands of the $\Kn$-motive of $Q$.
\item 
To explain the assumption on the dimension of $q$ in the statement of Proposition~\ref{prop:Kn-motive-q-small-dim_n}, we remark that for $\dim\,q<2^n$ one has $q'=-q$, $\alpha=0$ and tautologically $\Mker Q\cong\MKn{Q'}$.
\end{enumerate}
\end{Rk}

\begin{proof}
We may assume that $p:=(q\perp q')_{\mathrm{an}}$ is a general $(n+1)$-fold Pfister form passing to its leading field by Proposition~\ref{prop:reflects_MD}, and denote $K=k(p)$. 

Let $q'':=q'\perp\hyp^t$ such that $\dim\,q''=\dim\,q$, then the motives of $Q'$ and $Q''$ are related
by Lemma~\ref{lm:motive_of_isotropic_quadric}.
On the other hand, since $q_K=-q''_K$, Morava MDT of $q$ and $q''$ coincide by Corollary~\ref{cr:pfister_morava_mdt}, and, moreover, each indecomposable summand $N$ of $\Mker Q$ is either isomorphic to the corresponding indecomposable summand $N'$ of $\Mker{Q''}$, or to $N'\otimes L_\alpha$ by Theorem~\ref{th:iso_over_k(alpha)}. We will prove that the latter is always the case.

First,
observe that $\Mker Q$ has exactly $(2^n-\dim q')$ invertible summands, and each of them is either Tate, or isomorphic to a Tate twist of $L_\alpha$.
Assume that $\Mker Q$ contains a Tate summand and $\alpha \neq 0$. 
If we pass to the $\Kn$-kernel form of $q$, then $p$ remains anisotropic by Proposition~\ref{prop:reflects_MD}.
However, the existence of a Tate summand in $\Mker Q$ now implies that the anisotropic dimension of $q$ is less than $2^n$  by Theorem~\ref{th:prestableMDT} (cf. Remark~\ref{rk:th414-n1} for $n=1$).  
But then the dimension of $p=(q\perp q')_{\mathrm{an}}$ 
is less than $2^{n+1}$, which contradicts the Arason--Pfister Hauptsatz.

Second, we need to show that for every indecomposable direct summand $N'$ of $\MKn{Q'}$, 
the corresponding summand $N$ in $\Mker Q$ is isomorphic to $N'\otimes L_\alpha$.
Since $\dim q'\leq 2^n$, every indecomposable direct summand $N'$ of $\MKn{Q'}$
is the specialization of an indecomposable Chow motive $\widetilde{N}$ that is a summand of $\MCh{Q'}$ by Proposition~\ref{prop:unstable_quadric}, 
and by Karpenko's Theorem~\cite[Th.~1.1]{Karp-upper-outer}
we have that $\widetilde{N}$ is isomorphic to the upper motive
of some projective homogeneous variety $Y$.
There are two possibilities to consider.

1. The element $\alpha$ does not vanish over $k(Y)$. 
Then by passing to $k(Y)$ we see that a Tate summand splits off from $N'$,
however, since $\alpha_{k(Y)}\neq 0$ the $\Kn$-kernel motive $\Mker{Q_{k(Y)}}$ of $Q_{k(Y)}$ 
cannot have a Tate summand by the Arason--Pfister Hauptsatz, similarly to the discussion above. Hence $N$ cannot be isomorphic to $N'$. 

2. The element $\alpha$ vanishes over $k(Y)$. 
In this case we 
get that $\widetilde{N}\otimes R_\alpha \cong \widetilde{N}\oplus \widetilde{N}\sh{2^n-1}$ by Lemma~\ref{lm:product_with_rost_motive}, 
and specializing this isomorphism to $\Kn$-motives we get that 
$N\otimes (\un \oplus L_\alpha)\cong N \oplus N$, hence $N\cong N\otimes L_\alpha$.
\end{proof}

\begin{Ex}
\label{ex:morava-chow-mdt}
Recall that 
Chow motives of excellent quadrics decompose as sums of Tate twists of Rost motives~\cite[Prop.~4]{Rost-new},~\cite[Th.~7.2]{KarpMer-excel}. Therefore as an application of Proposition~\ref{prop:Kn-motive-q-small-dim_n} we determine $\Kn$-motivic decompositions of quadratic forms $q$ which are not excellent themselves, but congruent to an excellent form of dimension at most $2^n$ modulo $I^{n+1}(k)$. A particular example of such a form is discussed in Example~\ref{ex:chow-morava-mdt}.

Recall also that computation of $\Knt Q$ provides certain bounds on torsion in $\CHt Q$~\cite[Sec.~8]{SechSem}. For quadratic forms $q$ as above, computation of $\Knt Q$ now reduces to computation of $\Knt{(R_\beta)_{\Kn}\otimes L_\alpha}$. One case where this can be performed is where $\beta$ divides $\alpha$ in $\HH^*(k,\,\ZZ/2)$. Indeed $(R_\beta)_{\Kn}\otimes L_\alpha\cong(R_\beta)_{\Kn}$ by Lemma~\ref{lm:product_with_rost_motive}, and $\Knt{(R_\beta)_{\Kn}}$ can be computed due to~\cite{VishYag}, see~\cite[Prop.~6.2]{SechSem}.
\end{Ex}

\begin{Rk}
\label{rk:k0-motive-gen}
Note that for any quadratic form $q$ there exists a quadratic form $q'$ of dimension at most $2$ such that $q\perp q'\in I^2(k)$.
Thus, Proposition~\ref{prop:Kn-motive-q-small-dim_n} generalizes the description of $\KK$-motives of quadrics from Section~\ref{sec:K0-motive-quadric}.
\end{Rk}

\subsection{Occurrence of invertible motives and Kahn's descent of quadratic forms}
\label{sec:appearance_invertible_summands_quadrics}

We conjecture that the converse to Proposition~\ref{prop:Kn-motive-q-small-dim_n} also holds.

\begin{Conj}\label{conj:inv_summand_quadrics}
Let $Q$ be an anisotropic quadric over $k$.
If its $\Kn$-kernel motive $\Mker Q$  contains an invertible summand,
then $\dim_n q <2^n$.
\end{Conj}

This conjecture turns out to be a corollary of the
following conjecture of Kahn.
 Recall that for an irreducible projective $X\in\Smk$ and $K = k(X)$ its field of rational functions,
 an element of $W(K)$ is called unramified, 
 if its residue class\footnote{with respect to the second residue homomorphism of~\cite[Chapter~VI, Definition~2.5]{Scharlau}} at each codimension $1$ point of $x$ of $X$ is $0$. 
 The set of unramified elements is denoted by $\Wnr K$ 
 and called the unramified Witt ring.
  We say that the quadratic form $q$ over $K$ is defined over $k$
   if there exists a quadratic form $q'$ over $k$ such that $(q')_K$ is isometric to $q$.

\begin{Conj}[Kahn, {\cite[Conjecture 1]{Kahn-descent}}]\label{conj:kahn}
Let $q$ denote an anisotropic non-degenerate quadratic form over $k$, $K=k(q)$ and $r$ a quadratic form defined over $K$. Assume that 
$\mathrm{dim}\,r<\frac12\,\mathrm{dim}\,q$ and $r\in \Wnr K$.  

Then $r$ is defined over $k$.
\end{Conj}

The above conjecture predicts various phenomena in the theory of quadratic forms, 
cf.~\cite[Sec.~5.2]{SculHyper}, \cite[Rem.~5.4]{SculExt},
although it remains widely open (for the known cases see Theorem~\ref{prop:kahn-conj-known} below).
We will consider the following slightly weaker version of it. 

\begin{Conj}\label{conj:kahn_n}
Let $q$ denote an anisotropic quadratic form over $k$, $K=k(q)$ and $r$ a quadratic form defined over $K$. Assume that 
$\dim r<2^n<\frac12\dim q$ for some $n$, and $r\in \Wnr K$.  

Then $r$ is defined over $k$.
\end{Conj}

\begin{Th}[Kahn, Laghribi, Izhboldin--Vishik]
\label{prop:kahn-conj-known}
Conjecture~\ref{conj:kahn_n} holds for $n=1,2,3$.
\end{Th}
\begin{proof}
    The cases $n=1,2$ of Conjecture~\ref{conj:kahn} are shown by Kahn in~\cite[Th.~1]{Kahn-descent}.
    For the case $n=3$ see~\cite[Th\'eor\`eme principal]{Lagr} and~\cite[Th.~3.9]{IzhVish}. 
\end{proof}

We will now show that our Conjecture~\ref{conj:inv_summand_quadrics} follows from Conjecture~\ref{conj:kahn_n}. 
We start with the following observation (cf.~\cite[Prop.~1]{Kahn-descent}).

\begin{Prop}
\label{prop:dim_n_over_Kn-kernel-field}
For $n>0$ let $p$ denote an anisotropic quadratic form over $k$ of $\dim p>2^{n+1}$, $K=k(p)$, $q$ a quadratic form  over $k$ such that $\dim_{n} q_K<2^n$, and assume that Conjecture~\ref{conj:kahn_n} holds for $n$. 

Then $\dim_n q<2^n$.
\end{Prop}
\begin{proof}
Let $r$ denote a quadratic form over $K$ of dimension $\mathrm{dim}_n(q_K)<2^n$ such that $q_K\equiv r\mod I^{n+1}(K)$.
For a residue homomorphism $\partial$ associated to a point $x$ of $P$,
  one has $\partial(q_K)\equiv\partial(r)\mod I^{n}\big(\kappa(x)\big)$, 
  however, $q_K$ is unramified, therefore 
  $\partial(r)=0$ by the Arason--Pfister Hauptsatz, 
  i.e., $r\in\Wnr K$.
   Applying Conjecture~\ref{conj:kahn_n} we can find a form $r'$ over $k$ such that $r'_K=r$. 
   By Corollary~\ref{cr:KRS-Witt_mod_In+1_injectivity} we conclude that
    $q\perp-r'\in I^{n+1}(k)$.
\end{proof}

Recall that if the upper Chow motive of a projective quadric $Q$ is the Rost motive $R_\alpha$ for a non-zero pure symbol $\alpha$ corresponding to the anisotropic Pfister form $p_\alpha$, then $q$ is proportional to a subform of $p_\alpha$, that is, $q$ is a Pfister neighbor (e.g., by the Subform Theorem~\cite[Th.~22.5]{EKM}). We will need the following generalization of this fact.

\begin{Lm}
\label{lm:rost-kahn-dim}
Let $q$ denote a quadratic form of $\dim q\leq 2^{n+1}$, assume that $\MCh Q$ has a summand isomorphic
to a Tate twist of $R_\alpha$ for a non-zero pure symbol $\alpha\in\Hnpo k$, and assume that $r:=(-q\perp p_\alpha)_{\mathrm{an}}$ is of $\dim r\leq 2^{n+1}$. Then in fact $\dim r< 2^{n}$, in particular, $\dim_n q< 2^{n}$.
\end{Lm}
\begin{proof}
Observe that $q$ and $r$ have the same Morava MDT by Corollary~\ref{cr:pfister_morava_mdt}. 
Since $\Mker Q$ has an invertible summand isomorphic to a Tate twist of $L_\alpha$, 
then $\Mker R$ also has an invertible summand, which is either Tate or isomorphic to a Tate twist of $L_\alpha$ by Theorem~\ref{th:iso_over_k(alpha)}. In the former case we get the claim by Theorem~\ref{th:prestableMDT}.

In the latter case, $\MCh R$ has a summand isomorphic to a Tate twist of $R_\alpha$ by Theorem~\ref{th:prestableMDT}. Let $K=k_t$ be the field in the generic splitting tower of $q$ such that $(R_\alpha)_K$ is indecomposable, but $(R_\alpha)_{k_{t+1}}$ is split. Then $q_t$ and $(p_\alpha)_K$ are isotropic over function fields of each other, therefore the upper Chow motive of $Q_t$ is $(R_\alpha)_K$~\cite[Cor.~3.9]{Vish-quad}, and $q_t$ is a Pfister neighbor.

Let $q_t\perp s=a\cdot (p_\alpha)_K$ for some $a\in K$ and $s$ of $\dim s<2^n$. Since $q_K\perp r_K=(p_\alpha)_K\perp\hyp^i$, we conclude that $r_K\perp -s$ is Witt-equivalent to $(p_\alpha)_K\perp-a(p_\alpha)_K\in I^{n+2}(K)$. By the Arason-Pfister Hauptsatz, this is only possible if $r_K$ is Witt-equivalent to $s$. However, $(R_\alpha)_K$ cannot be a direct summand of $\MCh S$ for dimensional reasons.  
\end{proof}

\begin{Prop}
\label{kahn-unary} 
If Conjecture~\ref{conj:kahn_n} holds for some $n>0$, then  Conjecture~\ref{conj:inv_summand_quadrics} also holds for that $n$.
\end{Prop}
\begin{proof}
Let $q$ be a quadratic form over $k$
and assume that the $\Kn$-kernel of $Q$ contains an invertible summand. 
Using Proposition~\ref{prop:dim_n_over_Kn-kernel-field} 
we reduce to the case when $\dim q\leq 2^{n+1}$ by passing to the $\Kn$-kernel form of $q$.

If the invertible summand is trivial, then we have $\dim q < 2^n$ by Theorem~\ref{th:prestableMDT}, in particular $\dim_n q < 2^n$. 
Otherwise, 
$\MCh Q$ has an indecomposable binary summand $B$ by Theorem~\ref{th:prestableMDT}.
By the Vishik--Izhboldin Theorem~\ref{th:binary_chow_motives} there exists $\alpha \in \HH^{n+1}(k,\,\ZZ/2)$ such that 
$B_E$ is split iff $\alpha_E=0$ for any field extension $E/k$. Let $p\in I^{n+1}(k)$ be a quadratic form representing $\alpha$, and $K$ its leading field, so that $p_K$ is a general Pfister form, and $B_K\cong R_{\alpha_K}$ is the Rost motive. 
We pass to the first field $L$ in the generic splitting tower of $r:=-q_K\perp p_K$, 
such that $\dim (r_L)_{\mathrm{an}}\leq 2^{n+1}$. 
By~Lemma~\ref{lm:rost-kahn-dim} we conclude that $\dim_n q_L<2^n$, and
by Proposition~\ref{prop:dim_n_over_Kn-kernel-field} we get that $\dim_n q<2^n$ as well.
\end{proof}

From Theorem~\ref{prop:kahn-conj-known} we thus obtain the following result.

\begin{Cr}
\label{cr:inv_summand_quadrics_small_n}
Conjecture~\ref{conj:inv_summand_quadrics} is true for $n=1,2,3$.
\end{Cr}

\section{On cohomological invariants via invertible motives $L_\alpha$}\label{sec:coh_inv}

In this section we posit that the occurrence of the summands $L_\alpha$
in the $\Kn$-motive of a smooth projective $X$ (or, more generally, in the $\Kn$-specialization of a Chow motive $M$)
should be viewed as $\alpha$ being a cohomological invariant of $X$ (resp., of~$M$).

In Section~\ref{sec:coh_inv_history} we recall the history of the subject of cohomological invariants
emphasizing different approaches in the case of quadratic forms.
The connection between Morava motives and the splitting of cohomological invariants 
was formulated into a guiding principle in \cite{SechSem}, we prove  one direction of it 
in its strongest form in Section~\ref{sec:one_direction_guiding_principle}.
Finally, we formulate the conjecture that allows to associate cohomological invariants to some Chow motives
in Section~\ref{sec:morava_approach_coh_inv} and provide ample evidence for it 
coming from projective homogeneous varieties. 

\subsection{Different definitions of cohomological invariants}
\label{sec:coh_inv_history}
\subsubsection{Cohomological invariants of algebraic groups}
The notion of cohomological invariant goes back to Serre 
and was developed 
to organize the many known invariants systematically. 
We refer to~\cite{GMS} for the foundational material, as well as examples and historical information.

Let $G$ be an algebraic group, i.e.\ a smooth affine group scheme of finite type over $k$. Recall that in the present paper we always assume that $k$ has characteristic $0$. For a field extension $K/k$ we write $\HH^1(K,\,G)=\HH^1\big(\Gal K,\,G(\overline K)\big)$ and identify it with the set of isomorphism classes of $G_K$-torsors over $K$. {\sl A cohomological invariant of} $G$ of degree $i$ with coefficients in a discrete $\Gal k$-module $C$ is a natural transformation of functors
$$
\HH^1(-,\,G)\rightarrow\HH^i(-,\,C)
$$
on the category of field extensions of $k$. The set of such invariants is denoted by $\Inv GiC$. 
In other words, cohomological invariants permit to study $\HH^1(-,\,G)$ using ``more computable'' abelian cohomology groups.

A classical example  is given by the connecting homomorphism 
$
\HH^1(k,\mathrm{PGL}_n)\xrightarrow{\partial}\HH^2(k,\,\mathbb G_{\mathrm m})
$  
in the 
exact sequence of cohomology.
Identifying $\HH^1(k,\mathrm{PGL}_n)$ with the set of isomorphism classes of Severi--Brauer varieties of dimension $n-1$, and $\HH^2(k,\,\mathbb G_{\mathrm m})$ with the Brauer group, we recover the usual correspondence between the twisted forms of $\mathbb P^{n-1}$ and central simple algebras of degree $n$.

Cohomological invariants have played a prominent role in the research on algebraic groups over the past 30 years.
For instance, Garibaldi, Petrov and Geldhauser used them to detect rationality of parabolic subgroups for groups of type $\mathsf E_7$, solving a question of Springer~\cite{GPS}. 
They also appear in the work of Bayer-Fluckiger and Parimala on the Hasse Principle Conjecture~\cite{BP}, and in the theory of quadratic forms.

However, the full classification of cohomological invariants is far from complete.
All invariants can be computed for special classes of groups~\cite[Part~1]{GMS} and in small degrees~\cite[\S31]{KMRT},~\cite{BlinMer}~\cite[Part~2]{GMS}. 
Lourdeaux generalizes results of~\cite{BlinMer} on invariants of degree $2$ to non-perfect base field~\cite{Lour}. 
In a recent work~\cite{Totaro}, Totaro computes $\!\!\mod p$ cohomological invariants of various groups over a field of characteristic $p$. 
In a different direction, Pirisi replaces $G_K$-torsors over $K$ by $K$-points of an algebraic stack
and studies cohomological invariants of a stack in~\cite{Pirisi}.

\subsubsection{Cohomological invariants of quadratic forms}

Recall that for $G=\mathrm O_m$ the set $\HH^1(k,\,G)$ can be identified with the set of isometry classes of $m$-dimensional quadratic forms over $k$.

After the proof of Milnor conjectures~\cite{Milnor,Voe,OVV}, we have surjective homomorphisms 
$$
e_n\colon I^n(k)\rightarrow\HH^n(k,\,\ZZ/2)
$$ 
for $n\geq0$ satisfying $\Ker(e_n)=I^{n+1}(k)$. 
However, $e_n$ are {\sl not} the cohomological invariants of $\mathrm O_m$: they are only defined on a certain subset rather than the whole set $\HH^1(k,\,\mathrm O_m)$. Moreover, they cannot be continued to cohomological invariants of $\mathrm O_m$ for $n\geq3$.

In fact, $\bigoplus_{i\geq0}\Inv{\mathrm O_m}{i}{\ZZ/2}$ can be computed: it is a free $\HH^*(k,\,\mathbb Z/2)$-module with the basis 
given by Stiefel--Whitney classes $w_i\in\Inv{\mathrm O_n}{i}{\ZZ/2}$~\cite[Chapter~I, \S\,17]{GMS}. 
However,
if $-1$ is a square in $k$,
all Stiefel--Whitney classes of positive degree vanish on $I^3(k)$ \cite[Lm.~3.2]{Milnor}. 
Thus, for $n\geq3$ invariants $e_n$ cannot be continued to the whole set $\HH^1(k,\,\mathrm O_m)$.
For the case of the Arason invariant $e_3$ this was observed already in~\cite[p.~491]{Arason},
and obstructions to the extension of this invariant are discussed in~\cite{EKLV}. 
For a different type of invariants that do recover $e_n$, we refer the reader to~\cite{SmirVish}.

The example of $e_n$ shows that the definition of cohomological invariants, 
as stated in the previous section, might be too restrictive:
it is useful to consider a weaker version of it, 
a natural transformation from a subfunctor of $\HH^1(-,\,G)$. 
For example,  
$\omega_{n+1}$ of Kahn and Vishik (see Definition~\ref{def:dimn}) 
is defined for quadratic forms of $\dim_n<2^n$. 
In the case of algebras with orthogonal involutions 
an analogue of $e_3$ need not be defined for all torsors with trivial $e_1$ and $e_2$ 
~\cite{Queg},~\cite[\S\,3.4]{BPQ}, 
see also~\cite{Gar},
\cite[\S3.5]{Tig},~\cite{QuegTig}.
A different type of examples can be found in the works of Chernousov~\cite{Cher-e6e7,Cher-e8,Cher-f4}, where
the subset of those torsors $E\in\HH^1(k,\,G)$ for which the corresponding twisted from $\!\,_EG$ splits over quadratic field extension
is considered. 

\subsubsection{Voevodsky approach to cohomological invariants}

The work of Voevodsky~\cite{Voe_Z2,Voe_Zl}, Rost~\cite{Rost,Rost_special}, 
and others on the proof of the Bloch--Kato conjecture, in particular,  provided 
a bridge between cohomological invariants and motives. 
Given a ``symbol'' $\alpha$ in $\HH^n(k,\,\mu_p^{\otimes n-1})$ Rost
constructed a particularly nice splitting variety $X_\alpha$ 
(for $p = 2$ one can take norm quadrics, for arbitrary $p$ see~\cite{SusJou}).
The triviality of $\alpha_K$ over a field extension $K/k$ 
detects the existence of zero cycles on $(X_\alpha)_K$ of degree prime to $p$~\cite[Th.~6.19]{Voe_Zl}. 
In this sense, $\alpha$ is a cohomological invariant of $X_\alpha$.

These ideas were further developed  in~\cite{VishInt,IzhVish,Sem-coh-inv} 
permitting ``to go in the opposite direction'', namely,
to construct a cohomological invariant $\alpha$ starting from an appropriate variety $X$. 
For instance, starting from an outer binary summand $B$ of a Chow motive of an anisotropic quadric $Q$ of dimension $D=2^{n}-1$, Izhboldin--Vishik 
construct an element $\alpha\in\mathrm{K}^{\mathrm M}_{n+1}(k)/2$ such that $\alpha_K=0$ if and only if $B_K$ is split for all $K/k$~\cite[Sec.~6]{IzhVish}, cf. Section~\ref{sec:prelim:izh-vish}. 
We  refer the reader to~\cite[Th.~6.1]{Sem-coh-inv} for the case of odd $p$ and generalization of this technique to other projective homogeneous varieties.

\subsection{Vanishing of cohomological invariants: one direction of the guiding principle}
\label{sec:one_direction_guiding_principle}

The first connection between $\Kn$-motives and cohomological invariants was observed 
in~\cite{SechSem} by Geldhauser and the second author. The examples of projective homogeneous varieties 
that were studied there led to the following ``guiding principle'' [loc.\ cit., 1.3]:

``Let $X$ be a projective homogeneous variety, let $p$ be a prime
number and let $\Kn$ denote the corresponding Morava K-theory.
Then vanishing of cohomological invariants of $X$ with $p$-torsion coefficients in degrees no
greater than $n+1$ should correspond to the splitting of the $\Kn$-motive of $X$.''

To turn this principle into a conjecture, one needs to clarify what is meant by a cohomological invariant.
However, this is not so straightforward, as we already mentioned above, since some cohomological invariants 
turn out to be defined only when others vanish.
Nevertheless, one direction of the guiding principle --- namely, that if the $\Kn$-motive is split,
then there are no invariants --- can be proved for, arguably, the weakest possible definition of the cohomological invariant.

\begin{Def}\label{def:weak_coh_inf}
Let $C$ be a discrete $\Gal{k}$-module, and
$X$ be a $k$-smooth projective geometrically cellular variety. 

Then $\alpha \in \mathrm{H}^n(k,\,C)$ is a {\sl weak cohomological invariant} of $X$,
if for every field extension $K/k$ the condition that
the variety $X_K$ is cellular 
implies that $\alpha_K$ is zero.
\end{Def}
\begin{Rk}
Let $G$ be a split semi-simple group and $X$ a twisted form of a projective homogeneous variety for $G$, defined by $E\in\HH^1(k,\,G)$.

We warn the reader that 
for a cohomological invariant $\phi\in\Inv{G}{n}{C}$, an element $\alpha:=\phi(E)$ need not be a weak cohomological invariant of $X$ in the sense of Definition~\ref{def:weak_coh_inf}.
Indeed, 
$\alpha_K$ does {\sl not} necessarily vanish whenever $X_K$ splits --
for example, if $\phi$ is a 
constant cohomological invariant \cite[Ch.~I, Def.~4.4]{GMS}.
However, every invariant is canonically a direct sum of a constant and a normalized cohomological invariant~\textup{[loc.\ cit., Ch.~I, 4.5]}. 
The latter yield weak cohomological invariants in the situation above.    
\end{Rk}

\begin{Th}[One direction of the guiding principle]
\label{prop:guiding_principle}
\label{th:guiding_principle}
Let $X$ be a projective homogeneous variety for a semi-simple algebraic group $G$ of inner type.
Let $\Kn$ be the algebraic Morava K-theory at the prime $p$.

If $\Mot{\Kn}{X}$ is split, 
then there exists a field extension $K/k$
such that  
$X_K$ 
is split,
and such that the canonical morphism
$$ \HH^m(k,\,\mu_p^{\otimes s}) \rarr \HH^m(K,\,\mu_p^{\otimes s}) $$
is injective for 
all $m$ 
such that $2\le m\le n+1$,
and any $s$.

Moreover, if $\mu_{p^r}\subset k$
for some $r\in\NN$, 
then the morphism
$$ \HH^m(k,\,\ZZ/p^r) \rarr \HH^m(K,\,\ZZ/p^r) $$
is injective for 
all $m$ such that $2\le m\le n+1$. 

In particular, 
$X$ has no weak cohomological invariants in the 
cohomology 
groups listed above.
\end{Th}
\begin{Rk}
The key tool in the proof of this theorem is the result of Karpenko:
for $G$ of inner type 
the Chow motive of $X$ decomposes into upper motives~\cite[Th.~3.5]{Karp-upper}, cf.~\cite{Karp-upper-outer}. 
Recently there has been significant progress towards achieving a similar result in a more general situation,
see~\cite{A-upper},~\cite{Artin-shapes}. Whenever this can be done, the proof below should apply as well.
\end{Rk}

\begin{proof}
We denote $\Ch=\CH/p$. 
By~\cite[Th.~3.5]{Karp-upper}
 every indecomposable direct summand $M$ of $\Mot{\Ch}{X}$ is a Tate twist of  
the upper Chow motive of some projective homogeneous variety $Y$. 
We construct field $K$ by induction taking composites 
of the fields of functions of these varieties~$Y$.

Assume that $\Mot{\Ch}{X}$ is not split,
let $M$ be a non-split indecomposable summand of it, 
and 
$Y$  
the corresponding projective homogeneous variety 
as above.
By Lemma~\ref{lm:upper_Kn_split} below we get  
 that $l_0$ is $\Kn$-rational on $Y$. 
 By Corollary~\ref{cr:injectivity_galois_cohomology}
 we get the injectivity of the base change morphisms from $k$ to $k(Y)$
 on Galois cohomology groups of interest.
 
 However, over $k(Y)$ the Chow motive $M$ contains a Tate summand,
 and therefore the number of Tate summands of $\Mot{\Ch}{X_{k(Y)}}$ has increased.
 We continue the procedure by choosing a new indecomposable summand $M$ of $\Mot{\Ch}{X_{k(Y)}}$,
 and since the number of Tate summands of $\Mot{\Ch}{X_F}$, for any $F$,
 is bounded by that of $\Mot{\Ch}{\overline{X}}$,
this process ends after finitely many steps with the field $L$ 
over which the $\Ch$-motive of $X$ is split. 
Finally, 
there exists a finite field extension $K/L$ of degree prime to $p$,
over which $X_L$ is split.
The field $K$ satisfies the assumptions of the theorem.
 \end{proof}

\begin{Lm}\label{lm:upper_Kn_split}
Let $M$ be the upper $\Ch$-motive of a projective homogeneous variety $X$.
Assume that $M_{\Kn}$ is split.
Then $l_0$ is rational in $\Knt{\overline X}$.
\end{Lm}
\begin{proof}
By \cite{VishYag} (see Section~\ref{sec:prelim_vishik_yagita}) 
we can lift the projector $\pi^{\Ch}$ defining the upper motive of $X$
to a projector $\pi \in (\Omega^{\dim X}/p)(X\times X)$.
Let $\overline{\pi}$ denote its image in $(\Omega^{\dim X}/p)(\overline{X\times X})$.

First, note that $\overline{\pi}$ acts trivially on $l_0$, i.e.\ $\overline{\pi}\circ l_0 = l_0$.
Indeed, $(\Omega^{\dim X}/p)\cong \Ch^{\dim X}(X)$ by \cite[Th.~1.2.19]{LevMor},
and $\overline\pi^{\Ch}$ acts trivially on $l_0$ by definition of the upper motive.

Second, consider the image $\pi^{\Kn}$ of $\pi$ in $\Kn$.
By the assumption $\pi^{\Kn}$ defines a split summand of $\Mot{\Kn}{X}$,
and therefore $\pi^{\Kn}\circ \oKn(X)$ is isomorphic to $\pi^{\Kn} \circ \Kn(\overline{X})$
(one is a subgroup of the other and they  have the same rank). 
The latter contains $l_0$ by our observation above,
i.e.\ $l_0\in\oKn(X)$.
\end{proof}

\subsection{Morava motivic approach to cohomological invariants of motives}
\label{sec:morava_approach_coh_inv}

To address the lack of sufficiently many cohomological invariants
in the sense of~\cite{GMS}, 
we propose an alternative notion of cohomological invariant 
associated to certain motives. 
We show that various classically studied invariants can be incorporated within this framework
and conjecture that every direct summand of
the Chow motive of a projective homogeneous variety
admits at least one such invariant.
We investigate this claim in the case of quadrics.

The following definition of a cohomological invariant 
is based on the invertible $\Kn$-motives $L$ that are associated 
to elements in Galois cohomology. 
The main example of this paper is $L_\alpha$ constructed in Section~\ref{sec:construction}.
However, in the future work~\cite{SechInv} the second author will show how to define
 invertible $\Kn$-motives $L_\alpha$ for $\alpha \in \HH^{n+1}(k,\,\mu_{p^r}^{\otimes n})$ for any prime $p$ and any $r\ge 1$,
where $\Kn$ is the $n$-th algebraic Morava K-theory at the prime $p$ (with $v_n=1$).
We thus make the definition assuming the existence of these motives,
although in non-speculative parts of this section only $L_\alpha$ for $\alpha\in \HH^{n+1}(k,\,\ZZ/2)$ appear.

\begin{Def}\label{def:inv_mot_coh_inv}
Let $L_\alpha$ be an invertible $\Kn$-motive over $k$ for some $\alpha$ in the Galois cohomology, see above.

For a Chow motive $M$, we will say that $\alpha$ is a cohomological invariant of $M$,
if $L_\alpha\sh i$ is a direct summand of the $\Kn$-specialization of $M$ for some $i\in\ZZ$.  
\end{Def}

\begin{Ex}
\label{ex:quad-coh-inv}
For a quadratic form $q$ of 
$\dim_n(q) <2^n$, 
and $\alpha:=\omega_{n+1}(q)$, 
we have seen in 
Proposition~\ref{prop:Kn-motive-q-small-dim_n} 
that a Tate twist of $L_\alpha$ appears as a direct summand of $\MKn Q$. 
Therefore, $\omega_{n+1}(q)$ is a cohomological invariant of $\Mot{\CH}{Q}$ in the sense of Definition~\ref{def:inv_mot_coh_inv}.

In particular, for $q\in I^{n+1}(k)$, we obtain that $e_{n+1}(q)$ is a cohomological invariant of $\Mot{\CH}{Q}$.
\end{Ex}

\subsubsection{Invertible $\mathrm{K}_0$-motives and Tits algebras}
\label{sec:tits_k0_motives}

The case $n=1$, i.e.\ essentially $\KK$-motives, has been studied in~\cite{Panin}.

With a semi-simple algebraic group $G$ over $k$ of inner type, one associates
Tits algebras~\cite{Tits},
which are the endomorphism rings of irreducible representations of $G$~\cite[\S~27]{KMRT}.
They are central simple algebras over $k$ and their classes in $\Br(k)$ provide cohomological invariants of $G$ degree $2$.

If $X$ is the variety of Borel subgroups of $G$, it follows from the results of Panin~\cite{Panin} that the $\mathrm K_0$-motive 
of $X$ is isomorphic to a direct sum of invertible motives $L_w$ 
indexed by the elements $w$ of the the Weyl group $W$ of $G$. Moreover, $(L_w)_{\,\overline k}$ are Tate motives, and for each $w\in W$ one associates a Tits algebra $A_w$ such that the natural restriction $\mathrm K_0(L_w)\rightarrow\mathrm K_0((L_w)_{\,\overline k})\cong\ZZ$ is identified with the inclusion of $\mathrm{ind}(A_w)\ZZ$~\cite[Sec.~3.4]{SechSem}.

These motives $L_w$ induce invertible $\K1$-motives (where $v_1=1$ in $\K1$)
under identification $\K1\cong\mathrm K_0/p$, allowing to consider Tits algebras as cohomological invariants in the sense of Definition~\ref{def:inv_mot_coh_inv}. 
In particular, if $[A_w]\in\!\,_2\Br(k)$ one identifies $L_w$ with $L_{[A_w]}$ of Section~\ref{sec:construction} by Theorem~\ref{th:detect_L_alpha}.

\subsubsection{Cohomological invariants of projective homogeneous varieties}
\label{sec:coh_inv_phv}

To construct examples of Chow motives of projective homogeneous varieties 
that give rise to cohomological invariants in the sense of Definition~\ref{def:inv_mot_coh_inv},
we have two basic approaches at our disposal. 
Let $M$ be a Chow motive that is a direct summand of the Chow motive of a projective homogeneous variety.

\begin{enumerate}
    \item (Reduce to the Rost motives)
Let $K/k$ be a field extension such that
    $K/k$ is $\oKn$-universally bijective 
    and 
    $M_K$ decomposes into a direct sum of Tate twists of the Rost motives $(R_\alpha)_K$ for some $\alpha \in \HH^{n+1}(k,\,\ZZ/2)$. 
    Then $M_{\Kn}$ is isomorphic to a direct sum of Tate twists of $\un \oplus L_\alpha$ by Propositions~\ref{prop:prelim_L_alpha_Rost}, \ref{prop:reflects_MD}. 
    \item (Kill $\alpha\in\HH^{n+1}(k,\,\ZZ/2)$ to split the motive)
    Let $\alpha \in \HH^{n+1}(k,\,\ZZ/2)$ be an element such that 
    $(M_{\Kn})_{k(\alpha)}$ is split. Then by Theorem~\ref{th:detect_L_alpha} the motive $M_{\Kn}$ 
    decomposes into a direct sum of the Tate motives and Tate twists of the motives $L_\alpha$.    
\end{enumerate}

These approaches allow us to produce a series of examples of projective homogeneous varieties,
for which the $\Kn$-motivic decomposition can be computed
and the cohomological invariants in the sense of Serre appear via the motives $L_\alpha$. 
We collect them in this section.

Recall that for a simply connected simple $G$, Rost constructed a cohomological invariant $\rost$ with values in $\HH^3(k,\,\QQ/\ZZ(2))$
that generates the group $\Inv G3{\QQ/\ZZ(2)}$ of invariants of degree $3$~\cite{GMS, KMRT}. 
We denote by $\rost_p$ the $p$-primary component of the Rost invariant, 
by definition it has order $p^r$ for some $r$,
and for this $r$ we can identify $\rost_p$ with an element in $\HH^3(k,\,\mu_{p^r}^{\otimes 2})$. 

\begin{Ex}
\label{ex:rost-coh-inv}
Let $G$ be a split simply connected simple group {\sl not} of type $\mathsf E_7$, $\mathsf E_8$, and $E\in\HH^1(k,\,G)$. 
Denote by $X_0$ the variety of Borel subgroups of $G$, and by $X$ the twisted form of $X_0$ defined by $E$.  
Then 
$\rost_2(E)$
is a cohomological invariant of $\Mot{\CH}{X}$ in the sense of Definition~\ref{def:inv_mot_coh_inv}.
To this end, we apply our approach (1) below in order to compute the $\K{2}$-motivic decomposition of $X$:
it is a direct sum of Tate motives and Tate twists of $L_{\rost_2(E)}$.

For exceptional $G$ as above, 
the Rost invarinat $\rost_2(E)$ is a symbol in $\HH^3(k,\,\ZZ/2)$, 
    and the $\Ch$-motive of $X$ decomposes as a sum of Tate twists of the Rost motive corresponding to this symbol~\cite[Lm.~7.5]{PSZ}.

For spin groups, 
the Rost invariant 
    recovers the Arason invariant $e_3$ and its generalization, the invariant $\omega_3$.  
    More precisely, 
    for any $E\in\HH^1(k,\,\mathrm{Spin}_m)$, $m\geq5$, 
    let $q^E$ denote the quadratic form identified with the image of $E$ in $\HH^1(k,\,\mathrm{O}_m)$, then $\dim_2(q^E)\leq1$ and $\omega_3(q^E)$ 
    coincides with the Rost invariant of $E$~\cite[(31.41)]{KMRT}. 
   The field of definition of the $\K2$-kernel form $K/k$ of $q^E$ is $\overline{\K2}$-universally bijective, and $(q^E)_K$ is either a Pfister quadric, or a maximal Pfister neighbour.
   Then $\Mot{\Ch}{X_K}$ decomposes as a sum of Rost motives corresponding to the Rost invariant of $E_K$.
\end{Ex}

\begin{Rk}
\label{rk:wohlschlager}
For the groups in Example~\ref{ex:rost-coh-inv}, $\rost_2$ always takes values in $\HH^3(k,\,\ZZ/2)$. 
However, for the groups of type $\mathsf E_7$, $\mathsf E_8$,
in general $\rost_2$ takes values in $\HH^3(k,\,\ZZ/4)$. In this situation, the decomposition of $\Mot{\K2}{X}$
into invertible summands of order at most $4$
will be constructed in the forthcoming work~\cite{Alois} (in particular, their construction cannot be reduced to the Rost motives). 
\end{Rk}

One can also apply approach (2) in Example~\ref{ex:rost-coh-inv}, in a similar way as in the following example.

\begin{Ex}
Let $G$ be a split simply connected simple group of type $\mathsf E_7$, $\mathsf E_8$, and $E\in\HH^1(k,\,G)$ 
such that $\rost_2(E)$ takes values in $\HH^3(k,\,\ZZ/2)$.
Denote by $X_0$ the variety of Borel subgroups of $G$, and by $X$ the twisted form of $X_0$ defined by $E$. 
Then 
$\rost_2(E)$
is a cohomological invariant of $X$ in the sense of Definition~\ref{def:inv_mot_coh_inv}.

Indeed, in~\cite[Th.~9.1]{SechSem} 
it is shown
that the $\K2$-motive of $X$ is split iff $\rost_2(E)=0$. 
In particular, $\Mot{\K2}{X}$ becomes split over $k\big(\rost_2(E)\big)$ and we can apply our approach (2). 
\end{Ex}

\begin{Rk}
If $E_1$, $E_2$ are $G$-torsors, and $X_1$, $X_2$ the corresponding twisted forms of
$X_0$, Petrov and the first author proved in~\cite[Th.~6.8]{LavPet} that the $\K2$-motives of $X_1$ and $X_2$ are isomorphic iff $\rost_p(E_1)$ and $\rost_p(E_2)$ generate the same subgroup in $\HH^3(k,\,\QQ/\ZZ(2))$. 
It is, however, not immediate from this result that one can associate an invertible $\K2$-motive to $\rost_p(E)$.
\end{Rk}

\begin{Ex}
For a split simple group $G$ of type $\mathsf F_4$ Serre constructs two cohomological invariants with
 $\ZZ/2$-coefficients with values in symbols:
 $f_3$ of degree $3$, the $2$-component of the Rost invariant of $G$,
 and $f_5$ of degree $5$, which generate $\oplus_{i\geq0}\Inv{G}{i}{\ZZ/2}$ as an algebra over $\HH^*(k,\,\ZZ/2)$~\cite[Ch.~I, \S22]{GMS}. 

Let $P_4$ be the maximal parabolic subgroup of $G$ corresponding to $4$-th simple root in Bourbaki numbering, and $X_0=G/P_4$, cf.~\cite[Th.~2.4]{Macdonald}. Let $E\in\HH^1(k,\,G)$ and $X$ be a twisted form of $X_0$ defined by $E$. 
It follows from~\textup{[loc.\ cit.]} that the
$\Ch$-motive of $X$ for $p=2$ decomposes as a sum of the Rost motive corresponding to $f_5(E)$, and several Tate twists of the Rost motives corresponding to $f_3(E)$.

Indeed, recall that there exists $E'\in\HH^1(k,\,G)$ that becomes isomorphic to $E$ after an odd degree field extension $K$, and such that the Albert algebra corresponding to $E'$ is reduced~\cite{PetRac}. 
The absolute Galois group of $k$
acts trivially on the Chow motives of the corresponding projective homogeneous varieties $\overline{X}$, $\overline{X'}$,
therefore a $K$-rational isomorphism of their $\Ch$-motives yields a $k$-rational isomorphism by a transfer argument, cf.~\cite[Lm.~1.10, Lm.~1.12]{PSZ}.

Hence we can assume 
that $E$ corresponds to a reduced Albert algebra, and in this case we can apply~\cite[Th.~3.12]{Macdonald} to decompose $\MCh X$. 
Observe that $F_3^3$ of \textup{[loc.\ cit., Th.~3.12, Th.~3.2]} is a ``higher form'' $F_\phi$ of the motive of the quadric $\langle b_1,\,b_2,\,b_3\rangle$ in the sense of~\cite[Sec.~6.1]{Vish-quad}; by~\textup{[loc.\ cit., Th.~4.15]}, $F_3^3$ is isomorphic to the Rost motive corresponding to the general Pfister form $\phi\otimes\langle b_1,\,b_2,\,b_3,\,b_1b_2b_3\rangle$. The class of the latter form in $\HH^5(k,\,\ZZ/2)$ is by definition $f_5(E)$ ~\cite[Ch.~I, Th.~22.4]{GMS}. 

This allows us  
to conclude using approach (1) that 
$f_3(E)$ and $f_5(E)$
are cohomological invariants of $\Mot{\CH}{X}$ in the sense of Definition~\ref{def:inv_mot_coh_inv}.
\end{Ex}

\begin{Ex}
For a split simple group $G$ of type $\mathsf E_8$ 
and $K/k$ let us denote by $\HH^1(K,\,G)_0$ the set of $G_K$-torsors with trivial $2$-component of the Rost invariant.
Geldhauser constructs in~\cite{Sem-coh-inv} an invariant 
$$
u\colon \HH^1(-,\,G)_0\rightarrow \HH^5(-,\,\ZZ/2),
$$
which was used there to answer a question of Serre; for its further applications see~\cite{SemE8}. 

Let $X_0$ denote the variety of Borel subgroups of $G$, let $E\in\HH^1(k,\,G)_0$ and $X$ be the twisted form of $X_0$ defined by $E$. 
In~\cite[Th.~9.1]{SechSem} 
it is shown that the $\K4$-motive of $X$ is split iff $u(E)=0$. 
This allows us to conclude using approach (2) that
$u(E)$
is a cohomological invariant of $\Mot{\CH}{X}$ in the sense of Definition~\ref{def:inv_mot_coh_inv}.

In fact, approach (1) can be alternatively applied in this case as well.
\end{Ex}

\subsubsection{Conjecture}

We now turn to the explanation of how one can expect to use $\Kn$-motives to define cohomological invariants.
Let $p$ be a prime, let $\Kn$ be the algebraic Morava K-theory at the prime $p$ (with $v_n=1$).
For the formulation of the conjectural picture 
we assume the existence of $\Kn$-motives $L_\alpha$ for $\alpha \in \HH^{n+1}(k,\,\mu_{p^r}^{\otimes n})$ for all $r\ge 1$
mentioned above, which will be constructed in~\cite{SechInv}. 
Let also $\mathrm{K}(0)$ denote $\CH\ot \QQ$, sometimes referred to as the zeroth Morava K-theory.

\begin{Conj}\label{conj:coh_inv_morava}
Let $n\ge 1$.
Let $M$ be a Chow motive such that $M_{\overline{k}}$ is split.

Assume that $M_{\K{m}}$ splits for every $m$ satisfying $0\le m< n$.

Then $M_{\Kn}$ decomposes into a direct sum of motives $L_{\alpha}\sh j$ for some $j\ge 0$, $r\in\NN$ and $\alpha \in \HH^{n+1}(k, \mu_{p^r}^{\otimes n})$.
\end{Conj}

In other words, we conjecture that to every Chow motive as above it is possible 
to associate a cohomological invariant. 
The evidence for this conjecture 
comes from the study of projective homogeneous varieties:
case $n=1$ is essentially due to Panin (see Section~\ref{sec:tits_k0_motives}),
and many examples of projective homogeneous varieties
where the Chow motivic decomposition and the $\Kn$-motivic decomposition can be computed (see Section~\ref{sec:coh_inv_phv})
also confirm the conjecture. In fact, based on this conjecture,
we also predicted the $\K2$-decomposition of Remark~\ref{rk:wohlschlager} that was obtained by Wohlschlager~\cite{Alois}.

Nevertheless, even in the case $p=2$ and when $M$ is a direct summand of the Chow motive of a quadric
the conjecture remains open.
We show in Proposition~\ref{prop:from_Kn+1_to_Kn} that
for such $M$ there exists $n$ such that $M_{\Kn}$ contains a non-trivial invertible summand.
However, to identify it with a Tate twist of $L_\alpha$ 
we appeal to Kahn--Rost--Sujatha Conjecture~\ref{conj:KRS} (Corollary~\ref{cr:KRS_implies_main_conjecture_quadrics}).

Finally, we highlight a potentially interesting connection between motivic decompositions and invariants: 
if $X$ is a projective homogeneous variety of inner type (in particular, $\Mot{\K0}{X}$ is split)
and $\Mot{\CH}{X}$ admits $N$ classes of isomorphisms of its direct summands,
then one can associate to it $N$ possibly different cohomological invariants,
using the above conjecture. This will be investigated elsewhere.

\subsection{Cohomological invariants of the direct summands of quadrics}
\label{sec:coh_inv_direct_summands_quadrics}

Conjecture~\ref{conj:coh_inv_morava} implies that 
for every non-split direct summand $M$ of $\Mot{\Ch}{Q}$ such that $M_{\K0}$ is split,
there should exist an integer $n$ such that $M_{\Kn}$ contains a non-trivial invertible summand $L$.
We show that it is true.

\begin{Lm}
\label{lm:not-a-sum-of-invertibles}
Let $n\ge 2$, let $Q$ be a quadric over $k$, 
and let $M$ be an indecomposable direct summand of $\Mot{\Kn}{Q}$ of rank greater than $1$.

Assume that $Q'$ is a quadric of dimension $2^{n}-1$.

Then the motive $M_{k(Q')}$ does not decompose into invertible summands. 
\end{Lm}
\begin{proof}
Suppose, for contradiction, that $M_{k(Q')}$ decomposes into invertible summands.
In Proposition~\ref{prop:invertible_summ_quadrics_iso} we have seen that all invertible summands of the $\Kn$-kernel of a quadric
are isomorphic, i.e.\ $M_{k(Q')}\cong \oplus L(i)$ for some invertible $L$. 
Therefore the motive $(M\otimes M^\vee)_{k(Q')}$ is split. 

By Proposition~\ref{prop:splitting_over_quadrics_with_dim_2^n-1} we get that $M\otimes M^\vee$ decomposes into invertible summands
that trivialize over $k(Q')$. If there exists a non-trivial invertible summand that trivializes over $k(Q')$,
then by Proposition~\ref{prop:split_invertible_2n-1} the upper Chow motive of $Q'$ has to be binary. 
But in this case by 
Corollary~\ref{cr:pfister_morava_mdt} 
the motive $M$ does not decompose over $k(Q')$. 
Thus, we have to assume that $M\otimes M^\vee$ is split.

In this case $\Hom(M,M)$ does not change over any field extension
as it is isomorphic to $\Hom(\un, M\otimes M^\vee)$.
Therefore $M$ decomposes into (isomorphic) invertible summands,
contradicting our assumption.
\end{proof}

Let $n\ge 2$;  
for every indecomposable direct summand $M$ in $\Mot{\Kn}{Q}$, $M_{K_n}$ is a summand in the $\Kn$-specialization of
a direct summand $\widetilde{M}$ of $\Mot{\Ch}{Q_{K_n}}$ by Theorem~\ref{th:prestableMDT},
where $K_n$ is the field of definition of the $\Kn$-kernel form of $Q$.
Recall that $K_{n-1}$ is a field extension of $K_n$ (since for smaller $m$ 
the field of definition of the $\K{m}$-kernel form is further along the generic splitting tower of $Q$).
The $\K{n-1}$-specialization of 
$\widetilde{M}_{K_{n-1}}$ decomposes as a sum of Tate motives 
and a (possibly decomposable) direct summand of the $\K{n-1}$-kernel of $\Mot{\K{n-1}}{Q}$.
We call this latter summand the $\K{n-1}$-specialization of $M$ and denote it by $M_{\K{n-1}}$.
It is possible to describe explicitly how this specialization works on the level of rational projectors,
but we will not need it. 

\begin{Prop}\label{prop:from_Kn+1_to_Kn}
Let $n\ge 2$, let $Q$ be a quadric over $k$, 
and let $M$ be an indecomposable direct summand of $\Mot{\Kn}{Q}$ of rank greater than $1$.
Then $M_{\K{n-1}}$ is not split.
\end{Prop}
\begin{proof}
First we show that for any field extension $K$ 
in the generic splitting tower of $Q$ that sits in-between $k$ and $K_{n-1}$ 
the motive $M_{K}$ does not decompose into invertible summands. Indeed, 
if $Y$ is a quadric over a field $F$ with $\dim Y\ge 2^n-1$, $Q'$ is its subquadric of dimension $2^n-1$,
then the extension $F(Y\times Q')/F(Q')$ is purely transcendental, and
hence the decomposition of the base change of a motive to $F(Q')$ is finer than the one of the base change to $F(Y)$. In particular, a non-invertible indecomposable direct motivic summand of a quadric cannot decompose into a sum of invertible motives by Lemma~\ref{lm:not-a-sum-of-invertibles}.

Now consider a non-invertible indecomposable direct summand $N$ of $M_{K_{n-1}}$. 
The motive $N$ is a direct summand of the $\Kn$-motive of a quadric $\widetilde{Q}=(Q_{K_{n-1}})_{\an}$ of dimension less than $2^n-1$,
and hence it is a specialization of the indecomposable direct summand of $\mathrm{M}_{\Ch}\big({\widetilde{Q}}\big)$ (of the same rank; cf. Section~\ref{sec:unstable}).
However, a non-split summand of $\mathrm{M}_{\Ch}\big({\widetilde{Q}}\big)$ specializes to a non-split summand of $\Mot{\K{n-1}}{Q'}$,
i.e.\ $M_{ \K{n-1} }$ is non-split.
\end{proof}

\begin{Cr}
\label{cr:KRS_implies_main_conjecture_quadrics}
Assume that Kahn--Rost--Sujatha Conjecture~\ref{conj:KRS} holds for some $n$.
Then Conjecture~\ref{conj:coh_inv_morava} holds for this $n$
for every direct summand $M$ of the motive of a quadric,
i.e. if $M_{\K{m}}$ is split for $m\le n$,
then $M_{\K{n}}$ decomposes into a direct sum of the motives $L_\alpha(j)$, for some $\alpha \in \HH^{n+1}(k,\ZZ/2)$, $j\in \ZZ$.
\end{Cr}
\begin{proof}
    If $n=1$, then the decomposition of $M_{\K{1}}$ is known, see Section~\ref{sec:K0-motive-quadric}. 
    In particular, if $M_{\K0}$ is split,  $M_{\K1}$ cannot have non-invertible summands,
    and hence it is a direct sum of Tate motives and of the motives $L_{e_2(q)}$ 
    (note that the Tate motive can be seen as the motive $L_\alpha$ with $\alpha=0$).

    For $n\ge 2$ we get by Proposition~\ref{prop:from_Kn+1_to_Kn} 
    that all indecomposable summands in $M_{\Kn}$ are invertible. 
    By Proposition~\ref{prop:KRS_implies_invertible_summands_quadrics}
    Conjecture~\ref{conj:inv_summand_descent} holds, 
    and therefore invertible summands in $M_{\Kn}$ are isomorphic to Tate twists of $L_\alpha$
    for some $\alpha \in \HH^{n+1}(k,\,\ZZ/2)$.
\end{proof}

\appendix
\section{Geometric Rost Nilpotence Property.}\label{app:a}

We extend the result of \cite[Lm.~3.2]{VishZai} to arbitrary oriented cohomology theories
in characteristic 0. 
The strategy of the proof is the same as in loc.\ cit., except for one step
where refined pullbacks for Chow groups were used [loc.\ cit., Proof of Prop.~6.3].
Although refined pullbacks exist for algebraic cobordism,
 and hence for free theories as well by \cite[Sec.~6.6]{LevMor},
their construction is somewhat complicated and does not automatically extend to arbitrary oriented theories.
We provide a different proof for an analogue of \cite[Prop.~6.3]{VishZai}
that works in characteristic 0 for all oriented theories. 
In order for the article to be self-contained we also repeat the argument of \cite{VishZai}.

\begin{PropA}[{see \cite[Lm.~3.2]{VishZai} for the case of Chow groups}]
\label{prop:app_geometric_RNP}
Let $X\in \Sm_k$ be irreducible, let $j\colon U\hookrightarrow X$ be a non-empty open subset.
Let $A$ be an oriented cohomology theory on $\Sm_k$.

Then for every motive $M\in \CM_A(X)$ 
the restriction functor $j^*\colon\CM_A(X)\rarr \CM_A(U)$
induces a surjective homomorphism of rings
$$ \End(M) \rarr \End(M_U) $$
with nilpotent kernel.
\end{PropA}
\begin{proof}
We may assume that $M = M_A(Y\rarr X)$ for a smooth projective morphism $Y\rarr X$.

Let $Z$ be the (reduced) closed complement to $U$ in $X$,
then by resolution of singularities of Hironaka there exists a birational morphism $\pi\colon \tilde{X}\rarr X$
such that the total transform of $Z$ is an snc-divisor $\tilde{D}$ in $\tilde{X}$
and $\pi$ is an isomorphism away from $Z$.
Denote $\tilde{j}\colon U\hookrightarrow \tilde{X}$
the open complement of $\tilde{D}$.

Let $\tilde{Y}:=Y\times_X \tilde{X}$,
$Y_U:=Y\times_X U\cong \tilde{Y}\times_{\tilde{X}} U=:\tilde{Y}_U$.
Then we have the following commutative diagram of rings of correspondences:
 
\begin{center}
 \begin{tikzcd}
	\mathrm{End}(M(\tilde{Y}\rarr \tilde{X})) \arrow[r, "\tilde{j}^*"] & \mathrm{End}(M(\tilde{Y}_U\rarr U)) \\
	\mathrm{End}(M(Y\rarr X)) \arrow[u, "p^*"] \arrow[r, "j^*"] & \mathrm{End}(M(Y_U\rarr U))  \arrow[u, "\cong"]\\
 \end{tikzcd}
\end{center}

where $\mathrm{End}(M(W\rarr B)) = A(W\times_B W)$.
Note that $p\colon\widetilde{Y}\times_{\widetilde{X}} \widetilde{Y} \rarr Y\times_X Y$ is a birational morphism,
and hence $p^*$ is (split) injective:
it follows from \cite[Th.~4.4.7]{LevMor} that $p_*(1)-1$ is nilpotent, i.e.\ $p_*(1)$ is invertible,
and $p_*\circ p^*(-) = p_*(1)\times -$ by the projection formula.
Therefore the nilpotence of $\mathrm{Ker\ }j^*$ follows 
from the nilpotence of $\mathrm{Ker\ }\tilde{j}^*$,

Let $\alpha \in \mathrm{End}(M(\tilde{Y}\rarr \tilde{X}))$ vanish over $U$,
we will show that $\alpha^{\circ d} =0$ for $d=\dim X+1$.
This composition can be computed 
as the push-forward of $\prod_{i=1}^{d} p_{i,i+1}^*(\alpha) \in A(Y^{\times_X (d+1)})$
along $p_{1,d+1}$, where $p_{ij}$ denote canonical projections onto the $i$-th and the $j$-th components of $Y^{\times_X(d+1)}$.
All of the elements $p_{i,i+1}^*(\alpha)$ vanish over the preimage of $U$ in $Y^{\times_X (d+1)}$,
and thus the claim follows from Lemma~A.\ref{lmA} below. 
\end{proof}

\begin{LmA}[ {cf.~\cite[Prop.~6.1]{VishZai}} ]
\label{lmA}
Let $D=\cup_{t\in I} D_t$ be an snc-divisor in $X\in \Sm_k$ with the open complement $U$,
let $f\colon W\rarr X$ be a smooth morphism.

Let $a_j \in A(W)$, $j\in J$, such that $a_j$ vanishes when restricted to $W_U:=W\times_X U$.

Then $\prod_{j\in J} a_j$ is zero if $|J|\ge \dim X +1$. 
\end{LmA}
\begin{proof}
Let $i_t\colon W_t\hookrightarrow W$ denote the pullback of $D_t$ along $f$,
i.e.\ $\cup_{t\in I} W_t$ is an snc-divisor on $W$. 
Let also $W_{\tilde{I}}$ denote $\cap_{t\in \tilde{I}} W_t$ for $\tilde{I}\subset I$,
i.e.\ $W_{\tilde{I}}$ is the pullback of $\cap_{t\in \tilde{I}} D_t$ along $f$.
Since $a_j$ vanishes on the complement of $W$,
it can be written as a sum $\sum_{t\in I} (i_t)_*(b_t)$ where $b_t\in A(W_t)$.
Thus, it suffices to show that the product of $\dim X+1$ elements
of the form $(j_t)_*(b)$ for some $t\in I$ is zero in $A(W)$.

Let $N_t$ denote the normal line bundle of $D_t$ in $X$, and let $\lambda_t:=c_1^A(N_t)$ be its first Chern class.
Then $f^*N_t$ is isomorphic to the normal line bundle of $W_t$ in $W$, and $f^* \lambda_t = c_1^A(f^* N_t)$.

The  self-intersection formula (cf.~\cite[Lm.~5.1.11]{LevMor}) together with the projection formula yield that 
$$(j_t)_*(b_1) \cdot (j_t)_* (b_2) = (j_t)_* (b_1 \cdot b_2) \cdot f^*\lambda_t. $$

For $t_1, t_2 \in I, t_1\neq t_2$, the product of elements supported on the corresponding divisors
can be computed via the transversal base change formula:
$$ (j_{t_1})_*(b_1) \cdot (j_{t_2})_* (b_2) = (j_{\{t_1, t_2\}})_* ((b_1)|_{W_{\{t_1, t_2\}}} \cdot (b_2)|_{W_{\{t_1, t_2\}}})  $$
where $j_{\{t_1,t_2\}}\colon W_{\{t_1, t_2\}} \rarr W$ is the canonical closed embedding.

Combining these two formulas together,
we get that the product of $\dim X+1$ elements of  the form $(j_t)_*(b)$
can be written down as the push-forward from $W_{\tilde{I}}$ for some $\tilde{I}\subset I$
of the elements supported on it and  $\dim X +1 - |\tilde{I}|$ elements $\lambda_t$ 
restricted to $W_{\tilde{I}}$. Since $\lambda_t$ are defined over $X$,
their product over $W_{\tilde{I}}$ is the pullback of the product from $D_{\tilde{I}}$.
However, either $D_{\tilde{I}}$ is empty (in which case the claim is trivial)
or its dimension equals $\dim X - |\tilde{I}|$ (and is non-negative) 
and the product of at least $\dim D_{\tilde{I}}+1$ first Chern classes on it is zero.
\end{proof}

\section{$\mathrm K(2)$-Motivic Decomposition Types of quadrics}
\label{k2motives}

Proposition~\ref{prop:Kn-motive-q-small-dim_n} allows us to determine all possible $\K2$-Motivic Decomposition Types 
(see Section~\ref{sec:morava-mdt}) and indecomposable summands in $\K2$-motives of quadrics
in the case when $\dim_2(q) \le 4$. In Lemma~B.\ref{lm:K2_indec} below we show that in all other cases $\K2$-kernel motive of the quadric 
is indecomposable. We summarize these results in the following table for quadrics $Q$ of dimension at least 2.

To simplify notation, in the table $R_\alpha$ denotes the $\K2$-specialization of the corresponding Chow Rost motive.
We use the notation $d=\left\lfloor\frac{\dim Q}{2}\right\rfloor$, 
and like in Proposition~\ref{prop:Kn-motive-q-small-dim_n}, if $\dim_2(q)\leq4$,
we denote by $\alpha$ the class of $q\perp q'\in I^3(k)$ modulo $I^4(k)$, where $\dim(q')=\dim_2(q)$. 
Note that $\alpha$ can be equal to zero.

\begin{center}
\vspace{5pt}
\begin{tabular}{|c|c|c|}
\hline
$\phantom{\widetilde{\widetilde{M}}}$Kahn dimensions of $q\phantom{\widetilde{\widetilde{M}}}$&$\K{2}$-MDT of $Q$&Indecomposable summands of $\Mkertwo Q$\\
\hline
\hline
$\mathrm{dim}_2(q)=1$&
\begin{tikzpicture}
\filldraw [white] (-0.3,0) circle (1pt);
\filldraw [black] (0,0) circle (1pt);
\filldraw [black] (0.5,0) circle (1pt);
\filldraw [black] (1,0) circle (1pt);
\filldraw [white] (1.3,0) circle (1pt);
\end{tikzpicture}
&$\phantom{\widetilde{\widetilde{M}}}L_{\alpha}\sh{i},\ i=0,1,2\phantom{\widetilde{\widetilde{M}}}$\\
\hline
$\mathrm{dim}_2(q)=3$&
\begin{tikzpicture}
\filldraw [white] (-0.3,0) circle (1pt);
\filldraw [black] (0,0) circle (1pt);
\filldraw [black] (0.5,0) circle (1pt);
\filldraw [black] (1,0) circle (1pt);
\filldraw [white] (1.3,0) circle (1pt);
\draw (0.5,0) .. controls (0.75,0.1) .. (1,0);
\end{tikzpicture}
&$\phantom{\widetilde{\widetilde{M}}}L_{\alpha}\sh{d-1},\ R_{[C_0(q)]}\otimes L_{\alpha}\sh{d}\phantom{\widetilde{\widetilde{M}}}$\\
\hline
$\mathrm{dim}_2(q)>3$&
\begin{tikzpicture}
\filldraw [white] (-0.3,0) circle (1pt);
\filldraw [black] (0,0) circle (1pt);
\filldraw [black] (0.5,0) circle (1pt);
\filldraw [black] (1,0) circle (1pt);
\filldraw [white] (1.3,0) circle (1pt);
\draw (0.5,0) .. controls (0.75,0.1) .. (1,0);
\draw (0,0) .. controls (0.25,0.1) .. (0.5,0);
\end{tikzpicture}
&$\phantom{\widetilde{\widetilde{M}}}\Mkertwo Q\phantom{\widetilde{\widetilde{M}}}$\\
\hline
\hline
$\mathrm{dim}_2(q)=0$&
\begin{tikzpicture}
\filldraw [white] (0.7,0.3) circle (1pt);
\filldraw [black] (1,0) circle (1pt);
\filldraw [black] (1.5,0.2) circle (1pt);
\filldraw [black] (1.5,-0.2) circle (1pt);
\filldraw [black] (2,0) circle (1pt);
\filldraw [white] (2.3,-0.0) circle (1pt);
\end{tikzpicture}
&$L_{\alpha}\sh{i},\ i=0,1,2$\\
\hline
$\mathrm{dim}_2(q)=2$&
\begin{tikzpicture}
\filldraw [white] (0.7,0.3) circle (1pt);
\filldraw [black] (1,0) circle (1pt);
\filldraw [black] (1.5,0.2) circle (1pt);
\filldraw [black] (1.5,-0.2) circle (1pt);
\filldraw [black] (2,0) circle (1pt);
\filldraw [white] (2.3,-0.0) circle (1pt);
\draw (1.5,0.2) -- (1.5,-0.2);
\end{tikzpicture}
&$ L_{\alpha}\sh{i},\ i=d\pm1,\ R_{\mathrm{disc}(q)}\otimes L_{\alpha}\sh{d}$\\
\hline
$\mathrm{dim}_2(q)=4$, $\mathrm{dim}_1(q)=0$&
\begin{tikzpicture}
\filldraw [white] (0.7,0.3) circle (1pt);
\filldraw [black] (1,0) circle (1pt);
\filldraw [black] (1.5,0.2) circle (1pt);
\filldraw [black] (1.5,-0.2) circle (1pt);
\filldraw [black] (2,0) circle (1pt);
\filldraw [white] (2.3,-0.0) circle (1pt);
\draw (1,0) .. controls (1.25,0.15) ..(1.5,0.2);
\draw (2,0) .. controls (1.75,-0.15) ..(1.5,-0.2);
\end{tikzpicture}
&$R_{[C(q)]}\otimes L_{\alpha}\sh{i},\ i=d-1,\,d$\\
\hline
\begin{tabular}{c}
$\mathrm{dim}_2(q)=4$, $\mathrm{dim}_1(q)>0$\\ 
or $\mathrm{dim}_2(q)>4$
\end{tabular}
&
\begin{tikzpicture}
\filldraw [white] (0.7,0.5) circle (1pt);
\filldraw [black] (1,0) circle (1pt);
\filldraw [black] (1.5,0.2) circle (1pt);
\filldraw [black] (1.5,-0.2) circle (1pt);
\filldraw [black] (2,0) circle (1pt);
\filldraw [white] (2.3,-0.0) circle (1pt);
\draw (1.5,0.2) -- (1.5,-0.2);
\draw (1,0) .. controls (1.25,0.15) ..(1.5,0.2);
\draw (2,0) .. controls (1.75,-0.15) ..(1.5,-0.2);
\end{tikzpicture}
&$\Mkertwo Q$\\
\hline
\end{tabular}
\vspace{10pt}
\end{center}

\begin{LmB}
\label{lm:K2_indec}
        For a quadratic form $q$ of Kahn dimension $\dim_2(q)>4$, the $\K2$-kernel  motive of $Q$ is indecomposable. 
\end{LmB}

\begin{proof}
For $\dim(q)$ odd, if $\Mkertwo Q$ is decomposable, it has an invertible summand. This contradicts Conjecture~\ref{conj:inv_summand_quadrics} on invertible summands in motives of quadrics, which is known to be true for $n=2$ (see Corollary~\ref{cr:inv_summand_quadrics_small_n}).

For $\dim(q)$ even, it is enough to show that $\Mkertwo{Q'}$ is indecomposable for the $\K2$-kernel form $q'$ defined over $K/k$ by Proposition~\ref{prop:quad_MDT_stable}. Recall that $\dim(q')\geq4$ by Proposition~\ref{prop:dim_n_over_Kn-kernel-field}, since Kahn's Conjecture~\ref{conj:kahn_n} is proven for $n=2$. 

In the case $\dim_2(q')=4$,  we conclude by Proposition~\ref{prop:Kn-motive-q-small-dim_n} that $\Mkertwo{Q'}$ can be decomposable only if $q'\in I^2(K)$, in which case $C(q')$ has index $2$. Then by Merkurjev's index reduction formula~\cite{Mer} 
we conclude that $\mathrm{ind}\big(C(q)\big)=\mathrm{ind}\big(C(q_K)\big)$, i.e., $\dim_2(q)=4$ contradicting our assumption.

In other words, we may assume that $\dim_2(q')>4$, in particular, $\dim(q')=6$ or $8$,
in which case the splitting pattern of $q$ starts with  $(1,\ldots)$, resp., $(1,1,\ldots)$~\cite[Table~3]{Vish-quad}. 
It remains to use Vishik's excellent connections~\cite[Th.~1.3]{VishExc} for quadrics of dimensions $4$ and $2$ (resp., $6$, $4$ and $2$) 
to conclude that $\MCH{Q'}$ is indecomposable. Then $\Mkertwo{Q'}$ is indecomposable as well due to Theorem~\ref{th:prestableMDT}. 
\end{proof}

\end{document}